\documentclass[twoside]{amsbook}

\usepackage{enumerate, amsmath, amsfonts, amssymb, amsthm, mathtools, thmtools, wasysym, graphics, graphicx, xcolor, frcursive, xparse, comment, bm, stmaryrd, tensor, scalerel,rotating}

\usepackage{lscape}
\usepackage{scrextend}

\usepackage[english]{babel}
\usepackage{subfiles}

\usepackage{chngcntr}
\counterwithout{figure}{chapter}

\usepackage[all]{xy}
\usepackage[backref=page]{hyperref}
\newif\ifbackrefshowonlyfirst
\backrefshowonlyfirstfalse
\makeatletter
\let\BR@direct@old@hyper@natlinkstart\hyper@natlinkstart
\renewcommand*{\hyper@natlinkstart}{\phantomsection\BR@direct@old@hyper@natlinkstart}% note that the anchor will appear after any brackets at the start of the citation, but that's not really a big issue?
\let\BR@direct@oldBR@citex\BR@citex
\renewcommand*{\BR@citex}{\phantomsection\BR@direct@oldBR@citex}%

\long\def\hyper@page@BR@direct@ref#1#2#3{\hyperlink{#3}{#1}}

\ifx\backrefxxx\hyper@page@backref
    \let\backrefxxx\hyper@page@BR@direct@ref
    \ifbackrefshowonlyfirst
    \fi
\else
    \ifbackrefshowonlyfirst
    \fi
\fi

\RequirePackage{etoolbox}
\patchcmd{\Hy@backout}{Doc-Start}{\@currentHref}{}{\errmessage{I can't seem to patch backref}}
\makeatother
\usepackage{url, hypcap}
\hypersetup{colorlinks=true, citecolor=darkblue, linkcolor=darkblue}

\usepackage{tikz}
\usetikzlibrary{calc,through,backgrounds,shapes,matrix,angles}
\usetikzlibrary{positioning}
\usepackage{dynkin-diagrams}
\usepackage{braids} 

\numberwithin{equation}{chapter}
\counterwithout{footnote}{chapter}
\definecolor{darkblue}{rgb}{0.0,0,0.7} % darkblue color
\definecolor{darkred}{rgb}{0.7,0,0} % darkred color
\newcommand{\darkred}{\color{darkred}} % darkred command
\definecolor{lightgrey}{rgb}{0.7,0.7,0.7} % darkred color
\newcommand{\lightgrey}{\color{lightgrey}} % darkred command

\ExplSyntaxOn
\NewDocumentCommand{\varnewtheorem}{momo}
 {
  \IfValueTF{#4}
   {\newtheorem{#1}{#3}[#4]}
    {
     \IfValueTF{#2}
      {\newtheorem{#1}[#2]{#3}}
      {\newtheorem{#1}{#3}}
    }
  \newtheorem*{#1*}{#3}
  \newenvironment{#1+}
   {\pushQED{\qed}\begin{#1}}
   {\popQED\end{#1}}
  \newenvironment{#1*+}
   {\pushQED{\qed}\begin{#1*}}
   {\popQED\end{#1*}}
 }
\ExplSyntaxOff

\def\qedherecases{\par\vspace{-1.4\baselineskip}\qedhere}

\varnewtheorem{theorem}{Theorem}[section]
\varnewtheorem{theorems}[theorem]{Theorems}
\varnewtheorem{proposition}[theorem]{Proposition}
\varnewtheorem{corollary}[theorem]{Corollary}
\varnewtheorem{lemma}[theorem]{Lemma}
\varnewtheorem{observation}[theorem]{Observation}

\theoremstyle{definition}
\varnewtheorem{definition}[theorem]{Definition}
\varnewtheorem{example}[theorem]{Example}
\varnewtheorem{mywarning}[theorem]{\textbf{Warning}}
\varnewtheorem{conjecture}[theorem]{Conjecture}
\varnewtheorem{openproblem}[theorem]{Open Problem}
\varnewtheorem{remark}[theorem]{Remark}
\varnewtheorem{hope}[theorem]{Hope}

\usepackage[nameinlink]{cleveref}
\usepackage[all]{xy}
\usepackage[T1]{fontenc}

\crefname{theorem}{Thm.}{Thms.}
\Crefname{theorem}{Theorem}{Theorems}

\crefname{definition}{Def.}{Defs.}
\Crefname{definition}{Definition}{Definitions}

\crefname{corollary}{Cor.}{Cors.}
\Crefname{corollary}{Corollary}{Corollaries}

\crefname{proposition}{Prop.}{Props.}
\Crefname{proposition}{Proposition}{Propositions}

\crefname{conjecture}{Conj.}{Conjs.}
\Crefname{conjecture}{Conjecture}{Conjectures}

\crefname{observation}{Obser.}{Obsers.}
\Crefname{observation}{Observation}{Observations}

\Crefname{hope}{Hope}{Hopes}

\Crefname{mywarning}{Warning}{Warnings}

\newcommand\emm{m}
\newcommand\mhead{\texorpdfstring{$m$}{\protect\emm}}

\newcommand{\ind}{{\operatorname{ind}}}

\newcommand{\Ext}{{{\operatorname{Ext}}}}
\newcommand{\Hom}{{{\operatorname{Hom}}}}

\newcounter{nathancount}
\newcounter{christiancount}
\newcounter{hughcount}
\newcommand{\nathanside}[1]{
  \marginpar{
    {\small #1} \\%[-10pt]
    \hspace*{25pt}
    \includegraphics[height=.5in]{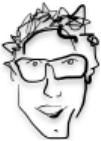}
    \hspace*{8pt}
    {\tiny\#\arabic{nathancount}}
  }
  \addtocounter{nathancount}{1}
}
\newcommand{\christianside}[1]{
  \marginpar{
    {\small #1} \\
    \hspace*{20pt}
    \includegraphics[height=.5in]{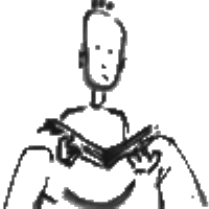}
    \hspace*{3pt}
    {\tiny\#\arabic{christiancount}}
  }
  \addtocounter{christiancount}{1}
}

\newcommand{\hughside}[1]{
   \marginpar{
    {\small #1} \\
    \hspace*{22pt}
    \ifthispageodd{
      \scalebox{-1}[1]{\includegraphics[height=.4in]{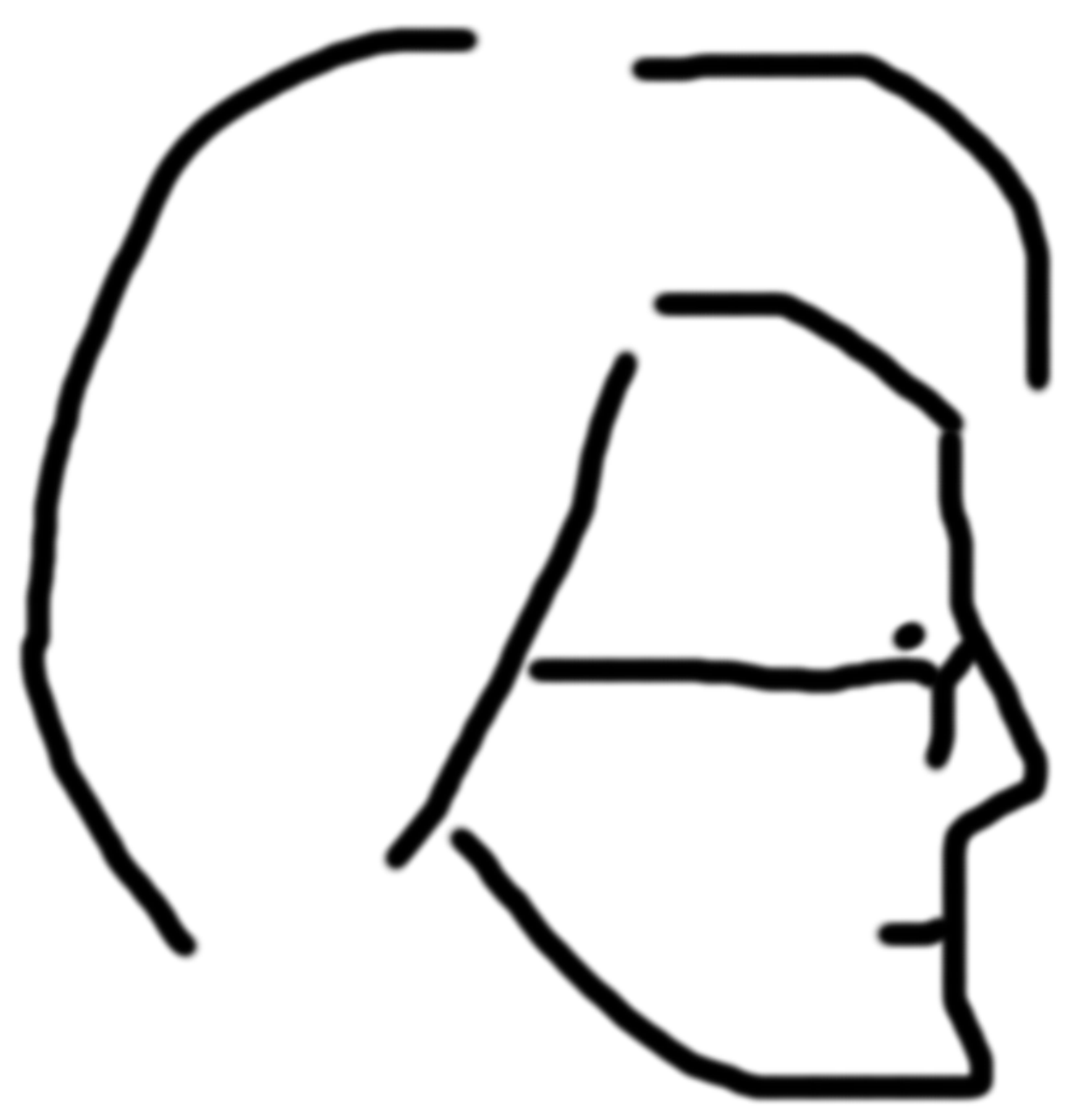}}
    }{
      \includegraphics[height=.4in]{hugh-avatar}
    }
    \hspace*{10pt}
    {\tiny\#\arabic{hughcount}}
  }
  \addtocounter{hughcount}{1}
}

\newcommand{\QQ}{\mathbb{Q}} % rationals
\newcommand{\C}{\mathbb{C}} % complexes
\newcommand{\R}{\mathbb{R}} % reals
\newcommand{\set}[2]{\left\{ #1 \; : \; #2 \right\}} % set notation
\newcommand{\bigset}[2]{\big\{ #1 \; : \; #2 \big\}} % big set notation
\newcommand{\id}{{1\!\!1}} % the identity
\newcommand{\one}{{e}} % the identity

\newcommand{\coveredref}{{\cov_{\downarrow}}}
\newcommand{\coveringref}{{\cov^{\uparrow}}}

\newcommand{\sq}[1]{{\sf #1}} % sequence of letters (squared symbols)
\newcommand{\Q}{\sq{Q}} % the word Q
\newcommand{\q}{\sq{q}} % the letter q
\newcommand{\s}{\sq{s}} % the letter s
\renewcommand{\r}{\sq{r}} % the letter r
\renewcommand{\t}{\sq{t}} % the letter t
\renewcommand{\u}{\sq{u}} % the word u

\newcommand{\wo}{{w_\circ}} % the longest element in W
\newcommand{\bwo}{\boldsymbol{w_\circ}} % the longest element in W
\NewDocumentCommand\wom{O{m}}{{w_\circ^{#1}}} % the longest element in W
\NewDocumentCommand\bwom{O{m}}{{\boldsymbol{w_\circ}^{#1}}} % the longest element in W
\NewDocumentCommand\wop{O{p}}{w_{#1}} % the longest element in W
\NewDocumentCommand\bwop{O{p}}{\boldsymbol{w}_{#1}} % the longest element in W
\NewDocumentCommand\cwop{O{p}}{\sq{w}_{#1}}

\newcommand{\refl}{\mathcal{R}}
\newcommand{\sref}{\mathcal{S}}
\newcommand{\arefl}{\bm{\mathcal{R}}}
\newcommand{\asref}{\bm{\mathcal{S}}}
\newcommand{\length}{\ell} % length
\newcommand{\lengthS}{\ell_\sref} % length
\newcommand{\lengthR}{\length_\refl} % reflection length
\newcommand{\leqref}{\leq_\refl} % reflection order
\newcommand{\geqref}{\geq_\refl} % reflection order
\newcommand{\subwordsS}{{\textsc{Sub}_\sref}}
\newcommand{\subwordsSB}{{\textsc{Sub}_{\asref}^{\Artingrp}}}
\newcommand{\subwordsSC}{{\textsc{Sub}_\sref^{\operatorname{col}}}}
\newcommand{\subwordsR}{{\textsc{Sub}_\refl}}
\newcommand{\reds}{{\mathrm{Red}_\refl}}
\newcommand{\flipGraph}[1][\Q,w,a]{\mathcal{G}(#1)} % flip graph

\newcommand{\FR}{{\operatorname{F\hspace*{-2pt}R}}} % Fomin-Reading map
\newcommand{\FRm}{{\FR}^{(m)}} % m-eralized Fomin-Reading map

\newcommand{\Root}[2]{\mathsf{r}_{#1}(#2)} % the root of #1 at position #2
\newcommand{\Roots}[1]{\mathsf{R}(#1)} % the root configuration of #1
\renewcommand{\c}{\sq{c}} % the word for the Coxeter element

\newcommand{\sw}[2]{\sq{#1}(\sq{#2})} % the #2-sorting word for #1
\newcommand{\cw}[1]{\sq{#1}\cwo[#1]} % the word for a cluster complex
\newcommand{\cwm}[1]{\sq{#1}\cwom[#1]} % the word for a generalized cluster complex
\NewDocumentCommand\cwo{O{c}}{\sw{\wo}{#1}} % the #1-sorting word for the longest element
\NewDocumentCommand\cwom{O{c}O{m}}{{\sq{w}_\circ^{#2}}{(\sq{#1})}} % the #1-sorting word for the longest element

\DeclareMathOperator{\inv}{inv} % inversion set
\DeclareMathOperator{\invs}{\mathsf{inv}} % inversion sequence
\usepackage[colorinlistoftodos]{todonotes}

\newcommand{\del}{{\operatorname{del}}}

\newcommand{\ie}{\textit{i.e.}} % id est
\newcommand{\defn}[1]{\emph{\darkred #1}} % emphasis of a definition
\renewcommand{\paragraph}[1]{\bigskip\noindent\textbf{#1.}} % paragraph

\def\Phipm{\Phi_{\geq -1}} % almost positive roots
\newcommand{\mPhipm}[1][m]{{\Phi^{(#1)}_{\geq -1}}} % colored almost positive roots

\newcommand{\alt}[1]{{[#1]}}

\newcommand{\toap}{{\phi}}
\NewDocumentCommand\BA{O{n}}{\boldsymbol{\mathfrak{S}}^+_{#1}}%{{{\sf Weak}^{(#1)}}}
\NewDocumentCommand\WA{O{n}}{\mathfrak{S}_{#1}}%{{{\sf Weak}^{(#1)}}}
\newcommand{\Cat}{{{\sf Cat}}}
\NewDocumentCommand\Catm{O{m}}{{{\sf Cat}^{(#1)}}}
\NewDocumentCommand\Catplusm{O{m}}{{{\sf Cat}^{(#1)}_+}}
\NewDocumentCommand\Catrat{O{p}}{{{\sf Cat}^{[#1]}}}

\newcommand{\sninv}{\overline{s}}
\newcommand{\sinv}{\overline{\s}}

\newcommand{\coxr}{{\sinv\c\s}}%{\s^{-1}\c \s}
\newcommand{\coxs}{{\sinv\c}}%{{}_\shortdownarrow \c}}%{\s^{-1}\c}
\newcommand{\coxrn}{{\sninv c s}}%{scs}
\newcommand{\coxsn}{{\sninv c}}%{{{}_\shortdownarrow c}}%{sc}

\newcommand{\coxsf}{{\c\sinv}}%{{}_\shortdownarrow \c}}%{\s^{-1}\c}
\newcommand{\coxrnf}{{sc\sninv}}%{scs}
\newcommand{\coxsnf}{{c\sninv}}%{{{}_\shortdownarrow c}}%{sc}

\DeclareMathOperator{\lcm}{lcm}
\DeclareMathOperator{\rev}{rev}

\DeclareMathOperator{\garside}{garside}

\newcommand\restr[2]{{% we make the whole thing an ordinary symbol
  \left.\kern-\nulldelimiterspace % automatically resize the bar with \right
  #1 % the function
  \vphantom{\big|} % pretend it's a little taller at normal size
  \right|^{#2} % this is the delimiter
  }}
\newcommand\restri[2]{{% we make the whole thing an ordinary symbol
  \left.\kern-\nulldelimiterspace % automatically resize the bar with \right
  #1 % the function
  \vphantom{\big|} % pretend it's a little taller at normal size
  \right|_{#2} % this is the delimiter
  }}

\newcommand{\Artinmon}{\boldsymbol{B}^+}
\newcommand{\Artingrp}{\boldsymbol{B}}
\newcommand{\Braidgrp}{\boldsymbol{\mathfrak{S}}}
\newcommand{\Weak}{{{\sf Weak}}}
\NewDocumentCommand\Weakm{O{m}}{{{\sf Weak}^{(#1)}}}
\NewDocumentCommand\Wm{O{m}}{{W}^{(#1)}}
\newcommand{\Abs}{{{\sf Abs}}}
\newcommand{\Assoc}{{{\sf Asso}}}

\newcommand{\NablaAssoc}{{\sf Asso}_\nabla}
\NewDocumentCommand\Assocm{O{m}}{{{\sf Asso}^{(#1)}}}
\NewDocumentCommand\Assocmpos{O{m}}{{{\sf Asso}^{(#1)}_+}}
\NewDocumentCommand\Assocrat{O{p}}{{{\sf Asso}^{[#1]}}}
\NewDocumentCommand\DeltaAssocm{O{m}}{{{\sf Asso}^{(#1)}_\Delta}}

\newcommand{\Shard}{{{\sf Shard}}}
\newcommand{\leqsh}{{\leq_{\operatorname{Sh}}}}
\newcommand{\geqsh}{{\geq_{\operatorname{Sh}}}}

\newcommand{\ncs}{{\underline w}}
\NewDocumentCommand\Krew{O{c}}{{\sf Krew}_{#1}}
\NewDocumentCommand\Krewpos{O{c}}{{\sf Krew}^{+}_{#1}}

\newcommand{\Camb}{{{\sf Camb}}}
\newcommand{\Cambsort}{{{\sf Camb}_\Sort}}
\newcommand{\Cambnc}{{{\sf Camb}_\NC}}

\NewDocumentCommand\Cambm{O{m}}{{{\sf Camb}^{(#1)}}}
\NewDocumentCommand\Cambsortm{O{m}}{{{\sf Camb}_\Sort^{(#1)}}}
\NewDocumentCommand\Cambncm{O{m}}{{{\sf Camb}_\NC^{(#1)}}}
\NewDocumentCommand\GCambncm{O{m}}{{{\mathcal{G}\sf Camb}_\NC^{(#1)}}}
\NewDocumentCommand\Cambassocm{O{m}}{{{\sf Camb}_\Assoc^{(#1)}}}
\newcommand{\FNC}{{\delta_\ncs}}
\NewDocumentCommand\deltaNCm{O{m}}{{{\sf NC}_{\delta}^{(#1)}}}
\newcommand{\DeltaNC}{{{\sf NC}_{\Delta}}}
\NewDocumentCommand\DeltaNCm{O{m}}{{{\sf NC}_{\Delta}^{(#1)}}}

\newcommand{\Rev}{{\operatorname{Fund}}}
\newcommand{\DesSet}{{{\sf des}}}
\newcommand{\AscSet}{{{\sf asc}}}

\newcommand{\Shift}{{{\sf Shift}}}
\newcommand{\PhiP}{{\Phi^+}}
\NewDocumentCommand\PhiPJ{O{J}}{{\Phi^+_{#1}}}
\newcommand{\NC}{{{\sf NC}}}
\NewDocumentCommand\NCm{O{m}}{{{\sf NC}^{(#1)}}}
\NewDocumentCommand\NCmpos{O{m}}{{{\sf NC}^{(#1)}_+}}
\NewDocumentCommand\NCrat{O{p}}{{{\sf NC}^{[#1]}}}
\newcommand{\NCL}{{{\sf NCL}}}
\newcommand{\flip}{{{\sf Flip}}}
\newcommand{\LW}{{\mathcal{L}(W)}} %= intersection lattice
\newcommand{\PW}{{\mathcal{P}(W)}} %= lattice of parabolic subgroups
\newcommand{\Park}{{{\sf Park}}} %= parking functions
\newcommand{\fixd}{\operatorname{Fix}}
\newcommand{\movd}{\operatorname{Mov}}

\newcommand{\pid}{{\pi^c_\downarrow}}
\newcommand{\piu}{{\pi_c^\uparrow}}
\newcommand{\Sort}{{{\sf Sort}}}
\NewDocumentCommand\Sortm{O{m}}{{{\sf Sort}^{(#1)}}}
\NewDocumentCommand\Sortinf{O{\infty}}{{{\sf Sort}^{(#1)}}}

\NewDocumentCommand\Sortmpos{O{m}}{{{\sf Sort}^{(#1)}_+}}
\NewDocumentCommand\Sortrat{O{p}}{{{\sf Sort}^{[#1]}}}
\NewDocumentCommand\Sortma{O{m}}{{{\sf Sort}_{{\sf fact}}^{(#1)}}}

\NewDocumentCommand\deltaSortm{O{m}}{{{\sf Sort}_{{\sf shard}}^{(#1)}}}
\newcommand{\cov}{{{\sf cov}}}

\newcommand{\order}{\operatorname{ord}}
\newcommand{\supp}{\operatorname{supp}}
\newcommand{\bij}{{\ \longrightarrow\ }}
\newcommand{\lrbij}{{\ \longleftrightarrow\ }}
\newcommand\XLeftrightarrow[2]{\xLeftrightarrow{\makebox[#1]{#2}}}
\newcommand\Xleftrightarrow[2]{\xleftrightarrow{\makebox[#1]{#2}}}
\newcommand{\cambbij}{{\ \Xleftrightarrow{12pt}{\tiny $c$}\ }}
\newcommand{\canbij}{{\ \Xleftrightarrow{12pt}{\tiny $\sim$}\ }}

\renewcommand{\S}{{\lightgrey\s}}
\newcommand{\T}{{\lightgrey\t}}
\newcommand{\U}{{\lightgrey\u}}

\usepackage{ifthen}
\newcommand\dyckpath[2]{
  \coordinate (last) at ($#1-(0.6,0.3)$);
  \draw[very thin, lightgrey] (last) -- ($(last)+(1.2,0.6)$);
  \filldraw (last) circle (1pt);
  \foreach \dir in {#2}{
    \ifnum\dir=0
      \coordinate (step) at (0,0.2);
      \coordinate (ext) at (0,0.013);
    \else
      \coordinate (step) at (0.2,0);
      \coordinate (ext) at (0.013,0);
    \fi
    \draw[line width=1pt] (last) -- ($(last)+(step)+(ext)$);
    \coordinate (last) at ($(last)+(step)$);
    \filldraw (last) circle (1pt);
  };
}

\newcommand{\rootposet}[6]{ % #1 = center
  \node[draw,rectangle,minimum width=2*#4*#2cm+#4cm,minimum height=#4*#2cm+#4cm,opacity=0.5] #5 at (#1) {};

  \coordinate (start) at ($(#1)-(#2*#4,#2*#4/2)+(#4,#4/2)$);
  \foreach \row in {#3}{
    \coordinate (row) at (start);
    \foreach \entry in \row{
      \ifnum\entry=0
        \node[draw,circle,inner sep =  7*#4pt] at (row) {};
      \fi
      \ifnum\entry=1
        \node[fill,circle,inner sep =  7*#4pt] at (row) {};
      \fi
      \ifnum\entry=2
        \node[fill,circle,inner sep =  7*#4pt] at (row) {};
        \node[fill,circle,inner sep = 12*#4pt,fill opacity=0.5] at (row) {};
      \fi
      \ifnum\entry=3
        \node[fill,circle,inner sep =  7*#4pt] at (row) {};
        \node[draw,circle,inner sep = 12*#4pt,opacity=0.5] at (row) {};
      \fi
      \coordinate (row) at ($(row)+(#4*2,#4*0)$);
    }
    \coordinate (start) at ($(start)+(#4,#4)$);
  }
}

\newcommand{\rootposetfill}[6]{ % #1 = center
  \coordinate (start) at ($(#1)-(#2*#4,#2*#4/2)+(#4,#4/2)$);
  \foreach \row in {#3}{
    \coordinate (row) at (start);
    \foreach \entry in \row{
        \node at (row) {\entry};
      \coordinate (row) at ($(row)+(#4*2,#4*0)$);
    }
    \coordinate (start) at ($(start)+(#4,#4)$);
  }
}

\NewDocumentCommand\Skipset{O{c}}{{\mathcal{C}_{#1}}} % the skip set
\NewDocumentCommand\Lastset{O{c}}{{\mathcal{C}^*_{#1}}} % the last occurrence
\NewDocumentCommand\Lr{O{c}}{{{\operatorname{Lr}}_{#1}}} % the last occurrence
\NewDocumentCommand\Lrplus{O{c}}{{{\operatorname{Lr}}^+_{#1}}} % the last occurrence

\NewDocumentCommand\taum{O{s}}{{\tau^{(m)}_{#1}}}
\NewDocumentCommand\num{O{s}}{{\nu^{(m)}_{#1}}}

\renewcommand{\th}{\textsuperscript{th}}

\newcommand{\polygon}[6]{ % #1 = center
  \node[circle,minimum size=#5cm] (#2) at  #1 {};

  \foreach \t in {1,...,#3} {
    \coordinate (#2\t) at ($#1+(90-\t*360/#3:#4)$);
  }
  \draw[thin,black,fill=white,opacity=0.3,densely dashed] #1 circle (#4);
  \setcounter{intege}{1}
  \pgfmathsetcounter{intege}{1}
  \foreach \object in {#6}{
    \ifthenelse{\not\equal{\object}{}}{
      \filldraw[black] ($#1+($(90-\theintege*360/#3:#4)$)$) circle(2pt);
    }{
    }
    \node[inner sep=0pt] at ($#1+($1.25*(90-\theintege*360/#3:#4)$)$) {$\scriptstyle\object$};
    \pgfmathsetcounter{intege}{\theintege+1}
    \setcounter{intege}{\theintege}
  }
}

\newcounter{intege}

\NewDocumentCommand\DeltaAssocinf{O{\infty}}{{{\sf Asso}^{(#1)}_\Delta}}
\NewDocumentCommand\DeltaAssocmp{O{m+1}}{{{\sf Asso}^{(#1)}_\Delta}}

\NewDocumentCommand\Cambncinf{O{\infty}}{{{\sf Camb}_\NC^{(#1)}}}
\NewDocumentCommand\DeltaNCinf{O{\infty}}{{{\sf NC}_{\Delta}^{(#1)}}}
\NewDocumentCommand\Cambsortinf{O{\infty}}{{{\sf Camb}_\Sort^{(#1)}}}

\NewDocumentCommand\DeltaNCinfpos{O{\infty}}{{{\sf NC}_{+}^{(#1)}}}

\newcommand{\DeltaNCmpos}{{{\sf NC}_{+}^{(m)}}}
\newcommand{\Vs}{\V^\sharp}

\newcommand{\dimp}{\underline{\operatorname{dim}}^\sharp}
\newcommand{\storf}{\operatorname{CoAisle}^{(\infty)}}
\newcommand{\storfp}{\operatorname{CoAisle}^{\sharp}}
\newcommand{\storfm}{\operatorname{CoAisle}^{(m)}}
\newcommand{\sta}{\mathcal V_{\textup{st}}}
\renewcommand{\mod}{\operatorname{mod}}
\newcommand{\Silt}{{\operatorname{Silt}}}
\newcommand{\Siltp}{\operatorname{Silt}^{(\infty)}}
\newcommand{\Homp}{\operatorname{Hom}_{\leq 0}^{(\infty)}}
\newcommand{\Siltm}{\operatorname{Silt}^{(m)}}
\newcommand{\Homm}{\operatorname{Hom}_{\leq 0}^{(m)}}

\newcommand{\Siltpos}{\operatorname{Silt}^{(\infty)}_+}
\newcommand{\Siltposm}{\operatorname{Silt}^{(m)}_+}
\NewDocumentCommand\NablaAssocinfpos{O{\infty}}{{{\sf Asso}^{(#1)}_{+}}}
\newcommand{\NablaAssocmpos}{{{\sf Asso}^{(m)}_{+}}}
\newcommand{\Hompos}{{\sf Hom}_{\leq 0,+}^{(\infty)}}
\newcommand{\Homposm}{{ \sf Hom}_{\leq 0,+}^{(m)}}

\NewDocumentCommand\NablaAssocinf{O{\infty}}{{{\sf Asso}^{(#1)}_\nabla}}
\NewDocumentCommand\NablaAssocm{O{m}}{{{\sf Asso}^{(#1)}_\nabla}}
\NewDocumentCommand\NablaAssocmp{O{m+1}}{{{\sf Asso}^{(#1)}_\nabla}}

\newcommand{\udim}{\underline{\operatorname{dim}}}
\renewcommand{\Rev}{\operatorname{Fund}}
\newcommand{\Revp}{\operatorname{Fund}} %I am no longer distinguishing this from
\newcommand{\bw}{\boldsymbol{w}}
\newcommand{\ba}{\boldsymbol{a}}
\newcommand{\bb}{\boldsymbol{b}}
\newcommand{\bu}{\boldsymbol{u}}
\newcommand{\bv}{\boldsymbol{v}}
\newcommand{\bone}{\boldsymbol{e}}

\newcommand{\cwoinf}{{\sq w}_\circ^{\infty}({\sq c})}
\newcommand{\V}{\mathcal V}

\newcommand{\TT}{\mathcal T}

\newcommand{\bs}{\boldsymbol{s}}
\newcommand{\br}{\boldsymbol{r}}
\newcommand{\bt}{\boldsymbol{t}}
\newcommand{\bc}{\boldsymbol{c}}
\newcommand{\BJ}{\boldsymbol{J}}

\newcommand{\Fm}{\mathcal F^{(m)}}
\newcommand{\Cm}{{\mathcal C^{(m)}}}
\newcommand{\fmp}{F^{(m)}_+}
\newcommand{\Cmp}{\mathcal C^{(m)}_+}
\newcommand{\Fmp}{\mathcal F^{(m)}_+}

\newcommand{\asc}{\operatorname{asc}}
\newcommand{\des}{\operatorname{des}}
\NewDocumentCommand\parallelcclassic{O{c}}{\parallel_{#1}}

\newcommand{\eqdef}{:=}

\NewDocumentCommand\Wres{O{s}}{W_{\langle #1\rangle}}

\NewDocumentCommand\Cambupflip{O{r_a}}{\flip^\uparrow_{#1}}
\NewDocumentCommand\Cambdownflip{O{r_a}}{\flip^\downarrow_{#1}}

\newcommand{\link}{\operatorname{lk}}

\newcommand{\shard}{\Sigma}
\newcommand{\shardbij}{\mathrm{shard}}

\newcommand\scalemath[2]{\scalebox{#1}{\mbox{\ensuremath{\displaystyle #2}}}}

\usepackage{environ}
\title%[Fu\ss-Cataland I]%Cataland I]%The noncrossing Fu\ss-Catalan universe]
  {Cataland: Why the Fu\ss?}%against Christian’s better judgement
\author[C.~Stump]{Christian Stump$^*$}
\address[C.~Stump]{Ruhr-Universit\"at Bochum, Germany}
\email{christian.stump@rub.de}
\thanks{$^*$Supported by DFG grants STU 563/2 ``Coxeter-Catalan combinatorics'' and STU 563/4-1 ``Noncrossing phenomena in Algebra and Geometry''. }

\author[H.~Thomas]{Hugh Thomas$^\dagger$}
\address[H.~Thomas]{Universit\'e du Qu\'ebec \`a Montr\'eal, Canada}
\email{hugh.ross.thomas@gmail.com}
\thanks{$^\dagger$Partially supported by an NSERC Discovery Grant and the
Canada Research Chairs program.}

\author[N.~Williams]{Nathan Williams}
\address[N.~Williams]{University of Texas at Dallas, USA}
\email{nathan.f.williams@gmail.com}

\date{\today}

\begin{document}

\maketitle

\cleardoublepage
\thispagestyle{empty}
\vspace*{13.5pc}
\begin{center}
   To Anke, Magali and Maria.
\end{center}
\setcounter{tocdepth}{1}
\tableofcontents

%%%%%%%%%%%%%%%%%%%%%%%%%%%%%%%%%%%%%%%%%%%%%%%%%%%%%%%%%%%%%%%%%%%%%%%%%%%%%%%%%%%%%
\chapter*{Abstract}
%%%%%%%%%%%%%%%%%%%%%%%%%%%%%%%%%%%%%%%%%%%%%%%%%%%%%%%%%%%%%%%%%%%%%%%%%%%%%%%%%%%%%

The three main objects in noncrossing Catalan combinatorics associated to a finite Coxeter system are noncrossing partitions, clusters, and sortable elements.
The first two of these have known Fu\ss-Catalan generalizations.
We provide new viewpoints for both and introduce the missing generalization of sortable elements by lifting the theory from the Coxeter system to the associated positive Artin monoid.  We show how this new perspective ties together all three generalizations, providing a uniform framework for noncrossing Fu\ss-Catalan combinatorics.  Having developed the combinatorial theory, we provide an interpretation of our generalizations in the language of the representation theory of hereditary Artin algebras.

%We then give a uniform treatment of positive analogues and propose conjectural rational constructions in the classical types.  

%%%%%%%%%%%%%%%%%%%%%%%%%%%%%%%%%%%%%%%%%%%%%%%%%%%%%%%%%%%%%%%%%%%%%%%%%%%%%%%%%%%%%
\chapter*{Acknowledgements}
%%%%%%%%%%%%%%%%%%%%%%%%%%%%%%%%%%%%%%%%%%%%%%%%%%%%%%%%%%%%%%%%%%%%%%%%%%%%%%%%%%%%%

This project began in December 2012 at the workshop ``Rational Catalan combinatorics'' at the  American Institute of Mathematics.
We thank AIM for financial support, and we thank the organizers D.~Armstrong, S.~Griffeth, V.~Reiner, and M.~Vazirani for the invitation to participate.

We made substantial progress in June 2014 at the workshop ``Non-crossing partitions in representation theory'' at Bielefeld University.   We thank the funding agencies for financial support, and we thank the organizers B.~Baumeister, A.~Hubery, and H.~Krause.

Part of the research for this monograph was carried out while we were staying at the Mathematical Research Institute Oberwolfach supported by the ``Research in Pairs'' program.
Finalizing this monograph was only possible because of the hospitality of the institute.

\medskip

We thank Drew Armstrong, Nathan Reading, and Vic Reiner for many helpful discussions.

%%%%%%%%%%%%%%%%%%%%%%%%%%%%%%%%%%%%%%%%%%%%%%%%%%%%%%%%%%%%%%%%%%%%%%%%%%%%%%%%%%%%%
\chapter{Introduction}
\label{sec:introduction}
%%%%%%%%%%%%%%%%%%%%%%%%%%%%%%%%%%%%%%%%%%%%%%%%%%%%%%%%%%%%%%%%%%%%%%%%%%%%%%%%%%%%%

\christianside{There are several \emph{alternative titles} in the source if you are interested. And maybe even other interesting stuff.}

\section{A brief historical overview}
This monograph takes the perspective that there are three distinct families of noncrossing Catalan objects---noncrossing partitions, triangulations, and stack-sortable permutations.\nathanside{Some people might argue there are a few other Catalan objects.}
The families are counted by the \defn{Catalan number}
\begin{equation*}
  \Cat_n\eqdef \frac{1}{n+1}\binom{2n}{n}.
\end{equation*}
Historically, these three families arose as follows:%we follow our chapter order, rather than giving a chronological account.

\smallskip

%\begin{enumerate}[$\circ$]
%	\item
	 In 1972, G.~Kreweras introduced and studied \emph{noncrossing set partitions} of the set $\{1,\ldots,n\}$~\cite{Kre1972}.  Besides giving many refined enumerations, he also proved that the noncrossing partitions form a lattice under refinement.  This curious property turns out to be intimately connected to the $K(\pi,1)$ problem for the braid group.  We refer to~\cite[Chapter 1]{Arm2006} and \cite{bessis2015finite} for details.
%	\item

\smallskip

	 In 1751, L.~Euler guessed the enumeration of \emph{triangulations of a convex $(n+2)$-gon}; J.~Segner discovered the standard combinatorial recurrence in~\cite{de1758enumeratio}, from which L.~Euler was able to prove the enumeration.  The proof appeared in 1761, and E.~C.~Catalan was born in 1814---some 50 years later.  We refer to~\cite[Appendix B, History of Catalan Numbers]{stanley2015catalan} for a comprehensive historical treatment.  % E.~Catalan studied the classical ballot problem in the guise  %For example, the five triangulations of a pentagon are drawn below.
% 	 \christianside{Is it Coxeter numbers or Segner numbers?}
%	\item 
\smallskip

	In 1968, D.~Knuth introduced the notion of \emph{stack-sortability} as those permutations in $\WA$ that can be sorted in a single pass through a stack.  In~\cite[Exercise 2.2.1.5]{Knu1973}, he gives the exercise of showing that such permutations are characterized as~$231$-avoiding.
%\end{enumerate}

\medskip

This monograph is the continuation of a program to simultaneously generalizing these three classical families in two orthogonal directions: the Fu\ss-Catalan direction, and the Coxeter-theoretic direction.

\paragraph{Fu\ss-Catalan generalizations}
The first direction is the combinatorial problem of finding generalizations of Catalan families that incorporate an additional nonnegative integral parameter~$m$.  We refer to such a generalization as an \defn{$m$-eralization}.
Such $m$-eralizations have been established for the first two families, but  we are not aware of any previous $m$-eralization of the stack-sortable elements:

\smallskip

For noncrossing partitions, G.~Kreweras considered chains in the noncrossing partition lattice~\cite[Section 6]{Kre1972}.  In 1980, P.~Edelman put a partial order on these chains in~\cite{Ede1980}, and interpreted them as those noncrossing partitions on $\{1,\ldots,mn\}$ with block sizes divisible by~$m$.

\smallskip

In 1791---as detailed in~\cite[Section 3.5]{Arm2006}---the $m$-eralization of triangulations goes back to N.~Fu\ss\, who counted the number of dissections of a convex $(mn{+}2)$-gon into $(m{+}2)$-gons~\cite{fuss1791solutio}.

\bigskip

%\end{enumerate}
 These two families are counted by the \defn{Fu\ss-Catalan number}\footnote{The Fu\ss-Catalan numbers---and this monograph---are named for N.~Fu\ss\ and E.~Catalan.}
 \begin{equation*}
  \Cat_n^{(m)} \eqdef \frac{1}{mn+1}\binom{(m+1)n}{n}.
%   \label{eq:class_fuss_cat_num}
\end{equation*}
%We are not aware of any generalization of the third flavor that isn't .

\paragraph{Coxeter-theoretic generalizations}
This second direction---as championed by R.~Si\-mion and V.~Reiner---is concerned with a generalization of Catalan objects to finite Coxeter groups.
All three classical Catalan objects have known Coxeter-theoretic generalizations:

\smallskip
%\begin{enumerate}[$\circ$]
%\item
 As is often the case, the Coxeter-theoretic generalizations of noncrossing partitions began with the study of symmetric objects.  In unpublished work from 1993 \cite{montenegro1993fixed}, C.~Montenegro counted centrally symmetric noncrossing partitions.\nathanside{C is for Carlos?}  In 1997, V.~Reiner associated these to the root system of type~$B$ \cite{Rei1997} and considered more refined enumerative properties.  He posed the problem of natural generalizations to all reflection groups.  The answer came from the theory of Artin groups.

In 1998, J.~Birman, K.~H.~Ko, and S.~J.~Lee gave a remarkable new presentation of the braid group~\cite{BKL1998}.  In the early 2000s---inspired by this and the $K(\pi,1)$ problem---T.~Brady and C.~Watt~\cite{brady2001partial,BW2002}, and D.~Bessis~\cite{Bes2003} gave a uniform definition of noncrossing partition lattices for all finite Coxeter groups.  This resolved V.~Reiner's question.%\nathanside{Awkward.}  %As a consequence, this resolving V.~Reiner's problem.
%connected this new presentation with the lattice of noncrossing partitions,

\smallskip

%\item
 Coxeter-theoretic generalizations of triangulations originated in the theory of polytopes.  Triangulations of an $(n{+}2)$-gon naturally index the vertices of an $n$-dimensional polytope called the \emph{associahedron}, famously studied by D.~Tamari after his 1951 thesis (as a poset, which he later showed was a lattice) and in J.~Stasheff's 1961 thesis (as a polytope).  The centrally symmetric triangulations of a $(2n)$-gon naturally index the vertices of a polytope J.~Stasheff called the \emph{cyclohedron}~\cite{stasheff1997operads}, which first appeared in 1994 work of R.~Bott and C.~Taubes~\cite{bott1994self}.  R.~Simion independently asked and solved the problem of finding a type $B$ analogue of the associahedron~\cite{simion2003type}.  The construction for general root systems came from the theory of cluster algebras.

As a consequence of their classification of finite-type cluster algebras~\cite{FZ2003}, S.~Fomin and A.~Zelevinsky gave uniform definitions of associahedra for all Weyl groups.  These were realized as polytopes by many researchers in many ways, as detailed in~\cite{HPS}.  We refer to the entire collection \cite{muller2012associahedra} for further background.

\smallskip

%\item
 The Coxeter-theoretic generalization of stack-sortable elements is intimately connected to polytopes, but also to lattice theory.  It is well-known that the associahedron can be constructed from the permutahedron by giving a map from permutations to triangulations~\cite{BW1997,tonks1997relating}.  Answering a question of R.~Simion from~\cite{simion2003type}, V.~Reiner constructed the cyclohedron from the permutahedron of type $B$ using a theory of equivariant fiber polytopes.  In particular, this construction defined a map from signed permutations to centrally symmetric triangulations~\cite{reiner2002equivariant}.

In 2005, N.~Reading uniformly generalized such maps to all finite Coxeter groups using the notion of \emph{Coxeter-sortability}.  He used the weak order to define a generalization of the Tamari lattices he called \emph{Cambrian lattices}, giving a combinatorial model for the cluster complex in the corresponding Coxeter group.
%\end{enumerate}

\medskip

Let $W$ be a finite Coxeter group.  Although no uniform proof is currently known, these Coxeter-Catalan families are counted by the \defn{Catalan number of type $W$},
\begin{equation*}
  \Cat(W) \eqdef \prod_{i=1}^n \frac{h+d_i}{d_i},
\label{eq:cat_num}
\end{equation*}\nathanside{'cause otherwise, calling them Catalan numbers wouldn't make sense.}
where $d_1\leq d_2 \leq \ldots \leq d_n$ are the degrees of the fundamental invariants of~$W$ and $h \eqdef d_n$ is the Coxeter number.
For the symmetric group $\WA$, this definition recovers the classical Catalan numbers.% = \order(c)
% \eqref{eq:catnum} recovers \eqref{eq:class_cat_num}.  
%

\medskip

\paragraph{Fu\ss-Catalan and Coxeter-theoretic generalizations together} 
%To avoid confusion, we refer to the case $m=1$ as the \defn{classical situation}\footnote{This is itself a generalization of the---then classical---notion of noncrossing set partitions.}, and
The Coxeter-theoretic generalizations of noncrossing partitions and triangulations have both previously been $m$-eralized: in his 2006 thesis~\cite{Arm2006}, D.~Armstrong defined \emph{$m$-eralized noncrossing partitions} by applying P.~Edelman's construction to the uniformly-defined noncrossing partition lattices for finite Coxeter groups; and in 2005, S.~Fomin and N.~Reading found an \emph{$m$-eralized cluster complex} for bipartite Coxeter elements~\cite{FR2005}.
These $m$-eralized objects are counted by the \defn{Fu{\ss}-Catalan numbers of type~$W$},
 % (See \Cref{sec:coxeter_elements} for the definition of bipartite Coxeter elements.)

%, we have the following   
%The noncrossing partition lattices and the cluster complex have been $m$-eralized.
%More precisely:

\begin{comment}
These noncrossing Catalan objects have two different lattice structures, each of which has an associated cyclic group action:
\begin{enumerate}[$\circ$]
  \item The $c$-noncrossing partition lattice $\NCL(W,c)$ and its anti-automorphism, the \emph{Kreweras complement} $\Krew$, of order $2h$.
  \item The $c$-Cambrian lattice $\Camb(W,c)$ and its underlying graph automorphism, the \emph{Cambrian rotation} $\Camb_c$, of order $h+2$.
\end{enumerate}
\end{comment}

%\medskip

%\begin{enumerate}[$\circ$]
%  \item

%\medskip
 
%\end{enumerate}

\begin{equation*}
  \Catm(W) \eqdef \prod_{i=1}^n \frac{m h+d_i}{d_i}.
\label{eq:fuss_cat_num}
\end{equation*}
%The important problem of finding a uniform proof of this formula is still open.
%\christian{this is a definition, not a formula}
%We refer to~\cite[Section 3.5 and Theorem~3.5.3]{Arm2006} for the history of the proof of this counting formula, and we remark that t

\medskip

%comparing the state of the art of the other noncrossing families and their $m$-eralizations shows that
%of noncrossing partitions and the noncrossing partition lattice~\cite{Arm2006}
While both $m$-eralizations have been thoroughly studied, there are several missing pieces.  Most importantly, the following structures have not been constructed previously:
\begin{itemize}
  \item An $m$-eralized cluster complex for general Coxeter elements,
  \item an $m$-eralization of sortable elements,
  \item an $m$-eralized Cambrian lattice, and
  \item bijections relating $m$-eralized noncrossing partitions to the above.
\end{itemize}
%Fix a finite Coxeter system~$(W,\sref)$ and a (standard) Coxeter element~$c \in W$---that is, a product of the simple generators~$\sref$ in any order.

Since the weak order on a finite Coxeter group~$W$ is so integral to the $m=1$ versions of the above, we should start our search with one more item:
\begin{itemize}
  \item an $m$-eralization of the weak order.
\end{itemize}
\christianside{This is the end of the historic treatment. It was a little quick, maybe too quick...}

\section{Summary of results}

This monograph introduces these missing $m$-eralizations by passing from the Coxeter group to its {\bf positive Artin monoid}.  In the remainder of this introduction, we sketch our main results. \nathanside{TL;DR: we replaced Coxeter groups with the positive Artin monoid and did some combinatorics.}

%As detailed in~\cite[Section 3.5}{armstrong}, t
%\medskip

%There are three \emph{noncrossing Catalan objects} associated to this data, namely
%\begin{itemize}
%    \item the $c$-noncrossing partitions $\NC(W,c)$,
%    \item the $c$-cluster complex $\Assoc(W,c)$, and
%    \item the $c$-sortable elements $\Sort(W,c)$.
%\end{itemize}
%The term ``noncrossing'' for historical reasons---noncrossing partitions generalize \emph{noncrossing set partitions} to all finite Coxeter systems and all Coxeter elements.
%We also attach this term to the other families, since they are in uniform bijection with the noncrossing partitions, and to distinguish them from the \emph{nonnesting Catalan objects}, which we will not consider in this monograph, but can be found in~\cite[Section~5.1]{Arm2006} and the references therein.
%We denote the cluster complex by $\Assoc(W,c)$ in order to emphasize that it generalizes the well-studied \emph{dual associahedron}.
%\medskip

%We refer to \Cref{sec:m-weak} and also to~\cite{Hum1990} for the needed background on finite Coxeter systems.

%\medskip

\medskip

\paragraph{The weak order}
In \Cref{sec:m-weak}, we introduce an $m$-eralization of the weak order using the positive Artin monoid~$\Artinmon$ corresponding to~$W$.  Let~$\bwo$ be the image
of the longest element of~$W$ in~$\Artinmon$, so that the weak order on~$W$ embeds into the weak order on~$\Artinmon$ as the interval consisting of those elements that are initial of $\bwo$.  We introduce the Fu\ss-Catalan parameter $m$ as follows.

%, which suggests the following definition.

\begingroup
\def\thedefinition{\ref{def:mweak}}
\begin{definition}
  $\Wm\eqdef\bigset{ \bw \in \Artinmon}{\bw\bu=\bwom \text{ for some } \bu \in \Artinmon }$.
\end{definition}
\addtocounter{definition}{-1}
\endgroup

We recall in \Cref{prop:weak_degree_characterization} that $\Wm$ contains exactly those elements of~$\Artinmon$ with at most~$m$ Garside factors, and show in \Cref{thm:weak_lattice} that the weak order on~$\Wm$ is a rank-symmetric and self-dual lattice.

%In \Cref{sec:coxeter_elements}, we define sorting words for elements in~$\Artinmon$.
%For a given Coxeter element~$c \in W$, the $\c$-sorting word $\sw{w}{c}$ of an element~$\bw \in \Artinmon$ is the lexicographically first (as a sequence of positions) subword of $\c^{\infty} = (\s_1 \cdots \s_n)^\infty$ that is a reduced expression for~$\bw$.

\paragraph{Subword complexes}
In \Cref{sec:chapter_subword_complexes}, we recall some background from simplicial complexes and then generalize the notion of subword complexes to the needed generality  in \Cref{def:subwordsS} and \Cref{prop:subwordcomplexesinjection}.
We finally introduce and study Coxeter initial subword complexes in \Cref{sec:initial_subword_complexes,,sec:initial_subword_complexes}.
These later turn out to be the generality in which we need subword complexes in this monograph.

\paragraph{Noncrossing partitions}
In \Cref{sec:noncrossing_partitions}, we review background on $m$-eralized noncrossing partitions, including their description in terms of a dual subword complex in \Cref{prop:nc_fuss_subwords}.
We $m$-eralize the notions of Cambrian rotation and recurrence, and provide a novel perspective---new even for $m=1$---by defining an $m$-eralized $c$-Cambrian poset structure on noncrossing partitions in \Cref{def:nc_cambrian_poset}.
%We then study several properties of this poset.%, and finish this section with a detailed review of $m$-eralized noncrossing parking functions.

\paragraph{Cluster complexes}
In \Cref{sec:clusters_and_subwords}, we consider the missing $m$-eralization of the $c$-cluster complex for arbitrary Coxeter elements~$c$.
We use an $m$-eralized $c$-compati\-bility relation whose existence and uniqueness are proven by $m$-eralizing the subword complex approach to $c$-cluster complexes given in~\cite{CLS2011,PS20112}.

%Fix a Coxeter element~$c \in W$ and an element~$\bw \in \Artinmon$.  The $\c$-sorting word $\sw{w}{c}$ of $\bw$ is the lexicographically first (as a sequence of positions) subword of $\c^{\infty} = (\s_1 \cdots \s_n)^\infty$ that is a reduced expression for~$\bw$.

\begingroup
\def\thedefinition{\ref{def:m-c-cluster}}
\begin{definition}
  The simplicial complex $\Assocm(W,c)$ is the set of subwords of the $c$-sorting word $\cwm{c}$ (see \Cref{def:c-sorting}) whose complements contain a word for $\bwom \in \Artinmon$.%, where $\cwom$ denotes the $c$-sorting word of $\bwom$ (\Cref{def:c-sorting}).%, and where $N = \lengthS(\wo) = |\PhiP|$.\nathan{make braid} of length $mN$
\end{definition}
\addtocounter{definition}{-1}
\endgroup

We show in \Cref{cor:m-c-assoc-m-assoc} that in the special case of bipartite Coxeter elements, this definition recovers the generalized cluster complex of S.~Fomin and N.~Reading~\cite{FR2005}.  This perspective allows us to give simple proofs of many of its known properties.
In \Cref{cor:m-c-shelling-order} we prove C.~Atha\-na\-si\-adis and E.~Tzanaki's result from~\cite{AT2008} that $\Assocm(W,c)$ is vertex-decomposable (and therefore shellable); and in \Cref{thm:assoc_homotopy} we show that $\Assocm(W,c)$ is a wedge of $\Catm[m-1](W)$ many spheres of dimension $n{-}1$~\cite[Proposition~11.1]{FR2005}.

\paragraph{Sortable elements}
In \Cref{sec:sortable_elements}, we provide the missing $m$-eralization of $c$-sortable elements by lifting N.~Reading's definition from~$W$ to $\Wm$ (\Cref{def:msort_sortable}).  
We denote the set of all \defn{$m$-eralized $c$-sortable elements} by $\Sortm(W,c)$, and characterize the individual Garside factors of an $m$-eralized $c$-sortable element in \Cref{def:factorsort} and in \Cref{cor:factorsort}.

%\begingroup
%\def\thedefinition{\ref{def:msort_sortable}}
%\begin{definition}
%  An element $\bw \in \Wm$ is \defn{$c$-sortable} if the $c$-sorting word $\sw{w}{c}$ for~$\bw$ defines a decreasing sequence of subsets of positions in~$\c$.
%\end{definition}
%\addtocounter{definition}{-1}
%\endgroup

We denote the set of all \defn{$m$-eralized $c$-sortable elements} by $\Sortm(W,c)$, and characterize the individual Garside factors of an $m$-eralized $c$-sortable element in \Cref{def:factorsort} and in \Cref{cor:factorsort}.

\paragraph{Bijections}
The following theorem relates D.~Armstrong's $m$-eralization of noncrossing partitions, S.~Fomin and N.~Reading's $m$-eralization of cluster complexes, and the new $m$-eralizations of sortable elements and subword complexes.

\begingroup
\def\thetheorems{\ref{thm:bij_sort_to_nc}, \ref{thm:bij_nc_to_sort}, \ref{thm:rootconf_skiptset}, \ref{thm:sortcluster}}
\begin{theorems}
  There are explicit, uniform, and natural bijections between the three families
  \begin{itemize}
    \item the $m$-eralized $c$-noncrossing partitions $\NCm(W,c)$,
    \item the $m$-eralized $c$-cluster complexes $\Assocm(W,c)$, and
    \item the $m$-eralized $c$-sortable elements $\Sortm(W,c)$.
  \end{itemize}
\label{thm:bij_intro}
\end{theorems}
\addtocounter{theorems}{-1}
\endgroup

We use the term ``explicit'' to mean that we provide a bijection (rather than relying on a counting argument), the term ``uniform'' to indicate that we do not use the classification theorem of finite Coxeter systems, and the term ``natural'' to mean that the given bijections respect the inductive parabolic structure on each family.
In fact, the proofs will very often be based on the inductive structure provided by a modification of Cambrian rotation called the \defn{Cambrian recurrence}, as in \Cref{prop:nc_cambrian_recurrence,,prop:assoc_cambrian_recurrence,,prop:msort_cambrian_recurrence}.

\paragraph{Cambrian lattices}
Although the flip graph of S.~Fomin and N.~Reading's $m$-eralized $c$-cluster complex can be used to define a Cambrian graph for bipartite Coxeter elements, no corresponding poset has been considered in the literature for $m>1$.
In particular, no orientation of the $m$-eralized exchange graph was known to be a lattice.

In \Cref{sec:nc_cambrian_posets,sec:m-cluster-subword,sec:sort_cambrian_lattices}, we give definitions of the $m$-eralized $c$-Cambrian lattice on each of $\NCm(W,c)$, $\Assocm(W,c)$, and $\Sortm(W,c)$.
For $\NCm(W,c)$ and $\Assocm(W,c)$, we construct these posets as the transitive closures of the objects under certain flips, while $\Sortm(W,c)$ inherits its poset structure from the weak order on~$\Wm$.  
The present construction $m$-eralizes N.~Reading's Cambrian lattices, which are themselves generalizations of the classical Tamari lattices.
As a case of particular interest, this construction provides a new $m$-eralization of the Tamari lattice, different from the $m$-Tamari lattice (\Cref{sec:tamari}).

\begingroup
\def\thetheorems{\ref{thm:sort_is_lattice}, \ref{thm:cambnccambsort}, \ref{thm:cambnccambassoc}}
\begin{theorems}
  The restriction of the weak order to $\Sortm(W,c)$ is a lattice.
  It is isomorphic to the increasing flip posets of $\NCm(W,c)$ and of $\Assocm(W,c)$.
\end{theorems}
\addtocounter{theorems}{-1}
\endgroup

%In summary, we root the program of $m$-eralizing noncrossing Coxeter-Catalan combinatorics in the $m$-weak order of the corresponding Artin group.

\paragraph{Positive and rational Catalan combinatorics}
In \Cref{sec:rat_positive}, we study \emph{positive} analogues of the three noncrossing Catalan families, and we explain a special symmetry on positive $m$-eralized $c$-noncrossing partitions.

D.~Armstrong has proposed that noncrossing Fu\ss-Catalan combinatorics should have a ``rational Catalan'' generalization to accommodate a parameter $p$ coprime to the Coxeter number $h$---the $m$-eralizations of noncrossing objects should then be recovered for $p=mh+1$, while the positive analogues should correspond to $p=mh-1$~\cite{ARW2013,ALW2014}.  In \Cref{sec:rat}, we give conjectural constructions in the classical types.\nathanside{Sorry, conjectural.}

\paragraph{Representation Theory}
\Cref{sec:representationtheory} discusses the link to the representation theory of hereditary Artin algebras.
For~$W$ crystallographic, $m$-eralized noncrossing partitions and $m$-eralized clusters have been given representation-theoretic interpretations in, respectively,~\cite{BRT2012} and \cite{thomas2007defining,zhu2008}.
We show that the $m$-eralized $c$-sortable elements also fit into this framework, generalizing the work of C.~Ingalls and H.~Thomas for $m=1$~\cite{IT2009}.
It turns out that the combinatorial bijections of \Cref{thm:rootconf_skiptset,thm:sortcluster} coincide in crystallographic types with representation-theoretic bijections~\cite{BRT2012,KV}.

\medskip

\Cref{fig:objects_and_maps} illustrates these objects and their bijections.\nathanside{\Cref{fig:objects_and_maps} is perfect for use as a desktop background, wall-hanging, or tattoo.}% between the various families of Fu\ss-Catalan objects.

\medskip

\paragraph{Gross omissions and open problems}
There are many directions in Coxeter-Catalan combinatorics that are outside the scope of this monograph, and there are also many open problems.  We give a short list here:

 %which we will not consider in this monograph, but can be found in~\cite[Section~5.1]{Arm2006} and the references therein.
%\item many poset-theoretic properties
\begin{enumerate}[$\circ$]
	\item We do not give type-by-type combinatorial constructions of our noncrossing Fu\ss-Catalan objects.  Enumerations in the classical types often make use of such constructions~\cite{FR2005,Arm2006}.
	\item We do not do any refined enumeration, nor do we address cyclic sieving~\cite{KS2018,BR2007}.
	\item We do not talk about \emph{nonnesting} Catalan objects (using root posets or the affine Weyl group)~\cite{Rei1997}, or touch on the vast field of $q,t$-Catalan combinatorics~\cite{Hai1994}.
	\item We do not address recent advances regarding noncrossing parking functions and parking spaces~\cite{ARR2015,Rho2014}.
	\item Outside of the historical overview, we do not mention polytopes or discuss realizations of associahedra.
	\item We do not construct $m$-eralized cluster \emph{algebras}.
	\item We do not consider our constructions for infinite Coxeter groups~\cite{RS2011}.  This would already be interesting in affine type~\cite{reading2015cambrian}.
\end{enumerate}
%\nathanside{To do: add a list of omissions from the list of omissions.}

\begin{landscape}
\centering
\vspace*{\fill}
\begin{figure}[h]
  \centering
  \scalebox{0.9}{
    \begin{tikzpicture}[xscale=3,yscale=4]
      \tikzstyle{rect}=[rectangle,draw,opacity=.5,fill opacity=1,text width=10em,align=center]
      \tikzstyle{sort}=[fill=black!20]
      \node[rect] (sortshard)  at (0,0) {$\deltaSortm(W,c)$ \\ \cref{def:sort_multichains}  \\ \small chains in shard intersection order};
      \node[rect] (sortfact)   at (2,0) {$\Sortma(W)$ \\ \cref{def:factorsort} \\ \small characterized using \\ Garside factorizations};
      \node[rect] (sort)  at (4,0) {$\Sortm(W,c)$ \\ \cref{def:msort_sortable}  \\ \small characterized using \\ sorting words in $\Artinmon$};
      \node[rect] (assnab) at (3,1) {$\NablaAssocm(W,c)$ \\ \cref{def:m-c-cluster} \\ \small $\refl$-words for $c^{-1}$ in \\ $\invs_\refl(\cwm{c})$};
      \node[rect] (assdel)   at (1,1) {$\DeltaAssocm(W,c)$\\ \cref{def:m-c-cluster} \\ $\asref$-words for $\bwom$ in \\ $\cwm{c}$};
      \node[rect] (nc) at (0,2) {$\NCm(W,c)$\\ \cref{def:nc_multichains} \\ \small chains in absolute order};
      \node[rect] (ncdelta) at (2,2) {$\deltaNCm(W,c)$\\ \cref{def:nc_delta} \\ \small delta sequences};
      \node[rect] (nccapdelta) at (4,2) {$\DeltaNCm(W,c)$ \\ \cref{def:nc_fuss_subwords} \\ $\refl$-words for $c$ in \\ $\invs_\refl(\cwom[c][m+1])$};
      \node[rect] (hom) at (6,2) {$\Homm(H)$ \\ \cref{def:hom} \\ \small $\Hom_{\leq 0}$ configurations};
	\node[rect] (silt) at (6,1) {$\Siltm(H)$ \\
\cref{def:silting} \\
\small silting objects};
\node[rect] (aisle) at (6,0) {$\storfm(H)$ \\ \cref{def:good-co} \\ \small co-aisles};
\draw[stealth-stealth] (silt) -- (aisle) node[midway,right] {\cref{th:silt-co}};
	\draw[stealth-stealth] (aisle) -- (sort) node[midway,below] {\cref{aisle-sort}};
	\draw [stealth-stealth] (sort) -- (assdel) node[midway,left] {\cref{thm:m-assoc_to_m_ncp}};	
	\draw[stealth-stealth] (hom) -- (silt) node[midway,right] {\cref{thm:rev}};
	 \draw [stealth-stealth] (hom) -- (nccapdelta) node[midway,above] {\cref{thm:hom-deltanc}};
      \draw [stealth-stealth] (sortshard) -- (sortfact) node[midway,below] {\cref{thm:bij_multichain_sort_to_sort}};
  	\draw [stealth-stealth] (sortfact) -- (sort) node[midway,below] {\cref{cor:factorsort}};
      \draw [stealth-stealth] (assnab) -- (assdel) node[midway,above] {\cref{prop:dualsubwordcomplexes}};
      \draw [stealth-stealth] (nc) -- (ncdelta) node[midway,above] {\text{\cref{prop:nc_delta}}};
 	\draw[stealth-stealth] (ncdelta) -- (nccapdelta) node[midway,above] {\text{\cref{prop:nc_fuss_subwords}}};
      \draw [stealth-stealth] (sortshard) -- (nc) node[midway,left] {\cref{thm:bij_sort_multichains_to_nc}};
		\draw [stealth-stealth] (assdel) -- (nccapdelta) node[midway,left] {\cref{thm:rootconf_skiptset}};
		\draw [stealth-stealth] (sort) -- (nccapdelta) node[pos=0.33,right] {\cref{thm:bij_sort_to_nc}};
		\draw[line width=5pt,white] (silt) -- (assnab);
	\draw[stealth-stealth] (silt) -- (assnab) node[midway,above] {\cref{prop:silt-to-assoc}};
    \end{tikzpicture}
  }
  \caption{A diagram of the Fu\ss--Catalan objects and their bijections.  The top line contains the noncrossing partition-type objects (\Cref{sec:noncrossing_partitions}), the middle line the cluster-type objects (\Cref{sec:clusters_and_subwords}), and the bottom line the sortable-type objects (\Cref{sec:sortable_elements}).  The rightmost column contains the representation-theoretic counterparts (\Cref{sec:representationtheory}).}
  \label{fig:objects_and_maps}
\end{figure}
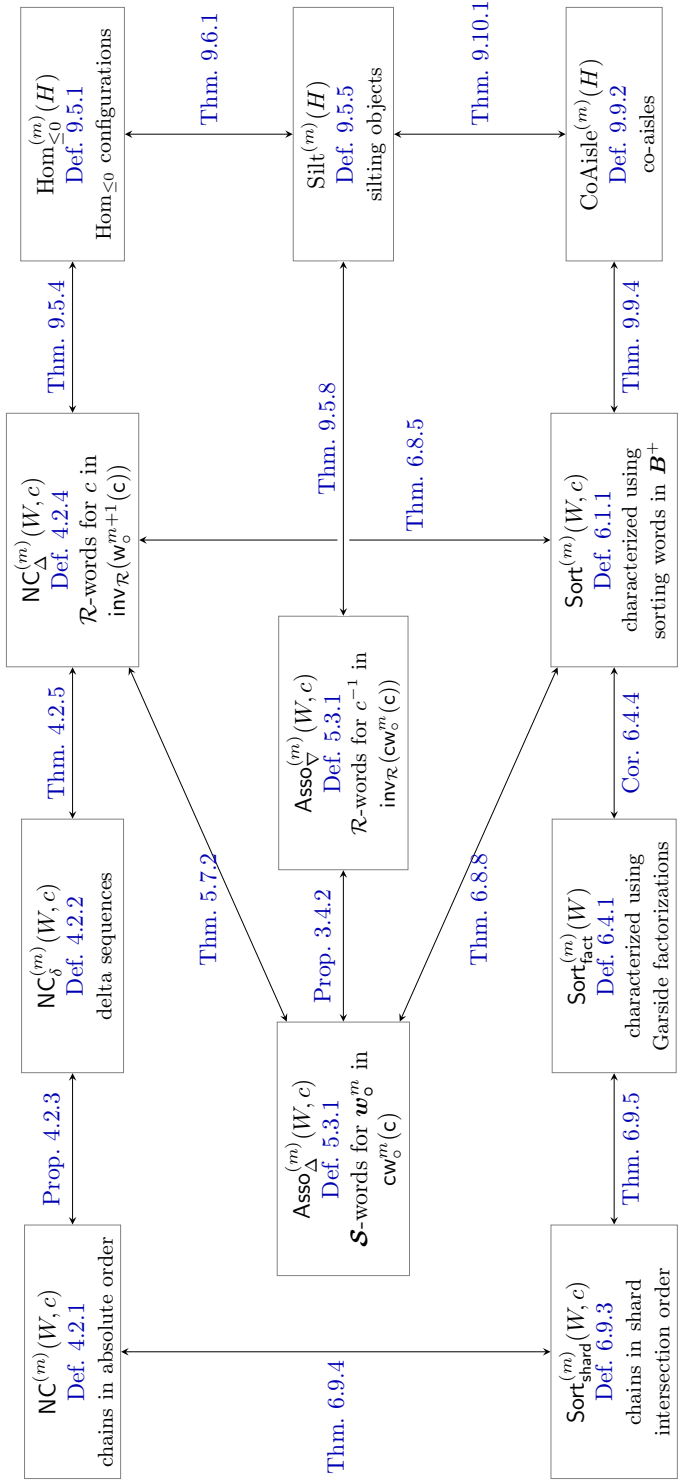
\vfill
\end{landscape}

%%%%%%%%%%%%%%%%%%%%%%%%%%%%%%%%%%%%%%%%%%%%%%%%%%%%%%%%%%%%%%%%%%%%%%%%%%%%%%%%%%%%%
\chapter{Background on Coxeter and Artin groups}
\label{sec:m-weak}
%%%%%%%%%%%%%%%%%%%%%%%%%%%%%%%%%%%%%%%%%%%%%%%%%%%%%%%%%%%%%%%%%%%%%%%%%%%%%%%%%%%%%

In this chapter, we review the theory of finite Coxeter and Artin groups.
Most of this material is classical and detailed background can be found in~\cite{Hum1990,BB2005,Deh2013}.\nathanside{If you can read Drew's monograph, you'll be fine.}
After recalling this background (\Cref{sec:coxeter_artin_definitions,sec:cox_length_support,sec:descentsandcovers,sec:garside_factorizations}), we extend the theory of sorting words to the positive Artin monoid (\Cref{sec:coxeter_elements}) and recall the geometry of Coxeter groups (\Cref{sec:geometry,sec:inversions}).
We then discuss the shard intersection order and give a new characterization in the Coxeter group (\Cref{sec:intersection_lattice}).
We conclude with the definition of the $m$-eralized weak order inside the positive Artin monoid (\Cref{sec:weak order}).

%%%%%%%%%%%%%%%%%%%%%%%%%%%%%%%%%%%%%%%%%%%%%%%%%%%%%%%%%%%%%%%%%%%%%%%%%%%%%%%%%%%%%
\section{Coxeter and Artin systems}
\label{sec:coxeter_artin_definitions}
%%%%%%%%%%%%%%%%%%%%%%%%%%%%%%%%%%%%%%%%%%%%%%%%%%%%%%%%%%%%%%%%%%%%%%%%%%%%%%%%%%%%%

A (finite) \defn{Coxeter system} $(W,\sref)$ of \defn{rank} $n \eqdef |\sref|$ is a finite group~$W$ together with a distinguished subset $\sref \subseteq W$ of generators and a presentation
\[
  W = \left\langle \sref  : s^2=\one, \quad \alt{s|t}^{m(s,t)}= \alt{t|s}^{m(t,s)} \text{ for } s,t \in \sref \text{ with } s \neq t \right\rangle,
\]
for integers $m(s,t)=m(t,s) \geq 2$, where $\alt{s|t}^{m(s,t)}$ consists of $m(s,t)$ alternating factors~$s$ and~$t$.\footnote{We will regularly use the notation $\alt{u|v}^{i}$ for a positive integer~$i$ to mean~$i$ alternating copies of~$u$ and~$v$---both being elements in the group~$W$, in the positive monoid~$\Artinmon$, or words in~$\sref$ or in~$\refl$.}
The elements in the set~$\sref$ are called \defn{simple generators} or \defn{simple reflections}. 
The relations $sts\cdots = tst\cdots$ are called \defn{braid relations}, and invoking one to rewrite a word in~$\sref$ is called a \defn{braid move}.  A braid relation of the form $st=ts$ is called a \defn{commutation relation}.

The set of \defn{reflections} in~$W$ is defined to be
\[
  \refl \eqdef \bigset{ s^w }{ s \in \sref, w \in W },
\]
where we write $u^w \eqdef w^{-1} u w$ and ${}^w u \eqdef wuw^{-1}$.

\medskip

The (spherical) \defn{Artin system} $(\Artingrp,\asref)$ corresponding to the finite Coxeter system $(W,\sref)$ is the group~$\Artingrp$ given by a formal copy of the generators~$\sref$---written $\asref$---subject to only the braid relations
\[
  \Artingrp = \left\langle \asref  : \alt{\bs|\bt}^{m(s,t)} = \alt{\bt|\bs}^{m(t,s)} \text{ for } s,t \in \sref \text{ with } s \neq t \right\rangle.
\]
  We mostly restrict to the submonoid generated by~$\asref$, called the \defn{positive Artin monoid}~$\Artinmon$.

\begin{mywarning*}
  While we used boldface above to distinguish between the simple generators $s\in \sref$ and the corresponding generator $\bs \in \asref$, we usuall do not distinguish between these two generating sets if there is no risk of confusion.  When confusion may arise, we still use boldface to distinguish elements $\bw \in \Artinmon$ and $w \in W$.
  \nathanside{This is great notation.  It even comes with a 90 day warranty.}
\end{mywarning*}

\begin{example}
  When~$W$ is the \defn{symmetric group} $\WA$,~$\Artingrp$ is the \defn{braid group} $\Braidgrp_n$, and $\Artinmon$ is the \defn{positive braid monoid} $\BA$.
  We often illustrate definitions and results using the running example of the symmetric group on three letters $\WA[3]$, which is the Coxeter group (of type~$A_2$ as given in \Cref{sec:classification}) generated by the simple reflections $\sref=\{s,t\}$ corresponding to the transpositions $s = (12)$ and $t = (23)$.
  We denote the third transposition by $u = (13)$, which completes the set of reflections
  \[
    \refl = \{ s,u,t\} = \big\{ (12),(13),(23)\big\}.
  \]
  In addition to these three reflections and the identity element~$\one$, $\WA[3]$ contains the two long cycles
  \[
    (123) = st = tu = us \quad \text{and}\quad (321) = ts = ut = su.
  \]
  For the symmetric group $\WA[n]$ on~$n$ letters, we sometimes use $\sref = s_1,\ldots,s_{n-1}$ with $s_i = (i\ i{+}1) \in \WA[n]$.
  The positive Artin monoid $\BA[3]$ contains all words in the letters~$\bs$ and~$\bt$, subject to the relation $\bs\bt\bs = \bt\bs\bt$, indicated as
  \begin{figure}[h]
    \centering
    \begin{tikzpicture}[scale=.7]
        \braid [rotate=90, style strands={1}{red},style strands={2}{blue}, style strands={3}{green}, line width=.4mm] (fir) at (0,-2.5) a_2 a_1 a_2;
        \node (equiv) at (1.8,1) {$\sim$};
        \braid [rotate=90, style strands={1}{red},style strands={2}{blue}, style strands={3}{green}, line width=.4mm] (sec) at (0, 2.5) a_1 a_2 a_1;
    \end{tikzpicture}\quad .
  \end{figure}
\end{example}

\subsection{Words in generators}

We use sans-serif letters to distinguish \emph{words} in generators from \emph{elements} in~$W$ or in~$\Artinmon$.  In this way, $\s_1 \cdots \s_p$ is an $\sref$-word, while $s_1 \cdots s_p \in W$ or $\bs_1 \cdots \bs_p \in \Artinmon$ are elements.\nathanside{We take typography very seriously.}\christianside{But we don't necessarily respect our own conventions.}
We call two $\sref$- or $\refl$-words $\Q$ and $\Q'$ \defn{commutation equivalent} if we may transform one into the other by a sequence of interchanges of consecutive commuting letters, \ie, letters that commute in the group.
In this case, we write $\Q \equiv \Q'$.
A word $\u$ (often a single letter) is \defn{initial} or \defn{final} in $\Q$ if $\u$ occurs as a prefix or, respectively, as a suffix of some $\Q' \equiv \Q$.
We denote the word obtained from a word $\Q$ by removing an initial letter~$\s$ by $\sinv\Q$ and write $\Q\sinv$ for the removal of a final letter~$\s$.

We use the notation $\psi(s) \eqdef w^\wo \in \sref$ for $s \in \sref$ and extend this notation to elements in~$W$ and~$\Artinmon$ and also to words.
We denote the reverse of a word $\Q$ by $\rev(\Q)$ and extend this notation to elements in~$\Artinmon$.

\subsection{Classifications of Coxeter groups}
\label{sec:classification}

The \defn{Coxeter diagram} of $(W,\sref)$ is the graph on~$\sref$ with an edge $s-t$ if $m(s,t) \geq 3$ labelled by $m(s,t)$.
Usually, the label $m(s,t) = 3$ is omitted.
The Coxeter system is \defn{simply-laced} if $m(s,t) \in \{2,3\}$ for all $s,t \in \sref$ and \defn{crystallographic} if $m(s,t) \in \{2,3,4,6\}$ for all $s,t \in \sref$.

A Coxeter system is called \defn{irreducible} if its Coxeter diagram is connected.
Irreducible finite Coxeter systems are classified, the classification is shown in \Cref{fig:classification}.
See~\cite[Appendix~A1]{BB2005} for details.\nathanside{Regardless of whether or not you know this content, you should probably be reading something else.}
\begin{figure}[t]
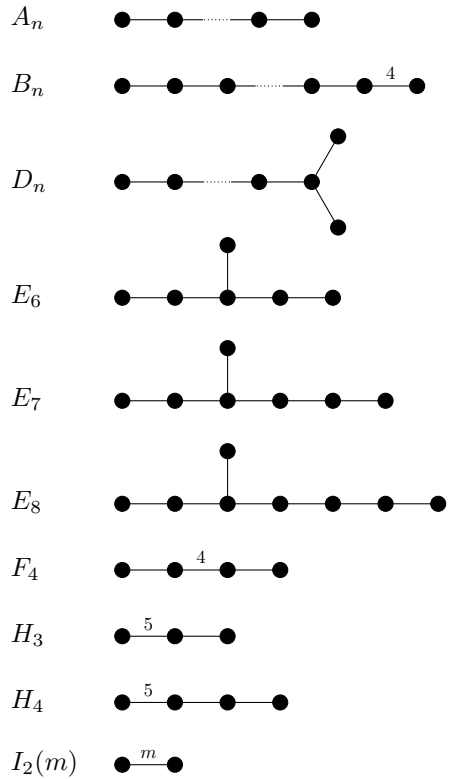

  \begin{center}
    \begin{tabular}{ll}
      $A_n$    \quad& \dynkin[Coxeter, scale=2]{A}{} \\[10pt]
      $B_n$    \quad& \dynkin[Coxeter, scale=2]{B}{} \\[10pt]
      $D_n$    \quad& \dynkin[Coxeter, scale=2]{D}{} \\[10pt]
      $E_6$    \quad& \dynkin[Coxeter, scale=2]{E}{6} \\[10pt]
      $E_7$    \quad& \dynkin[Coxeter, scale=2]{E}{7} \\[10pt]
      $E_8$    \quad& \dynkin[Coxeter, scale=2]{E}{8} \\[10pt]
      $F_4$    \quad& \dynkin[Coxeter, scale=2]{F}{4} \\[10pt]
      $H_3$    \quad& \dynkin[Coxeter, scale=2]{H}{3} \\[10pt]
      $H_4$    \quad& \dynkin[Coxeter, scale=2]{H}{4} \\[10pt]
      $I_2(m)$ \quad& \dynkin[Coxeter, scale=2,,gonality=m]{I}{}
    \end{tabular}
  \end{center}
  \caption{Classification of the irreducible Coxeter groups and diagrams.}
  \label{fig:classification}
\end{figure}
%%%%%%%%%%%%%%%%%%%%%%%%%%%%%%%%%%%%%%%%%%%%%%%%%%%%%%%%%%%%%%%%%%%%%%%%%%%%%%%%%%%%%
\section{The weak order}
\label{sec:cox_length_support}
%%%%%%%%%%%%%%%%%%%%%%%%%%%%%%%%%%%%%%%%%%%%%%%%%%%%%%%%%%%%%%%%%%%%%%%%%%%%%%%%%%%%%

\subsection{Words in simple reflections}

The (Coxeter) \defn{length} of an element~$w$ in~$W$ or in $\Artinmon$ is the length~$\lengthS(w)$ of a shortest expression for~$w$ as a product of the generators in~$\sref$.  Examples are given in \Cref{fig:A2,fig:A22}.
An $\sref$-word $\s_1 \cdots \s_p$ is \defn{reduced} if $w = s_1 \cdots s_p$ and $p = \lengthS(w)$.
A factorization $w = u \cdot v$ is \defn{reduced} if $\lengthS(w)=\lengthS(u)+\lengthS(v)$.
In this case,~$u$ is \defn{initial} and~$v$ is \defn{final} in~$w$ and we set $\overline{u}w \eqdef u^{-1}w$, where we emphasize that $\overline{\bu} \bw \in \Artinmon$ for $\bw \in \Artinmon$.
A sequence $s_1,\dots,s_p \in \sref$ is an \defn{initial sequence} for~$w$ if $s_{i+1}$ is initial in $\sninv_i\dots\sninv_2\sninv_1 w s_1s_2\dots s_i$.
We use analogous notation for~$s$ final in~$w$ and final sequences for~$w$.

\medskip

E.~Brieskorn and K.~Saito proved the following essential lemma.\nathanside{Double induction, oh my God! $\Cap$}% \url{https://www.youtube.com/watch?v=OQSNhk5ICTI}}

\begin{lemma+}[{\cite[Lemma~2.1 \& Proposition~2.3]{brieskorn1972artin}}]
\label{lem:brieskorn_saito}
   Let $\bu,\bv \in \Artinmon$.
   \begin{enumerate}[(1)]
     \item If $\ba \bu \bb = \ba \bv \bb$ for $\ba,\bb \in \Artinmon$, then $\bu = \bv$.
     \item\label{lem:brieskorn_saito1} If $\bs \bu=\bt \bv$ for $\bs,\bt \in \asref$, then there exists $\bw \in \Artinmon$ such that $\bu=\alt{\bt|\bs}^{m(s,t)-1}\bw$ and $\bv=\alt{\bs|\bt}^{m(s,t)-1}\bw$. \qedhere
   \end{enumerate}
\end{lemma+}

Any reduced $\sref$-word of an element in~$W$ may be transformed to any other by a sequence of braid moves~\cite[Theorem~3.3.1]{BB2005}.  We define the \defn{support} $\supp(w)$ for $w \in W$ or $\Artinmon$ to be the set $\{ s_1,\ldots,s_p\} \subseteq \sref$ of simple reflections contained in any reduced word $\s_1\cdots\s_p$ for~$w$.  Support is well-defined, since reduced words are connected under braid moves, and braid moves preserve the set of simple reflections in a reduced word.

\medskip

Identifying the generating sets for $W$ and $\Artinmon$ gives a natural injection
\begin{equation}
  W \hookrightarrow \Artinmon.
\label{eq:artininjection}
\end{equation}
Since~$\Artinmon$ is subject to the braid relations, any two reduced $\sref$-words for $w \in W$ specify the same element $\bw \in \Artinmon$.  The unique \defn{longest element} in~$W$ is denoted by~$\wo$.
The corresponding element $\bwo \in \Artinmon$ is sometimes called the \defn{fundamental element} in Garside theory.
% \nathanside{Not by us.}
We denote its length by $N \eqdef \lengthS(\wo) = |\refl|$.

\subsection{The weak order}

The (right) \defn{weak order} on~$W$ is the partial order $\Weak(W)$ defined by $u \le w$ if~$u$ is initial in~$w$.
The weak order on the positive Artin monoid $\Weak(\Artinmon)$ is defined analogously.  We record the following fundamental fact about the weak order.

\begin{theorem+}
\label{thm:artinmon_lattice}
  For~$W$ a finite Coxeter group with corresponding Artin mo\-noid~$\Artinmon$, $\Weak(W)$ and $\Weak(\Artinmon)$ are lattices.
\end{theorem+}

The injection $W \hookrightarrow \Artinmon$ gives a poset isomorphism
\[
  \Weak(W) \cong [\bone,\bwo]_{\Weak(\Artinmon)} \subseteq \Weak(\Artinmon).
\]

\medskip

The \defn{meet} of two elements $w,w'$ in $\Weak(W)$ or in $\Weak(\Artinmon)$ is denoted by $w \wedge w'$, and their \defn{join} is denoted by $w \vee w'$.  We recall that~$w$ is called \defn{join-irreducible} if it covers exactly one other element.  The join-irreducible elements are exactly those elements~$w$ such that $w \neq w' \vee w''$ for $w',w'' < w$ and $w\ne e$.

\section{Descents and cover reflections}
\label{sec:descentsandcovers}

We define the \defn{left descent set}, \defn{right descent set}, \defn{left ascent set}, and \defn{right ascent set} of $w$ by
\begin{alignat*}{3}
  \DesSet_L(w) &\eqdef \set{s \in \sref }{ s \text{ initial in }  w },& \hspace{2em} \DesSet_R(w) &\eqdef \set{s \in \sref }{ s \text{ final in }  w }, \\
  \AscSet_L(w) &\eqdef \sref \setminus \DesSet_L(w),&  \hspace{2em} \AscSet_R(w) &\eqdef \sref \setminus \DesSet_R(w).
\end{alignat*}

\begin{proposition}
\label{prop:des_contained}
  Let~$s \in S$.
  For an element $w$ of $W$ or $\Artinmon$ such that $\lengthS(sw) > \lengthS(w)$, we have
  \[
    \DesSet_L(sw) \subseteq \{s\} \cup \DesSet_L(w).
  \]
\end{proposition}

\begin{proof}
  If $s,t \in \DesSet_L(u)$ for $u$ in $W$ or in $\Artinmon$, then $s \vee t = \alt{s|t}^{m(s,t)} = \alt{t|s}^{m(s,t)}$ is initial in~$u$ by~\Cref{lem:brieskorn_saito}\eqref{lem:brieskorn_saito1}.
  Applying this to the element $u = sw$ yields that $s \vee t$ is initial in~$sw$ for any $t \in \DesSet_L(sw)$.
  If $t \neq s$, we conclude that $\alt{t|s}^{m(s,t)-1}$ is initial in~$w$.
\end{proof}

For $w \in W$, the \defn{covered reflections} and the \defn{covering reflections} are, respectively,
\begin{equation}
\begin{aligned}
  \coveredref(w)  &\eqdef \set{ {}^w s }{ s \in \DesSet_R(w) }, \hfill \coveringref(w) &\eqdef \set{{}^w s  }{ s \in \AscSet_R(w) }.
\end{aligned}\label{eq:coverreflections}
\end{equation}\nathanside{Cue wild applause.  We also do not define motives.} 
We do not define colored versions of covered and covering reflections for elements of $\Artinmon$. \Cref{fig:A2} illustrates several examples.
\christianside{Protip: if you understand the examples, you understand the theory.}

\begin{figure}[t]
  \begin{center}
    \begin{tabular}{c|c|c|c|c|c|c|c|c}
      $w$ & $\sw{w}{st}$ & $\lengthS$ & $\lengthR$ & $\supp$ & $\DesSet_L$ & $\DesSet_R$ & $\coveredref$ & $\coveringref$\\
      \hline
      $e$ &
      $\S\T|
        \S\T|
        \S\T$ &
      0 & 0 & $-$ & $-$ & $-$ & $-$ & $s,t$ \\
     $s$ &
      $\s\T|
        \S\T|
        \S\T$&
      1 & 1 &$s$ &$s$ &$s$ &$s$ &~$u$ \\
     $t$ &
      $\S\t|
        \S\T|
        \S\T$&
      1 & 1 &$t$ &$t$ &$t$ &$t$ &~$u$ \\
      $st$ &
      $\s\t|
        \S\T|
        \S\T$&
      2 & 2 & $s,t$ &$s$ &$t$ &~$u$ &$t$ \\
      $ts$ &
      $\S\t|
        \s\T|
        \S\T$&
      2 & 2 & $s,t$ &$t$ &$s$ &~$u$ &$s$ \\
      $sts$ &
      $\s\t|
        \s\T|
        \S\T$&
      3 & 1 & $s,t$ & $s,t$ & $s,t$ & $s,t$ & $-$
    \end{tabular}
  \end{center}
  \caption{The $st$-sorting word, length, reflection length, left descents, right descents, covered reflections, and covering reflections of the elements of $\WA[3]$.}
  \label{fig:A2}
\end{figure}
  \begin{figure}[t]
    \begin{center}
      \begin{tabular}{c|c|c|c}
        $\garside$ & $\sw{w}{st}$ & $\DesSet_L$ & $\DesSet_R$ \\
        \hline
        $s\cdot s$ &
        $\s\T|\s\T|\S\T$ &
       $s$ &$s$ \\
        $s\cdot st$ &
        $\s\T|\s\t|\S\T$ &
       $s$ &$t$ \\
        $t\cdot t$ &
        $\S\t|\S\t|\S\T$ &
       $t$ &$t$ \\
        $t\cdot ts$ &
        $\S\t|\S\t|\s\T$ &
       $t$ &$s$ \\
        $st\cdot t$ &
        $\s\t|\S\t|\S\T$ &
       $s$ &$t$ \\
        $st\cdot ts$ &
        $\s\t|\S\t|\s\T$ &
       $s$ &$s$ \\
        $ts\cdot s$ &
        $\S\t|\s\T|\s\T$ &
       $t$ &$s$ \\
      \end{tabular}
      \quad
      \begin{tabular}{cc|c|c|c}
        & $\garside$ & $\sw{w}{st}$ & $\DesSet_L$ & $\DesSet_R$ \\
        \hline
        & $ts\cdot st$ &
        $\S\t|\s\T|\s\t$ &
       $t$ &$t$ \\
        & $sts\cdot s$ &
        $\s\t|\s\T|\s\T$ &
        $s,t$ &$s$ \\
        & $sts\cdot t$ &
        $\s\t|\s\t|\S\T$ &
        $s,t$ & $s,t$ \\
        & $sts\cdot st$ &
        $\s\t|\s\T|\s\t$ &
        $s,t$ & $s,t$ \\
        & $sts\cdot ts$ &
        $\s\t|\s\t|\s\T$ &
        $s,t$ & $s,t$ \\
        & $sts\cdot sts$ &
        $\s\t|\s\t|\s\t$ &
        $s,t$ & $s,t$ \\
        \\
      \end{tabular}
    \end{center}
    \caption{The $\s\t$-sorting word and descent sets of those elements in $\BA[3]$ having exactly two nontrivial Garside factors.}
    \label{fig:A22}
  \end{figure}

\section{Parabolic subgroups}

%\nathanside{\url{https://mathoverflow.net/questions/24960/why-are-parabolic-subgroups-called-parabolic-subgroups}.}
When $J \subseteq \sref$, we use the notation $W_J$ for the \defn{standard parabolic subgroup} generated by~$J$, and $\Artinmon_{\BJ}$ for the corresponding \defn{standard parabolic positive submonoid}.  A standard parabolic subgroup is called \defn{maximal} if it is generated by $\langle s \rangle\eqdef\sref \setminus \{s\}$ for some $s \in \sref$.  Conjugates of standard parabolic subgroups are called \defn{parabolic subgroups}.

The \defn{parabolic quotient} corresponding to a standard parabolic subgroup is
\begin{align*}
   W^{J}  \eqdef \set{w \in W }{ \DesSet_L(w) \cap J = \emptyset }.
\end{align*}
For $J \subseteq \sref$ and an element $w \in W$, we write $w_J \eqdef w \wedge \wo(J)$, where $\wo(J)$ is the longest element of~$W_J$.  Every $w \in W$ has a unique \defn{parabolic decomposition}
\begin{equation}
   w = w_J w^J \text{ for } w_J \in W_J \text{ and } w^J \in W^J.\label{eq:parabolic_decomposition_W}
\end{equation}

\section{Garside factorizations}
\label{sec:garside_factorizations}

The \defn{Garside factorization}\nathanside{That's Major Garside to you.} is a certain factorization of an element $\bw \in \Artinmon$ as a product of elements in $[\bone,\bwo]$, given as follows: set $\bw_1=\bw$.
For $i=1,2,3,\ldots$, as long as $\bw_{i} \neq \bone$, let
\[
  \bw^{(i)} = \bw_{i}\wedge\bwo, \quad \bw_{i+1} = (\bw^{(i)})^{-1} \bw_{i} 
\]
and set~$k = i$ for the last $i$ with $\bw_{i} \neq \bone$.
Then
\[
  \garside(\bw) = \bw^{(1)} \cdot \bw^{(2)} \cdot\ \cdots\ \cdot \bw^{(k)},
\]
where the \defn{Garside factors} $\bw^{(i)}$ are separated by a centered dot.  The \defn{Garside degree} $\deg(\bw)$ is defined to be~$k$.
By construction, every factor $\bw^{(i)}$ is initial in~$\bwo$ and so can be treated as an element $w^{(i)}\in W$.

\medskip

The following characterization of Garside factorizations appears in~\cite[Corollary~4.2]{Mil1999}.

\begin{theorem+}
\label{thm:garsidefactorization}
  A factorization $\bv_1 \cdot \bv_2 \cdot\ \cdots\ \cdot \bv_k$ with $\bv_i \leq \bwo$ is the Garside factorization of the element $\bw = \bv_1 \cdots \bv_k \in \Artinmon$ if and only if
  \begin{equation*}
    \DesSet_R(\bv_{i-1}) \supseteq \DesSet_L(\bv_{i}).\qedhere
  \end{equation*}
\end{theorem+}

%%%%%%%%%%%%%%%%%%%%%%%%%%%%%%%%%%%%%%%%%%%%%%%%%%%%%%%%%%%%%%%%%%%%%%%%%%%%%%%%%%%%%
\section{Coxeter elements and sorting words}
\label{sec:coxeter_elements}
%%%%%%%%%%%%%%%%%%%%%%%%%%%%%%%%%%%%%%%%%%%%%%%%%%%%%%%%%%%%%%%%%%%%%%%%%%%%%%%%%%%%%

A (standard) \defn{Coxeter element}~$c$ for $(W,\sref)$ is defined to be the product of the elements of~$\sref$ in any order.\nathanside{Coxeter was known to do 50 pushups a day, even into his 90s.  Perhaps now's a good time to see how many pushups you can do.}
All Coxeter elements in~$W$ are conjugate, and we denote their common order by the \defn{Coxeter number}~$h$.  For $\c=\s_1 \s_2 \cdots \s_n$ and $J \subseteq \sref$, define the \defn{restriction} $\restri{\c}{J}$ as the subword of $\c$ consisting of those simple reflections lying in~$J$.  The restriction of the corresponding element $c$ is defined similarly.

\medskip

Finite Coxeter groups have a \defn{bipartite decomposition} $\sref = \sref_L \sqcup \sref_R$ with the property that all reflections in~$\sref_L$ pairwise commute, as do all those in~$\sref_R$.  For irreducible finite Coxeter systems, this decomposition is unique up to the interchanging of $\sref_L$ and  $\sref_R$.  A Coxeter element is called \defn{bipartite} if it is the product of the reflections in~$\sref_L$ followed by the product of the reflections in~$\sref_R$, or vice versa.

We recall N.~Reading's definition of $\c$-sorting words~\cite{Rea2007}, and extend it to $\Artinmon$.

\nathanside{This definition is \emph{much} nicer than the notation suggests.}
\begin{definition}\label{def:c-sorting}
  Let~$w$ be an element of~$W$ or~$\Artinmon$.
  The \defn{$\c$-sorting word} $\sw{w}{c}=\restri{\c}{I_1} \cdots \restri{\c}{I_k}$ for~$w$ is the lexicographically first (as a sequence of positions) subword of $\c^{\infty} = (\s_1 \cdots \s_n)^\infty$ that is a reduced expression for~$w$.  (It is common to separate the different copies of $\s_1 \cdots \s_n$ by vertical bars when writing $\c$-sorting words.)
\end{definition}

The $\c$-sorting word is attached to a particular $\sref$-reduced word for~$c$ rather than to~$c$ itself.
However, since all reduced words for~$c$ are commutation equivalent, the different choices of reduced words for a fixed Coxeter element~$c$ give commutation equivalent sorting words.
% \nathanside{Deal with it.}
\Cref{fig:A2,fig:A22} list the $\s\t$-sorting words of the six elements in~$\WA[3]$, and the 13 elements in $\BA[3]$ with exactly two nontrivial Garside factors.

This definition gives the following greedy procedure to compute the $\c$-sorting word of an element $w$ of $W$ or $\Artinmon$:
\[
  \sw{w}{c} = \begin{cases}
                \s\ \sw{u}{\coxrn} &\text{if } s \in \DesSet_L(w) \text{ with } w = su,\\
                \sw{w}{\coxrn} &\text{if } s \notin \DesSet_L(w).
              \end{cases}
\]
That is---for $s$ initial in $c$---the $\c$-sorting word for $w$ begins with $s$ if and only if~$s$ is a left descent of $w$.  The remainder of $\sw{w}{c}$ then coincides with the $\coxrn$-sorting word of an element of shorter or equal length.
\christianside{This is \emph{the} way to think of the sorting word.}

\medskip

The following lemmas summarize many previously known properties of sorting words.  The first lemma is immediate from the greedy procedure for computing $c$-sorting words.

\begin{lemma+}
\label{lem:stayinginitial}
  Let $w=s_1\cdots s_\ell \in \Artinmon$ and let $w \leq v \in \Artinmon$.
  If $\s_1\cdots\s_\ell$ is initial in $\c^\infty$ for a Coxeter element $c \in W$, then $\s_1\cdots\s_\ell$ is also initial in the $\c$-sorting word $\sw{v}{\c}$ of~$v$.
\end{lemma+}

The following lemma is surprisingly difficult to prove.
\nathanside{Even David Speyer needed seven pages.}
\begin{lemma+}[{\cite[Corollary~4.1]{speyer2009powers}}]
\label{lem:wosortingword}
  The $c$-sorting word $\cwo$ of the longest element $\wo \in W$ is initial in $\c^\infty$.
\end{lemma+}

\begin{example}
  We consider two examples in $\WA[5]$ for \Cref{lem:wosortingword}.
  First, let $\c=\s_2\s_4\s_1\s_3\s_5$.  Underlining the letters in $\cwo$ as a subword of $\c^\infty$ gives
  \[
    \scalemath{0.8}{
    \begin{array}{rccccc|ccccc|ccccc|ccccc|c}
       & \underline{\s_2} & \underline{\s_4} & \underline{\s_1} & \underline{\s_3} & \underline{\s_5} & \underline{\s_2} & \underline{\s_4} & \underline{\s_1} & \underline{\s_3} & \underline{\s_5} & \underline{\s_2} & \underline{\s_4} & \underline{\s_1} & \underline{\s_3} & \underline{\s_5} & \s_2 & \s_4 & \s_1 & \s_3 & \s_5 & \cdots\ .\\
    \end{array}
    }
  \]
  The element $\cwo$ is a prefix of $\c^\infty$.
  The Coxeter element $\c=\s_1\s_2\s_3\s_5\s_4$ gives $\cwo$ as the subword
  \[
    \scalemath{0.8}{
    \begin{array}{rccccc|ccccc|ccccc|ccccc|c}
       & \underline{\s_1} & \underline{\s_2} & \underline{\s_3} & \underline{\s_5} & \underline{\s_4} & \underline{\s_1} & \underline{\s_2} & \underline{\s_3} & \underline{\s_5} & \underline{\s_4} & \underline{\s_1} & \underline{\s_2} & \underline{\s_3} & \s_5 & \s_4 & \underline{\s_1} & \underline{\s_2} & \s_3 & \s_5 & \s_4 & \cdots\ \\
    \end{array}
    },
  \]
  which---though not a prefix---is still initial in $\c^\infty$.
\end{example}

We collect the several elementary properties of sorting words for later reference.

\begin{lemma}
\label{lem:reflection_order}
  Let~$c$ be a Coxeter element and let~$s$ be initial in~$c$.
  Let $\c=\c_1$ and $\c_2$ be two reduced $\sref$-words for~$c$ such that the first letter of~$\c$ is~$s$, and let $\coxr$ be a reduced $\sref$-word for~$\coxrn$.   Let $w \in \Artinmon$.
  Then \nathanside{Brace yourself!  This will keep the machine running smoothly.}
  \begin{enumerate}
    \item $\sw{w}{\c_1}\equiv \sw{w}{\c_2}$;
    \label{it:reflection_order1}

    \item $\c$ is initial in $\cwo$;
    \label{it:reflection_order2}

    \item $\psi(\c)$ is final in $\cwo$;
    \label{it:reflection_order3}

    \item $\psi(\cwo) \equiv \cwo[\psi(\c)]$;
    \label{it:reflection_order4}

    \item $\cwo \psi(\cwo) \equiv \c^h$;
    \label{it:reflection_order5}

    \item $\cwo[\rev(\psi(\c))] \equiv \rev(\cwo)$; and
    \label{it:reflection_order6}

    \item $\cwo[\coxr] \equiv \sinv\cwo\psi(\s)$.
    \label{it:reflection_order7}
  \end{enumerate}
\end{lemma}

\begin{proof}
  \eqref{it:reflection_order1} was shown in the discussion after~\Cref{def:c-sorting}.

  \eqref{it:reflection_order2} follows from the fact that $c \leq \wo$ in weak order together with the observation that $\c$ is lexicographically minimal in $\c^\infty$.

  For \eqref{it:reflection_order4}, we note that given two reduced words $\Q$ and $\Q'$ for~$\wo$, $\Q$ lexicographically precedes $\Q'$ inside $\c^\infty$ if and only if $\psi(\Q)$ precedes $\psi(\Q')$ inside $\psi(\c)^\infty$.
  As $\psi(\Q)$ and $\psi(\Q')$ are also reduced words for~$\wo$ and $\cwo$ is lexicographically minimal inside $\c^\infty$, we conclude that $\psi(\cwo)$ is lexicographically minimal inside $\psi(\c)^\infty$.

  For the other items, we rely on \Cref{lem:wosortingword}.

  For \eqref{it:reflection_order7}, we use that $\cwo = \s\q_2\ldots\q_N$ is initial in $\c^\infty$, and thus, $\q_2\ldots\q_N$ is initial in $(\coxr)^\infty$.
  Therefore, \Cref{lem:stayinginitial} implies that $\q_2\ldots\q_N$ is initial in $\cwo[\coxr]$.
  Only one letter---$\psi(s)$---is left to obtain the $\coxr$-sorting word, so that $\cwo[\coxr] \equiv \q_2\ldots\q_N \psi(\s)$.

  \eqref{it:reflection_order3} is obtained from~\eqref{it:reflection_order2} by applying~\eqref{it:reflection_order7}~$n$ times.
  Since~$\c$ is initial in $\cwo$ by~\eqref{it:reflection_order2}, $\overline{\c}\cwo\psi(\c) \equiv \cwo$.

  \eqref{it:reflection_order5} follows from~\eqref{it:reflection_order2},~\eqref{it:reflection_order3} and~\eqref{it:reflection_order4}:
  First, we have that $\cwo$ is initial in~$\c^\infty$ and $\cwo[\psi(\c)]$ is initial in~$\psi(\c)^\infty$.
  As $\psi(\c)$ is final in $\cwo$, this final $\psi(\c)$ is directly followed by the initial copy of $\psi(\c)$ in $\cwo[\psi(\c)]$ and we obtain that $\cwo\cwo[\psi(\c)]$ is initial in~$\c^\infty$.
  As~$\c$ is now final in $\cwo\cwo[\psi(\c)]$ and this $\c$ can only consist of a consecutive copy of~$\c$ inside~$\c^\infty$, this word equals~$\c^k$ for some~$k$.
  As its length is $nh = 2N = 2\lengthS(\wo)$, we conclude that $k=h$.

  Finally, \eqref{it:reflection_order6} is obtained from~\eqref{it:reflection_order2} and~\eqref{it:reflection_order3}. We have that $\cwo[\rev(\psi(\c))]$ is initial in $\big(\rev(\psi(\c))\big)^\infty$ and $\cwo$ is initial in $\c^\infty$.  As $\psi(\c)$ is final in $\cwo$, we deduce that $\rev(\cwo)$ is initial in $\big(\rev(\psi(\c))\big)^\infty$ as well.    For $s\in S$, write $a_s$ for the number of occurrences of $\s$ in $\cwo[\rev(\psi(\c))]$ and $b_s$ for the number of occurrences of $\s$ in $\rev(\cwo)$.

By forgetting all simple reflections other than $\s$ and $\t$, we note that a subword $\Q$ of $\c^\infty$ is initial if and only if it has the property that for each pair of noncommuting simple reflections $\s$ and $\t$ with $\s$ preceding $\t$ in $\c$, the occurrences of $\s$ and $\t$ alternate within $\Q$, starting with~$\s$.  %In particular, the number of occurrences of $\s$ in $\Q$ is equal to or one greater than the number of occurrences of $\t$ in $\Q$.

This applies to both $\cwo[\rev(\psi(\c))]$ and $\rev(\cwo)$. In particular, if $s,t$ are a pair of non-commuting reflections with $s$ appearing before $t$ in $\rev\psi(\c)$, then $a_t \leq a_s \leq a_t+1$, and similarly $b_t \leq b_s \leq b_t+1$.  Whether $a_s=a_t$ or $a_s=a_t+1$ can be determined from the fact that $\rev(\c)=\psi(\rev(\psi(\c)))$ is final in $\cwo[\rev(\psi(\c))]$.  Since $\rev(\c)$ is also final in $\rev(\cwo)$, $b_s=b_t$ if and only if $a_s=a_t$, and $b_s=b_t+1$ if and only if $a_s=a_t+1$.  It follows that the $n$-tuples $a$ and $b$ differ only by an overall additive constant.  The sum of the entries in $a$ is the length of $\cwo[\rev(\psi(\c))]$, and similarly the sum of the entries in $b$ is the length of $\rev(\cwo)$, and these two quantities are therefore equal.  Thus $a=b$, and since $\rev(\cwo)$ and $\cwo[\rev(\psi(\c))]$ are initial in the same word, they must coincide.  
\end{proof}

We leave it to the reader to check \Cref{lem:reflection_order} using \Cref{fig:coxsortex}.
\begin{figure}[t]
  \begin{center}
    \begin{tabular}{c|c|c|c}
      $\c$ & $\psi(\c)$ & $\cwo$ & $\rev(\cwo[\rev(\psi(c))])$\\
      \hline
      $\s_1\s_2\s_3$ & $\s_3\s_2\s_1$ & $\s_1\s_2\s_3\s_1\s_2\s_1$ &
                                  $\s_1\s_2\s_1\s_3\s_2\s_1$ \\
      $\s_1\s_3\s_2$ & $\s_3\s_1\s_2$ & $\s_1\s_3\s_2\s_1\s_3\s_2$ 
                                  %&\s_1\s_3\s_2\s_3\s_1\s_2 \equiv
                                  %\s_3\s_1\s_2\s_1\s_3\s_2 \equiv
                                  &$\s_3\s_1\s_2\s_3\s_1\s_2$ \\
      $\s_3\s_1\s_2$ & $\s_1\s_3\s_2$ & $\s_3\s_1\s_2\s_3\s_1\s_2$
                                  &$\s_1\s_3\s_2\s_1\s_3\s_2$  \\
      $\s_2\s_1\s_3$ & $\s_2\s_3\s_1$ & $\s_2\s_1\s_3\s_2\s_1\s_3$ %\equiv
                                  %\s_2\s_1\s_3\s_2\s_3\s_1 \equiv
                                  %\s_2\s_3\s_1\s_2\s_1\s_3 \equiv
                                  &$\s_2\s_3\s_1\s_2\s_3\s_1$ \\
      $\s_2\s_3\s_1$ & $\s_2\s_1\s_3$ & $\s_2\s_3\s_1\s_2\s_3\s_1$
                                  &  $\s_2\s_1\s_3\s_2\s_1\s_3$\\
      $\s_3\s_2\s_1$ & $\s_1\s_2\s_3$ & $\s_3\s_2\s_1\s_3\s_2\s_3$ %\equiv
                                  & $\s_3\s_2\s_3\s_1\s_2\s_3$
    \end{tabular}
  \end{center}
  \caption{\label{fig:coxsortex}The $\sref$-words for the Coxeter elements in $\WA[4]$ and associated sorting words.}
\end{figure}

\section{The absolute order}

In addition to studying the group~$W$ using the simple reflections~$\sref$, one may also consider words in general reflections~$\refl$.
The \defn{reflection length} of an element~$w \in W$ is the length~$\lengthR(w)$ of a shortest expression for~$w$ as a product of reflections in $\refl$.
An $\refl$-word $\r_1 \cdots \r_p$ is \defn{reduced} if $w = r_1 \cdots r_p$ and $p = \lengthR(w)$.
The \defn{absolute order} $\Abs(W) = (W,\leq_\refl)$ is defined by
\begin{align*}
  u \leqref w &\Longleftrightarrow \text{ there exists } v\in W \text{ with } uv=w \text{ and }\lengthR(u)+\lengthR(v)=\lengthR(w) \\
              &\Longleftrightarrow \text{ there exists } v'\in W \text{ with } v'u=w \text{ and }\lengthR(v')+\lengthR(u)=\lengthR(w).
\end{align*}
The equivalence follows from invariance of $\lengthR$ under conjugation.

\subsection{Dual braid presentations}
\label{sec.dual_braid_rel}

%\christian{NEW!}
The idea of using $\refl$ as a generating set has been exploited to develop alternative presentations for~$W$, $\Artingrp$, and~$\Artinmon$~\cite{BKL1998,brady2001partial,BW2002,Bes2003}.%\nathanside{This was an amazingly good idea.}
We briefly recall this presentation here, following~\cite{Bes2003}.

For a Coxeter element $c \in W$, denote by $\reds(c)$ the set of its reduced $\refl$-words.
There is an action of the classical braid group $\Braidgrp_n$ on~$n$ strands on $\reds(c)$ (not to be confused with the Artin group~$\Artingrp$ of~$W$).
To avoid confusion, set $\boldsymbol\sigma_i$ be the generator of $\Braidgrp_n$ crossing strand $i+1$ over strand $i$, corresponding to $s_i = (i\ i+1) \in \WA[n]$.
This action is called the \defn{Hurwitz action}, and is given by
\[
  \boldsymbol\sigma_i(r_1,\dots,r_n) \mapsto (r_1,\dots,r_{i-1},\ r_{i+1},\ r_{i+1}^{-1}r_ir_{i+1},\ r_{i+2},\dots,r_n) \in \reds(c).
\]
The Hurwitz action action is transitive on $\reds(c)$.  For $c$ a Coxeter element, the Artin group $\Artingrp$ has a \defn{dual braid presentation} 
\[
  \Artingrp = \Big\langle \arefl : \br_1\br_2 = \br_2\br_3 \text{ for } r_1,r_2,r_3 \in \refl \text{ with } r_1r_2=r_2r_3 \leq_\refl c \Big \rangle. 
% r_1,r_2,r_3=r_2^{-1}r_1r_2 \in \refl \text{ such that } r_1r_2 \leq_\refl c \Big\rangle.
\]
%This presentation is called theand depends on the chosen Coxeter element~$c$.
A relation of the form $\br_1\br_2 = \br_2\br_3$ is called a \defn{dual braid relation}.

%%%%%%%%%%%%%%%%%%%%%%%%%%%%%%%%%%%%%%%%%%%%%%%%%%%%%%%%%%%%%%%%%%%%%%%%%%%%%%%%%%%%%
\section{Root systems}
\label{sec:geometry}
%%%%%%%%%%%%%%%%%%%%%%%%%%%%%%%%%%%%%%%%%%%%%%%%%%%%%%%%%%%%%%%%%%%%%%%%%%%%%%%%%%%%%

The study of finite Coxeter groups is closely related to the study of finite root systems.
We refer to \cite[Section~5]{Hum1990} for a detailed treatment of this relation.\nathanside{It's probably fine if you want think of everything using reflections.}

\subsection{Root systems and reflection arrangements}
Let \[\Delta \subseteq \Phi^+ \subseteq \Phi \subset V\] be a \defn{root system} inside a Euclidian vector space~$V$ associated to the finite Coxeter system $(W,\sref)$, where $\Delta = \set{ \alpha_s}{s \in \sref}$ denotes the \defn{simple roots} and $\Phi^+$ denotes the \defn{positive roots}.  We write $|\beta|$ for the positive root in $\{\pm \beta\}$.

As there is not necessarily a unique root system associated to a finite reflection group~$W$, we consider a Coxeter system to be given with a fixed associated root system.
For $J \subseteq \sref$, the \defn{standard parabolic root subsystem} is given by $\Phi_J \eqdef \Phi \cap \R \set{ \alpha_s  }{ s \in J }$ with positive roots $\Phi_J^+ \eqdef \Phi_J \cap \PhiP$.

There is a bijection $\refl \leftrightarrow \Phi^+$ obtained by sending a reflection $r \in \refl$ to the unique positive root~$\beta$ such that $r(\beta) = -\beta$, and write $r = s_\beta$ and $\beta = \alpha_r$ in this case.
We define the corresponding \defn{reflecting hyperplane} $H_r \eqdef \beta^\perp$ as the orthogonal complement of~$\beta$.  The \defn{reflection arrangement} is the collection of all such reflecting hyperplanes, and the \defn{fundamental chamber} is the connected component of the complement of the reflection arrangement defined by the intersection of the positive half spaces corresponding to the simple roots.
This fixes a bijection between the set of connected components $V \setminus \cup_{r \in \refl} H_r$ and~$W$; under this bijection, we may refer to a connected component by its corresponding element.
A \defn{gallery} is a walk on these connected components, where two components are adjacent if they share a common hyperplane, and any hyperplane is crossed at most once.

\begin{figure}[t]
  \begin{center}
    \begin{tikzpicture}[scale=1]
      \draw[line width=0.5mm] (0:2cm) -- (180:2cm);
      \draw[line width=0.5mm] (60:2cm) -- (240:2cm);
      \draw[line width=0.5mm] (120:2cm) -- (300:2cm);

      \draw[line width=0.2mm, dotted] (0:1.5cm) -- (120:1.5cm) -- (240:1.5cm) -- (0:1.5cm);
      \draw[line width=0.5mm] (60:2cm) -- (240:2cm);
      \draw[line width=0.5mm] (120:2cm) -- (300:2cm);

      \draw[line width=0.2mm, ->] (0:0cm) -- ( 30:1.5cm);
      \draw[line width=0.2mm, ->] (0:0cm) -- ( 90:1.5cm);
      \draw[line width=0.2mm, ->] (0:0cm) -- (150:1.5cm);
      \draw[line width=0.2mm, ->] (0:0cm) -- (210:1.5cm);
      \draw[line width=0.2mm, ->] (0:0cm) -- (270:1.5cm);
      \draw[line width=0.2mm, ->] (0:0cm) -- (330:1.5cm);

      \node at (270:1.7cm) {$\gamma$};
      \node at (210:1.7cm) {$\beta$};
      \node at (330:1.7cm) {$\alpha$};

      \draw[fill=gray, opacity=0.2] (0:0cm) -- ([shift=(240:2cm)]0,0) arc (240:300:2cm) -- (0:0cm);

      \node at (  0:2.5cm) {$H_{u}$};
      \node at (240:2.4cm) {$H_{s}$};
      \node at (300:2.4cm) {$H_{t}$};
%       \node at (270:2.2cm) {$\one$};
%       \node at (210:2.2cm) {$s$};
%       \node at (330:2.2cm) {$t$};
%       \node at (150:2.2cm) {$st$};
%       \node at ( 30:2.2cm) {$ts$};
%       \node at ( 90:2.2cm) {$u=sts$};
    \end{tikzpicture}
  \end{center}
  \caption{The root system of type~$A_2$.}
  \label{fig:A2roots}
\end{figure}
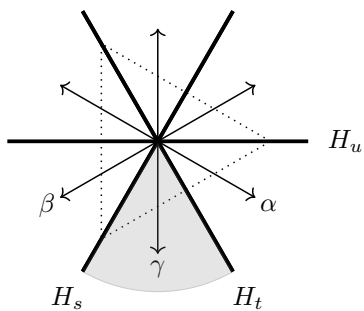
\begin{example}
\label{ex:a2roots}
  The Coxeter group $\WA[3]$ can be realized as the \defn{dihedral group} of isometries of an equilateral triangle.
  Its root system contains the simple and positive roots
  \[
    \Delta = \{ \alpha,\beta\} \quad\text{and}\quad \PhiP = \{ \alpha, \gamma, \beta\},
  \]
  where we take $\alpha = e_2-e_1, \beta = e_3-e_2$, and $\gamma = e_3-e_1 = \alpha+\beta$.
  Here, $e_1,$ $e_2,$ and $e_3$ denote the standard basis of $\mathbb{R}^3$---but, as usual, we restrict to the hyperplane $x+y+z=0$.
  The corresponding reflections are
  \[
    s = s_\alpha, \quad t = s_\beta, \quad \text{and} \quad u = s_\gamma = s_\alpha s_\beta s_\alpha = s_\beta s_\alpha s_\beta.
  \]
%   \nathanside{That's quite explicit!}
  See \Cref{fig:A2roots} for an illustration of this example with shaded fundamental chamber.
\end{example}

\subsection{The geometry of the absolute order}
The \defn{fixed space} of an element $w \in W$ is defined to be
\[
  \fixd(w) \eqdef \operatorname{ker}(w-\id),
\]
while its \defn{moved space} is
\[
  \movd(w) \eqdef \operatorname{im}(w-\id).
\]
As~$W$ acts by orthogonal transformations, we have $\movd(w) = \fixd(w)^\perp$.  The following properties relate reflection length with moved and fixed spaces.
%In order to prove this proposition, we need the following properties relating the reflection length with the moved and fixed space, as introduced in \Cref{sec:finite_reflection_groups}.
%All these are well-known and were proven in~\cite{BW2002}.

\nathanside{Just think about cycle decomposition in the symmetric group.}
\begin{lemma+}[{\cite[Section~2]{BW2002}}]
\label{lem:refllenprops}
  For $w,w_1,w_2,w_3 \in W$ and $r,r_1,r_2\in \refl$, we have
  \begin{enumerate}[(i)]
    \item\label{it:reflenprops1} $\lengthR(w)=\dim\movd(w)$,
    \item\label{it:reflenprops2} $w_1 \leqref w_2 \Rightarrow \movd(w_1) \subseteq \movd(w_2)$ and $\fixd(w_1) \supseteq \fixd(w_2)$,
    \item\label{it:reflenprops3} $w_1 \leqref w_2 \leqref w_3 \Rightarrow w_2^{-1}w_3 \leqref w_1^{-1}w_3$, 
    \item\label{it:reflenprops4} if $r_1\leqref w$ and $r_1\ne r_2$, then $r_1 r_2 \leqref w \Leftrightarrow r_2 \leqref r_1 w$, and
    \item\label{it:reflenprops5} $\movd(r) \subseteq \movd(w) \Rightarrow r \leqref w$.
    \qedhere
  \end{enumerate}
\end{lemma+}

%%%%%%%%%%%%%%%%%%%%%%%%%%%%%%%%%%%%%%%%%%%%%%%%%%%%%%%%%%%%%%%%%%%%%%%%%%%%%%%%%%%%%
\section{Inversions and colored inversions}
\label{sec:inversions}
%%%%%%%%%%%%%%%%%%%%%%%%%%%%%%%%%%%%%%%%%%%%%%%%%%%%%%%%%%%%%%%%%%%%%%%%%%%%%%%%%%%%%

We recall how Coxeter groups act on root systems, and extend this action to the corresponding Artin monoid.

\subsection{Inversion sets for Coxeter groups}
\label{sec:inv_sets}
%of positive roots sent to negative roots by~$w^{-1}$
The \defn{inversion set} of an element $w \in W$ is the set $\inv(w) \eqdef \PhiP \cap w(\Phi^-)$.
% \footnote{Often, the inversion set is defined as $\PhiP \cap w^{-1}(\Phi^-)$ but we chose this form in agreement with the right weak order in Eq.~\eqref{eq:rightweakorder}.}
The Coxeter length of an element $w \in W$ is equal to the cardinality of its inversion set.
The choice of a reduced $\sref$-word $\s_1 \cdots \s_\ell$ for $w \in W$ induces a total order on the inversion set of~$w$
\begin{equation}
  \invs(\s_1 \cdots \s_\ell) \eqdef \big( \alpha_{s_1}, s_1(\alpha_{s_2}),\ldots, s_1 \cdots s_{\ell-1}(\alpha_{s_\ell})\big),\label{eq:inv_seq}
\end{equation}
so that
\[
  \inv(w) = \big\{ \alpha_{s_1}, s_1(\alpha_{s_2}), \ldots, s_1 \cdots s_{\ell-1}(\alpha_{s_\ell}) \big\} \subseteq \PhiP.
\]
The right weak order in~$W$ can be described in terms of inversion sets as
\begin{equation}
  w \le w' \Leftrightarrow \inv(w) \subseteq \inv(w'). \label{eq:rightweakorder}
\end{equation}
Subsets of positive roots that are inversion sets of elements were characterized by P.~Papi in~\cite{Pap1994}.
A subset $A \subseteq \PhiP$ is called \defn{biclosed} if, for $\alpha, \beta, \gamma \in \PhiP$ such that $\gamma = a\alpha + b\beta$ with $a,b \in \R_+$, we have
\begin{equation}
  \alpha,\beta \in A \Rightarrow \gamma \in A \text{ and } \gamma \in A \Rightarrow (\alpha \in A \text{ or } \beta \in A ). \label{eq:biclosed} 
\end{equation}
The map $w \mapsto \inv(w)$ is a bijection between elements in~$W$ and biclosed subsets of~$\PhiP$.

\subsection{Root orders and reflection orders}

The inversion sequence of a reduced $\sref$-word for~$\wo$ specifies a total order on all positive roots which we call a \defn{root order}.
Root orders on~$\PhiP$ were shown by P.~Papi to be exactly those orderings of~$\PhiP$ with the property that if $\alpha,\beta,\gamma=a\alpha+b\beta \in\PhiP$ with $a,b \in \R_+$ then~$\gamma$ is in between~$\alpha$ and~$\beta$ in the ordering, see~\cite{Pap1994}.
Every such root order induces a corresponding reflection order on all reflections.
Recall that an order~$\prec$ on~$\refl$ is a \defn{reflection order} as defined by M.~Dyer in~\cite[Definition~2.1]{Dye1993} if for any two reflections $r_1,r_2 \in \refl$ such that the positive roots $\alpha_{r_1},\alpha_{r_1r_2r_1},\ldots,\alpha_{r_2r_1r_2},\alpha_{r_2}$ are in the positive real span of $\alpha_{r_1}$ and $\alpha_{r_2}$, either
\[
  r_1 \prec r_1r_2r_1 \prec \ldots \prec r_2r_1r_2 \prec r_2
\]
holds, or the same statement with $r_2$ and $r_1$ swapped holds.
The reflection orders on~$\refl$ are exactly those orders that come from root orders.

We use the symbol $\leq_{\c}$ for the total order $\invs(\cwo)$ on positive roots and for the corresponding reflection order.  A reflection order is compatible with a Coxeter element~$c$ if it is commutation equivalent to $\invs(\cwo)$ for some $\sref$-word $\s_1\cdots\s_n$ for $c = s_1\cdots s_n$.\nathanside{Order or partial order?  You decide.}

\begin{comment}
\nathan{fix this}
Even though root and reflection orders are total orders, their defining property only depends on roots $\alpha,\beta$ and reflections $r_\alpha,r_\beta$ such that $r_\alpha r_\beta \neq r_\beta r_\alpha$ as otherwise, $\gamma \notin \PhiP$ for any $a,b \in\R_+$ and $r_\alpha = r_\beta r_\alpha r_\beta$. 
We think of root and reflection orders as the partial orders given by the intersection of all orders coming from reduced $\sref$-words for~$\wo$ that are commutation equivalent.\end{comment}

We have the following addition to \Cref{lem:reflection_order}, proven in~\cite[Lemmas~3.7 and~3.8]{RS2011}.
\begin{lemma+}
\label{lem:reflection_order4}
  Let~$c$ be a Coxeter element and let~$s$ be initial in~$c$.
  If $r_1,r_2 \in \refl \cap W_{\langle s \rangle}$ then
  \[
    r_1 <_{c} r_2 \Leftrightarrow r_1 <_{\coxrn} r_2 \Leftrightarrow r_1 <_{\coxsn} r_2 \text{ in } W_{\langle s \rangle}. \qedhere
  \]
\end{lemma+}

\subsection{Inversion sets for positive Artin monoids}
\label{sec:inv-set-pos-Art}
If $\Q$ is not a reduced $\sref$-word for an element in~$W$, the definition of inversion sequence in~\eqref{eq:inv_seq} still makes sense, but may now contain negative or repeated roots.
We keep track of this additional information using \defn{colored positive roots}, defined as the set $\Phi^{(\infty)}\eqdef \set{\beta^{(k)} }{ \beta\in \PhiP, k \in \mathbb{N}}$ consisting of positive roots \defn{colored} by a nonnegative integer superscript.  We write $|\beta^{(k)}|$ for the uncolored positive root $\beta$.

For a set~$T$ of colored positive roots and a parabolic subgroup $W_J$ with $J \subseteq \refl$, we write $T_J \eqdef \set{\beta^{(i)} \in T }{ \beta \in \PhiPJ }$.
A simple reflection~$s \in \sref$ acts on a colored positive root by
\begin{equation}
  s(\beta^{(k)}) \eqdef \begin{cases}
                     \big[s(\beta)\big]^{(k)} & \text{if } \beta \neq \alpha_s \text{ and}\\
                     \beta^{(k+1)} & \text{if } \beta = \alpha_s
                   \end{cases}\ . \label{eq:colored_action}
\end{equation}
The \defn{colored inversion sequence} and \defn{colored reflection sequence} of an $\sref$-word $\Q = \s_1\cdots\s_p$ are defined to be
\begin{equation}
\begin{aligned}
  \invs(\Q) &\eqdef \left( \beta_1^{(m_1)}, \ldots, \beta_p^{(m_p)} \right), \text{ and}\\
  \invs_\refl(\Q) &\eqdef \left( r_1^{(m_1)}, \ldots, r_p^{(m_p)} \right),
\end{aligned} \label{eq:colored_inversion_sequence}
\end{equation}
where $\beta_i^{(m_i)}=s_1\cdots s_{i-1}(\alpha_{s_i}^{(0)})$ and $r_i = r_{\beta_i}$.

%Generalizing~\Cref{eq:coverreflections}, the \defn{colored cover and covering reflections} of an element $w \in \Artinmon$ are
%\begin{equation}
%\begin{aligned}
%  \coveredref(w)  &\eqdef \set{ r^{(k)}_\beta }{ \beta^{(k)}=w(\alpha_s^{(0)}) \text{ for } s \in \DesSet_R(w) }, \hfill \coveringref(w) &\eqdef \set{r^{(k)}_\beta }{\beta^{(k)}=w(\alpha_s^{(0)}) \text{ for }  s \in \AscSet_R(w) },
%\end{aligned}\label{eq:coverreflections}
%\end{equation}

The \defn{colored inversion set} of an element $\bw \in \Artinmon$ is the set
\[
  \inv(\bw) \eqdef \big\{ \beta_1^{(m_1)}, \ldots, \beta_\ell^{(m_\ell)} \big\}
\]
for any reduced $\sref$-word $\s_1 \cdots \s_\ell$ for $\bw$, as in~\eqref{eq:colored_inversion_sequence}.

\begin{lemma}
\label{lem:coloredinversionsArtin}
  The colored inversion set of $\bw \in \Artinmon$ is well-defined.
\end{lemma}\nathanside{Well, defined.}
\begin{proof}
  This follows, as in the case of $w \in W$, from the fact that any two words for $\bw \in \Artinmon$ are obtained from each other by braid moves.
  The situation therefore reduces to checking that the colored inversion set stays unchanged under a braid move, which is immediate.
\end{proof}

The definition of colored inversion sequence can be rephrased as follows.

\begin{lemma}
\label{lem:uncoloredinversions}
  Let $\Q = \s_1\cdots\s_p$ be an $\sref$-word with $\invs(\Q) = \left( \beta_1^{(m_1)}, \ldots, \beta_p^{(m_p)} \right)$.
  Then the following two properties hold:
  \begin{enumerate}[(i)]
    \item\label{it:uncoloredinversions1} For any $1 \leq i \leq p$, we have $\beta_i = \big| s_1\cdots s_{i-1}(\alpha_{s_i}) \big| \in \PhiP$.% for $s_1\cdots s_{i-1} \in W$.
    \item\label{it:uncoloredinversions2} For any $\beta \in \PhiP$ let $1 \leq k_0 < \ldots < k_a \leq p$ be all indices such that $\beta_{k_i} = \beta$.
    Then $m_{k_i} = i$ for $0 \leq i \leq a$.
  \end{enumerate}
\end{lemma}
\begin{proof}
  Item~\eqref{it:uncoloredinversions1} follows from the observation that the action of~$s \in \sref$ in~\eqref{eq:colored_action} differs from the usual action of~$s$ on roots only in the second case where it changes the color instead of the sign.

   We now turn to item~\eqref{it:uncoloredinversions2}.  For $w \in W$ and positive roots $\alpha \neq \beta$,
  \begin{equation}
    |w(\beta)| \neq |w(\alpha)|.\label{eq:wbetaneqwalpha}%\tag{$\star$}
  \end{equation}
 % This follows from the fact that $|w(\gamma)| = |w(\gamma')|$ for any $\gamma,\gamma' \in \Phi$ implies that $|\gamma| = |\gamma'|$ and the above $\beta$ and $\alpha$ are both positive.
  Fix $\beta \in \PhiP$ and let $1 \leq k_0 < \ldots < k_a \leq p$ be the indices for which $\beta_{k_i} = \beta$ as in the statement.  Fix $0\leq i\leq a$.
  By~\eqref{eq:wbetaneqwalpha}, we obtain that, when we calculate the color of $\beta_{k_i}^{(m_{k_i})}=s_1\cdots s_{{k_i}-1}(\alpha_{s_{k_i}}^{(0)})$, we add one exactly in the positions $k_0,\ldots,k_{i-1}$, implying that $m_{k_i} = i$, as desired.
\end{proof}

This description of colored inversion sequences has two immediate consequences.
First, no colored roots inside $\invs(\Q)$ for $\Q = \s_1\cdots\s_p$ appear multiple times, \ie, $|\inv(\bw)| = p$ for $\bw = \bs_1 \cdots \bs_p \in \Artinmon$.
Second, the colored inversion sequence can be computed by first doing the calculation without colors---recording a sequence of positive roots---and then filling in the colors so the colors of each positive root that appears increase consecutively starting from zero.
We may therefore suppress the colors in colored inversion sequences and write
\[
  \invs(\s_1\cdots\s_p) = ( \beta_1, \ldots, \beta_p )
\]
with $\beta_i = | s_1\cdots s_{i-1}(\alpha_{s_i})|$.%\nathanside{Wait, what?}
In the case that $\s_1\cdots\s_p$ is a reduced word for $w \in W$, this agrees with~\eqref{eq:inv_seq}.  To refer to the words $\left( r_1^{(m_1)}, \ldots, r_p^{(m_p)} \right)$ and $(r_1,\ldots,r_p)$, we also use the notation $\r_1^{(m_1)}\cdots\r_p^{(m_p)}$ and $\r_1\cdots\r_p$, respectively.

\begin{example}
\label{ex:inversionsets}
  The reduced $\sref$-word $\s\t\s$ for the element $\wo \in \WA[3]$ has reflection sequence $\invs_\refl(\s\t\s)=\s\u\t$ and inversion sequence $\invs(\s\t\s) = \left(\alpha,\gamma,\beta\right)$.

  The $\sref$-word $\s\t\s\s\t\s$ specifies the element $\bwom[2] \in \BA[3]$ with colored reflection sequence
  \begin{align*}
    \invs_\refl(\s\t\s\s\t\s) &= \s^{(0)}\u^{(0)}\t^{(0)}\t^{(1)}\u^{(1)}\s^{(1)}
  \intertext{and colored inversion sequence}
    \invs(\s\t\s\s\t\s) &= \left(\alpha^{(0)},\gamma^{(0)},\beta^{(0)},\beta^{(1)},\gamma^{(1)},\alpha^{(1)}\right).
  \end{align*}
  In contrast to the situation for the Coxeter group in~\Cref{sec:inv_sets}, colored inversion sets do not generally distinguish positive braids.\nathanside{Bad hash.}
  For example, $\bs\bs\bt\bt \neq \bt\bt\bs\bs$ in $\BA[3]$, even though
  \[
    \inv(\bs\bs\bt\bt) = \inv(\bt\bt\bs\bs) = \{\alpha^{(0)},\alpha^{(1)},\beta^{(0)},\beta^{(1)}\}.
  \]
\end{example}

Nevertheless, we have the following lemma.

\begin{lemma}
\label{lem:coloredinversionsprojection}
  Let $\bw_1,\bw_2 \in \Artinmon$ with $\inv(\bw_1) = \inv(\bw_2)$, and let $w_1, w_2$ denote their images inside~$W$.
  Then $w_1 = w_2$.
\end{lemma}

\begin{proof}
%  Given a word $\s_1\cdots\s_p$ for an element $w \in W$, 
A reduced $\sref$-word for~$w \in W$ can be obtained from a (possibly non-reduced $\sref$-word) $\s_1\cdots\s_p$ by applying braid moves and by removing two consecutive equal letters $\s\s$.
  A word $\Q$ for $w \in W$ is reduced if and only if its colored inversion set $\inv(\Q)$ only has the color~$0$.
  By \Cref{lem:coloredinversionsArtin}, braid moves do not change the colored inversion set.
  On the other hand, removing two consecutive letters $\s\s$ removes the same positive root from the colored inversion set twice, with two consecutive colors.

  Let $\bw \in \Artinmon$ and write $w$ for its image in~$W$.
  Then $\inv( w )$ is given by all positive roots in $\inv( \bw )$ that appear an odd number of times.
  As elements in~$W$ are uniquely determined by their inversion sets, this shows that $w$ only depends on $\inv(\bw)$ and not on~$\bw$ itself.
\end{proof}

%%%%%%%%%%%%%%%%%%%%%%%%%%%%%%%%%%%%%%%%%%%%%%%%%%%%%%%%%%%%%%%%%%%%%%%%%%%%%%%%%%%%%
\section{Shards and the shard intersection order}
\label{sec:intersection_lattice}
%%%%%%%%%%%%%%%%%%%%%%%%%%%%%%%%%%%%%%%%%%%%%%%%%%%%%%%%%%%%%%%%%%%%%%%%%%%%%%%%%%%%%

The set of all parabolic subgroups, when ordered by inclusion, forms a lattice $\PW$.
On the other hand, the \defn{intersection lattice} $\LW$ is the lattice of flats, \ie, of all intersections of the hyperplanes in the reflection arrangement of~$W$ ordered by reverse inclusion.
These two lattices are well-known to be related in the following way.

\begin{theorem+}[\cite{BI1999}]\label{thm:intersection_lattice_parabolic}
  There is an order-preserving bijection \begin{align*}\LW &\cong \PW \\ X &\mapsto W_X \\ V^U &\mapsfrom U, \end{align*} where $W_X \eqdef \bigset{w \in W }{ X \subseteq \fixd(w)}$ and $V^U \eqdef \bigcap_{w \in U} \fixd(w).$
\end{theorem+}

\subsection{Shards}

In~\cite{Rea2011}, N.~Reading defined a delicate slicing procedure on simplicial hyperplane arrangements that cuts hyperplanes into several pieces called \defn{shards}.  We are interested in this construction in the case of a reflection arrangement.
Recall the identification of chambers and elements in~$W$ from~\Cref{sec:geometry}.\nathanside{Seven years bad luck.  Until your next sabbatical.}

A hyperplane~$H$ in a subarrangement of the reflection arrangement is called \defn{basic} if the connected region containing the fundamental chamber is bounded by~$H$ in the subarrangement.
For any two hyperplanes $H,H'$, define $\mathcal{A}(H,H')$ to be the subarrangement consisting of all hyperplanes containing $H \cap H'$.
One says that~$H'$ \defn{cuts}~$H$ if~$H'$ is a basic hyperplane of $\mathcal{A}(H,H')$ while~$H$ is not.
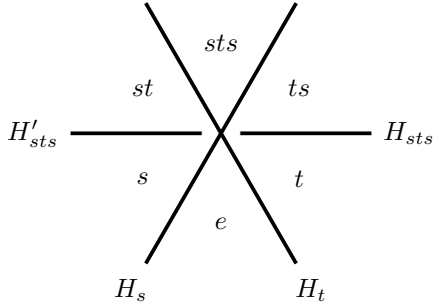
\begin{figure}[t]
  \begin{center}
    \begin{tikzpicture}[scale=1]
      \draw[line width=0.5mm] (0:2cm) -- (180:2cm);
      \draw[white,fill=white] (0,0) circle (.25cm);
      \draw[line width=0.5mm] (60:2cm) -- (240:2cm);
      \draw[line width=0.5mm] (120:2cm) -- (300:2cm);
      \node at (  0:2.5cm) {$H_{sts}$};
      \node at (180:2.5cm) {$H'_{sts}$};
      \node at (240:2.4cm) {$H_{s}$};
      \node at (300:2.4cm) {$H_{t}$};

      \node at (270:1.2cm) {$\one$};
      \node at (210:1.2cm) {$s$};
      \node at (330:1.2cm) {$t$};
      \node at (150:1.2cm) {$st$};
      \node at ( 30:1.2cm) {$ts$};
      \node at ( 90:1.2cm) {$sts$};
    \end{tikzpicture}
  \end{center}
  \caption{The shards of type~$A_2$.}
  \label{fig:shardA2}
\end{figure}
In this way, all hyperplanes are cut into \defn{shards}, defined as the closures of the connected pieces
\[
  H \setminus \bigcup_{H' \text{ cuts } H} H'
\]
of hyperplanes. \Cref{fig:shardA2} illustrates the slicing of the reflection arrangement of type~$A_2$ into shards.

A \defn{lower shard} for an element $w \in W$ is the shard~$\shard$ inside a hyperplane $H_r$ such that $r \in \coveredref(w)$ and the gallery from~$w$ to $r w$ crosses $\shard$.
It was shown in~\cite[Proposition~3.3]{Rea2011} that shards are in bijection with the join-irreducible elements of the weak order $\Weak(W)$.
A shard~$\shard$ is associated to the join-irreducible element~$w$ that has~$\shard$ as its unique lower shard.
In type~$A_2$, the join-irreducible elements are $s,t,st,ts$, which correspond in the obvious way to the four shards shown in \Cref{fig:shardA2}.

\begin{lemma}
\label{lem:gallery_shard}
  Any gallery from~$e$ to~$w$ crosses every lower shard of~$w$.
\end{lemma}

\begin{proof}
  Let~$H$ be the hyperplane corresponding to a lower shard~$\shard$ of~$w$ and consider a different shard $\shard'$ in~$H$.
  Then there is a hyperplane~$H'$ which cuts~$H$ such that~$\shard$ and~$e$ are on one side of~$H'$ and $\shard'$ is on the other side.
  Then it is clear that a gallery from~$e$ to~$w$ will not cross~$H'$ and so will not cross $\shard'$.  On the other hand, since $e$ and $w$ are on opposite sides of $H$, it must cross some shard in $H$.  Therefore it must cross $\shard$.
\end{proof}

\subsection{The shard intersection order}
The \defn{shard intersection order} is the set $\Shard(W)$ of all intersections of shards, ordered by reverse inclusion.
N.~Reading constructed in~\cite[Proposition~4.7]{Rea2011} a bijection between~$W$ and the set of shard intersections.

\begin{theorem+}
\label{thm:lower_shards_and_elements}
  There is a bijection
  \begin{align*}
    \shardbij: W &\bij \Shard(W)\\
    w &\longmapsto \bigcap_{\shard \text{ a lower shard for } w} \shard. \qedhere
  \end{align*}
\end{theorem+}

\Cref{fig:shardA2} shows the correspondence between $s,t,st,ts$ and the four shards.
To complete the bijection, we observe that~$\one$ corresponds to the empty intersection, while~$sts$ corresponds to the intersection $\{0\} = H_s \cap H_t$.

\medskip

Since both shards and shard intersections have purely group-theoretic interpretations as elements of~$W$, it is reasonable to ask for a definition of the shard intersection order directly on the group~$W$.

\begin{definition}
\label{def:shard_on_group}
  The \defn{shard intersection order} on~$W$ is defined by $u \leqsh v$ if
  \[
    \inv(u) \subseteq \inv(v) \text{ and } W_{\coveredref(u)} \subseteq W_{\coveredref(v)}.
  \]
\end{definition}

This definition is justified by the following theorem, which has not previously appeared in the literature.

\begin{theorem}
\label{thm:shard_intersection_order}
  Let $u,v \in W$.
  Then
  \[
    u \leqsh v \quad\Leftrightarrow\quad \shardbij(u) \subseteq \shardbij(v)
  \]
\end{theorem}\nathanside{Inversion set + cover reflections = shards.}

By \eqref{eq:rightweakorder}, $\inv(u) \subseteq \inv(v)$ if and only if $u \leq v$ in weak order---so this describes the shard intersection order on~$W$ explicitly as a weakening of the weak order, as shown in~\cite[Proposition~4.7 (ii)]{Rea2011}.
\begin{figure}[t]
  \begin{center}
    \begin{tikzpicture}[yscale=1.5]
      \tikzstyle{rect}=[rectangle,draw,opacity=.5,fill opacity=1]
      \tikzstyle{sort}=[fill=black!20]

      \node[rect,sort] (e)   at (0,0) {\shortstack{$\one$ \\ $\emptyset$}};
      \node[rect,sort] (1)   at (-3,1) {\shortstack{$s$ \\ $\{s\}$}};
      \node[rect,sort] (12)  at (-1,1) {\shortstack{$st$ \\ $\{u\}$}};
      \node[rect] (21)  at ( 1,1) {\shortstack{$ts$ \\ $\{u\}$}};
      \node[rect,sort] (2)   at ( 3,1) {\shortstack{$t$ \\ $\{t\}$}};
      \node[rect,sort] (121) at ( 0,2) {\shortstack{$sts$ \\ $\{s,t\}$}};
      \draw (e) -- (1)  -- (121);
      \draw (e) -- (2)  -- (121);
      \draw (e) -- (12) -- (121);
      \draw (e) -- (21) -- (121);
    \end{tikzpicture}
  \end{center}
  \caption{The shard intersection order $\Shard(\WA[3])$ with each element labelled by its cover reflections.  The $st$-sortable elements are shaded.  These are defined in \Cref{sec:sortable_elements}.}
  \label{fig:shardorderA2}
\end{figure}
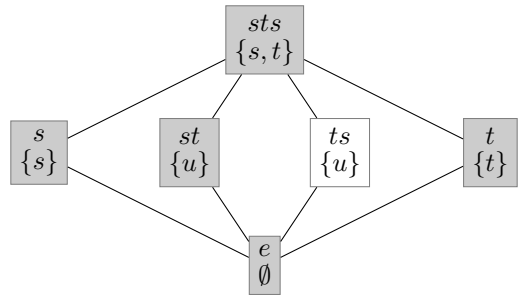
\Cref{fig:shardorderA2} shows the shard intersection order of type~$A_2$ with each element labelled by its cover reflections.

\begin{proof}[Proof of \Cref{thm:shard_intersection_order}]
  We first show that the two properties $\inv(u) \subseteq \inv(v)$ and $W_{\coveredref(u)} \subseteq W_{\coveredref(v)}$ imply $u \leqsh v$.
  By the correspondence in \Cref{thm:intersection_lattice_parabolic} between the lattice of parabolic subgroups and the intersection lattice of the hyperplane arrangement of~$W$, $W_{\coveredref(u)} \subseteq W_{\coveredref(v)}$ implies that 
  \begin{equation}
    \bigcap_{H \in \coveredref(u)} H \supseteq \bigcap_{H \in \coveredref(v)} H.\label{eq:covered_ref}
  \end{equation}
  If we in addition have $\inv(u) \subseteq \inv(v),$ then there exists a gallery from~$e$ to~$u$ to~$v$.
  Let $\shard$ be the union of the set of shards that this gallery crosses, and note that there is exactly one shard for each hyperplane in $\{ H_r \mid r \in \inv(v)\}$ (since each hyperplane is crossed at most once).
  For any chamber~$r$ with lower hyperplane~$H$, the entire facet of~$r$ corresponding to~$H$ is part of the same shard.

  By \Cref{lem:gallery_shard}, taking the intersection of both sides of~\eqref{eq:covered_ref} with $\shard$ implies that the intersection of the lower shards of~$u$ contains the intersection of the lower shards of~$v$.

  We now show that  $u \leqsh v$ implies that $W_{\coveredref(u)} \subseteq W_{\coveredref(v)}$ and $\inv(u) \subseteq \inv(v)$.
  By~\cite[Proposition~5.5]{Rea2011}, we know that $W_{\coveredref(u)} \subseteq W_{\coveredref(v)}$, since sending shards to their containing hyperplane induces an order-preserving map from the shard intersection order to the intersection lattice of~$W$.
  Finally, by~\cite[Proposition~4.7 (ii)]{Rea2011}, sending regions to the intersection of their lower shards is an order-preserving map from the shard intersection order to the weak order.
\end{proof}

We recall the following proposition.

\begin{proposition+}[{\cite[Proposition~1.2]{Rea2011}}]
\label{prop:shard_interval_isomorphism}
  \[
    [\one,w]_{\Shard(W)} \cong \Shard(W_{\coveredref(w)})\cong \Shard(W_{\DesSet_R(w)}). \qedhere
  \]
\end{proposition+}

The isomorphism $[\one,w]_{\Shard(W)} \cong \Shard(W_{\coveredref(w)})$ is given by sending~$u \leqsh w$ to $u_{\coveredref(w)}$.
The bijection $r \mapsto r^{w}$ from $\coveredref(w)$ to $\DesSet_R(w)$
%\begin{align*}
%  \coveredref(w) &\bij \DesSet_R(w) \\
%  r &\longmapsto r^{w}
%\end{align*}
 induces the poset isomorphism
\begin{align*}
  W_{\coveredref(w)} &\cong   W_{\DesSet_R(w)} \\
  u &\mapsto u^{w},
\end{align*}
from which the isomorphism $[\one,w]_{\Shard(W)} \cong \Shard(W_{\DesSet_R(w)})$ follows.

%%%%%%%%%%%%%%%%%%%%%%%%%%%%%%%%%%%%%%%%%%%%%%%%%%%%%%%%%%%%%%%%%%%%%%%%%%%%%%%%%%%%%
\section{The \mhead-eralized weak order}
\label{sec:weak order}
%%%%%%%%%%%%%%%%%%%%%%%%%%%%%%%%%%%%%%%%%%%%%%%%%%%%%%%%%%%%%%%%%%%%%%%%%%%%%%%%%%%%%

By~\eqref{eq:artininjection}, a finite Coxeter group~$W$ injects into its Artin monoid~$\Artinmon$ as the interval
$
  \Weak(W) \cong [\bone,\bwo]_{\Weak(\Artinmon)}.
$
This injection suggests the following $m$-eralization, previously considered by P.~Dehornoy in an enumerative context~\cite{Deh2007}.\nathanside{If you like permutations, then you will \emph{love} the $m$-eralized weak order!}

\begin{definition}
\label{def:mweak}
  The (right) \defn{$m$-eralized weak order} is defined as the interval
  \[
    \Weakm(W)\eqdef[\bone,\bwom]_{\Weak(\Artinmon)}.
  \]
  We denote the elements of $\Weakm(W)$ by $\Wm$.% \eqdef \bigset{\bw \in \Artinmon }{ \bone \leq \bw \leq \bwom}$.
\end{definition}

\Cref{fig:weakA22} illustrates the Hasse diagram of $\Weakm[2](\WA[3])$.
\begin{figure}[t]
  \begin{center}
    \begin{tikzpicture}[scale=1.2]
      \tikzstyle{rect}=[rectangle,draw,opacity=.5,fill opacity=1]
      \tikzstyle{sort}=[fill=black!20]
      \node[rect,sort] (e)   at (0,0) {$\one$};
      \node[rect,sort] (s)   at (-1,1) {$s$};
      \node[rect,sort] (t)   at ( 1,1) {$t$};
      \node[rect,sort] (st)  at (-1,2) {$st$};
      \node[rect] (ts)  at ( 1,2) {$ts$};
      \node[rect,sort] (sts) at ( 0,3) {$sts$};
      \node[rect,sort] (stss)   at ( 1,4) {$sts\cdot s$};
      \node[rect,sort] (stst)   at (-1,4) {$sts\cdot t$};
      \node[rect] (stsst)  at ( 1,5) {$sts\cdot st$};
      \node[rect,sort] (ststs)  at (-1,5) {$sts\cdot ts$};
      \node[rect,sort] (stssts) at ( 0,6) {$sts\cdot sts$};
      \node[rect,sort] (ss)   at (-2,2) {$s\cdot s$};
      \node[rect,sort] (tt)   at ( 2,2) {$t\cdot t$};
      \node[rect] (sst)   at (-3,3) {$s\cdot st$};
      \node[rect] (tts)   at ( 3,3) {$t\cdot ts$};
      \node[rect,sort] (stt)  at (-2,3) {$st\cdot t$};
      \node[rect] (tss)  at ( 2,3) {$ts\cdot s$};
      \node[rect] (stts)  at (-3,4) {$st\cdot ts$};
      \node[rect] (tsst)  at ( 3,4) {$ts\cdot st$};

      \draw (e) to (s) to (st) to (sts) to (stss) to (stsst) to (stssts);
      \draw (e) to (t) to (ts) to (sts) to (stst) to (ststs) to (stssts);
      \draw (s) to (ss) to (sst) to (stst);
      \draw (t) to (tt) to (tts) to (stss);
      \draw (st) to (stt) to (stts) to (ststs);
      \draw (ts) to (tss) to (tsst) to (stsst);
    \end{tikzpicture}
  \end{center}
  \caption{The $m$-eralized weak order $\Weakm[2](\WA[3])$ for $m=2$.  The $st$-sortable elements are shaded.  These are defined in \Cref{sec:sortable_elements}.}
  \label{fig:weakA22}
\end{figure}
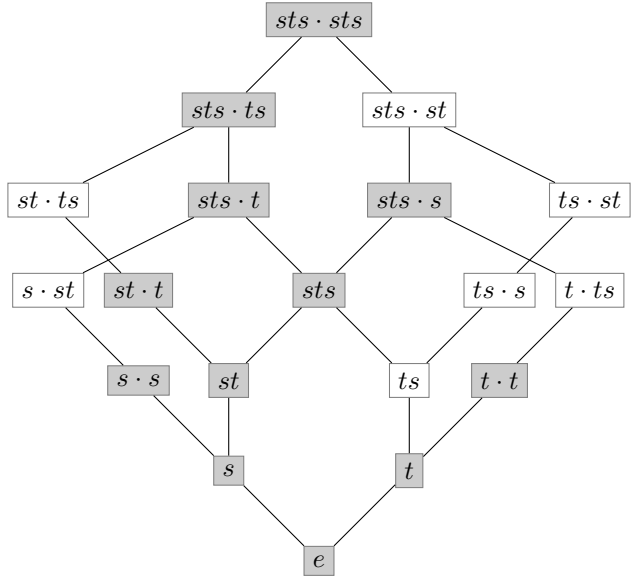

\begin{remark}
\label{rem:weakm_in_B+}
  Both $\Weak(W) = \Weakm[1](W)$ and $\Weak(\Artinmon) = \Weakm[\infty](W)$ are known to have beautiful rank-generating functions, which follow respectively from invariant theory~\cite[Section 3.9]{Hum1990} and from an inclusion-exclusion argument~\cite[Lemma 2.1]{SD2008}.
  The corresponding rank-generating functions are given by
\[
    \Weak(W;q)          = \prod_{i=1}^n [d_i]_q \quad \text{ and }\quad
    \Weak(\Artinmon;q) = \left(\sum_{J \subseteq S} (-1)^{|J|} q^{\lengthS(w_0(J))}\right)^{-1}.
\]
%  \Cref{prop:weak_degree_characterization} and \Cref{thm:garsidefactorization} allow small examples of $\Weakm(W;q)$ to be computed~\cite{Deh2007}. % but we do not know of any similar formula for the rank-generating function for $\Weakm(W)$ when $1<m<\infty$.
  The rank generating function $\Weakm[2](\WA[4];q)$ is an irreducible polynomial over $\QQ$ of degree~$12$, suggesting that no nice formula may exist for general $m$.
\end{remark}

\subsection{The \mhead-eralized weak order and Garside degree}
We now show that consideration of the Garside degree of $\bw \in \Artinmon$ is enough to determine membership in $\Wm$.    We first require a technical lemma, extending~\cite[Lemma~2.10]{EM1994} to all spherical Artin monoids.

\begin{lemma}
\label{lem:lemma210}
  Let $\bu,\bw \in \Artinmon$ such that $\bu \bw \geq \bwo \in \Artinmon$, and let $\bw^{(1)}$ be the first Garside factor of~$\bw$.
  Then $\bu \bw^{(1)} \geq \bwo$.
\end{lemma}
\begin{proof}
  We argue by induction on $\lengthS(\bu)$, the base case when $\bu=\bone$ being trivial.
  Otherwise, let $\bu=\bu'\bs$ and $\bw'=\bs\bw$ for some $\bs \in \asref$, so that $\bu'\bw' = \bu\bw \geq \bwo$.   By induction, $\bu' (\bw')^{(1)} \geq \bwo$.
  Since $\bs \in \DesSet_L(\bw')$, the first Garside factor $(\bw')^{(1)}=\bw' \wedge \bwo$ has $\bs \in \DesSet_L((\bw')^{(1)})$.
  We may therefore write $(\bw')^{(1)}=\bs \bv$ for some $\bv \leq \bwo$ with $\bv \leq \bw$.
  Since $\bw^{(1)}=\bw \wedge\bwo$, $\bv \leq \bwo$, and $\bv \leq \bw$, we have that $\bv \leq \bw^{(1)}$.
  Then $\bu\bw^{(1)} \geq \bu \bv = \bu' (\bw')^{(1)} \geq \bwo$.
\end{proof}

We now extend~\cite[Theorem~2.11]{EM1994} to all spherical Artin monoids.

\begin{proposition}
\label{prop:weak_degree_characterization}
  Let $\bu \in \Artinmon$.
  Then $\bu \in \Wm$ if and only if~$\bu$ has Garside degree at most~$m$.
\end{proposition}\nathanside{So you can think of an element of $\Wm$ as an $m$-tuple of elements in $W$ satisfying \Cref{thm:garsidefactorization}.}

\begin{proof}
  We show that
  \begin{enumerate}[(i)]
    \item $\bw \leq \bwom[r] \Rightarrow \deg(\bw) \leq r$ and \label{it:garside1}
    \item $\deg(\bw) = r \Rightarrow \bw \leq \bwom[r]$. \label{it:garside2}
  \end{enumerate}
  We prove~\eqref{it:garside1} by induction on~$r$.
  The base case $r=0$ is trivial, so we assume that $r \geq 1$ and let $\garside(\bw) = \bw^{(1)}\cdot \bw^{(2)}\cdot\ \cdots\ \cdot \bw^{(k)}$.
  Since $\bw \leq \bwom[r]$, there exists $\bu \in \Artinmon$ with $\bu \bw = \bwom[r]$.
  Then $\bu \bw \geq \bwo,$ so that $\bu \bw^{(1)} = \bwo \bu'$ by \Cref{lem:lemma210}.
  We can therefore write $\bu \bw = \bwo \bu' \bw^{(2)} \cdots \bw^{(k)},$ from which it follows that $\bu' \bw^{(2)} \cdots \bw^{(k)} = \bwom[r-1]$, giving $\bw^{(2)} \cdots \bw^{(k)} \leq \bwom[r-1]$.
  By induction, we obtain that $k-1 \leq r-1$ implying that $\deg(\bw) = k \leq r$.

  On the other hand, we show~\eqref{it:garside2} as follows.
  Let $\garside(\bw) = \bw^{(1)}\cdot \bw^{(2)}\cdot\ \cdots\ \cdot \bw^{(r)}$
  and let $\bu_{1} \in \Artinmon$ be the element such that $\bu_1 \bw^{(1)} = \bwo \in \Artinmon$.
  Push this copy of~$\bwo$ to the right to obtain the factorization
		\[\bu_1 \bw = (\bw^{(2)})^{\bwo} (\bw^{(3)})^{\bwo} \cdots  (\bw^{(k)})^{\bwo} \bwo.\]
  Iterating, we obtain the desired conclusion.
\end{proof}

\subsection{Lattice properties of \mhead-eralized weak order}

We conclude the discussion of $\Weakm(W)$ with some of its poset-theoretic properties.

\begin{theorem}
\label{thm:weak_lattice}
  $\Weakm(W)$ is a rank-symmetric, self-dual lattice.
\end{theorem}

Just as the rank-symmetry of $\Weak(W)$ is demonstrated using the anti-auto\-morphism given by acting by the longest element~$\wo$, we prove \Cref{thm:weak_lattice} by showing that $\Weakm(W)$ has an anti-automorphism given by acting by~$\bwom$.

\begin{lemma}
\label{lem:leftandright}
  The left and the right factors of $\bwom \in \Artinmon$ coincide.
\end{lemma}

\begin{proof}
  For $\bs \in \asref$, it holds that $\bwom \bs =  \psi^m(\bs) \bwom \in \Artinmon$
  Any reduced expression $\bu \bs_1 \cdots \bs_k = \bwom \in \Artinmon$, $\psi^m(\bs_1) \cdots \psi^m(\bs_k) \bu$ is again a reduced expression for~$\bwom$.
\end{proof}

\begin{lemma}
\label{lem:antiautoleftright}
  The map $\bw \mapsto \bwom \overline{\bw}$ on $\Artinmon$ is an anti-isomorphism from right to left weak order when restricted to $\Weakm(W)$.
\end{lemma}

\begin{proof}
  Set $\phi(\bw) = \bwom \overline{\bw}$.
  By \Cref{lem:leftandright}, $\bw \in \Wm$ is both initial and final in~$\bwom$.
  We also have
  \[
    \lengthS(\bwom \overline{\bw}) = \lengthS(\overline{\bw} \wom) = \lengthS(\bwom)-\lengthS(\bw) = mN - \lengthS(\bw).
  \]
  To show that the map reverses the order, it suffices to consider the case for $\bw \leq \bw \bs \in \Wm$.
  Then $\lengthS(\phi(\bw)) = \lengthS(\bwom)-\lengthS(\bw)$ and
  \[
    \lengthS(\phi(\bw \bs)) = \lengthS(\bwom)-\lengthS(\bw \bs) = \lengthS(\bwom)-\lengthS(\bw)-1,
  \]
  so that $\lengthS(\phi(\bw \bs))+1 = \lengthS(\phi(\bw))$.
  Furthermore,
  \[
    \phi(\bw \bs) = \bwom \overline{\bs} \overline{\bw} =  \overline{\psi^m(\bs)} \bwom \overline{\bw} = \overline{\psi^m(\bs)} \phi(\bw),
  \]
  so that $\phi(\bw) = \psi^m(\bs) \phi(\bw \bs)$.
\end{proof}

Representing an element of~$\Artinmon$ as an $\sref$-word, the reverse map converts between left weak order and right weak order.

\begin{lemma}
\label{lem:autoleftright}
  The map~$\rev$ is an isomorphism from left weak order to right weak order that preserves the interval $\Weakm(W)$.
\end{lemma}

\begin{proof}
  Since $\rev(\bs\bw)=\rev(\bw)\bs$, $\rev$ is an isomorphism from left weak order to right weak order.
  If $\bw \leq \bwom$, then there exists $\bu \in \Artinmon$ such that $\bu\bw = \bwom$ so that $\rev(\bw)\rev(\bu)=\rev(\bu\bw)=\rev(\bwom)=\bwom$ and $\rev(\bw)\leq \bwom$.
\end{proof}

\begin{proposition}
  The composition $(\rev \circ \phi)$ is an anti-isomorphism of the $m$-eralized weak order $\Weakm(W)$.
\end{proposition}

\begin{proof}
  This follows from \Cref{lem:antiautoleftright,lem:autoleftright}.
\end{proof}

\begin{proof}[Proof of \Cref{thm:weak_lattice}]
  Since it is an interval, $\Weakm(W)$ inherits the lattice property of \Cref{thm:artinmon_lattice} from $\Weak(\Artinmon)$.   Self-duality and rank-symmetry are a direct consequence of the existence of the anti-isomorphism $(\rev \circ \phi)$ and the fact the $\Weakm(W)$ is graded by length.
  %Finally, $\Weakm(W)$ is rank-symmetric since it is graded by the Coxeter length.
\end{proof}
%%  Any interval in a lattice is again a lattice, thus $\Weakm(W)$ is a lattice by \Cref{thm:artinmon_lattice}.

%\section{Parabolic decompositions in $\Wm$}

%For an Artin system $(B,\sref)$, the \defn{parabolic subgroup} $B_{J}$ is the subgroup of~$B$ generated by~$J$.
%The \defn{positive parabolic submonoid} is denoted $\Artinmon_{J}$.

%The parabolic quotient of the $m$-weak order is
%\[
%  W_{(m)}^J\eqdef\Wm \cap B_+^J
%\]
%with induced weak order denoted $\Weakm(W^J)$.

%\begin{remark}
%  Parabolic quotients are not well-behaved in~$\Wm$.
%  Although the minimal coset representatives have a natural (left) order, there is generally not a unique maximal element, unlike in the situation of $\Weak(W)$.
%  Moreover, because of the restriction that elements in $\Weakm(W)$ have only~$m$ Garside factors, not all cosets have the same size.
%\end{remark}

%%%%%%%%%%%%%%%%%%%%%%%%%%%%%%%%%%%%%%%%%%%%%%%%%%%%%%%%%%%%%%%%%%%%%%%%%%%%%%%%%%%%%
\chapter{Subword complexes}
\label{sec:chapter_subword_complexes}
%%%%%%%%%%%%%%%%%%%%%%%%%%%%%%%%%%%%%%%%%%%%%%%%%%%%%%%%%%%%%%%%%%%%%%%%%%%%%%%%%%%%%

In this chapter, we define and study several variants of subword complexes.\nathanside{We're not done with the backgound?}
We begin by reviewing properties of simplicial complexes (\Cref{sec:simplicial}).
Additional background can be found in~\cite[Section 9]{bjorner1995topological} and its references.
We then extend the theory of subword complexes to positive Artin monoids (\Cref{sec:subword_complexes,,sec:root_configurations}), and introduce dual subword complexes (\Cref{sec:dual_subword_complexes}) and Coxeter initial subword complexes (\Cref{sec:initial_subword_complexes}).
We conclude with topological properties of Coxeter initial subword complexes (\Cref{sec:topologysubword}).

\section{Simplicial complexes}
\label{sec:simplicial}

A \defn{simplicial complex} with ground set $[m] = \{1,\ldots,m\}$ is a collection $\mathcal{C} \subset 2^{[m]}$ such that
\[
  A \subseteq B \in \mathcal{C} \Rightarrow A \in \mathcal{C}.
\]
The elements of $\mathcal{C}$ are called \defn{faces}, the containment-wise maximal faces are called \defn{facets}.  We regularly identify a simplicial complex with its set of facets. 

Faces of cardinality one are called \defn{vertices}, and we denote a vertex $\{i\}$ simply by~$i$.
Every face~$F \in \mathcal{C}$ together with all its subfaces $G \subseteq F$ again forms a simplicial complex $2^F$.
Simplicial complexes of this form are called \defn{simplices}.

The \defn{dimension} of a face $F$ is given by $\dim(F) \eqdef |F| -1$ and the dimension of $\mathcal{C}$ is the maximal dimension among its faces.
The complex $\mathcal{C}$ is \defn{pure} if all its facets have the same dimension, and it is called \defn{flag} if all its containment-wise minimal non-faces have cardinality two.
The \defn{one-skeleton} of a simplicial complex is its subcomplex consisting of its $0$- and $1$-dimensional faces.
A simplicial complex is flag if and only if it is the clique complex of its one-skeleton.

\subsection{Shellability and vertex-decomposability}

If $\mathcal{C}$ is pure, than it is said to be \defn{shellable} if there is a linear \defn{shelling order} $F_1,\ldots,F_k$ of its facets such that for $2 \leq i \leq k$, we have
\[
  F_i \cap \Big( \bigcup_{j < i} F_j \Big)
\]
is a nonempty union of facets of the boundary of the simplex~$2^{F_i}$.
A shellable simplicial complex is contractible or has the homotopy type of a wedge of spheres.

\medskip

Define the \defn{deletion} of a vertex~$i$ of $\mathcal{C}$ by
\[
  \del_i(\mathcal{C}) \eqdef \set{ B \in \mathcal{C}}{ i \notin B},
\]
and the \defn{link} of $i$ by
\[
  \link_i(\mathcal{C}) \eqdef \set{ B \in \mathcal{C}}{ i \notin B, B \cup \{i\} \in \mathcal{C}}.
\]
Finally, a pure complex $\mathcal{C}$ of dimension~$d$ is called \defn{vertex-decomposable} if either $\mathcal{C} = \{\emptyset\}$, or if there exists a vertex~$i \in\mathcal{C}$ such that
\begin{enumerate}[(i)]
  \item $\del_i(\mathcal{C})$ is pure of dimension~$d$ and again vertex-decomposable, and
  \item $\link_i(\mathcal{C})$ is pure of dimension~$d-1$ and again vertex-decomposable.
\end{enumerate}

The lexicographic order on facets of a vertex-decomposable complex (regarded as increasing tuples of vertices) induced by the vertex ordering arising from the decomposition is a shelling order.
We refer to~\cite{BW1996,BW1997} and the references therein for a detailed treatment of shellable and vertex-decomposable simplicial complexes.

\section{Subword complexes}
\label{sec:subword_complexes}

Subword complexes were introduced by A.~Knutson and E.~Miller in~\cite{KM2005,KM2004} in the context of Gr\"obner geometry.  
In this monograph, we consider the following generalization.

%\begin{definition}
%\label{def:subwordsSKM}
%  Let $\Q=\s_1\cdots \s_p$ be an $\sref$-word and let $w \in W$.
%  The \defn{subword complex} $\subwordsS(\Q,w)$ is the simplicial complex with facets given by all subsets of (positions of) letters in~$\Q$ whose complements form a reduced $\sref$-word for~$w$.
%\end{definition}

\begin{definition}
\label{def:subwordsS}
  Let $\Q=\s_1\cdots \s_p$ be an $\sref$-word, let $w \in W$, and let $a = \lengthS(w)+2g$ for some $g \geq 0$.
  The \defn{subword complex} $\subwordsS(\Q,w,a)$ is the simplicial complex with facets given by subsets of (positions of) letters in~$\Q$ whose complements form an $\sref$-word for~$w$ of length~$a$. We call $\Q$ the \defn{search word} and we write $\subwordsS(\Q,w)$ for $\subwordsS(\Q,w,\lengthR(w))$.
\end{definition}

 When $g=0$, $\subwordsS(\Q,w)$ recovers A.~Knutson and E.~Miller's original definition.   Subword complexes for $g=0$ were studied in connection to Catalan combinatorics in~\cite{SS2012,CLS2011,PS20112}, but related objects have appeared in a number of contexts, including in S.~Billey, W.~Jockush, and R.~Stanley's remarkable formula for Schubert polynomials~\cite{billey1993some} and in J.~Morse and A.~Schilling's crystal model for~$\mathfrak{sl}_n$~\cite{morse2016crystal}.

\medskip

When two search words are commutation equivalent, they evidently yield isomorphic complexes (see \cite[Proposition~3.8]{CLS2011}).\nathanside{Deal with it.}

\begin{proposition+}
\label{prop:isomorphicsubwordcomplexes1}
  Let $\Q \equiv \Q'$ be commutation equivalent $\sref$-words.
  Then there is a canonical isomorphism $\subwordsS(\Q,w,a) \cong \subwordsS(\Q',w,a)$.  This is given by the canonical identification of letters in $\Q$ and $\Q'$.  (To have this canonical identification, we forbid two copies of the same letter from commuting.)
\end{proposition+}

\begin{example}
  \label{ex:subwordA22-first}
%   Let $\Q = \s\t\s\t\s$ be an $\sref$-word for $\WA[3]$ and let $\wo = sts = tst$.
  The facets of the subword complex $\subwordsS(\s\t\s\t\s,\wo)$ for $\WA[3]$ are given by
  \[
    \big\{ \{1,2\},\{2,3\},\{3,4\},\{4,5\},\{1,5\}\big\},
  \]
  as the complements of the subword in positions $i$ and $i+1$ for $i \in \{1,2,3,4\}$ give the reduced $\sref$-word $\s\t\s$, while the complement of the subword in positions $1$ and $5$ gives the reduced word $\t\s\t$.
  The first column of \Cref{fig:assocA22} on page~\pageref{fig:assocA22} lists the $12$ facets of the subword complex $\subwordsS(\s\t\s\t\s\t\s\t,e,6)$ for $\WA[3]$.
%   Let $\Q=$ be an $\sref$-word for $\WA[3]$, and let $w=e=\wom[2]$ and $a=6$.
%   The  are listed in the first column of .
\end{example}

\subsection{Alternative possible definitions for subword complexes}
\label{sec:alternative_subword_complexes}

As we have seen in \Cref{ex:inversionsets}, an element of $\Artinmon$ is not uniquely determined by its colored inversion set, which suggests two further generalizations of \Cref{def:subwordsS} for a given search word $\Q$:%=\s_1\cdots \s_p
\begin{enumerate}[$\bullet$]
  \item For $\bw \in \Artinmon$, let $\subwordsSB(\Q,\bw)$ be the \defn{Artin subword complex} with facets given by subsets of (positions of) letters in~$\Q$ whose complements form an $\sref$-word for~$\bw$.
  \item For a set~$X$ of colored positive roots, let $\subwordsSC(\Q,X)$ be the \defn{inversion set subword complex} with facets given by (positions of) letters in~$\Q$ whose complement is a word for a $\bw \in \Artinmon$ with $\inv(\bw) = X$.
\end{enumerate}

We choose to work with \Cref{def:subwordsS} because central properties of subword complexes do not generally hold for the other definitions (\Cref{lem:flipsfromrootvector}, \Cref{prop:dualsubwordcomplexes}).  In the special circumstances of \Cref{sec:initial_subword_complexes}, all three will coincide.

\begin{proposition}
\label{prop:subwordcomplexesinjection}
  Let $\Q=\s_1\cdots \s_p$ be a search word and let $\bw \in \Artinmon$.
  There are injections
  \[
    \subwordsSB(\Q,\bw) \hookrightarrow \subwordsSC(\Q,\inv(\bw)) \hookrightarrow \subwordsS(\Q,w, \lengthS(\bw)),
  \]
  where $\inv(\bw)$ is the colored inversion set of~$\bw$ and $w$ is the projection of~$\bw$ in~$W$.
\end{proposition}
\christianside{They want to be the same, but they simply aren't in general.}

We will discuss in \Cref{sec:initial_subword_complexes} that these notions of subword complex coincide in the situations we consider in this monograph.

\begin{proof}[Proof of \Cref{prop:subwordcomplexesinjection}]
  The first injection is clear, as the facets of $\subwordsSC(\Q,\inv(\bw))$ are given by the disjoint union of the facets of all $\subwordsSB(\Q,\bw')$ such that $\inv(\bw') = \inv(\bw)$.
  On the other hand, this argument also shows that this injection is a bijection if and only if there is a unique $\bw \in \Artinmon$ with $\inv(\bw) = X$ for a given colored inversion set~$X$.

  The second embedding follows from \Cref{lem:coloredinversionsprojection}.
  In this case, one also has that the facets of $\subwordsS(\Q,w,a)$ for $a = \lengthS(\bw)$ are given by the disjoint union of the facets of all $\subwordsSC(\Q,X)$ such that the colored inversion set $X$ has cardinality $a$, contains the roots in $\inv(\bw)$ an odd number of times, and contains all other roots an even number of times (as described at the end of the proof of \Cref{lem:coloredinversionsprojection}).
\end{proof}

\begin{example}
  For $\WA[3]$, for any search word~$\Q$
  \[
    \subwordsSC(\Q,\{\alpha^{(0)},\alpha^{(1)},\beta^{(0)},\beta^{(1)}\}) = \subwordsSB(\Q,\bs\bs\bt\bt) \sqcup \subwordsSB(\Q,\bt\bt\bs\bs),
  \]
  where the disjoint union should be interpreted on the level of facets.
\end{example}

\section{Root configurations}
\label{sec:root_configurations}

\christianside{Skip this section if you know root configurations for subword complexes. It's still the same story.}\nathanside{Same good taste, but now with 50\% less sodium!}
Fix a subword complex $\subwordsS(\Q,w)$ with $\Q=\s_1 \cdots \s_p$. Following~\cite[Definition~3.2]{CLS2011} and~\cite[Definition~3.1]{PS20112}, we modify the colored inversion sequence of the $\sref$-word $\Q$ to account for the choice of a facet $I = \{ i_1 < \ldots < i_{p-a}\} \subseteq \{1,2,\ldots,p\}$.
The \defn{root vector} of~$I$ is the tuple of colored positive roots $\big(\Root{I}{1},\ldots,\Root{I}{p} \big)$ defined by
\[
  \Root{I}{i} \eqdef s_1\cdots\widehat s_{i_1}\cdots\widehat s_{i_k}\cdots s_{i-1}\big(\alpha_{s_i}^{(0)}\big),
\]
%\christian{use widehat everywhere for missing letters}
where the letters $\s_{i_1},\ldots,\s_{i_k}$ with $I \cap [i-1] = \{i_1 < \ldots < i_k\}$ are omitted.
The \defn{root configuration} of the facet~$I$ is the set
\begin{equation}
  \Roots{I} = \bigset{ \Root{I}{i} }{ i \in I }, \label{eq:root_configuration}
\end{equation}
which is totally ordered by~$I$ as a subset of positions of $\Q$.
The root configurations of the facets of the subword complex from \Cref{ex:subwordA22-first} are listed in the fourth column of \Cref{fig:assocA22} on page \pageref{fig:assocA22}.

\begin{lemma}
\label{lem:uncoloredrootvector}
  Let $\Root{I}{\cdot} = \big(\beta_1^{(m_1)},\ldots,\beta_p^{(m_p)}\big)$ be a root vector.
  \begin{enumerate}[(i)]
    \item\label{it:uncoloredrootvector1} The map $\Root{I}{\cdot} : [p]\setminus I \bij \inv( \s_1\cdots\widehat\s_{i_1}\cdots\widehat\s_{i_{p-a}}\cdots\s_{p} )$ is a bijection.
    Moreover, \Cref{lem:uncoloredinversions}\eqref{it:uncoloredinversions2} shows how to recover the colors $m_i$ for $i \notin I$.
    \item\label{it:uncoloredrootvector2} For $i \in I$, the color $m_i$ equals the number of indices $j < i$ with $j \notin I$ and $\beta_j = \beta_i$.
  \end{enumerate}
\end{lemma}

\begin{proof}
  Item~\eqref{it:uncoloredrootvector1} is a direct consequence of the definition.
  Item~\eqref{it:uncoloredrootvector2} follows from the observation that $\Root{I}{i}$ only depends on~$I \cap [i-1]$.
  This means that $\Root{I}{i}$ for $i \in I$ is given by $\beta_i^{(m_i)}$ with $\beta_i = | w(\alpha_{s_i}) | \in \PhiP$ for $w = s_1 \cdots \widehat s_{i_1} \cdots \widehat s_{i_k} \cdots s_{i-1} \in W$ and the color $m_i$ is equal to the number of occurrences of the root $\beta_i$ in position $j$ with $j<i$ and $j \not \in I$.%to the left of position~$i$ inside $I^c$.
\end{proof}

\subsection{Flips}
\label{sub:flips}

We say that two facets~$I$ and~$J$ are \defn{adjacent} if $I \setminus \{i\} = J \setminus \{j\}$ for some positions $1\leq i,j \leq p$ with $i \neq j$.
If~$I$ and~$J$ are adjacent, the \defn{flip} from~$I$ to~$J$ is \defn{increasing} if $i<j$ and \defn{decreasing} if $i>j$.
The \defn{direction} of the flip is given by the root vector $\Root{I}{i}$.  
% These notions directly generalize to $g> 0$.

The root configuration captures the notions of increasing and decreasing flips in subword complexes, described for $g=0$ in~\cite{CLS2011,PS20112}.   Note that when $g=0$, each vertex can be flipped in only one way~\cite{CLS2011, PS20112}.  This is no longer true when $g>0$---it may be possible to find three distinct facets $I$, $J$, $K$ with
$I\setminus\{i\} = J\setminus \{j\}=K\setminus\{k\}$.%, which cannot happen in the setting of~\cite{CLS2011,PS20112}.

\begin{lemma}
\label{lem:rootupdate}
  Let $I$~and~$J$ be two adjacent facets of~$\subwordsS(\Q,w)$ with~$I \setminus i = J \setminus j$ and write $\beta=|\Root{I}{i}|$.
  Then $|\Root{J}{k}|=|s_{\beta}(\Root{I}{k})|$ for $\min(i,j) \le k \le \max(i,j)$ as uncolored roots, and $\Root{J}{k} = \Root{I}{k}$ otherwise.
 % \[
 %   |\Root{J}{k}| = \begin{cases}
 %                     |s_{\beta}(\Root{I}{k})| & \text{if } \min(i,j) \le k \le %\max(i,j),\\
  %                    |\Root{I}{k}| & \text{otherwise}.
   %                 \end{cases}
  %\]
%  and the colors are given as described in \Cref{lem:uncoloredrootvector}.
%Furthermore, $\Root{J}{k} = \Root{I}{k}$ in the second case.
\end{lemma}
\begin{proof}
  This is straightforward from the definition and from the description of colors in~\Cref{lem:uncoloredrootvector}.
\end{proof}

\begin{lemma}[{see also~\cite[Lemma~3.3(2)]{PS20112}}]
\label{lem:roots&flips}
  Let~$I$ be a facet of the subword complex~$\subwordsS(\Q,w,a)$ with $i \in I$, and write $\beta^{(k)}=\Root{I}{i}$.
  If $J$ is adjacent to~$I$ with $I \setminus \{i\} = J \setminus \{j\}$, then~$\Root{I}{j} = \beta^{(\ell)}$ for some color $\ell$.
  Furthermore, $i < j$ if and only if $k \leq \ell$.
\end{lemma}
\begin{proof}
  Assume that $i<j$, the case $i>j$ being completely analogous.
  Let $\s_1\dots\s_p$ be the word obtained from~$\Q$ by removing all letters in $I \cap J = I\setminus i = J\setminus j$ and let the two remaining letters from~$I$ and~$J$ in~$\s_1\dots\s_p$ be $\s_{i'}$ and $\s_{j'}$, respectively.
  Since $i<j$, we also have $i' < j'$.
  By definition, we have $p = a+1$ and
  \[
    w = s_1 \cdots \widehat s_{i'} \cdots s_p = s_1 \cdots \widehat s_{j'} \cdots s_p
  \]
  and obtain $s_{i'+1} \cdots s_{j'-1} s_{j'} = s_{i'} s_{i'+1} \cdots s_{j'-1}$.
  This gives $s_{i'}^\rho = s_{j'}$ for $\rho = s_{i'+1}\cdots s_{j'-1} \in W$, implying that $\alpha_{s_{i'}} = \rho(\alpha_{s_{j'}})$ and
  \[
    | \Root{I}{i} | = | s_1 \cdots s_{i'-1}(\alpha_{s_{i'}})| = | s_1 \cdots s_{i'-1} s_{i'+1} \cdots s_{j'-1}(\alpha_{s_{j'}}) | = | \Root{I}{j} |.
  \]
  The statement about the colors follows from \Cref{lem:uncoloredrootvector}\eqref{it:uncoloredrootvector2}.
\end{proof}

\begin{lemma}
\label{lem:flipsfromrootvector}
  Let~$I$ be a facet of the subword complex~$\subwordsS(\Q,w,a)$ and let $i \in I$ and $j \notin I$ such that $|\Root{I}{i}| = | \Root{I}{j} |$.
  Then $(I \setminus \{i\} ) \cup \{j\}$ is again a facet of $\subwordsS(\Q,w,a)$.
\end{lemma}

\begin{proof}
  The proof is the reverse of the argument in the proof of \Cref{lem:roots&flips}.
  Assume again that $i<j$, the case $i>j$ being completely analogous.
  Let $\s_1\cdots\s_p$ be the word obtained from~$\Q$ by removing all letters in $I\setminus i$, let $\s_{i'}$ be the remaining letter from~$I$ inside~$\s_1\dots\s_p$, and let $\s_{j'}$ with $i' < j'$ be the letter in $\Q$ corresponding to $j \notin I$.
  We then have $w = s_1\cdots \widehat s_{i} \cdots s_p$ and
  \[
    | s_1\cdots s_{i'-1}(\alpha_{s_{i'}}) | = |\Root{I}{i}| = | \Root{I}{j} | = | s_1\cdots \widehat s_{i'} \cdots s_{j'-1}(\alpha_{s_{j'}}) |.
  \]
  Thus, $s_{i'} = \rho s_{j'} \rho^{-1}$ for $\rho = s_{i'+1} \cdots s_{j'-1}$.
  This gives
  \begin{align*}
    w &= s_1 \cdots \widehat s_{i'} \cdots s_p \\
      &= s_1 \cdots s_{i'-1} s_{i'} s_{i'} s_{i'+1} \cdots s_p \\
      &= s_1 \cdots s_{i'} \rho s_{j'} \rho^{-1} \rho s_{j'} \cdots s_p \\
      &= s_1 \cdots \widehat s_{j'} \cdots s_p. \qedhere
  \end{align*}
\end{proof}

\begin{lemma}[{see also~\cite[Lemma~3.4 and Remark~3.5]{PS20112}}]
\label{lem:rootconfigurationdetermination}
  A facet~$I$ of the subword complex~$\subwordsS(\Q,w,a)$ is uniquely determined by its root configuration~$\Roots{I}$.
\end{lemma}
\begin{proof}
  The proof is completely analogous to the proof of \cite[Lemma~3.4]{PS20112}.
  Let $I$~and~$J$ be two distinct facets of~$\subwordsS(\Q,w,a)$, and let~$i$ be the smallest index which is in $I \cup J$ but not in $I \cap J$.
  Assume without loss of generality that $i \in I$.
  We then have that $\Root{I}{i} = \Root{J}{i}$.
  Also, \Cref{lem:uncoloredrootvector}\eqref{it:uncoloredrootvector2} shows that $\Root{J}{i} = \Root{J}{j}$ implies $j \leq i$.
  Consequently, $\Root{I}{i}$ appears at least once more in~$\Roots{I}$ than in~$\Roots{J}$.
\end{proof}

\begin{definition}
  The \defn{increasing flip graph} $\flipGraph$ of $\subwordsS(\Q,w,a)$ is the directed graph whose vertices are the facets of $\subwordsS(\Q,w,a)$ with directed edges given by increasing flips.  The \defn{increasing flip poset} is the transitive closure of the increasing flip graph.
\end{definition}

\Cref{fig:cambassocA22} on page~\pageref{fig:cambassocA22} illustrates the increasing flip poset in~$\WA[3]$ for the subword complex $\subwordsS(\s\t\s\t\s\t\s\t,e,6)$.
Increasing flips are drawn upwards.

\section{Dual subword complexes}
\label{sec:dual_subword_complexes}

\begin{definition}
\label{def:subwordsT}
  Let $\Q$ be an $\refl$-word, let $w \in W$, and let $a = \lengthR(w)+2g$ for some $g \geq 0$.
  The \defn{dual subword complex} $\subwordsR(\Q,w,a)$ is the simplicial complex with facets being all subsets of (positions of) letters in~$\Q$ that are an $\refl$-word for~$w$ of length~$a$.  We write $\subwordsR(\Q,w)$ for $\subwordsR(\Q,w,\lengthR(w))$.
\end{definition}\nathanside{These want to be defined in the dual braid group. They usually can't, unless they're of a very special form\dots}

The analogue of~\Cref{prop:isomorphicsubwordcomplexes1} evidently holds for dual subword complexes.

\subsection{Duality of subword complexes}

\Cref{def:subwordsS,def:subwordsT} are related by the following proposition, which associates to any subword complex an isomorphic dual subword complex.
In fact, both complexes coincide as abstract complexes on the ground set $\{1,\ldots,p\}$.

\begin{proposition}
\label{prop:dualsubwordcomplexes}
  Let $\Q = \s_1\cdots \s_p$ be an $\sref$-word, let $w \in W$, and let $a = \lengthS(w)+2g$ for some $g \geq 0$.
  Then
  \[
    \subwordsS(\Q,w,a) = \subwordsR(\invs_\refl(\Q),w',b),
  \]
  where $b = p-a$ and $w' = w s_p\cdots s_1 \in W$.
\end{proposition}
\begin{proof}
  This proof uses the same argument as the proofs of~\cite[Lemma~3.2]{IS2010} and of~\cite[Proposition~2.8]{CLS2011}.
  Let $\invs_\refl(\Q) = (\r_1,\ldots, \r_p)$ and let $I=\{i_1,\cdots,i_b\}$ with $1 \leq i_1 \leq \ldots \leq i_{b} \leq p$. %, so that
  For each $j$ from $1$ to $b$, replace $\widehat{s}_{i_j}$ by $s_{i_j} s_{i_j}$ and move one copy to the left by conjugation.
  We obtain
  \[
    s_1 \cdots \widehat s_{i_1} \cdots \widehat s_{i_b} \cdots s_p = r_{i_1}\cdots r_{i_b} s_1\cdots s_p.
  \]
  Therefore, the set $\{i_1,\ldots,i_{b}\}$ is a facet of $\subwordsS(\Q,w,a)$ if and only if
  \[
    w = r_{i_1}\cdots r_{i_b} s_1\cdots s_p
  \]
  or, equivalently, $r_{i_1}\cdots r_{i_b} = w s_p \cdots s_1$.
  This is the case if and only if $\{i_1,\ldots,i_b\}$ is a facet of $\subwordsR(\invs_\refl(\Q),w',b)$.
\end{proof}

Even though subword complexes and their associated dual subword complexes are identical because we are working in the Coxeter group $W$, they provide two different viewpoints in $\Artingrp$.  This will be essential in the proof of~\Cref{thm:subwordcomplexesequality}.

\begin{example}
  \label{ex:subwordA22}
  As in \Cref{ex:subwordA22-first}, let $\Q = \s\t\s\t\s\t\s\t$ be an $\sref$-word with $\invs_\refl(\Q) = \s\u\t\s\u\t\s\u$.
  The~$12$ facets of $\subwordsR(\invs_\refl(\Q),su)$ are listed in the second column of \Cref{fig:assocA22} on page~\pageref{fig:assocA22}.
\end{example} 

\section{Coxeter initial subword complexes}
\label{sec:initial_subword_complexes}

We now turn to a class of subword complexes that behave particularly nicely, and which play a central role in the $m$-eralized noncrossing theory.

\begin{theorem}
\label{thm:subwordcomplexesequality}
  Let~$c$ be a Coxeter element with $\sref$-word $\c=\s_1\dots \s_n$, and let $\Q = \s_1\dots \s_n \s_{n+1} \dots\s_p$ be initial in $\c^\infty$.
  Set $\bw = \bs_{n+1}\dots \bs_p \in \Artinmon$ and its projection $w = s_{n+1}\dots s_p \in W$.
  Then
  \[
    \subwordsSB(\Q,\bw) = \subwordsSC(\Q,\inv(\bw)) = \subwordsS(\Q,w,p-n). \qedhere
  \]
\end{theorem}\nathanside{\dots and in that special case it doesn't matter which one you use!  (Learn about the dual braid monoid before reading the proof.)}

We call subword complexes of the form given in this theorem \defn{$c$-initial subword complexes}.

\medskip

We use the relationship between subword and dual subword complexes given in~\Cref{prop:dualsubwordcomplexes} to prove this theorem.
The core is a lift of~\Cref{prop:dualsubwordcomplexes} to the Artin group $\Artingrp$.
This uses a certain embedding of the reflections $\refl \subseteq W$ into~$\Artingrp$ which depends on the choice of a Coxeter element~$c$.
This construction is essentially due to D.~Bessis~\cite{Bes2003}, and we emphasize that this embedding is \emph{not} the injection $\refl \subseteq W \hookrightarrow \Artinmon \subseteq \Artingrp$ defined in~\eqref{eq:artininjection}.

\medskip

Let $\Q = \s_1\dots\s_p$ be initial in $\c^\infty$ and set
\begin{equation}
  \label{eq:braid_reflections}
  \br_i \eqdef \bs_1\dots\bs_{i-1}\bs_i\bs_{i-1}^{-1}\dots\bs_1^{-1} \in \Artingrp
\end{equation}
using the identification between $\sref$ and $\asref$.

\begin{proposition}
\label{prop:braidequality}
  Let~$c = s_1\dots s_n$ be a Coxeter element with $\sref$-word $\c$ and let $\bc = \bs_1\dots\bs_n \in \Artingrp$.
  Let $\Q = \s_1\dots\s_p$ be a word initial in $\c^\infty$ starting with $\s_1\dots\s_n$, and let $\invs_\refl(\Q) = \r_1\dots\r_p$ be the inversion sequence of~$\Q$.
  Let $1 \leq i_1 < \dots < i_n \leq p$ such that $r_{i_n}\dots r_{i_1} = c \in W$.
  Then
  \begin{equation}
    \br_{i_n}\dots \br_{i_1} = \bc \in \Artingrp.
    \label{eq:coxeterinartin}
  \end{equation}
  Equivalently,
  \[
    \bs_1 \dots \widehat \bs_{i_1} \dots \widehat \bs_{i_n} \dots \bs_p = \bc^{-1} \bs_1\dots\bs_p \in \Artinmon.
  \]
\end{proposition}

\begin{proof}
  This proposition follows for bipartite Coxeter elements from~\cite[Theorem~2.2.5]{Bes2003}, we only provide a sketch of the argument here.
  Assume the Coxeter element~$c$ to be bipartite.
  D.~Bessis shows that the elements
  \[
    \arefl = \big\{ \bc^k \bs \bc^{-k} \mid k \in \mathbb{Z}, \bs \in \asref\big\} \subset \Artingrp
  \]
  are in bijection with $\refl \subseteq W$ and that these satisfy the dual braid relations as defined in \Cref{sec.dual_braid_rel}.
  Since~$c$ is bipartite and $\Q$ is initial in~$\c^\infty$, we have $\{ \br_1,\dots,\br_p \} \subseteq \arefl$ for $\br_i$ as defined in~\eqref{eq:braid_reflections}, implying that these also satisfy the dual braid relations.

  We now deduce~\eqref{eq:coxeterinartin} for bipartite Coxeter elements by recalling that the Hurwitz action is transitive on $\reds(c)$. Therefore
  $r_{i_n}\dots r_{i_1}$ and $s_1\dots s_n$ are related by dual braid relations, implying that $\br_{i_n}\dots \br_{i_1}$ and $\bs_1\dots\bs_n$
  are also related by dual braid relations.
  As these relations are satisfied in~$\Artingrp$, we obtain that $\br_{i_n}\dots \br_{i_1} = \bc$, as desired.
  We finally compute
  \[
    \bs_1 \dots \widehat \bs_{i_1} \dots \widehat \bs_{i_n} \dots \bs_p
    = \br_{i_1}^{-1}\dots \br_{i_n}^{-1} \bs_1\dots\bs_p = \bc^{-1} \bs_1\dots\bs_p.
  \]

  The proposition for general Coxeter elements now follows from the claim that if the proposition holds for the Coxeter element~$c$ than it also holds for the Coxeter element $\coxrn$ because every Coxeter element can be obtained from any other by such a procedure.

  We obtain this reduction as follows.
  Let $\Q$ be an $\sref$-word and let $\widehat\Q$ be the word obtained from $\Q$ by removing the first letter.
  Let~$c$ be a Coxeter element with word $\c=\s_1\dots \s_n$.
  \Cref{lem:reflection_order}\eqref{it:reflection_order7} then implies that $\Q$ is initial in $\c^\infty$ if and only if $\widehat\Q$ is initial in $\coxr^\infty$, so we assume both to be initial.

  Given indices $2 \leq i_1 < \dots < i_n \leq p$ such that $r_{i_n}\dots r_{i_1} = c \in W$.
  If the conclusion holds for the Coxeter element~$c$, then
  $
    \br_{i_n}\dots\br_{i_1} = \bc \in \Artingrp.
  $
  We obtain that%for $\widehat\br_i = \bs_1^{-1} \br_i^{\bs_1}$ that
  \[
    \br_{i_n}^{\bs_1}\dots \br_{i_1}^{\bs_1} = \bs_1^{-1}\br_{i_n}\dots\br_{i_1}\bs_1 = \bs_1^{-1}\bc\bs_1 = \bs_2\dots\bs_n\bs_1 \in \Artingrp.
  \]
  This implies the conclusion also for the Coxeter element~$\coxrn$.%a sequence of conjugations by initial simple reflections
\end{proof}

\begin{proof}[Proof of \Cref{thm:subwordcomplexesequality}]
  This is a direct consequence of \Cref{prop:dualsubwordcomplexes} using \Cref{prop:braidequality}.
\end{proof}

\begin{example}
\label{ex:initialsubwordcomplex}
  Let $c = st \in \WA[3]$ be a Coxeter element and consider the word $\Q = \s\t\s\t\s\t$ with reflection sequence $\invs_\refl(\Q) = \s\u\t\s\u\t$.
  \christianside{This example is worth going through to avoid confusion later.}
  Then
  \begin{align*}
    \subwordsR(\invs_\refl(\Q),c^{-1}) &= \{ 12, 15, 23, 26, 34, 45, 56 \} \\
    \subwordsR(\invs_\refl(\Q),c\hspace*{11pt} ) &= \{ 13, 16, 24, 35, 46 \}.
  \end{align*}
  The first is a $c$-initial subword complex and we have
  \begin{align*}
    \subwordsR(\invs_\refl(\Q),c^{-1})
      &= \subwordsS(\Q,c^{-1},4) \\
      &= \subwordsSC(\Q,\{\alpha^{(0)},\gamma^{(0)},\beta^{(0)},\alpha^{(1)}\}) \\
      &= \subwordsSB(\Q,\bs\bt\bs\bt).
  \end{align*}
  On the other hand, $\subwordsR(\invs_\refl(\Q),c) = \subwordsS(\Q,c,4)$ is \emph{not} $c$-initial and does not coincide with the corresponding Artin subword complex.
\end{example}

\section{Topology of Coxeter initial subword complexes}
\label{sec:topologysubword}

When $g=0$, A.~Knutson and E.~Miller showed in \cite[Theorem~2.5 \& Corollary~3.8]{KM2004} that subword complexes are vertex-decomposable spheres or balls.  Subword complexes are generally neither spheres nor balls for $g>0$, and are not even necessarily vertex-decomposable. % The following example shows subword complexes may not be vertex-decomposable.%that both Artin and inversion set

\begin{example}
\label{ex:artinsubwordcomplex}
  Consider the word $\Q = \t\s\s\t\t\s$ for $\WA[3]$ and the element $\bw = \bs\bs\bt\bs = \bs\bt\bs\bt = \bt\bs\bt\bt \in \BA[3]$.
  Then the subword complex is
  \begin{align*}
    \subwordsS(\Q,w,4) &= \{ 14,15,23,26,36,45 \},
    \intertext{while the Artin subword complex is}
    \subwordsSB(\Q,\bw) &= \{ 14,15,26,36 \}.
  \end{align*}
  Since the only words with colored inversion set $\inv(\bw) = \{ \alpha^{(0)},\alpha^{(1)},\beta^{(0)},\gamma^{(0)}\}$ are reduced $\sref$-words for $\bw$, we therefore have $\subwordsSB(\Q,\bw) = \subwordsSC(\Q,\inv(\bw))$.
  These simplicial complexes are disconnected but of positive dimension, and are therefore not vertex-decomposable.
\end{example}

\begin{theorem}
\label{thm:vertex-decomposability}
  Let~$c$ be a Coxeter element with $\sref$-word $\c=\s_1\dots \s_n$, and let $\Q = \s_1\dots \s_n \s_{n+1} \dots\s_p$ be initial in $\c^\infty$.
  Set $w = s_{n+1}\dots s_p \in W$.
  Then the Coxeter initial subword complexes $\subwordsS(\Q,w,p-n)$ is vertex-decomposable.
  Moreover, the lexicographic order on the facets is a shelling order.
\end{theorem}
\christianside{You see---these behave like subword complexes always did.}
\nathanside{word.}

\begin{proof}
  We show that both the link and the deletion of the first vertex are vertex-decomposable subword complexes by simultaneous induction on the rank of~$W$ and on the length of~$\Q$.
  Without loss of generality, we may assume that $s_{n+1} = s_1$.
  The base case of the empty word $\Q$ is clear, so we may assume $p > n$.

  We first address the link.
  The facets of the link of the first vertex of the complex $\subwordsS(\Q,w,p-n)$ are given by
  \[
    \bigset{ I \setminus 1 }{ I \text{ a facet of } \subwordsS(\Q,w,p-n) \text{ with } 1 \in I }.
  \]
  Let $\widehat\Q$ be the restriction of $\Q$ to those letters $\s_i$ with ${}^{s_1 \cdots s_{i-1}}s_i \in W_{\langle s \rangle}$.  Then it is clear that
  \[
    \subwordsR(\Q, s_n\cdots s_2) \cong \subwordsR(\widehat\Q, s_n\cdots s_2).
  \]
  This complex is vertex-decomposable by induction on the rank of~$W$.

  We next consider the deletion.
  Let~$I$ be a facet of $\subwordsS(\Q,w,p-n)$ with $1 \in I$, so that $I \setminus 1$ is a face of the deletion of the vertex~$1$.
  We show that there is a facet~$J$ of $\subwordsS(\Q,w,p-n)$ with $J \setminus j = I \setminus 1$ and $j > 1$.
  Observe that $\Root{I}{1} = \alpha_{s_1}^{(0)}$.
  By \Cref{lem:roots&flips}, we thus have to show that $\alpha_{s_1}$ appears in the inversion sequence of $\Q$ with the letters in positions~$I$ deleted.
  This follows from \Cref{thm:subwordcomplexesequality} because the inversion sequence of $\Q$ with the letters in positions~$I$ deleted is the same as the inversion sequence of $\s_{n+1}\dots\s_p$.
  By construction, $s_{n+1} = s_1$ and thus $\invs_\refl(\s_{n+1}\dots\s_p)$ starts with $\alpha_{s_1}$.
\end{proof}

\begin{corollary}
\label{cor:wedgeofspheres}

  Let~$c$ be a Coxeter element with $\sref$-word $\c=\s_1\dots \s_n$, and let $\Q = \s_1\dots \s_n \s_{n+1} \dots\s_p$ be initial in $\c^\infty$.
  Set $w = s_{n+1}\dots s_p \in W$.
  Let $\widehat\Q$ of length~$q$ be obtained from $\Q$ by removing the longest initial subword that is also initial in $\cwo$.
  Then $\subwordsS(\Q,w,p-n)$ has the homotopy type of a wedge of~$k$ spheres where~$k$ is the number of facets of $\subwordsS(\widehat\Q,w,q-n)$.
\end{corollary}

Observe in this corollary, that $\cwo$ is initial in $\c^\infty$ by \Cref{lem:reflection_order}\eqref{it:reflection_order5}.
Therefore, $\widehat\Q$ is obtained from $\Q$ by removing the intersection of $\cwo$ and $\Q$ as initial segments of $\c^\infty$.

\begin{proof}
  Because the complex is vertex-decomposable and every facet contains~$n$ positions, it has the homotopy type of a wedge of $(n-1)$-dimensional spheres.
  We count the number of such spheres using the technique of \emph{homology facets}, developed by A.~Bj{\"o}rner and M.~Wachs in~\cite{BW1996} (see also~\cite[Theorem~3.1.3]{Wac2007}).
  The number of spheres is given by the number of facets whose entire boundary is contained in the union of the earlier facets in the lexicographic shelling order of \Cref{thm:vertex-decomposability}.

  We show that the homology facets of $\subwordsS(\Q,w,p-n)$ are exactly those facets that do not contain a position corresponding to one of the letters in the initial copy of $\cwo$.
  These facets are in canonical bijection with facets of $\subwordsS(\widehat\Q,w,q-n)$.

  We first consider the case when~$I$ is a facet of $\subwordsS(\Q,w,p-n)$ that does not contain a position corresponding to a letter of the initial $\cwo$.
  Then $k>0$ for $\beta^{(k)} = \Root{I}{i}$ for any $i \in I$.
  Since the root vectors in this initial copy of $\cwo$ are the positive zero-colored roots, we can use \Cref{lem:flipsfromrootvector} to flip any position $i \in I$ into a letter in this initial copy of $\cwo$ to obtain another facet~$J$ with $I \setminus i = J \setminus j$ for some position~$j$.
  Since $j < i$ by construction, the obtained facet~$J$ is lexicographically smaller and contains the boundary face $I \setminus i$ of~$I$.
  The facet~$I$ is therefore a homology facet.

  Suppose now that~$I$ is a facet containing a position~$i$ of a letter in this initial $\cwo$.
  The root vector~$\beta^{(k)} = \Root{I}{i}$ is now colored by $k=0$, and the position~$i \in I$ cannot be flipped to a lexicographically smaller facet.
  The boundary face $I \setminus i$ is therefore not contained in any lexicographically smaller facet, and~$I$ is thus not a homology facet.
\end{proof}

%%%%%%%%%%%%%%%%%%%%%%%%%%%%%%%%%%%%%%%%%%%%%%%%%%%%%%%%%%%%%%%%%%%%%%%%%%%%%%%%%%%%%
\chapter{Noncrossing partitions}
\label{sec:noncrossing_partitions}
%%%%%%%%%%%%%%%%%%%%%%%%%%%%%%%%%%%%%%%%%%%%%%%%%%%%%%%%%%%%%%%%%%%%%%%%%%%%%%%%%%%%%

In this chapter, we study $m$-eralized noncrossing partitions.
We begin by reviewing known constructions (\Cref{sec:classicalnc,sec:ncm_chains_deltas_subwords}).
Detailed background and historical information can be found in~\cite{Arm2006}.
We then describe the theory in terms of dual subword complexes and conclude with a definition of the Cambrian recurrence and Cambrian poset on $m$-eralized noncrossing partitions (\Cref{sec:colored_facts}).
%We conclude with a discussion of $m$-eralized noncrossing parking functions (\Cref{sec:parking}).\nathan{maybe}

%%%%%%%%%%%%%%%%%%%%%%%%%%%%%%%%%%%%%%%%%%%%%%%%%%%%%%%%%%%%%%%%%%%%%%%%%%%%%%%%%%%%%
\section{Classical noncrossing partitions}
\label{sec:classicalnc}
%%%%%%%%%%%%%%%%%%%%%%%%%%%%%%%%%%%%%%%%%%%%%%%%%%%%%%%%%%%%%%%%%%%%%%%%%%%%%%%%%%%%%

A set partition of $\{1,\ldots,n\}$ is called \defn{noncrossing} if the convex hulls of its blocks do not overlap when drawn on a circle.\nathanside{This sentence is clarified by the picture, until you notice that we haven't drawn convex hulls.}  G.~Kreweras introduced and studied the lattice of noncrossing partitions ordered by containment in~\cite{Kre1972}---the noncrossing partition lattice of $\{1,2,3,4\}$ is illustrated below.

\begin{center}
  \resizebox{\textwidth}{!}{
  \begin{tikzpicture}
    \tikzstyle{rect}=[rectangle,draw,opacity=.5,fill opacity=1];

    \polygon{( 0,0)}{e}{4}{0.5}{1.4}{$ $,$ $,$ $,$ $};

    \polygon{(-5,3)}{12}{4}{0.5}{1.4}{$ $,$ $,$ $,$ $};
      \draw[fill=black,fill opacity=0.1] (121) to[bend right=30] (122) to[bend left=90, looseness=1.5] (121);
    \polygon{( 1,3)}{13}{4}{0.5}{1.4}{$ $,$ $,$ $,$ $};
      \draw[fill=black,fill opacity=0.1] (131) to[bend right=30] (133) to[bend right=30] (131);
    \polygon{(-1,3)}{34}{4}{0.5}{1.4}{$ $,$ $,$ $,$ $};
      \draw[fill=black,fill opacity=0.1] (343) to[bend right=30] (344) to[bend left=90, looseness=1.5] (343);
    \polygon{( -3,3)}{23}{4}{0.5}{1.4}{$ $,$ $,$ $,$ $};
      \draw[fill=black,fill opacity=0.1] (232) to[bend right=30] (233) to[bend left=90, looseness=1.5] (232);
    \polygon{( 3,3)}{24}{4}{0.5}{1.4}{$ $,$ $,$ $,$ $};
      \draw[fill=black,fill opacity=0.1] (242) to[bend right=30] (244) to[bend right=30] (242);
    \polygon{( 5,3)}{14}{4}{0.5}{1.4}{$ $,$ $,$ $,$ $};
      \draw[fill=black,fill opacity=0.1] (144) to[bend right=30] (141) to[bend left=90, looseness=1.5] (144);

    \polygon{(-5,6)}{12-34}{4}{0.5}{1.4}{$ $,$ $,$ $,$ $};
      \draw[fill=black,fill opacity=0.1] (12-341) to[bend right=30] (12-342) to[bend left=90, looseness=1.5] (12-341);
      \draw[fill=black,fill opacity=0.1] (12-343) to[bend right=30] (12-344) to[bend left=90, looseness=1.5] (12-343);
    \polygon{(-3,6)}{123}{4}{0.5}{1.4}{$ $,$ $,$ $,$ $};
      \draw[fill=black,fill opacity=0.1] (1231) to[bend right=30] (1232) to[bend right=30] (1233) to[bend left=30] (1231);
    \polygon{( 1,6)}{134}{4}{0.5}{1.4}{$ $,$ $,$ $,$ $};
      \draw[fill=black,fill opacity=0.1] (1343) to[bend right=30] (1344) to[bend right=30] (1341) to[bend left=30] (1343);
    \polygon{(-1,6)}{234}{4}{0.5}{1.4}{$ $,$ $,$ $,$ $};
      \draw[fill=black,fill opacity=0.1] (2342) to[bend right=30] (2343) to[bend right=30] (2344) to[bend left=30] (2342);
    \polygon{( 3,6)}{124}{4}{0.5}{1.4}{$ $,$ $,$ $,$ $};
      \draw[fill=black,fill opacity=0.1] (1244) to[bend right=30] (1241) to[bend right=30] (1242) to[bend left=30] (1244);
    \polygon{( 5,6)}{14-23}{4}{0.5}{1.4}{$ $,$ $,$ $,$ $};
      \draw[fill=black,fill opacity=0.1] (14-232) to[bend right=30] (14-233) to[bend left=90, looseness=1.5] (14-232);
      \draw[fill=black,fill opacity=0.1] (14-234) to[bend right=30] (14-231) to[bend left=90, looseness=1.5] (14-234);

    \polygon{( 0,9)}{1234}{4}{0.5}{1.4}{$ $,$ $,$ $,$ $};
    \draw[fill=black,fill opacity=0.1] (12341) to[bend right=30] (12342) to[bend right=30] (12343) to[bend right=30] (12344) to[bend right=30] (12341);

    \draw [black] (e) to (12);
    \draw [black] (e) to (13);
    \draw [black] (e) to (14);
    \draw [black] (e) to (23);
    \draw [black] (e) to (24);
    \draw [black] (e) to (34);

    \draw [black] (12) to (123);
    \draw [black] (12) to (124);
    \draw [black] (12) to (12-34);

    \draw [black] (13) to (123);
    \draw [black] (13) to (134);

    \draw [black] (24) to (124);
    \draw [black] (24) to (234);

    \draw [black] (14) to (124);
    \draw [black] (14) to (134);
    \draw [black] (14) to (14-23);

    \draw [black] (23) to (123);
    \draw [black] (23) to (234);
    \draw [black] (23) to (14-23);

    \draw [black] (34) to (134);
    \draw [black] (34) to (234);
    \draw [black] (34) to (12-34);

    \draw [black] (123) to (1234);
    \draw [black] (124) to (1234);
    \draw [black] (134) to (1234);
    \draw [black] (234) to (1234);
    \draw [black] (12-34) to (1234);
    \draw [black] (14-23) to (1234);
  \end{tikzpicture}
  }
\end{center}

\subsection{Noncrossing partitions for Coxeter groups}

In~\cite{Rei1997}, V.~Reiner interpreted the noncrossing partition lattice as a type~$A$ phenomenon using the identification of the lattice of set partitions with the intersection lattice of the type~$A_{n-1}$ reflection arrangement.
Together with C.~Athanasiadis, he extended this notion to types~$B_n$ and~$D_n$ in~\cite{Rei1997,AR2004}.
Noncrossing partitions were subsequently defined in full generality for finite Coxeter systems independently by T.~Brady and C.~Watt in~\cite{BW2002} and by D.~Bessis in~\cite{Bes2003}.

Let $c \in W$ be a Coxeter element.
An element $w \in W$ is a \defn{$c$-noncrossing partition} if $w \leqref c$, and we denote the set of all $c$-noncrossing partitions by $\NC(W,c)$.  The following theorem was proven independently by different methods by D.~Bessis~\cite[Fact~2.3.1]{Bes2003}, T.~Brady and C.~Watt~\cite[Theorem~7.8]{BW2008}, C.~Ingalls and H.~Thomas~\cite[Theorem 4.2]{IT2009}, and N.~Reading~\cite[Corollary 8.6]{Rea2011}.

\begin{theorem+}
  The $c$-noncrossing partitions are a lattice under absolute order.
\end{theorem+}

We denote this \defn{$c$-noncrossing partition lattice} by
\[
  \NCL(W,c)\eqdef[\one,c]_{\Abs(W)}.
\]
In type~$A_{n-1}$ with $c = (1,2,\ldots,n) \in \WA[n]$, the original lattice of noncrossing set partitions of $\{1,\ldots,n\}$ is recovered by sending a permutation to the set partition given by its cycles.
\Cref{fig:ncA3} gives an example for $c = (1,2,3,4) \in \WA[4]$.

\medskip

Noncrossing partitions have the following useful characterization, which appeared in a slightly different form in~\cite[Lemma~4.8]{BW2008}.

\begin{proposition}
\label{prop:ncpairs}
  Let $r_1,\ldots,r_k \in \refl$.
  Then
  \[
    \begin{array}{c}
      r_1 \cdots r_k \leqref c \text{ with }\lengthR(r_1 \cdots r_k) = k \\
      \Leftrightarrow  \\
      r_1,\dots,r_k \text{ are pairwise distinct and } r_a r_b \leqref c
      \text{ for all } 1 \leq a < b \leq k. 
    \end{array}
  \]
\end{proposition}
\nathanside{Go back and read the lemma about fixed and moved spaces.}

For clarity of the argument, we extract the following lemma before proving~\Cref{prop:ncpairs}.

\begin{lemma}
\label{lem:nclen2}
  Let $r_1,r_2 \in \refl$ with $r_1 \neq r_2$.
  Let $\beta_1,\beta_2 \in \PhiP$ such that $r_1 = s_{\beta_1}$ and $r_2 = s_{\beta_2}$ and let~$\mu$ be a generator of $\fixd(r_1 c)$ such that $\langle \beta_1, \mu\rangle = 1$.%so that $\beta_1$ and $\beta_2$ are the positive roots that generate $\movd(r_1)$ and $\movd(r_2)$, respectively
  Then
  \[
    r_1 r_2 \leqref c \Leftrightarrow \langle \mu, \beta_2 \rangle = 0.
  \]
\end{lemma}
\christianside{Geometry is the key.}
\nathanside{Thanks, Brady and Watt!}
\begin{proof}
  By \Cref{lem:refllenprops}\eqref{it:reflenprops1}, $\fixd(r_1 c)=\movd(r_1 c)^\perp$ is one-dimensional.  By \Cref{lem:refllenprops}\eqref{it:reflenprops5}, $\beta_1$ is not perpendicular to $\fixd(r_1 c)$, so that $\mu$ is well-defined:
\begin{align*}
  r_1 \not \leqref r_1c\ \Rightarrow\ \movd(r_1) \not \subseteq \movd(r_1c) \ \Rightarrow\  \fixd(r_1c) \not \subseteq \fixd(r_1)\ \Rightarrow\ \beta_1 \not \perp \fixd(r_1c).
\end{align*}
  Then
  \begin{align*}
    r_1 r_2 \leqref c \ \Leftrightarrow\  r_2 \leqref r_1 c
                      \ \Leftrightarrow\  \movd(r_2) \subseteq \movd(r_1 c)
                      \ \Leftrightarrow\  \langle \beta_2, \mu \rangle = 0,
  \end{align*}
  as desired.
  The first implication follows from \Cref{lem:refllenprops}\eqref{it:reflenprops4}, the second from~\eqref{it:reflenprops2} and~\eqref{it:reflenprops5}, and the third from the definitions of~$\beta_2$ and~$\mu$.
\end{proof}

\begin{proof}[Proof of \Cref{prop:ncpairs}]
  The forward implication is clear: the distinctness of the reflections follows from the fact that $\lengthR(r_1\dots r_k)=k$, and for any $1 \leq a < b \leq k$, the pair $r_a r_b$ can be moved to the left by conjugation to reveal that $r_a r_b \leqref r_1 \cdots r_k \leqref c$.  

  The reverse implication is more delicate.
  As in~\Cref{lem:nclen2}, define $\beta_a \in \PhiP$ such that $r_a = r_{\beta_a}$ and let $\mu_a$ be the generator of $\fixd(r_a c)$ such that $\langle \beta_a, \mu_a \rangle = 1$.  (As above, $\mu_a$ is well-defined because $\beta_a$ is not perpendicular to $\fixd(r_a c)$.)

  Consider the $(k \times k)$-matrix $\langle \alpha_a, \mu_b \rangle_{1 \leq a,b \leq k}$.
  By \Cref{lem:nclen2} and the assumption that $r_a r_b \leqref c$ with $r_a \neq r_b$ for all $1 \leq a < b \leq k$, this matrix is upper-triangular with ones on the main diagonal.
  Therefore, $\{\mu_1,\ldots,\mu_k\}$ are linearly independent.

  We complete the proof by showing by induction that $r_1 \cdots r_{k'} \leqref c$ for $1 \leq k' \leq k$.
  The case $k'=1$ is trivial.
  Assume that $k' \geq 2$.
  By induction, for $1 \leq a < k'$, we have
  \begin{align*}
    r_a \leqref r_1 \cdots r_{k'-1} \leqref c
      &\Rightarrow r_a c \geqref r_{k'-1}\cdots r_1 c \\
      &\Rightarrow \fixd(r_a c) \subseteq \fixd(r_{k'-1}\cdots r_i c) \\
      &\Rightarrow \mu_a \in \fixd(r_{k'-1}\cdots r_1 c).
  \end{align*}
  Now $\{\mu_1,\ldots,\mu_{k-1}\}$ are linearly independent by the calculation above and it follows from the inductive hypothesis that $\dim\fixd(r_{k'-1}\cdots r_1 c) = k'-1$. Therefore, $\{\mu_1,\ldots,\mu_{k'-1}\}$ is a basis of $\fixd(r_{k'-1}\cdots r_1 c)$.
  As $\beta_{k'}$ is orthogonal to $\mu_a$ for $a < k'$, we have that $\beta_{k'} \in \movd(r_{k'-1}\cdots r_1 c)$.
  This proves that $r_{k'} \leqref r_{k'-1}\cdots r_1 c$, so that $r_1 \cdots r_{k'} \leqref c$.
\end{proof}

\subsection{The Kreweras complement}

The \defn{Kreweras complement} is the bijection on noncrossing partitions%$\Krew: \NC(W,c) \to \NC(W,c)$ is defined by $\Krew(w)\eqdef cw^{-1}$.
\begin{equation}
  \begin{aligned}
    \Krew: \NC(W,c) &\to \NC(W,c) \\
    w        &\mapsto c w^{-1}.
  \end{aligned}
  \label{eq:krewerasclassical}
\end{equation}
  This operation is an anti-auto\-morphism of the lattice $\NCL(W,c)$---if $w \leqref wr$ for some $r \in \refl$, then $\Krew(wr)=\Krew(w) \cdot {}^w r \leqref \Krew(w)$.

\medskip

%The Kreweras complement $\Krew$ is so named because 
In the graphical description of noncrossing partitions, the Kreweras complement sends a noncrossing partition to the coarsest ``complementary'' noncrossing partition, as indicated below.  Its square $\Krew^2(w) = {}^c w$ is conjugation by~$c$, which corresponds to the cyclic rotation $i \mapsto i+1$.  

\begin{figure}[h]
  \begin{center}
  \begin{tikzpicture}
    \polygon{(-4,0)}{A}{16}{1.3}{4}{1,,2,,3,,4,,5,,6,,7,,8,};
    \draw[fill=black,fill opacity=0.1] (A1) to[bend right=30] (A3) to[bend right=30] (A9) to[bend left=10] (A1);
    \draw[fill=black,fill opacity=0.1] (A11) to[bend right=0] (A15) to[bend left=15] (A11);

    \polygon{( 0,0)}{B}{16}{1.3}{4}{1,$ $,2,$ $,3,$ $,4,$ $,5,$ $,6,$ $,7,$ $,8,$ $};
    \draw[fill=black,fill opacity=0.1] (B1) to[bend right=30] (B3) to[bend right=30] (B9) to[bend left=10] (B1);
    \draw[fill=black,fill opacity=0.1] (B11) to[bend left=2] (B15) to[bend left=15] (B11);

    \draw[fill=black,opacity=0.5,fill opacity=0.05] (B4) to[bend right=30] (B6) to[bend right=30] (B8) to[bend left=30] (B4);
    \draw[fill=black,opacity=0.5,fill opacity=0.05] (B10) to[bend left=10] (B16) to[bend left=15] (B10);
    \draw[fill=black,opacity=0.5,fill opacity=0.05] (B12) to[bend right=2] (B14) to[bend left=25] (B12);

    \polygon{( 4,0)}{C}{16}{1.3}{4}{,2,,3,,4,,5,,6,,7,,8,,1};
    \draw[fill=black,fill opacity=0.2] (C4) to[bend right=30] (C6) to[bend right=30] (C8) to[bend left=30] (C4);
    \draw[fill=black,fill opacity=0.2] (C10) to[bend left=10] (C16) to[bend left=15] (C10);
    \draw[fill=black,fill opacity=0.2] (C12) to[bend right=2] (C14) to[bend left=25] (C12);

    \draw[|->] (A.south east) to[bend right=25] (C.south west);
    \node at (0,-2.3) {$\Krew[(12345678)]$};
  \end{tikzpicture}
  \end{center}
\label{fig:graphicalkreweras}
\end{figure}

\subsection{Noncrossing partitions as dual subword complexes}

 C.~Atha\-na\-siadis, T.~Brady, and C.~Watt used the edge labelling of the noncrossing partition lattice by reflections to prove that $\NCL(W,c)$ is \emph{EL-shellable}~\cite{ABW2007}.
 This edge labelling gives a unique factorization of each element $w \in \NC(W,c)$ into reflections that increase with respect to the reflection order $\leq_{\c}$, as shown in~\cite[Theorem 3.5]{ABW2007}.
Such factorizations were also considered by N.~Reading in~\cite[Remark~6.8]{Rea20072}.
\begin{proposition+}
\label{prop:nc_unique_factorization}
  Each $w \in \NC(W,c)$ has a unique reduced $\refl$-word $\r_1 \r_2 \cdots \r_p$ with $r_1<_{\c} r_2<_{\c} < \cdots <_{\c} r_p$.
\end{proposition+}

Since $\Krew(w) \cdot w = c$, by combining a noncrossing partition with its Kreweras complement and invoking the unique factorization of \Cref{prop:nc_unique_factorization}, we obtain a refined description of noncrossing partitions considered in~\cite{Arm2006,Tza2008,BRT2012}.
We $m$-eralize this description in \Cref{def:nc_fuss_subwords} below.

\begin{proposition}
\label{prop:nc_subwords}
  There is a canonical bijection between $c$-non\-crossing partitions $\NC(W,c)$ and facets of the dual subword complex
  \begin{align*}
    \DeltaNC(W,c) \eqdef \subwordsR\big(\invs_\refl(\c^h), c\big),
  \end{align*}
  for any reduced $\sref$-word $\c$ for~$c$.
  \nathanside{Having access to individual reflections is going to come in handy.}
\end{proposition}

The construction of $\DeltaNC(W,c)$ depends on the chosen reduced word $\c$ for~$c$, but the complexes for all possible choices are canonically isomorphic by \Cref{lem:reflection_order}\eqref{it:reflection_order1} and \Cref{prop:isomorphicsubwordcomplexes1}.
We therefore prefer to attach this complex to the Coxeter element~$c$ itself rather than to any specific reduced word.

\begin{proof}[Proof of \Cref{prop:nc_subwords}]
  The map $w \mapsto (\Krew(w),w)$ is a bijection from $\NC(W,c)$ to the set of pairs $(\delta_0,\delta_1)$ such that $\delta_0, \delta_1 \in \NC(W,c)$, $\delta_0 \delta_1 = c$, and $\lengthR(\delta_0)+\lengthR(\delta_1)=\lengthR(c)$.
  Let $(\delta_0(c),\delta_1(c))$ be the pair obtained by applying the unique factorizations from \Cref{prop:nc_unique_factorization} to each of $\delta_0$ and $\delta_1$, each specifying certain positions in $\invs_\refl(\cwo)$.
  By construction, this yields a bijection to facets of a dual subword complex
  \[
    \NC(W,c) \bij \subwordsR(\invs_\refl(\cwo)\invs_\refl(\cwo), c).
  \]
  Finally, it follows from \Cref{lem:reflection_order}\eqref{it:reflection_order5} that
  \[
    \invs_\refl(\cwo) \invs_\refl(\cwo) = \invs_\refl \left(\cwo \psi(\cwo) \right) \equiv \invs_\refl(\c^h). \qedhere
  \]
\end{proof}

\begin{example}
  We illustrate \Cref{prop:nc_subwords} for $\WA[4]$ in \Cref{fig:ncA3}, using
  \begin{align*}
    \invs_\refl(c^h) &\equiv \invs_\refl(\cwo) \invs_\refl(\cwo) \\
      &= (12)(13)(14)(23)(24)(34) \cdot (12)(13)(14)(23)(24)(34).
  \end{align*}
  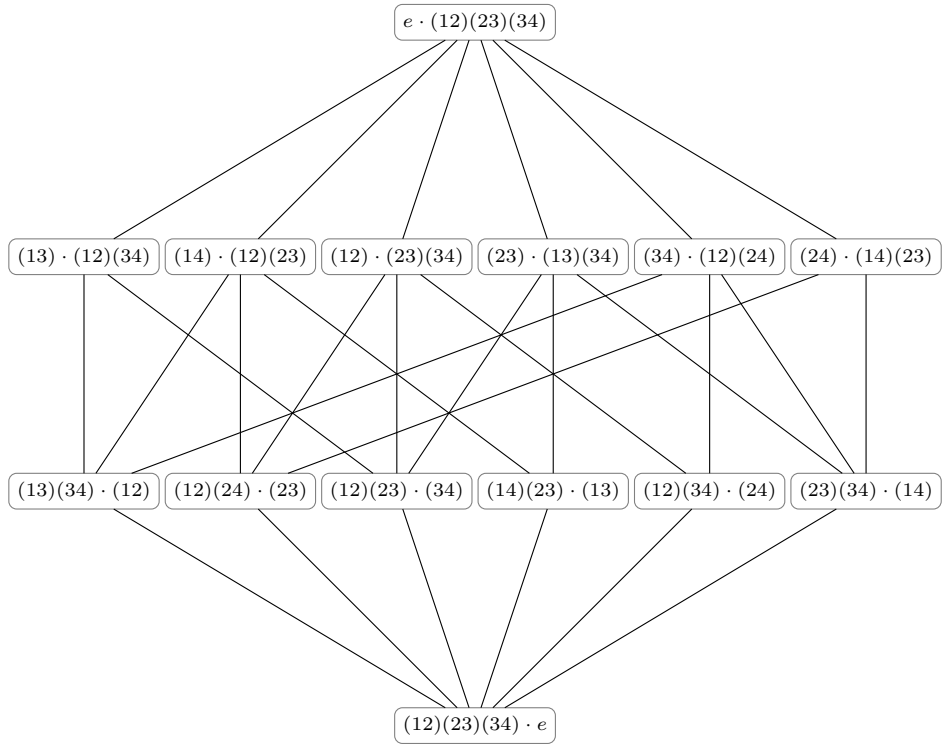
\begin{figure}[t]
    \begin{center}
      \resizebox{\textwidth}{!}{
      \begin{tikzpicture}
        \tikzstyle{rect}=[rectangle,draw,opacity=.5,fill opacity=1,inner sep=3pt,outer sep=0pt,rounded corners=0.1cm];

        \node[rect] (e)     at ( 0,0) {\scriptsize $(12)(23)(34) \cdot e$};

        \node[rect] (12)    at (-5,3) {\scriptsize $(13)(34)    \cdot (12)$};
        \node[rect] (13)    at ( 1,3) {\scriptsize $(14)(23) \cdot (13)$};
        \node[rect] (34)    at (-1,3) {\scriptsize $(12)(23)    \cdot (34)$};
        \node[rect] (23)    at (-3,3) {\scriptsize $(12)(24)    \cdot (23)$};
        \node[rect] (24)    at ( 3,3) {\scriptsize $(12)(34) \cdot (24)$};
        \node[rect] (14)    at ( 5,3) {\scriptsize $(23)(34)    \cdot (14)$};

        \node[rect] (12-34) at (-5,6) {\scriptsize $(13) \cdot (12)(34)$};
        \node[rect] (123)   at (-3,6) {\scriptsize $(14) \cdot (12)(23)$};
        \node[rect] (134)   at ( 1,6) {\scriptsize $(23) \cdot (13)(34)$};
        \node[rect] (234)   at (-1,6) {\scriptsize $(12) \cdot (23)(34)$};
        \node[rect] (124)   at ( 3,6) {\scriptsize $(34) \cdot (12)(24)$};
        \node[rect] (14-23) at ( 5,6) {\scriptsize $(24) \cdot (14)(23)$};

        \node[rect] (1234)  at ( 0,9) {\scriptsize $e \cdot (12)(23)(34)$};

        \draw [black] (e) to (12);
        \draw [black] (e) to (13);
        \draw [black] (e) to (14);
        \draw [black] (e) to (23);
        \draw [black] (e) to (24);
        \draw [black] (e) to (34);

        \draw [black] (12) to (123);
        \draw [black] (12) to (124);
        \draw [black] (12) to (12-34);

        \draw [black] (13) to (123);
        \draw [black] (13) to (134);

        \draw [black] (24) to (124);
        \draw [black] (24) to (234);

        \draw [black] (14) to (124);
        \draw [black] (14) to (134);
        \draw [black] (14) to (14-23);

        \draw [black] (23) to (123);
        \draw [black] (23) to (234);
        \draw [black] (23) to (14-23);

        \draw [black] (34) to (134);
        \draw [black] (34) to (234);
        \draw [black] (34) to (12-34);

        \draw [black] (123) to (1234);
        \draw [black] (124) to (1234);
        \draw [black] (134) to (1234);
        \draw [black] (234) to (1234);
        \draw [black] (12-34) to (1234);
        \draw [black] (14-23) to (1234);
      \end{tikzpicture}
      }
    \end{center}
    \caption{\label{fig:ncA3}The noncrossing partition lattice $\NCL \big(\WA[4],(1234)\big)$.  Each noncrossing partition is represented by its corresponding $1$-delta sequence.}
\end{figure}
\end{example}

\subsection{The Cambrian recurrence}
 For~$s$ initial in~$c$, there is an isomorphism

\begin{align*}
  \NCL(W,c) &\bij \NCL(W,\coxrn) \\
  w         &\longmapsto \sninv ws.
\end{align*}
A modification of this conjugation map produces a simple defining recurrence for $\NC(W,c)$ called the \defn{$c$-Cambrian recurrence}~\cite[Theorem~6.1]{Rea2006}.
\begin{proposition+}
\label{prop:nc_recurrence}
  Let~$s$ be initial in~$c$.
  Then
  \[
    w \in \NC(W,c) \Leftrightarrow
    \begin{cases}
      sw \in \NC(W_{\langle s \rangle},\coxsn) & \text{if } sw \leqref w \\
      \sninv ws \in \NC(W,\coxrn) & \text{otherwise}
    \end{cases}\ .
  \]
  \qedherecases
\end{proposition+}
\christianside{This \emph{natural} recurrence shows up everywhere in noncrossing Catalan combinatorics!}

%%%%%%%%%%%%%%%%%%%%%%%%%%%%%%%%%%%%%%%%%%%%%%%%%%%%%%%%%%%%%%%%%%%%%%%%%%%%%%%%%%%%%
\section{\mhead-eralized noncrossing partitions}
\label{sec:ncm_chains_deltas_subwords}
%%%%%%%%%%%%%%%%%%%%%%%%%%%%%%%%%%%%%%%%%%%%%%%%%%%%%%%%%%%%%%%%%%%%%%%%%%%%%%%%%%%%%

We now review the $m$-eralization of noncrossing partitions.  Generalizing a construction given by P.~Edelman in~\cite{Ede1980} from type~$A$ to all finite Coxeter groups, D.~Armstrong defined $m$-eralized $c$-noncrossing partitions as $m$-multichains in absolute order.

\subsection{$m$-eralized noncrossing partitions as chains}

\begin{definition}[{\cite[Definition~3.2.2(1)]{Arm2006}}]
\label{def:nc_multichains}
The \defn{$m$-eralized $c$-noncrossing partitions} are the~$m$-multichains
\[
  \NCm(W,c) \eqdef \bigset{ (w_1 \geqref w_2 \geqref \cdots \geqref w_m) }{ w_i \in \NC(W,c) }.
\]
The \defn{support} of $(w_1 \geqref w_2 \geqref \cdots \geqref w_m)$ is $\supp(w_1)\subseteq \sref$.
\end{definition}

\subsection{$m$-eralized noncrossing partitions as delta sequences}

D.~Armstrong also gave an equivalent construction encoding the intervals between consecutive elements of the $m$-multichain.  This $m$-eralizes the pairing of a noncrossing partition with its Kreweras complement.

\begin{definition}[{\cite[Definition~3.2.2(2)]{Arm2006}}]
\label{def:nc_delta}
  A sequence $\delta=\left(\delta_0,\delta_1,\ldots,\delta_m\right)$ with $\delta_i \in \NC(W,c)$ is an \defn{$m$-delta sequence} if
  \[
    \delta_0 \delta_1 \cdots \delta_m = c \quad\text{and}\quad \sum_{i=0}^m \lengthR(\delta_i) = \lengthR(c).
  \]
 The \defn{support} of $\delta$ is defined to be $\supp(\delta)\eqdef \supp(\delta_1\cdots\delta_m) \subseteq \sref$, and we denote the set of all such sequences by $\deltaNCm(W,c)$.
\end{definition}

The following proposition relates \Cref{def:nc_multichains,def:nc_delta}, and is immediate from the definitions.

\begin{proposition+}[{\cite[Lemma~3.2.4]{Arm2006}}]
\label{prop:nc_delta}
  There is a canonical bijection
  \begin{align*}
    \NCm(W,c) &\canbij \deltaNCm(W,c) \\
     \big( w_1 \geqref \ldots \geqref w_m \big) &\longmapsto \big(cw_1^{-1},w_1w_2^{-1},\ldots,w_{m-1}w_m^{-1},w_m \big)\\
 \big( \delta_1\cdots\delta_m \geqref \cdots \geqref \delta_{m-1}\delta_m \geqref \delta_m \big) &\longmapsfrom  \big(\delta_0,\ldots,\delta_m\big).
 \qedhere
  \end{align*}
\end{proposition+}

D.~Armstrong $m$-eralized the lattice $\NCL(W,c)$ by considering $\deltaNCm(W,c)$ under componentwise absolute order.
\begin{equation}
  \left(\delta_0,\delta_1,\ldots,\delta_m\right) \leqref \left(\delta'_0,\delta'_1,\ldots,\delta'_m\right) \Leftrightarrow \delta_i \leqref \delta'_i \text{ for all } 1 \leq i \leq m.
\label{eq:ncl_meet_semilattice}
\end{equation}
(Note that we do not compare the zero-th components.)  This order on $\deltaNCm(W,c)$ is a graded meet semilattice~\cite[Theorem 3.4.4]{Arm2006}.  For $\WA$ it coincides with the refinement order on \defn{$m$-shuffle noncrossing partitions}---the noncrossing partitions of $m(n+1)$ such that the elements of each block are congruent modulo~$m$~\cite[Section~4.3.1, Figure~4.6]{Arm2006}.

\medskip

It is straightforward to $m$-eralize the \defn{Kreweras complement} by
\[
  \Krew(\delta) \eqdef \left( c\delta_mc^{-1}, \delta_0, \delta_1, \ldots, \delta_{m-1} \right),
\]
so that $\order(\Krew) = (m+1)h$.
On the combinatorial model of $m$-shuffle noncrossing partitions, $\Krew^{m+1}$ acts as a rotation by~$m$.

\medskip

\subsection{\mhead-eralized noncrossing partitions as dual subword complexes}
\label{sec:colored_facts}
%%%%%%%%%%%%%%%%%%%%%%%%%%%%%%%%%%%%%%%%%%%%%%%%%%%%%%%%%%%%%%%%%%%%%%%%%%%%%%%%%%%%%

Using dual subword complexes, we generalize \Cref{prop:nc_subwords} and refine $m$-delta sequences to the level of reflections.
\begin{definition}
\label{def:nc_fuss_subwords}
  Define the dual subword complex
  \[
    \DeltaNCm(W,c) \eqdef \subwordsR(\invs_\refl(\cwom[c][m+1]), c).
  \]
\end{definition}

As in \Cref{prop:nc_subwords}, the construction depends on the chosen reduced $\sref$-word but the complexes for all possible choices are canonically isomorphic---so we again prefer to attach this complex to the element~$c$ itself.\nathanside{Deal with it.}
As seen in \Cref{ex:initialsubwordcomplex}, the dual subword complexes for $m$-eralized noncrossing partitions are \emph{not} Coxeter initial subword complexes.

\medskip

It is convenient to represent a facet~$I$ of $\DeltaNCm(W,c)$ as a word of colored reflections $\r_1^{(i_1)}\cdots \r_n^{(i_n)}$, where
\begin{itemize}
  \item a reflection $r_j$ is colored according to the copy of $\invs_\refl(\cwo)$ to which it belongs,
  \item $r_a<_{\c} r_b$ if $a<b$ and $i_a=i_b$, and
  \item $r_1 \cdots r_n = c$.
\end{itemize}
For notational simplicity, we write $r \in I$ to mean that $r \in \{r_1,\ldots,r_n\}$.  The \defn{support} of $I$ is defined to be $\supp(I)\eqdef \supp(\prod r_k) \subseteq \sref$, where the product is over all roots $r_k \in I$ with color $i_k>0$.

\begin{theorem}
\label{prop:nc_fuss_subwords}
  There is a canonical bijection
  \[
     \deltaNCm(W,c) \canbij \DeltaNCm(W,c).
   \]
\end{theorem}
\begin{proof}
  As in the proof of \Cref{prop:nc_subwords}, apply the factorization of \Cref{prop:nc_unique_factorization} separately to each component of the $m$-delta sequence.
\end{proof}

\Cref{prop:nc_delta,prop:nc_fuss_subwords} give canonical preserving bijections between the three variants of noncrossing partitions, and we move freely between them.
\Cref{fig:ncA22} shows all $12$ elements of
\[
  \NCm[2](\WA[3],st)\ \canbij\ \deltaNCm[2](\WA[3],st)\ \canbij\ \DeltaNCm[2](\WA[3],st)
\]
with their supports.
\christianside{Each noncrossing partition definition has a slightly different flavor\dots}\nathanside{\dots and equally delicious in its own unique way.}
\begin{figure}[t]
  \begin{center}
    \begin{tabular}{c|c|c|c}
      $\NCm[2](\WA[3],st)$ & $\deltaNCm[2](\WA[3],st)$ & $\DeltaNCm[2](\WA[3],st)$ & $\supp(w)$ \\
      \hline
      $e \geqref e$ &
      $(st , e , e)$ &
      $\s\U\t.\S\U\T.\S\U\T$ &
      $-$
      \\
      $st \geqref st$ &
      $(e , e , st)$ &
      $\S\U\T.\S\U\T.\s\U\t$ &
      $s,t$
      \\
      $st \geqref s$ &
      $(e , u , s)$ &
      $\S\U\T.\S\u\T.\s\U\T$ &
      $s,t$
      \\
      $u \geqref  e$ &
      $(t , u , e)$ &
      $\S\U\t.\S\u\T.\S\U\T$ &
      $s,t$
      \\
      \hline
      $s \geqref  e$ &
      $(u , s , e)$ &
      $\S\u\T.\s\U\T.\S\U\T$ &
      $s$
      \\
      $t \geqref t$ &
      $(s , e , t)$ &
      $\s\U\T.\S\U\T.\S\U\t$ &
      $t$
      \\
      $st \geqref u$ &
      $(e , t , u)$ &
      $\S\U\T.\S\U\t.\S\u\T$ &
      $s,t$
      \\
      $st \geqref e$ &
      $(e , st , e)$ &
      $\S\U\T.\s\U\t.\S\U\T$ &
      $s,t$
      \\
      \hline
      $t \geqref e$ &
      $(s , t , e)$ &
      $\s\U\T.\S\U\t.\S\U\T$ &
      $t$
      \\
      $u \geqref  u$ &
      $(t , e , u)$ &
      $\S\U\t.\S\U\T.\S\u\T$ &
      $s,t$
      \\
      \hline
      $s \geqref s$ &
      $(u , e , s)$ &
      $\S\u\T.\S\U\T.\s\U\T$ &
      $s$
      \\
      $st \geqref t$ &
      $(e , s , t)$ &
      $\S\U\T.\s\U\T.\S\U\t$ &
      $s,t$
    \end{tabular}
  \end{center}
  \caption{The three variants of the $m$-eralized $st$-noncrossing partitions for $\WA[3]$ with $m=2$, together with their supports.  They are arranged according to their orbits under Cambrian rotation, defined in~\Cref{sec:nc-m-cambrian-recurrence}.}
  \label{fig:ncA22}
\end{figure}

%%%%%%%%%%%%%%%%%%%%%%%%%%%%%%%%%%%%%%%%%%%%%%%%%%%%%%%%%%%%%%%%%%%%%%%%%%%%%%%%%%%%%
\section{The Cambrian rotation and recurrence}
\label{sec:nc-m-cambrian-recurrence}
%%%%%%%%%%%%%%%%%%%%%%%%%%%%%%%%%%%%%%%%%%%%%%%%%%%%%%%%%%%%%%%%%%%%%%%%%%%%%%%%%%%%%

For~$s$ initial in~$c$, define the bijection
%\christian{@nathan: what does ``by lemma ...'' mean? Also, the change of the colors should be explained}
%\christian{unfinished sentence}
\begin{align*}
  \Shift_{s}: \DeltaNCm(W,c) &\bij \DeltaNCm(W,\coxrn)\\
\r_1^{(i_1)}\r_2^{(i_2)}\cdots \r_n^{(i_n)} &\longmapsto
    \begin{cases}
      \r_2^{(i_2)}\cdots \r_n^{(i_n)} \s^{(m)} & \text{if } r_1^{(i_1)} = s^{(0)} \\
      \t_1^{(j_1)}\cdots \t_n^{(j_n)} & \text{otherwise}
    \end{cases}\ ,
\end{align*}
where $t_k=r_k^s \text{ and } j_k=\begin{cases} i_k & \text{if } t_k \neq s \\ i_k-1 & \text{if } t_k = s\end{cases}$.  By~\Cref{lem:reflection_order}\eqref{it:reflection_order7} and since $\s$ is final in $\invs_\refl(\cwo[\coxrn])$, the ordering on the reflections in the word in the image of this map is compatible with $\coxrn$. 

\medskip

\begin{example}
\label{ex:ncshiftorbit}
  Alternately applying $\Shift_s$ and $\Shift_t$ to $\s^{(0)}\t^{(0)} \in \DeltaNCm[2](\WA[3],st)$ gives the orbit
  \[
    \begin{array}{ccccc}
      \s\U\t.\S\U\T.\S\U\T
      &\xmapsto{\Shift_s}&
      \t\U\S.\T\U\S.\T\U\s
      &\xmapsto{\Shift_t}&
      \S\U\T.\S\U\T.\s\U\t \\
      &\xmapsto{\Shift_s}&
      \T\U\S.\T\U\s.\T\u\S
      &\xmapsto{\Shift_t}&
      \S\U\T.\S\u\T.\s\U\T \\
      &\xmapsto{\Shift_s}&
      \T\U\S.\t\U\s.\T\U\S
      &\xmapsto{\Shift_t}&
      \S\U\t.\S\u\T.\S\U\T \\
      &\xmapsto{\Shift_s}&
      \T\u\S.\t\U\S.\T\U\S
      &\xmapsto{\Shift_t}&
      \s\U\t.\S\U\T.\S\U\T
    \end{array}.
  \]
\end{example}

\begin{definition}
\label{def:nc_cambrian_rotation}
  For $\c = \s_1\cdots \s_n$, the \defn{$m$-eralized $c$-Cambrian rotation} is
  \[
    \Camb_\c \eqdef \Shift_{s_n} \circ \cdots \circ \Shift_{s_1} : \DeltaNCm(W,c) \bij \DeltaNCm(W,c).
  \]
\end{definition}
As usual, this composition does not depend on the chosen reduced word $\c$.
The elements in \Cref{fig:ncA22} are arranged according to their orbits under Cambrian rotation.

\medskip

A modification of the first case in the definition of the shift operator gives an inductive characterization of $\DeltaNCm(W,c)$ called the \defn{$m$-eralized $c$-Cambrian recurrence}.

\begin{proposition}\label{prop:nc_cambrian_recurrence}
  Let~$s$ be initial in~$c$ and let $I = \r_1^{(i_1)}\r_2^{(i_2)}\cdots \r_n^{(i_n)}$.
  Then
  \[
    I \in \DeltaNCm(W,c)
    \Leftrightarrow
    \begin{cases}
      \r_2^{(i_2)}\cdots \r_n^{(i_n)} \in \DeltaNCm(\Wres,\coxsn) & \text{if } r_1^{(i_1)} = s^{(0)} \\
      \Shift_s(I) \in \DeltaNCm(W,\coxrn) & \text{otherwise}
    \end{cases}\ .
  \]
\end{proposition}

\begin{proof}
 This follows from \Cref{lem:reflection_order}\eqref{it:reflection_order7} and \Cref{lem:reflection_order4}.
\end{proof}

\begin{mywarning*}
  Some care must be taken to correctly run this recurrence in reverse.
  It might seem that the facet $\s^{(0)}\r_2^{(i_2)}\cdots \r_n^{(i_n)} \in \DeltaNCm(W,c)$ is produced by both $\r_2^{(i_2)}\cdots \r_n^{(i_n)} \in \DeltaNCm(\Wres,\coxsn)$ and by $\r_2^{(i_2)}\cdots \r_n^{(i_n)} \s^{(m)} \in \DeltaNCm(W,\coxrn)$.  But facets of $\DeltaNCm(W,\coxrn)$ of the form $\r_2^{(i_2)}\cdots \r_n^{(i_n)} \s^{(m)}$ are not in the image of the recurrence, and so it is crucial not to use them when running the recurrence in reverse.

  We emphasize here that when reversing any of the Cambrian recurrences in \Cref{prop:assoc_cambrian_recurrence} and in \Cref{prop:sort_cambrian_recurrence,,prop:msort_cambrian_recurrence,,prop:aligned_cambrian_recurrence,,prop:skipset}, the elements on the right hand side are required to have the form specified by the recurrence.\nathanside{For later: the parabolic subgroup gets Cambrian rotated to the top of the Cambrian lattice.}
\end{mywarning*}

\begin{example}
\label{ex:ncdeltacambrian}
  The Cambrian recurrence for the $(e,s,t) \in \deltaNCm[2](\WA[3],st)$ is computed as
  \[
    \underbrace{(e,s,t)}_{st}
    \hspace*{-1pt}\mapsto\hspace*{-1pt}
    \underbrace{(s,e,u)}_{ts}
    \hspace*{-1pt}\mapsto\hspace*{-1pt}
    \underbrace{(u,e,s)}_{st}
    \hspace*{-1pt}\mapsto\hspace*{-1pt}
    \underbrace{(t,s,e)}_{ts}
    \hspace*{-1pt}\mapsto\hspace*{-1pt}
    \underbrace{(e,s,e)}_{s}
    \hspace*{-1pt}\mapsto\hspace*{-1pt}
    \underbrace{(s,e,e)}_{s}
    \hspace*{-1pt}\mapsto\hspace*{-1pt}
    \underbrace{(e,e,e)}_{\one},
  \]
where the subscript identifies the (parabolic) Coxeter element.
\end{example}

%%%%%%%%%%%%%%%%%%%%%%%%%%%%%%%%%%%%%%%%%%%%%%%%%%%%%%%%%%%%%%%%%%%%%%%%%%%%%%%%%%%%%
\section{Cambrian posets}
\label{sec:nc_cambrian_posets}
%%%%%%%%%%%%%%%%%%%%%%%%%%%%%%%%%%%%%%%%%%%%%%%%%%%%%%%%%%%%%%%%%%%%%%%%%%%%%%%%%%%%%

In this section, we define a new poset structure $m$-eralizing N.~Reading's Cambrian lattices.

\subsection{The weak order on covering and covered reflections}

Starting with $\coveringref(\one) = \sref$ and $\coveredref(\one) = \emptyset$, we build $\Weak(W)$ on the covering and covered reflections of the elements in~$W$ by describing cover relations.  Supposing that $\coveringref(w)$ and $\coveredref(w)$ are known, there will be one cover of $w$ for each $r \in \coveringref(w)$.  Choose any $r \in \coveringref(w)$ and let $\alpha_r$ be the associated positive root.
The element $wr$ covering $w$ has $\coveringref(rw)$ and $\coveredref(rw)$ given by
\begin{align*}
  \coveringref(rw) &= \set{ u^r }{ u \in \coveringref(w), r(\alpha_u) > 0 } \cup  \set{ u^r }{ u \in \coveredref(w), r(\alpha_u) < 0 },\\
  \coveredref(rw) &= \set{ u^r }{ u \in \coveredref(w), r(\alpha_u) > 0 } \cup  \set{ u^r }{ u \in \coveringref(w), r(\alpha_u) < 0 }.
\end{align*}
Finally, an element $w \in W$ can be reconstructed from its covered and covering reflections---these tell us exactly which hyperplanes bound the corresponding chamber, which determines~$w$ up to multiplication by~$\wo$.  Since we have distinguished covering and covered reflections, we have therefore uniquely specified~$w$.

\subsection{Cambrian posets}

We mimic this characterization of $\Weak(W)$ with a small twist.
Delta sequences in $\deltaNCm[1](W,c)$ have two components, $\FNC = (\delta_0,\delta_1)$.
The analogy we wish to draw is that the factorization of $\delta_0$ of \Cref{prop:nc_unique_factorization} should be thought of as a variant of $\coveringref(w)$, while the factorization of $\delta_1$ behaves like $\coveredref(w)$.  We now define the cover relations, jumping immediately to the definition for general~$m$.

\medskip

Let~$I=\r_1^{(i_1)}\cdots\r_a^{(i_a)}\cdots \r_n^{(i_n)}$ be a facet of $\DeltaNCm(W,c)$.  If~$i_a \neq m$, the \defn{increasing flip} of~$I$ in the direction $r_a$ is given by
\begin{equation}
  \flip^{\uparrow}_{r_a}(I)\eqdef\r_1^{(i_1)}\cdots\r_{a-1}^{(i_{a-1})}\ \underbrace{\t_{a+1}^{(i'_{a+1})} \cdots \t_{b}^{(i'_b)}\ \r_a^{(i_a+1)}}_{\text{modified}}\ \r_{b+1}^{(i_{b+1})}\cdots\r_n^{(i_n)},
  \label{eq:increasingflip}
\end{equation}
where $b$ is chosen maximally so that $r_b^{(i_b)} <_{\c} r_a^{(i_a+1)}$.  We therefore have
\begin{align*}
  \underbrace{r_{a+\ell} <_{\c} \cdots  <_{\c} r_b}_{\text{color }i_a+1} &<_{\c} r_a <_{\c} \underbrace{r_{a+1} <_{\c} \cdots  <_{\c} r_{a+\ell-1}}_{\text{color }i_a},
\end{align*}
and for $a+1 \leq j \leq b$, the modified portion of the colored factorization is given by $t_j=r_j^{r_a}$ and
\[
  i'_{j} =  \begin{cases}
              i_{j} & \text{if } r_a(\alpha_{r_{j}}) \in \PhiP \\
              i_j - 1 & \text{if } r_a(\alpha_{r_{j}}) \in \Phi^-
            \end{cases}\ .
\]\nathanside{An innocent composition of Hurwitz moves that somehow mirrors some rather risqu\'e quiver mutations.}
% \hugh{Does the case in the second line ever happen?  I don't think so.  Also, I feel like in this case, I would want to say $i_j'$ should be $i_j-1$. Though if it actually never arises, it doesn't exactly matter, but, to me, that's what would fit with the logic.}\nathan{yeah, I think this never happens.  Worth a remark?}\hugh{My preference would be to state here something which is definitely well-defined.  I would therefore remove the condition $i_a=i_j-1$.  That way it's clear that we are in one of the two cases.}

The \defn{decreasing flip} $\flip^{\downarrow}_{r_a}(I)$ for~$i_a\neq 0$ is defined analogously.
Examples of flips are illustrated in \Cref{fig:cambncA22} on page~\pageref{fig:cambncA22}.
Before showing that this notion of flips is well-defined on $m$-eralized noncrossing partitions, we remark that this is \emph{not} the notion of flip we defined for subword complexes in \Cref{sub:flips}.

\begin{proposition}
\label{prop:nc_flip_is_subword}
  If $I=\r_1^{(i_1)}\cdots\r_a^{(i_a)}\cdots \r_n^{(i_n)}$ is a facet of $\DeltaNCm(W,c)$ with $i_a \neq m$, then $\Cambupflip(I) \in \DeltaNCm(W,c)$.
\end{proposition}

\begin{proof}
  For $\Cambupflip(I)$ to be a subword, we must check that---up to commutation of commuting letters---it respects the reflection order $\leq_{\c}$.
% it respects the partial order on positive roots induced by $c$.
%  \begin{equation}
%    \label{eq:root_nc_flip}
%    \underbrace{t_{a+\ell'} <_{\c} \cdots <_{\c} t_{b}}_{\text{color }i_a+1} <_{\c} r_a <_{\c} \underbrace{t_{a+1} <_{\c}\cdots <_{\c} t_{a+\ell'-1}}_{\text{color }i_a}.
%\end{equation}
  The idea is simple.
  We rotate $\cwo$ so that $r_a$ is conjugated to a simple reflection, perform the flip in this simple system where it is straightforward to check the condition, and then rotate back.
  By assumption, we have that the reflections in the colored factorization
  \begin{equation}
  \r_a^{(i_a)} \r_{a+1}^{(i_{a+1})} \cdots \r_b^{(i_b)} \label{eq:r1}
  \end{equation}
  are ordered by
    $r_{a+\ell} \cdots <_{\c} r_b <_{\c} r_a <_{\c} r_{a+1} <_{\c} \cdots <_{\c} r_{a+\ell-1}.$

  Let~$\sq{w}=\s_1 \s_2 \cdots \s_{i-1}$ be the prefix of $\cwo=\s_1\cdots \s_N$ so that ${}^w s_i=r_a$.   Applying \Cref{lem:reflection_order}\eqref{it:reflection_order7}, there is some Coxeter element $c'$ so that the cyclic rotation $\s_i \cdots \s_N \s_1 \cdots \s_{i-1} \equiv \cwo[\c']$.
  Conjugating~\eqref{eq:r1} by $w$ shifts all reflections into the same color giving the colored factorization
  \begin{equation}
    \s_i^{(i_a)} (r_{a+1}^w)^{(i_a)} \cdots (r_{b}^w)^{(i_a)}. \label{eq:r2}
  \end{equation}
  Up to commutation of commuting letters, this factorization respects the reflection order $\leq_{\c'}$.   By \Cref{lem:reflection_order4}, conjugating~\eqref{eq:r2} by $s_i$ gives us the colored factorization
  \begin{equation}
    \underbrace{(r_{a+1}^{ws_i})^{(i_a)} \cdots (r_{b}^{ws_i})^{(i_a)}}_{\text{in } W_{\langle s_i \rangle}} \s_i^{(i_a+1)}\label{eq:r3}
  \end{equation}
  that respects the reflection order $\leq_{\sinv\c'\s}$, up to commutations.  For any $\beta\in \Phi^+_{\langle s \rangle}$, $s_i(\beta) = \beta+k\alpha_i$ for $k\geq 0$.  By \Cref{lem:reflection_order4}, $s_i(\beta)$ therefore occurs weakly earlier in the orders $\leq_{\c'}$ and $\leq_{\sinv\c'\s}$.

  We now undo the applications of \Cref{lem:reflection_order}\eqref{it:reflection_order7} to recover $\cwo$, which has the effect of conjugating~\eqref{eq:r3} by~$w$.
  Since ${}^w s_i=r_a$, the colored factorization
  $\t_{a+1}^{(i'_{a+1})} \cdots \t_b^{(i'_b)} \r_a^{(i_a)+1}$
  respects the reflection order $\leq_\c$.  Because reflections have only moved weakly earlier in $\leq_\c$ as a result of the sequence of conjugations, it is only possible for colors to remain the same or decrease by one.  A color decreases by one exactly when the corresponding root changed sign when reflected by ${}^w s_i = r_a$, as desired.
 %The roots that change color are those that change sign under ${}^w s_i = r_a$, as desired.
\end{proof}

By analogy to the weak order, flips define the cover relations of an $m$-eralized $c$-Cambrian poset.

\begin{definition}
\label{def:nc_cambrian_poset}
  The \defn{$m$-eralized $c$-Cambrian poset} $\Cambncm(W,c)$ is the partial order on the elements of $\DeltaNCm(W,c)$ with minimal element $\s_1^{(0)}\cdots \s_n^{(0)}$ and with covering relations
  \[
    \r_1^{(i_1)}\cdots \r_a^{(i_a)}\cdots \r_n^{(i_n)} \lessdot \Cambupflip \left(\r_1^{(i_1)}\cdots \r_a^{(i_a)}\cdots \r_n^{(i_n)}\right)
  \]
  for $i_a<m$.
  Cover relations are labelled by the colored reflection $r_a^{(i_a)}$.
\end{definition}

\begin{mywarning*}
  Note that this notion of flip does \emph{not} come from the flip on the subword complex corresponding to $\DeltaNCm(W,c)$ defined in~\Cref{prop:dualsubwordcomplexes}.
  \christianside{We never use the subword complex flip for noncrossing partitions, so don't confuse these two notions of flip!}
\end{mywarning*}

It follows from the definitions that $\Shift_s$ and $\Cambupflip[r]$ indeed interact nicely, as the following proposition shows.
\begin{proposition}
\label{prop:nc_cambrian_graph_isomorphism}
  For~$s$ initial in~$c$, let $I=\r_1^{(i_1)}\cdots \r_a^{(i_a)}\cdots \r_n^{(i_n)} \in \DeltaNCm(W,c)$, and suppose $I \lessdot \Cambupflip[r_a] \left(I\right)$.
  Then for $b+c=m+1$,
  \[
    \begin{cases}
      \left(\Cambdownflip[r_a] \right)^b \circ \Shift_s(I)= \Shift_s \circ \left(\Cambupflip[r_a]\right)^c (I)  & \text{if } r_1^{(i_1)} = s^{(0)} \text{ and } r_a = s\\
      \Cambupflip[r_a] \circ \Shift_s(I) = \Shift_s \circ \Cambupflip[r_a](I)  & \text{if } r_1^{(i_1)} = s^{(0)} \text{ and } r_a \neq s\\
      \Cambupflip[r_a^s] \circ \Shift_s(I) = \Shift_s \circ \Cambupflip[r_a](I) & \text{otherwise}.
    \end{cases}
  \]
\end{proposition}
%   \[
%     \begin{cases}
%       \Shift_s(I)=\left(\flip^\uparrow_s \right)^m \left(\Shift_s(\Cambupflip[s](I))\right)  & \text{if } r_1^{(i_1)} = s^{(0)} \text{ and } r_a = s\\
%       \Cambupflip[r_a]\left(\Shift_s(I)\right)=\Shift_s\left(\Cambupflip[r_a](I)\right)  & \text{if } r_1^{(i_1)} = s^{(0)} \text{ and } r_a \neq s\\
%       \Cambupflip[r_a^s]\left(\Shift_s(I)\right) = \Shift_s\left(\Cambupflip[r_a](I)\right) & \text{otherwise}
%     \end{cases}\ .
%   \]
%\hugh{In the first case above, what about, instead of saying ``apply flip, then shift, then m flips'', we just said: apply b flips then shift then c flips, such that b+c=m+1, with b and c both positive.}
%\christian{I moved the up-flip to a down-flip on the other side to make the cases look as similar as possible. Is that better?}

\begin{proof}
  This follows immediately from the definition of $\Shift_s$ and the definition of $\flip^\uparrow_r$.
\end{proof}

\Cref{fig:cambncA22} illustrates the $m$-eralized $c$-Cambrian poset for $\WA[3]$ for $m=2$.
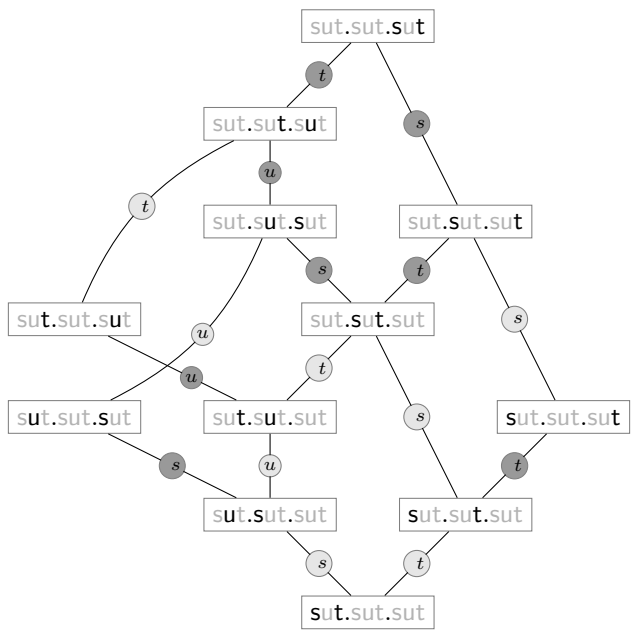
\begin{figure}[t!]
  \begin{center}
    \begin{tikzpicture}[scale=1.3]
      \tikzstyle{rect}=[rectangle,draw,opacity=.5,fill opacity=1]
      \tikzstyle{sort}=[]
      \tikzstyle{flip1}=[inner sep = 1pt,circle,draw,opacity=.5,fill opacity=1,fill=black!10]
      \tikzstyle{flip2}=[inner sep = 1pt,circle,draw,opacity=.5,fill opacity=1,fill=black!40]
        \node[rect,sort] (e)   at (0,0)
          {$\s\U\t.\S\U\T.\S\U\T$};
        \node[rect,sort] (s)   at (-1,1)
          {$\S\u\T.\s\U\T.\S\U\T$};
        \node[rect,sort] (t)   at ( 1,1)
          {$\s\U\T.\S\U\t.\S\U\T$};
        \node[rect,sort] (st)  at (-1,2)
          {$\S\U\t.\S\u\T.\S\U\T$};
        \node[rect,sort] (sts) at ( 0,3)
          {$\S\U\T.\s\U\t.\S\U\T$};
        \node[rect,sort] (stss)   at ( 1,4)
          {$\S\U\T.\s\U\T.\S\U\t$};
        \node[rect,sort] (stst)   at (-1,4)
          {$\S\U\T.\S\u\T.\s\U\T$};
        \node[rect,sort] (ststs)  at (-1,5)
          {$\S\U\T.\S\U\t.\S\u\T$};
        \node[rect,sort] (stssts) at ( 0,6)
          {$\S\U\T.\S\U\T.\s\U\t$};
        \node[rect,sort] (ss)   at (-3,2)
          {$\S\u\T.\S\U\T.\s\U\T$};
        \node[rect,sort] (tt)   at ( 2,2)
          {$\s\U\T.\S\U\T.\S\U\t$};
        \node[rect,sort] (stt)  at (-3,3)
          {$\S\U\t.\S\U\T.\S\u\T$};

        \draw (e) to node[flip1] {\scriptsize~$s$} (s) to node[flip1] {\scriptsize $u$} (st) to node[flip1] {\scriptsize~$t$} (sts) to node[flip2] {\scriptsize~$t$} (stss) to node[flip2] {\scriptsize~$s$} (stssts);
        \draw (e) to node[flip1] {\scriptsize~$t$} (t) to node[flip1] {\scriptsize~$s$} (sts) to node[flip2] {\scriptsize~$s$} (stst) to node[flip2] {\scriptsize $u$} (ststs) to node[flip2] {\scriptsize~$t$} (stssts);
        \draw (s) to node[flip2] {\scriptsize~$s$} (ss) to[bend right=20] node[flip1] {\scriptsize $u$} (stst);
        \draw (t) to node[flip2] {\scriptsize~$t$} (tt) to node[flip1] {\scriptsize~$s$} (stss);
        \draw (st) to node[flip2,pos=0.35] {\scriptsize $u$} (stt) to[bend left=20] node[flip1] {\scriptsize~$t$} (ststs);
    \end{tikzpicture}
  \end{center}
  \caption{
    The flip poset $\Cambncm[2](\WA[3],st)$ on the noncrossing partitions $\NCm[2](\WA[3],st)$.
    The colored reflections labelling the edges correspond to the direction of the flips, with the color indicated by the shading.}
  \label{fig:cambncA22}
\end{figure}
We define a Cambrian graph by allowing longer flips that send a reflection from the $i$\th\ copy of $\wo(c)$ to the $j$\th\ copy for $i<j$.

\begin{definition}
\label{def:nc_cambrian_graph}
  Fix $\c=\s_1\cdots \s_n$ a reduced $\sref$-word for a Coxeter element~$c$.
  The \defn{$m$-eralized $c$-Cambrian graph} $\GCambncm(W,c)$ is the directed graph on the elements of $\DeltaNCm(W,c)$ with source $\s_1^{(0)}\cdots \s_n^{(0)}$ and edges
  \[
    \r_1^{(i_1)}\cdots \r_a^{(i_a)}\cdots \r_n^{(i_n)} \to \left(\Cambupflip\right)^k \left(\r_1^{(i_1)}\cdots \r_a^{(i_a)}\cdots \r_n^{(i_n)}\right)
  \]
  for $k>0$ with $i_a+k\leq m$.
\end{definition}
\christianside{The Cambrian lattice defined directly on noncrossing partitions!}

\begin{corollary}
  Let $c,c'$ be two Coxeter elements.
  Then there is an \emph{undirected} graph isomorphism $\GCambncm(W,c) \bij \GCambncm(W,c')$.
\end{corollary}

\begin{proof}
  Since any two Coxeter elements $c, c'$ are conjugate in~$W$, there exists an initial sequence $s_1,\dots,s_p$ for~$c$ such that $c' = \sninv_p\cdots\sninv_1cs_1\cdots s_p$.
  The undirected graph isomorphism is induced by the composition of the shift operations for the sequence $s_1,\ldots,s_p$.
\end{proof}

\section{Some enumerations}

We collect the following enumeration results for $m$-eralized $c$-noncrossing partitions and their Cambrian lattices.

\begin{theorem}
\label{thm:m-counting}
  We have
  \begin{equation*}
    (1-q)^{n+1}\sum_{m=0}^\infty  \big|\NCm(W,c)\big| q^m = \sum_{\r_1\cdots \r_n \in \reds(c)} q^{\mathrm{des}(\r_1\cdots \r_n)}
  \end{equation*}
  where $\des(\r_1,\ldots,\r_n)$ is the number of descents $r_i >_c r_{i+1}$ in the reflection order induced by the Coxeter element~$c$.
\end{theorem}
\nathanside{Is this trivial? It feels trivial.}
\begin{proof}
  This is a direct consequence of \Cref{prop:nc_fuss_subwords}, after observing that this formula is equivalent to
  \[
    \big|\NCm(W,c)\big| = \sum_{\r_1\cdots\r_n \in \reds(c)} \binom{m+1+\asc(\r_1\cdots\r_n)}{n}
  \]
  where $\asc(\r_1,\ldots,\r_n)$ is the number of ascents $r_i <_c r_{i+1}$.
  Given a reduced $\refl$-word $\r_1\cdots\r_n$ of~$c$, one can place each reflection~$r_i$ into one copy of the $m+1$ copies of $\invs_\refl(\cwo)$ inside the search word $\invs_\refl(\cwom[c][m+1])$, and two consecutive reflections~$r_i$ and~$r_{i+1}$ possibly into the same copy if and only if $r_i <_c r_{i+1}$.
  This shows that the number of ways to place the $\refl$-word $\r_1\cdots\r_n$ into the search word equals the number of tuples $1 \leq a_1 \leq \ldots \leq a_n \leq m+1$ with weak inequalities $a_i \leq a_{i+1}$ if $r_i <_c r_{i+1}$ and strict inequalities otherwise.
  The statement then follows by summing over all factorizations as such tuples $1 \leq a_1 \leq \ldots \leq a_n \leq m+1$ are clearly counted by the above binomial.
\end{proof}

\begin{theorem}
\label{prop:cambnc}
  We have that 
  \begin{enumerate}[(i)]
    \item the posets $\Cambncm(W,c)$ and $\Cambncm(W,\psi(c))$ are dual, and
    \label{eq:cambncitem1}

    \item the number of elements of $\Cambncm(W,c)$ with~$k$ upper covers equals the number of elements $(w_1,\ldots,w_m) \in \NCm(W,c)$ such that $\lengthR(w_1) = k$.
    \label{eq:cambncitem3}
  \end{enumerate}
\end{theorem}

\begin{proof}
  The first item~\eqref{eq:cambncitem1} follows from \Cref{lem:reflection_order}\eqref{it:reflection_order6}.
  The second item~\eqref{eq:cambncitem3} follows from \Cref{def:nc_cambrian_poset}.
% of the $m$-eralized $c$-Cambrian posetobserve that the definition of $\Cambncm$ yields
%   \begin{equation*}
%     \sum q^{\lengthR(w)} &= \bigset{ F \in \Cambncm(W,c) }{ F \text{ has exactly } i \text{ decreasing covers } }
%   \end{equation*}
%   where the sum ranges over all $(\delta_0,\ldots,\delta_m) \in \deltaNCm(W,c)$ and
%   \[
%     w'=\delta_0\cdots\delta_{m-1}, \quad w=\delta_1\cdots\delta_{m}.
%   \]
%   The description in terms of $m$-delta sequences follows from the observation that the reflection lengths $\lengthR(\delta_0\cdots\delta_{m-1})$ and $\lengthR(\delta_1\cdots\delta_{m})$ are equidistributed on all $m$-delta sequences in $\deltaNCm(W,c)$.
%   The description in terms of $\NCm(W,c)$ is obtained by applying the isomorphism $\NCm(W,c) \bij \deltaNCm(W,c)$.
\end{proof}

\chapter{Cluster complexes}
\label{sec:clusters_and_subwords}
%%%%%%%%%%%%%%%%%%%%%%%%%%%%%%%%%%%%%%%%%%%%%%%%%%%%%%%%%%%%%%%%%%%%%%%%%%%%%%%%%%%%%

In this chapter, we study $m$-eralized cluster complexes.
We review the theory of cluster complexes as subword complexes (\Cref{sec:classicalcluster}) and define $m$-eralized compatibility relations and $m$-eralized cluster complexes (\Cref{sec:gen_cs_clusters,,sec:m-cluster-subword}).
We define the Cambrian recurrence and the Cambrian poset on $m$-eralized cluster complexes (\Cref{sec:m-assoc-cambrian-recurrence}).  We use this recurrence to show that our definition recovers known constructions of generalized cluster complexes for bipartite Coxeter elements (\Cref{sec:subwords_and_compatiblity}).
We conclude with topological properties of cluster complexes in our framework (\Cref{sec:clustertopology}), and give a natural bijection to noncrossing partitions (\Cref{sec:m-assoc-ncp}).

%%%%%%%%%%%%%%%%%%%%%%%%%%%%%%%%%%%%%%%%%%%%%%%%%%%%%%%%%%%%%%%%%%%%%%%%%%%%%%%%%%%%%
\section{Classical cluster complexes}
\label{sec:classicalcluster}
%%%%%%%%%%%%%%%%%%%%%%%%%%%%%%%%%%%%%%%%%%%%%%%%%%%%%%%%%%%%%%%%%%%%%%%%%%%%%%%%%%%%%

Let $\Phipm = \PhiP \cup -\Delta$ be the set of \defn{almost positive roots}.  For $s \in \sref$, define a bijection
\begin{align}
  \tau_s : \Phipm &\bij \Phipm \nonumber \\
  \beta &\longmapsto
  \begin{cases}
    \beta     & \text{if } \beta \in -(\Delta\setminus \alpha_s) \\
    s(\beta)  & \text{otherwise}
  \end{cases}\ .  \label{eq:taus}
\end{align}

\christianside{This is a really beautiful story, I recommend you look at these references!}
In their study of finite type cluster algebras, S.~Fomin and A.~Zelevinsky used~$\tau_s$ to define a binary relation on $\Phipm$~\cite{FZ2002,FZ2003}.
R.~Marsh, M.~Reineke, and A.~Zelevinsky~\cite{marsh2003generalized} and, independently, N.~Reading~\cite{Rea20072} interpreted this as the bipartite case of a more general family of relations, depending on a Coxeter element~$c$ (or, equivalently, on an orientation of the Coxeter diagram).

\begin{definition}%[{\cite{marsh2003generalized},\cite[Section~7]{Rea20072}}]
  \label{def:reading_compatibility}
  The \defn{$c$-compatibility relations} are the unique family of relations $\parallelcclassic$ on $\Phipm$ characterized by the following two properties:
  \begin{enumerate}[(i)]
    \item\label{it:compatibility1} for $\alpha \in \Delta$,
      \[
        -\alpha \parallelcclassic \beta \Leftrightarrow \beta \in \Phi_{\langle s_\alpha \rangle},
      \]
    \item\label{it:compatibility2} for~$s$ final in~$c$,
      \[
        \beta_1 \parallelcclassic \beta_2 \Leftrightarrow \tau_s(\beta_1) \parallelcclassic[\coxrnf] \tau_s(\beta_2).
      \]
  \end{enumerate}
  The \defn{$c$-cluster complex} $\Assoc(W,c)$ is the simplicial complex given by all collections of pairwise $c$-compatible almost positive roots.
\end{definition}

The two properties overdetermine this binary relation.
For a proof that~$\parallelcclassic$ is well-defined, we refer to~\cite[Section~7]{Rea20072}.  By symmetry,~\eqref{it:compatibility2} can be equivalently stated for~$s$ initial in~$c$.

A \defn{$c$-cluster} is a facet of the $c$-cluster complex $\Assoc(W,c)$---that is, a maximal subset of almost positive roots which are pairwise $c$-compatible.
In crystallographic type, the $c$-cluster complexes are isomorphic to the cluster complex defined in~\cite{FZ2002}.
We $m$-eralize \Cref{def:reading_compatibility} in \Cref{def:parallelc}, and show in \Cref{cor:m-c-assoc-m-assoc} that this $m$-eralization of compatibility relations reduces for generalized cluster complexes (see~\cite{FR2005}) and bipartite Coxeter elements to S.~Fomin and N.~Reading's compatibility relation recalled in \Cref{def:fr_compatibility}.

\medskip

K.~Igusa and R.~Schiffler found an explicit rule for the compatibility of two roots under $\parallelcclassic$ in~\cite[Theorem 2.5]{IS2010}, using the connection between representation theory and clusters introduced in~\cite{BMRRT2006}; see also~\cite[Section 8]{BW2008} when~$c$ is bipartite.
This combinatorics was studied and extended in a purely combinatorial setting beyond crystallographic type by C.~Ceballos, J.-P.~Labb\'{e}, and C.~Stump in~\cite{CLS2011}, obtaining isomorphisms
  \[
    \Assoc(W,c) \cong \subwordsS(\cw{c},\wo) = \subwordsR(\invs_\refl(\cw{c}),c^{-1}).
  \]
The root configuration gives a bijection between the facets of these subword complexes and noncrossing partitions~\cite[Theorem 6.23]{PS20112}.
These bijections are $m$-eralized in \Cref{sec:m-assoc-ncp} and in \Cref{sec:subwords_and_compatiblity}, respectively.

%%%%%%%%%%%%%%%%%%%%%%%%%%%%%%%%%%%%%%%%%%%%%%%%%%%%%%%%%%%%%%%%%%%%%%%%%%%%%%%%%%%%%
\section{\mhead-eralized compatibility relations}
\label{sec:gen_cs_clusters}
%%%%%%%%%%%%%%%%%%%%%%%%%%%%%%%%%%%%%%%%%%%%%%%%%%%%%%%%%%%%%%%%%%%%%%%%%%%%%%%%%%%%%

We construct in this section an $m$-eralized compatibility relation simmultaneously for all Coxeter elements.
This ties together the compatibility relation from \Cref{def:reading_compatibility} for all Coxeter elements and the compatibility relation of S.~Fomin and N.~Reading for bipartite Coxeter elements.
After having established the existence of the proposed relation, we show in \Cref{sec:fomin_reading} that our construction recovers the Fomin-Reading relation in the bipartite case.

\subsection{Colored almost positive roots}

Let
\[
  \mPhipm \eqdef \bigset{ \alpha^{(k)} }{ \alpha \in \PhiP, 0 \leq k < m } \cup \bigset{ \alpha^{(m)} }{ \alpha \in \Delta }
\]
be the set of \defn{$m$-colored almost positive roots}.
The color~$m$ plays the role of the negative simple roots.
In particular, we use the identification
\[
  \mPhipm[1] \longleftrightarrow \Phipm, \quad \textrm{where }\beta^{(0)} \mapsto \beta, \textrm{ and }\alpha^{(1)} \mapsto -\alpha.
\]
%\christian{another notion for bijection here!}
Given this notation of colored roots, a similar inductive description yields an $m$-eralization of \Cref{def:reading_compatibility} for general Coxeter elements.

For $s \in \sref$, define the bijection
\begin{align}
  \taum : \mPhipm &\bij \mPhipm \nonumber\\
  \beta^{(k)} &\longmapsto
  \begin{cases}
    \beta^{(k-1)} &\text{ if } \beta = \alpha_s \text{ and } k>0 \\
  \beta^{(m)} & \text{ if } \beta = \alpha_s \text{ and } k=0 \\
  \beta^{(k)} & \text{ if } \beta \neq \alpha_s \text{ and } k=m \\
    [s(\beta)]^{(k)} &\text{ otherwise}
  \end{cases}\ .\label{eq:taum}
\end{align} 

\begin{remark}
  It follows from the definition that the order of~$\taum$ is given by $\lcm\big\{ m+1, 2\big\}$.
  Thus,~$\taum$ is an involution if and only if $m=1$, and~$\taum$ reduces to $\tau_s$ in this case.
\end{remark}

For later usage, we record the following elementary property of~$\taum$.

\begin{lemma}
\label{lem:taum_restr}
  Let $s,s' \in \sref$ with $s \neq s'$.
  Then $\taum[s]$ restricts to a bijection on $\left(\Phi_{\langle s' \rangle}\right)_{\geq -1}^{(m)}$.
\end{lemma}
\begin{proof}
  All four definitions in~\eqref{eq:taum} map $\left(\Phi_{\langle s' \rangle}\right)_{\geq -1}^{(m)}$ to itself.
  As~$\taum[s]$ is a bijection, the property follows.
\end{proof}

\subsection{The compatibility relation}

\begin{definition}\label{def:parallelc}
  The \defn{$m$-eralized $c$-compatibility relation}~$\parallel_c$ is the unique family of relations~$\parallel_c$ on $\mPhipm$ by the following two properties.
  For~$s$ final in~$c$, we have
  \begin{enumerate}[(i)]
    \item $\alpha_s^{(m)} \parallel_c \beta^{(k)} \Leftrightarrow \beta \in \Phi_{\langle s \rangle}$, and \label{it:mcompatibility1}
    \item $\beta_1^{(k)} \parallel_c \beta_2^{(\ell)} \Leftrightarrow \taum(\beta_1^{(k)}) \parallel_{\coxrnf} \taum(\beta_2^{(\ell)})$. \label{it:mcompatibility2}
  \end{enumerate}
\end{definition}
\christianside{So that $s$ is now final---not initial.  There was a reason for this choice, I promise.}\nathanside{We just don't remember it.}

In contrast to the case $m=1$, \Cref{def:parallelc}\eqref{it:mcompatibility2} is not equivalent to the corresponding statement for $s$ initial in $c$.
We now show that property~\eqref{it:mcompatibility1}---which assumes $s$ to be final in~$c$---implies the analogue of \Cref{def:reading_compatibility}\eqref{it:compatibility1}.

\begin{proposition}
\label{prop:mcompatibility1new}
  For $s \in \sref$, we have
  \[
    \alpha_s^{(m)} \parallel_c \beta^{(k)} \Leftrightarrow \beta \in \Phi_{\langle s \rangle}.
  \]
  For $m=1$, the $m$-eralized $c$-compatibility relation~$\parallel_c$ therefore recovers \Cref{def:reading_compatibility}.
\end{proposition}

\begin{proof}
  Assume that~$s$ is not final in~$c$ and write $c = s_1\cdots s_n$ with $s = s_p$ for $1 \leq p < n$.
  As $s_i \neq s$ for $i > p$, we obtain
  \begin{align*}
    \alpha_s^{(m)} \parallel_c \beta^{(k)}
      &\Leftrightarrow \alpha_s^{(m)} \parallel_{c'} \taum[s_{p+1}] \circ \cdots \circ \taum[s_n] \big(\beta^{(k)}\big) \\
      &\Leftrightarrow \taum[s_{p+1}] \circ \cdots \circ \taum[s_n] \big(\beta^{(k)}\big) = \gamma^{(\ell)} \text{ with } \gamma \in \Phi_{\langle s \rangle} \text{ and some color } \ell \\
      &\Leftrightarrow \beta \in \Phi_{\langle s \rangle},
  \end{align*}
  for $c' = s_{p+1} \cdots s_{n}c\sninv_n \cdots \sninv_{p+1}$.
  The first equivalence follows from the defining property~\eqref{it:mcompatibility2} and the fact that $\taum[s_{p+1}]\circ\cdots\circ\taum[s_n]\big(\alpha_s^{(m)}\big) = \alpha_s^{(m)}$.  The second equivalence follows from the defining property~\eqref{it:mcompatibility1}, and the third equivalence follows from \Cref{lem:taum_restr}.
\end{proof}

Analogously to~\cite[Lemma~7.1]{Rea20072}, we show that a binary relation on $\mPhipm$ satisfying the two defining properties of~$\parallel_c$ is unique, if it exists.
Existence is deferred to \Cref{thm:m-c-compatibility}.

\medskip

As~$\taum$ coincides with  $\tau_s$ on non-simple roots of a given color, the following can be seen as the $m$-eralized version of~\cite[Lemma~7.1]{Rea20072}.

\begin{lemma}
\label{lem:relation1}
  Let~$c$ be a Coxeter element, and $\beta^{(k)} \in \mPhipm$.
  Then there exists a final sequence $s_1,\dots,s_p,s$ for~$c$ such that
  \[
    \taum[s_p] \circ \cdots \circ \taum[s_1](\beta^{(k)}) = \alpha_{s}^{(k)}.
  \]
\end{lemma}

\begin{proof}
  We first consider the case~$k=m$.
  In this case, $\beta = \alpha_s$ for some $s \in \sref$.
  Write $c = s_n\dots s_1$ and set $p < n$ such that $s_{p+1} = s$.
  As $s_i \neq s$ for $i \leq p$, we obtain the desired property by~\eqref{eq:taum}.

  We next consider the case~$k<m$.
  Let $\c = \s_n\dots\s_1$ and consider
  \[
    \cwo[\rev(\psi({\c}))] = \s_1\cdots\s_N.
  \]
  \Cref{lem:reflection_order}\eqref{it:reflection_order7} implies that $s_1,\ldots,s_N$ is a final sequence for~$c$.
  Set~$0 \leq p < N$ such that $s_1\cdots s_{p}(\alpha_{s_{p+1}}) = \beta$ and compute
  \[
    \taum[s_p] \circ \cdots \circ \taum[s_1](\beta^{(k)}) = \big[s_p\cdots s_1(\beta)\big]^{(k)} = [\alpha_{s_{p+1}}]^{(k)}
  \]
  by~\eqref{eq:taum}.
\end{proof}

\begin{lemma}
\label{lem:relation2}
  Let~$c$ be a Coxeter element and let $\beta \in \PhiP$.
  There there exist a final sequence $s_1,\dots,s_p,s$ for~$c$ such that
  \[
    \taum[s_p] \circ \dots \circ \taum[s_1](\beta^{(k)}) = \alpha_s^{(m)}.
  \]
\end{lemma}
\begin{proof}
  Applying \Cref{lem:relation1}, we obtain a final sequence $s_1,\dots,s_{p},s$ for~$c$ such that
  \[
    \taum[s_{p}] \circ \cdots \circ \taum[s_1](\alpha^{(k)}) = \alpha_{s}^{(k)}.
  \]
  If $k=m$, we conclude the result.
  Otherwise, we apply $\taum[s]$ to obtain $\alpha_{s}^{(k-1)}$, and proceed by again applying \Cref{lem:relation1}.
\end{proof}

\begin{proposition}
\label{prop:unique_compatible}
  The $m$-eralized $c$-compatibility relation~$\parallel_c$ is uniquely determined by its two defining properties.
\end{proposition}
\christianside{Remember this, forget the lemmas.}
\begin{proof}
  This is an immediate consequence of \Cref{lem:relation2}.
\end{proof}

%%%%%%%%%%%%%%%%%%%%%%%%%%%%%%%%%%%%%%%%%%%%%%%%%%%%%%%%%%%%%%%%%%%%%%%%%%%%%%%%%%%%%
\section{\mhead-eralized cluster complexes}
\label{sec:m-cluster-subword}
%%%%%%%%%%%%%%%%%%%%%%%%%%%%%%%%%%%%%%%%%%%%%%%%%%%%%%%%%%%%%%%%%%%%%%%%%%%%%%%%%%%%%

We construct $m$-eralized cluster complexes as subword complexes and in terms of colored almost positive roots.

\subsection{Cluster complexes as subword complexes}
E.~Tzanaki encoded $m$-eralized clusters as certain colored factorizations for bipartite $c$~\cite[Theorem~1.1]{Tza2008}, $m$-eralizing the result for $m=1$ in~\cite{IS2010,BW2008}.
We extend this approach to general Coxeter elements.

\begin{definition}
\label{def:m-c-cluster}
  Let~$c$ be a Coxeter element and let~$\c$ be a reduced word for~$c$.
  The \defn{$m$-eralized $c$-cluster complex} is the $c$-initial subword complex defined by either
  \begin{align*}
    \DeltaAssocm(W,c) \eqdef\ & \subwordsSB(\cwm{c}, \bwom),\\
                           =\ & \subwordsS(\cwm{c}, \wom, mN),  \text{ or }\\
    \NablaAssocm(W,c) \eqdef\ & \subwordsR(\invs_\refl(\cwm{c}), c^{-1}).
  \end{align*}
The \defn{$m$-eralized $c$-Cambrian poset} $\Cambassocm(W,c)$ is the increasing flip poset of $\DeltaAssocm(W,c)$.
\end{definition}

The equality $\DeltaAssocm(W,c) = \NablaAssocm(W,c)$ follows from the general equality proven in \Cref{prop:dualsubwordcomplexes}, where we again emphasize the different flavours of both.
We may represent a face~$I$ of $\NablaAssocm(W,c)$ as a word of colored reflections $\r_1^{(i_1)}\cdots \r_k^{(i_k)}$, where $k \leq n$ and
\begin{itemize}
  \item a reflection $r_j$ is colored by the copy of $\invs_\refl(\cwo)$ to which it belongs, or by~$m$ if it belongs to the final copy of $\invs_\refl(\c)$;
  \item $r_a<_{\c} r_b$ when $a<b$ and $i_a=i_b$; and
  \item $r_k \cdots r_1 \leq_\refl c$ with $\lengthR(r_k\cdots r_1) = k$.
\end{itemize}
By construction, every face is of the form above.  We will prove in \Cref{prop:dualfaces} that these three properties characterize the faces of $\NablaAssocm(W,\c)$.

%\medskip
%\subsection{\mhead-eralized cluster complexes and compatibility}

\subsection{Cluster complexes and colored almost positive roots}
We now begin to relate the above construction to $\parallel_c$.  Let $\c=\s_1\cdots \s_n$ for a Coxeter element~$c$.
We know from \Cref{lem:reflection_order}\eqref{it:reflection_order2} and~\eqref{it:reflection_order3}
that $\c\cwo \equiv \cwo\psi(\c)$ and thus that
\[
  \c \cwom \equiv \cwom \psi^m(\c).
\]
We therefore have%, similar to the situation in~\eqref{eq:wom},
\begin{equation}
  \invs_\refl(\cwm{c}) \equiv  \underbrace{\invs_\refl(\cwo) \invs_\refl(\cwo) \cdots \invs_\refl(\cwo)}_{\text{colors }0,\ldots,m-1} \underbrace{\invs_\refl(\c)}_{\text{color }m}, \label{eq:cwm}
\end{equation}
where the reflections in the $i$\th\ copy of $\invs_\refl(\cwo)$ have color~$(i-1)$, and the~$n$ roots of color~$m$ all come from the final copy of $\psi^m(\c)$.
This defines a bijection
\begin{align}
\toap_c: \invs_\refl(\cwm{c}) & \bij \mPhipm \nonumber \\
  r^{(k)} &\longmapsto
    \begin{cases}
      \alpha_r^{(k)} &\text{if } k<m \\
      \big[s_{i-1}\cdots s_{1}(\alpha_r)\big]^{(k)}  &\text{if } k=m
    \end{cases}\ , \label{rem:almost_pos_are_these}
\end{align}
where $i$ is chosen so that $\alpha_r = s_1\cdots s_i(\alpha_{s_{i+1}})$ if $\c = \s_1\cdots \s_n$.
% is the $(i+1)$\st\ letter in the final 

%We make the following definition only to draw the connection to S.~Fomin and N.~Reading's compatibility relation \Cref{sec:fomin_reading}.

\begin{definition}
\label{def:assocm}
  Let $\Assocm(W,c)$ be the image of the faces of $\NablaAssocm(W,c)$ under the map $\toap_c$ defined in~\eqref{rem:almost_pos_are_these}.
  That is,
  \begin{align*}
    \toap_c: \NablaAssocm(W,c) &\bij \Assocm(W,c)\\
    \r_1^{(i_1)}\cdots \r_k^{(i_k)} &\longmapsto \{\toap_c(r_1^{(i_1)}),\ldots,\toap_c(r_k^{(i_k)})\}.
  \end{align*}
\end{definition}

%\Cref{prop:dualfaces} will establish the isomorphisms
%\[
%  \Assocm(W,c) \cong \NablaAssocm(W,c) \cong \DeltaAssocm(W,c).
%\]

%\medskip

We show in \Cref{thm:m-c-compatibility} that two $m$-colored almost positive roots are compatible under~$\parallel_c$ if and only if they appear together in a face of $\Assocm(W,c)$.

\begin{definition}
\label{def:m-assoc-support}
  The \defn{support} of a facet~$I$ of $\DeltaAssocm(W,c)$ is
  \[
    \supp(I)\eqdef\set{ s \in \sref }{ \alpha_s^{(0)} \notin \Roots{I}}.
  \]
  The support on $\NablaAssocm(W,c)$ and on $\Assocm(W,c)$ are defined analogously.
\end{definition}

\Cref{fig:assocA22} shows all $12$ facets of
\[
  \DeltaAssocm[2](\WA[3],st) \cong \NablaAssoc^{(2)}(\WA[3],st) \cong \Assocm[2](\WA[3],st),
\]
with their supports.
\begin{figure}[t]
  \begin{center}
    \begin{tabular}{c|c|c|c|c}
      $\DeltaAssocm[2](\WA[3],st)$ & $\NablaAssocm[2](\WA[3],st)$  & $\Assocm[2](\WA[3],st)$  & $\Roots{I}$ & $\supp(I)$ \\
      \hline
      $\s\t.\S\T\S.\T\S\T$ &
      $\s\u\T.\S\U\T.\S\U$ &
      $\alpha^{(0)},\gamma^{(0)}$ &
      $\alpha^{(0)},\beta^{(0)}$ &
      $-$
      \\
      $\S\T.\S\T\S.\T\s\t$ &
      $\S\U\T.\S\U\T.\s\u$ &
      $\alpha^{(2)},\beta^{(2)}$ &
      $\alpha^{(2)},\beta^{(2)}$ &
      $s,t$
      \\
      $\S\T.\S\T\s.\t\S\T$ &
      $\S\U\T.\S\u\t.\S\U$ &
      $\gamma^{(1)},\beta^{(1)}$ &
      $\gamma^{(1)},\alpha^{(2)}$ &
      $s,t$
      \\
      $\S\T.\s\t\S.\T\S\T$ &
      $\S\U\t.\s\U\T.\S\U$ &
      $\beta^{(0)},\alpha^{(1)}$ &
      $\beta^{(0)},\gamma^{(1)}$ &
      $s,t$
      \\
      \hline
      $\S\t.\s\T\S.\T\S\T$ &
      $\S\u\t.\S\U\T.\S\U$ &
      $\gamma^{(0)},\beta^{(0)}$ &
      $\gamma^{(0)},\alpha^{(1)}$ &
      $s$
      \\
      $\s\T.\S\T\S.\T\S\t$ &
      $\s\U\T.\S\U\T.\S\u$ &
      $\alpha^{(0)},\beta^{(2)}$ &
      $\alpha^{(0)},\beta^{(2)}$ &
      $t$
      \\
      $\S\T.\S\T\S.\t\s\T$ &
      $\S\U\T.\S\U\t.\s\U$ &
      $\beta^{(1)},\alpha^{(2)}$ &
      $\beta^{(1)},\gamma^{(2)}$ &
      $s,t$
      \\
      $\S\T.\S\t\s.\T\S\T$ &
      $\S\U\T.\s\u\T.\S\U$ &
      $\alpha^{(1)},\gamma^{(1)}$ &
      $\alpha^{(1)},\beta^{(1)}$ &
      $s,t$
      \\
      \hline
      $\s\T.\S\T\s.\T\S\T$ &
      $\s\U\T.\S\u\T.\S\U$ &
      $\alpha^{(0)},\gamma^{(1)}$ &
      $\alpha^{(0)},\beta^{(1)}$ &
      $t$
      \\
      $\S\T.\s\T\S.\T\s\T$ &
      $\S\U\t.\S\U\T.\s\U$ &
      $\beta^{(0)},\alpha^{(2)}$ &
      $\beta^{(0)},\gamma^{(2)}$ &
      $s,t$
      \\
      \hline
      $\S\t.\S\T\S.\t\S\T$ &
      $\S\u\T.\S\U\t.\S\U$ &
      $\gamma^{(0)},\beta^{(1)}$ &
      $\gamma^{(0)},\alpha^{(2)}$ &
      $s$
      \\
      $\S\T.\S\t\S.\T\S\t$ &
      $\S\U\T.\s\U\T.\S\u$ &
      $\alpha^{(1)},\beta^{(2)}$ &
      $\alpha^{(1)},\beta^{(2)}$ &
      $s,t$
    \end{tabular}
  \end{center}
  \caption{The three variants of the $m$-eralized $st$-clusters for $\WA[3]$ with $m=2$, together with their root configurations and supports.  They are arranged according to their orbits under Cambrian rotation, defined in~\Cref{sec:m-assoc-cambrian-recurrence}.}
  \label{fig:assocA22}
\end{figure}

%%%%%%%%%%%%%%%%%%%%%%%%%%%%%%%%%%%%%%%%%%%%%%%%%%%%%%%%%%%%%%%%%%%%%%%%%%%%%%%%%%%%%
\section{The Cambrian rotation and recurrence}
\label{sec:m-assoc-cambrian-recurrence}
%%%%%%%%%%%%%%%%%%%%%%%%%%%%%%%%%%%%%%%%%%%%%%%%%%%%%%%%%%%%%%%%%%%%%%%%%%%%%%%%%%%%%

\subsection{The shift operator}
Let~$\c$ be a reduced word for~$c$ with first letter~$\s$.  It is immediate from \Cref{lem:reflection_order}\eqref{it:reflection_order5} and~\eqref{it:reflection_order7} that for $\cwm{c}=\s_1 \s_2 \dots \s_{n+mN}$,
\[\s_2\dots \s_{n+mN} \psi^m(\s_1) \equiv \cwm{\coxr}.\] %where $\s'= \s^{(\wom)}$.
This identification defines a canonical isomorphism
\[
  \Shift_s : \DeltaAssocm(W,c)\bij\DeltaAssocm(W,\coxrn).
\]
This isomorphism is explicitly described on $\NablaAssocm(W,c) \bij \NablaAssocm(W,\coxrn)$ by
\begin{align}
  \Shift_s : \invs_\refl(\cwm{c}) &\bij \invs_\refl(\cwm{\coxr}) \nonumber \\
  r^{(k)} &\longmapsto
    \begin{cases}
      s^{(k-1)}  & \text{if } r=s \text{ and } k>0 \\
      \left[\ {}^{\coxsn}s\right]^{(m)}  & \text{if } r=s \text{ and } k=0 \\
      [r^s]^{(k)} & \text{if } r \neq s \\
    \end{cases} \ .\label{eq:shifts}
\end{align}
\christianside{Everything here behaves as expected, no pitfalls anywhere. Notationally a little heavy, though.}
\nathanside{In an initial draft, we invented the notation $w^s=sws^{-1}$.  A referee set us straight, but there were a lot of changes to be made.}

\begin{example}
  We visualize the definition of~$\Shift_s$ for the initial $s = s_1$ in $c = s_1s_2s_3$ with $\coxrn = s_2s_3s_1$.
  In this case, we obtain
  \begin{center}
    \begin{tikzpicture}[xscale=1.2,yscale=.7]
      \tikzstyle{rect}=[shape=rectangle,minimum width=3cm, text width=3cm]
      \node (U1) at (1,2) {$\s_1$};
      \node (U2) at (2,2) {$\s_2$};
      \node (U3) at (3,2) {$\s_3$};
      \node (U4) at (4,2) {$\s_1$};
      \node (U5) at (5,2) {$\s_2$};
      \node (U6) at (6,2) {$\s_3$};
      \node (U7) at (7,2) {$\s_1$};
      \node (U8) at (8,2) {$\s_2$};
      \node (U9) at (9,2) {$\s_1$};

      \node (UR1) at (1,1) {$(12)^{(0)}$};
      \node (UR2) at (2,1) {$(13)^{(0)}$};
      \node (UR3) at (3,1) {$(14)^{(0)}$};
      \node (UR4) at (4,1) {$(23)^{(0)}$};
      \node (UR5) at (5,1) {$(24)^{(0)}$};
      \node (UR6) at (6,1) {$(34)^{(0)}$};
      \node (UR7) at (7,1) {$(12)^{(1)}$};
      \node (UR8) at (8,1) {$(13)^{(1)}$};
      \node (UR9) at (9,1) {$(14)^{(1)}$};

      \node (DR2) at (2,-1) {$(23)^{(0)}$};
      \node (DR3) at (3,-1) {$(24)^{(0)}$};
      \node (DR4) at (4,-1) {$(13)^{(0)}$};
      \node (DR5) at (5,-1) {$(14)^{(0)}$};
      \node (DR6) at (6,-1) {$(34)^{(0)}$};
      \node (DR7) at (7,-1) {$(12)^{(0)}$};
      \node (DR8) at (8,-1) {$(23)^{(1)}$};
      \node (DR9) at (9,-1) {$(24)^{(1)}$};
      \node (DR1) at (10,-1) {$(13)^{(1)}$};

      \node (D2) at (2,-2) {$\s_2$};
      \node (D3) at (3,-2) {$\s_3$};
      \node (D4) at (4,-2) {$\s_1$};
      \node (D5) at (5,-2) {$\s_2$};
      \node (D6) at (6,-2) {$\s_3$};
      \node (D7) at (7,-2) {$\s_1$};
      \node (D8) at (8,-2) {$\s_2$};
      \node (D9) at (9,-2) {$\s_1$};
      \node (D1) at (10,-2) {$\s_3$};

      \draw (UR1.south) edge[out=270,in=90,looseness=.5,->] (DR1.north);
      \draw (UR2.south) edge[out=270,in=90,->] (DR2.north);
      \draw (UR3.south) edge[out=270,in=90,->] (DR3.north);
      \draw (UR4.south) edge[out=270,in=90,->] (DR4.north);
      \draw (UR5.south) edge[out=270,in=90,->] (DR5.north);
      \draw (UR6.south) edge[out=270,in=90,->] (DR6.north);
      \draw (UR7.south) edge[out=270,in=90,->] (DR7.north);
      \draw (UR8.south) edge[out=270,in=90,->] (DR8.north);
      \draw (UR9.south) edge[out=270,in=90,->] (DR9.north);
    \end{tikzpicture}
  \end{center}
  where the first two rows show $\cwo$ and $\invs_\refl(\cwo)$, the third row shows the map $\Shift_{(12)}$, and the last two rows show $\invs_\refl(\cwo[\coxr])$ and $\cwo[\coxr]$.
  (Observe that in the later, the last two commuting reflections $s_1$ and $s_3$ are interchanged.)
\end{example}

\subsection{The Cambrian rotation}

For $s$ initial in~$c$, we extend the definition of~$\Shift_s$ to words in $\invs_\refl(\cwm{c})$---and thus to faces---by
\begin{align}
  \Shift_s : \NablaAssocm(W,c) &\bij \NablaAssocm(W,\coxrn) \nonumber\\
  \r_1^{(i_1)}\r_2^{(i_2)} \cdots \r_k^{(i_k)} &\longmapsto
    \begin{cases}
      \Shift_s(\r_2^{(i_2)} \cdots \r_k^{(i_k)}\s^{(0)}) & \text{if } \r_1^{(i_1)}= \s^{(0)}\\
      \Shift_s(\r_1^{(i_1)}\r_2^{(i_2)} \cdots \r_k^{(i_k)}) & \text{otherwise} \end{cases}\ .  \label{eq:ass_shift}
\end{align}

\begin{example}
\label{ex:assocshiftorbit}
  Parallel to \Cref{ex:ncshiftorbit}, alternately applying $\Shift_s$ and $\Shift_t$ to the facet $\{1,2\} \in \DeltaAssocm[2](\WA[3],st)$ gives the orbit
  \[
    \begin{array}{ccccc}
      \s\t\S\T\S\T\S\T
      &\xmapsto{\Shift_s}&
      \t\S\T\S\T\S\T\s
      &\xmapsto{\Shift_t}&
      \S\T\S\T\S\T\s\t \\
      &\xmapsto{\Shift_s}&
      \T\S\T\S\T\s\t\S
      &\xmapsto{\Shift_t}&
      \S\T\S\T\s\t\S\T \\
      &\xmapsto{\Shift_s}&
      \T\S\T\s\t\S\T\S
      &\xmapsto{\Shift_t}&
      \S\T\s\t\S\T\S\T \\
      &\xmapsto{\Shift_s}&
      \T\s\t\S\T\S\T\S
      &\xmapsto{\Shift_t}&
      \s\t\S\T\S\T\S\T
    \end{array}.
  \]
  The same computation in $\NablaAssoc^{(2)}(\WA[3],st)$ gives
  \[
    \begin{array}{ccccc}
      \s^{(0)} \u^{(0)}
      &\xmapsto{\Shift_s}&
      \t^{(0)} \u^{(2)}
      &\xmapsto{\Shift_t}&
      \s^{(2)} \u^{(2)} \\
      &\xmapsto{\Shift_s}&
      \s^{(1)} \t^{(2)}
      &\xmapsto{\Shift_t}&
      \u^{(1)} \t^{(1)}\\
      &\xmapsto{\Shift_s}&
      \t^{(1)} \u^{(1)}
      &\xmapsto{\Shift_t}&
      \t^{(0)} \s^{(1)} \\
      &\xmapsto{\Shift_s}&
      \u^{(0)} \s^{(0)}
      &\xmapsto{\Shift_t}&
      \s^{(0)} \u^{(0)}
    \end{array}.
  \]
\end{example}

\begin{definition}
\label{def:cs_c_cambrian_rotation}
  The \defn{$m$-eralized $c$-Cambrian rotation} is
  \[
    \Camb_c \eqdef \Shift_{s_n} \circ \cdots \circ \Shift_{s_1} :  \Assocm(W,c)\bij \Assocm(W,c).
  \]
\end{definition}

This composition evidently does not depend on the chosen reduced word.  The elements in \Cref{fig:assocA22} are arranged according to their orbits under Cambrian rotation.

\begin{proposition}
\label{prop:cambrian_recurrence_order}
  \[
    \order(\Camb_c) =
    \begin{cases}
      mh+2     &\text{if } \psi \not\equiv \id \text{ and } m \text{ odd}\\
      (mh+2)/2 &\text{otherwise}
    \end{cases}\ .
  \]
\end{proposition}
\begin{proof}
  The statement follows from the facts that~$\wo$ is an involution on~$\sref$, $2N = nh$, and the length of the word $\cwm{c}$ is $n+mN$.
\end{proof}

\subsection{The Cambrian recurrence}
We modify the first case in the definition of the shift operator to obtain an inductive characterization of $\Assocm(W,c)$, the \defn{$m$-eralized $c$-Cambrian recurrence}.

\begin{proposition}
\label{prop:assoc_cambrian_recurrence}
  Let~$s$ be initial in~$c$ and let $I=\r_1^{(i_1)} \r_2^{(i_2)} \cdots \r_n^{(i_n)}$.
  If $r^{(i_1)} = s^{(0)}$, we set $I_{\langle s \rangle} \eqdef [\r_2^s]^{(i_2)} \cdots [\r_n^s]^{(i_n)}$.
  Then
  \[
    I \in \NablaAssocm(W,c)
    \Leftrightarrow
    \begin{cases}
      \Shift_s(I_{\langle s \rangle}) \in \NablaAssocm(W_{\langle s \rangle},\coxsn) & \text{if } r_1^{(i_1)} = s^{(0)} \\
      \Shift_s(I) \in \NablaAssocm(W,\coxrn) & \text{otherwise}
    \end{cases}\ .
  \]
\end{proposition}
\begin{proof}
  If $\r_1^{(i_1)}\ne \s^{(0)}$, the statement is clear, so assume otherwise.
%   \hugh{I added the first sentence of the proof.}
  We compute that $r_2^s \cdots r_n^s = (r_1 \cdots r_n) s = c^{-1}s = (sc)^{-1}$.
  Each $r_i^s$ is a reflection inside the parabolic subgroup $W_{\langle s \rangle}$, since $\s \r_n^s \cdots \r_2^s$ is a reduced $\refl$-word for~$c$.
  It is a facet by \Cref{lem:reflection_order}\eqref{it:reflection_order7} and \Cref{lem:reflection_order4}.
\end{proof}

\begin{example}
\label{ex:assoccambrian}
  Parallel to \Cref{ex:ncdeltacambrian}, the Cambrian recurrence for the facet $s^{(1)} u^{(2)} \in \NablaAssocm[2](\WA[3],st)$ is computed as
  \[
    \underbrace{\s^{(1)} \u^{(2)}}_{st} \mapsto \underbrace{\s^{(0)} \t^{(2)}}_{ts} \mapsto \underbrace{\u^{(0)} \t^{(1)}}_{st} \mapsto \underbrace{\t^{(0)} \u^{(1)}}_{ts} \mapsto \underbrace{\s^{(1)}}_{s} \mapsto \underbrace{\s^{(0)}}_{s} \mapsto \underbrace{-}_{\one}.
  \]
  The same computation for $\S\T\S\t\S\T\S\t \in \DeltaAssocm[2](\WA[3],st)$ is
  \[
    \underbrace{\S\T\S\t\S\T\S\t}_{st}
    \mapsto
    \underbrace{\T\S\t\S\T\S\t\S}_{ts}
    \mapsto
    \underbrace{\S\t\S\T\S\t\S\T}_{st}
      \mapsto
    \underbrace{\t\S\T\S\t\S\T\S}_{ts}
    \mapsto
    \underbrace{\S\s\S}_{s}
      \mapsto
    \underbrace{\s\S\S}_{s}
    \mapsto
    \underbrace{-}_{\one}.
  \]
\end{example}

\Cref{fig:cambassocA22} illustrates $\Cambassocm[2](\WA[3],st)$.
\begin{figure}[t]
  \begin{center}
    \begin{tikzpicture}[scale=1.3]
      \tikzstyle{rect}=[rectangle,draw,opacity=.5,fill opacity=1]
      \tikzstyle{sort}=[]
      \tikzstyle{flip1}=[inner sep = 1pt,circle,draw,opacity=.5,fill opacity=1,fill=black!10]
      \tikzstyle{flip2}=[inner sep = 1pt,circle,draw,opacity=.5,fill opacity=1,fill=black!40]
      \node[rect,sort] (e)   at (0,0)
        {$\s\t\S\T\S\T\S\T$};
      \node[rect,sort] (s)   at (-1,1)
        {$\S\t\s\T\S\T\S\T$};
      \node[rect,sort] (t)   at ( 1,1)
        {$\s\T\S\T\s\T\S\T$};
      \node[rect,sort] (st)  at (-1,2)
        {$\S\T\s\t\S\T\S\T$};
      \node[rect,sort] (sts) at ( 0,3)
        {$\S\T\S\t\s\T\S\T$};
      \node[rect,sort] (stss)   at ( 1,4)
        {$\S\T\S\t\S\T\S\t$};
      \node[rect,sort] (stst)   at (-1,4)
        {$\S\T\S\T\s\t\S\T$};
      \node[rect,sort] (ststs)  at (-1,5)
        {$\S\T\S\T\S\t\s\T$};
      \node[rect,sort] (stssts) at ( 0,6)
        {$\S\T\S\T\S\T\s\t$};
      \node[rect,sort] (ss)   at (-3,2)
        {$\S\t\S\T\S\t\S\T$};
      \node[rect,sort] (tt)   at ( 2,2)
        {$\s\T\S\T\S\T\S\t$};
      \node[rect,sort] (stt)  at (-3,3)
        {$\S\T\s\T\S\T\s\T$};

      \draw (e) to node[flip1] {\scriptsize $\alpha$} (s) to node[flip1] {\scriptsize $\gamma$} (st) to node[flip1] {\scriptsize $\beta$} (sts) to node[flip2] {\scriptsize $\beta$} (stss) to node[flip2] {\scriptsize $\alpha$} (stssts);
      \draw (e) to node[flip1] {\scriptsize $\beta$} (t) to node[flip1] {\scriptsize $\alpha$} (sts) to node[flip2] {\scriptsize $\alpha$} (stst) to node[flip2] {\scriptsize $\gamma$} (ststs) to node[flip2] {\scriptsize $\beta$} (stssts);
      \draw (s) to node[flip2] {\scriptsize $\alpha$} (ss) to[bend right=20] node[flip1] {\scriptsize $\gamma$} (stst);
      \draw (t) to node[flip2] {\scriptsize $\beta$} (tt) to node[flip1] {\scriptsize $\alpha$} (stss);
      \draw (st) to node[flip2,pos=0.35] {\scriptsize $\gamma$} (stt) to[bend left=20] node[flip1] {\scriptsize $\beta$} (ststs);
    \end{tikzpicture}
  \end{center}
  \caption{
    The flip poset $\Cambassocm[2](\WA[3],st)$ on the facets of $\Assoc^{(2)}(\WA[3],st)$.
      The colored roots labelling the edges correspond to the direction of the flips, with the color indicated by the shading.}
  \label{fig:cambassocA22}
\end{figure}
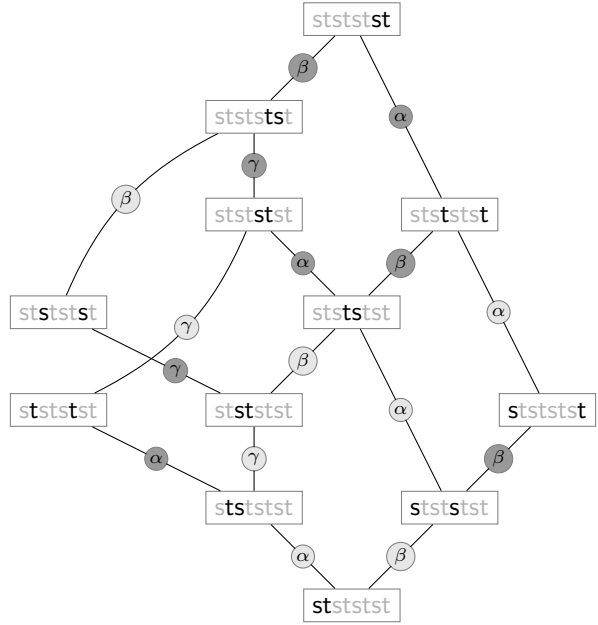

\begin{proposition}
\label{prop:mclusterisomorphism}
  For any two Coxeter elements $c,c'$, the two $m$-eralized cluster complexes $\DeltaAssocm(W,c)$ and $\DeltaAssocm(W,c')$ are isomorphic.
\end{proposition}
\begin{proof}
  Let $s_1,\ldots,s_p$ be an initial sequence for~$c$.
  Then,
  \[
    \Shift_{s_p} \circ \cdots \circ \Shift_{s_1} : \DeltaAssocm(W,c)\bij\DeltaAssocm(W,\sninv_p\cdots\sninv_1cs_1\cdots s_p)
  \]
  is the desired isomorphism.
\end{proof}

\begin{remark}
\label{rem:graphautomorphism}
  For any two Coxeter elements~$c$ and $c'$, the underlying unoriented flip graphs of the complexes $\Assocm(W,c)$ and $\Assocm(W,c')$ are isomorphic.
  The isomorphism is induced by the shift operation, and is given by the isomorphism in \Cref{prop:mclusterisomorphism}.
  In particular, the map $\Camb_c$ is a graph automorphism.
  For $m=1$, the increasing flip graph coincides with the Hasse diagram of the Cambrian poset, while for $m\geq 2$ the increasing flip graph is no longer transitively reduced.
  Therefore, the shift operation does not induce a isomorphism between the unoriented Hasse diagrams of $\Cambassocm(W,c)$ and $\Cambassocm(W,c')$.
\end{remark}

\subsection{Flagness of \mhead-eralized cluster complexes}

We conclude this section by using the Cambrian recurrence to intrinsically describe the faces of the complex $\NablaAssocm(W,\c)$.
We then deduce that the $m$-eralized $c$-cluster complex is flag, and so can be reconstructed as the clique complex of its edges.\footnote{This statement was inadvertently omitted in~\cite{CLS2011}.}
\christianside{Never forget to forget the flagness.}
\nathanside{Is flagness is a word?  Please run a spellchecker before submission.}

\begin{proposition}
\label{prop:dualfaces}
  Let $\r_1^{(i_1)}\cdots \r_k^{(i_k)}$ be a subword of $\invs_\refl(\cwm{c})$.
  Then $\r_1^{(i_1)}\cdots \r_k^{(i_k)}$ is a face of $\NablaAssocm(W,c)$ if and only if $r_k \cdots r_1 \leqref c$ with $\lengthR(r_k\cdots r_1) = k$.
\end{proposition}

\begin{proof}
  We have already seen in \Cref{sec:m-cluster-subword} that these are necessary conditions for a subword to be a face, so we only need to show that these are also sufficient.
  For the empty word, the statement is clear, so we can assume $k \geq 1$ and that $r_k \cdots r_1 \leqref c$ with $\lengthR(r_k\cdots r_1) = k$.
  We have to show that this subword of $\invs_\refl(\cwm{c})$ can be extended to a subword that is a reduced $\refl$-word for~$c^{-1}$.

  But this is a direct consequence of the Cambrian recurrence.
  Either the first letter $r_1^{(i_1)}$ is not equal to $s^{(0)}$ in which case we conclude the statement by induction on the second equivalence in \Cref{prop:assoc_cambrian_recurrence}.
  Otherwise, the first letter $\r_1^{(i_1)}$ is $\s^{(0)}$.  In this case, we conclude the statement by induction from the first equivalence in \Cref{prop:assoc_cambrian_recurrence} as follows.
  Removing this first letter $\s^{(0)}$ from the facet and conjugating all other letters by~$s$ gives the facet $I_{\langle s \rangle} = [\r_2^s]^{(i_2)}\cdots [\r_k^s]^{(i_k)} \in \NablaAssocm(W,\coxsn)$.  By induction, $I_{\langle s \rangle}$ can be extended to a reduced subword of $\invs_\refl(\cwm{\coxs})$ whose product is $c^{-1}s$.  Lifting this subword back to $\NablaAssocm(W,c)$ has the effect of conjugating by $s$, giving a reduced subword of $\invs_\refl(\cwm{\c})$ whose product is $sc^{-1}$.  Adding back the first letter $\s^{(0)}$ gives a reduced subword for $c^{-1}$.
\end{proof}

\begin{example}
  The cluster complex $\NablaAssocm[1](\WA[4],s_1s_2s_3)$ has search word
  \[
    \invs_\refl(\c\cwo) = (12)^{(0)}(13)^{(0)}(14)^{(0)}(23)^{(0)}(24)^{(0)}(12)^{(1)}(34)^{(0)}(13)^{(1)}(14)^{(1)}.
  \]
  There are $21$ subwords of this search word of the form $\r_1^{i_1}\r_2^{i_2}$ with $r_2r_1 \leqref c$ and $r_1\neq r_2$.
  Two such examples are $(23)^{(0)}(14)^{(1)}$ and $(24)^{(0)}(12)^{(1)}$. 
\end{example}

\begin{corollary}
\label{cor:flag}
  $\Assocm(W,c)$ is flag---that is, all its minimal nonfaces have cardinality~$2$.
\end{corollary}\nathanside{Note the dependence on~\Cref{prop:ncpairs}.  Subtle.}

\begin{proof}
  We first check that every letter is contained in a facet of $\DeltaAssocm(W,c)$.
  Write $\cwm{c} = \c \restri{\c}{I_1} \cdots \restri{\c}{I_k}$ where $\restri{\c}{I_1} \cdots \restri{\c}{I_k}$ is the $c$-sorting word of~$\wom$ decomposed into its subwords of the individual copies of~$\c$.
  By definition, the letters in the first copy of~$\c$ form a facet.   Furthermore,~\Cref{lem:wosortingword} implies that $I_1 \supseteq I_2 \supseteq \cdots \supseteq I_k$---so the word obtained from $\cwm{c}$ omitting any $\restri{\c}{I_j}$ still contains a copy of $\wom$.  Therefore, every letter is contained in facet.

  Now let $\r_1^{(i_1)} \cdots \r_k^{(i_k)}$ be a subword of $\invs_\refl(\cwm{c})$ such that $\r_a^{(i_a)} \r_b^{(i_b)}$ for $1 \leq a < b \leq k$ is a face of $\NablaAssocm(W,c)$.
  Therefore $r_b r_a \leqref c$ with $r_b \neq r_a$ by \Cref{prop:dualfaces}.
  By \Cref{prop:ncpairs}, $r_k \cdots r_1 \leq c$ with $\lengthR(r_k \cdots r_1) = k$.  Therefore, $\r_1^{(i_1)} \cdots \r_k^{(i_k)}$ is a face of $\NablaAssocm(W,c)$ by \Cref{prop:dualfaces}.
\end{proof}

%%%%%%%%%%%%%%%%%%%%%%%%%%%%%%%%%%%%%%%%%%%%%%%%%%%%%%%%%%%%%%%%%%%%%%%%%%%%%%%%%%%%%
\section{Subword complexes and compatibility}
\label{sec:subwords_and_compatiblity}
%%%%%%%%%%%%%%%%%%%%%%%%%%%%%%%%%%%%%%%%%%%%%%%%%%%%%%%%%%%%%%%%%%%%%%%%%%%%%%%%%%%%%

\subsection{Compatibility via the \mhead-eralized cluster complex}
We show that the construction of $\Assocm(W,c)$ given in \Cref{def:m-c-cluster} implies the existence of the~$\parallel_c$ relation proposed in \Cref{def:parallelc}.  Write~$\parallel'_c$ for the relation on $\Phi^{(m)}_{\geq -1}$ given by
\begin{equation}
  \label{eq:reltwo}
  r^{(k)} \parallel'_c t^{(\ell)} \text{ if } \{\r^{(k)},\t^{(\ell)}\} \text{ is a face of } \Assocm(W,c).
\end{equation}

\begin{lemma}
  \label{lem:shift_eq_taum}
  The bijection~$\toap: \invs_\refl(\cwm{c}) \bij \mPhipm$ defined in~\eqref{rem:almost_pos_are_these} sends $\Shift_s$ to~$\taum$.
  That is, for $r^{(k)} \in \invs_\refl(\c \cwom)$ and~$s$ initial in~$c$, we have
  \[
    \toap_{\coxrn}\big(\Shift_s(r^{(k)})\big) = \taum\big(\toap_c(r^{(k)})\big).
  \]
\end{lemma}
\begin{proof}
  This follows immediately from the definitions~\eqref{eq:taum} and~\eqref{eq:shifts}.
  Note that the third and fourth cases in~\eqref{eq:taum} are merged together as the third case in~\eqref{eq:shifts} using~\eqref{rem:almost_pos_are_these}.
\end{proof}

\begin{theorem}
\label{thm:m-c-compatibility}
  %$r^{(k)} \parallel'_c t^{(\ell)}$ if and only if $\toap_c(r^{(k)}) \parallel_c \toap_c(t^{(\ell)})$.
  The $m$-eralized $c$-compatibility relation~$\parallel_c$ exists and is symmetric.  Moreover, \[r^{(k)} \parallel'_c t^{(\ell)} \Leftrightarrow \toap_c(r^{(k)}) \parallel_c \toap_c(t^{(\ell)}).\]
\end{theorem}

\begin{proof}
  Applying $\phi^{-1}_c$ and using~\Cref{lem:shift_eq_taum}, we prove that $\parallel'_c$ satisfies the defining properties \Cref{def:parallelc} of $\parallel_c$ with $\taum$ replaced by $\Shift_s$.
  That is, for~$s$ final in~$c$
  \begin{enumerate}[(i)]
\item $[{}^{sc}s]^{(m)} \parallel'_c t^{(k)} \Leftrightarrow \phi_c(t^{(k)}) \in \Phi_{\langle s \rangle}$, and
\item $\quad t^{(k)} \parallel'_c r^{(\ell)} \Leftrightarrow \Shift_s(t^{(k)}) \parallel_{\coxrnf} \Shift_s(r^{(\ell)})$,
\end{enumerate}
with $t^{(k)},r^{(\ell)} \in \invs_\refl(\cwm{c})$.  When $k=m$, the first property holds because the colored reflections of color $m$ form a facet of $\Assocm(W,c)$ and the roots of color $m$ in $\Phipm$ are simple.  For $k<m$ we compute
  \begin{align*}
    [{}^{sc}s]^{(m)} \parallel'_c t^{(k)}
      &\XLeftrightarrow{30pt}{\eqref{eq:ass_shift}} \Shift_s^{-1}\big([{}^{sc}s]^{(m)}\big) \parallel'_{\coxrnf} \Shift_s^{-1}\big(t^{(k)}\big) \\
      &\XLeftrightarrow{30pt}{\eqref{eq:shifts}} s^{(0)} \parallel'_{\coxrnf} \Shift_s^{-1}\big(t^{(k)}\big) \\
      &\XLeftrightarrow{30pt}{\ref{prop:assoc_cambrian_recurrence}} \Shift_s\Shift_s^{-1}\big(t^{(k)}\big) = t^{(k)} \in \NablaAssocm(W_{\langle s \rangle},\coxsnf) \\
      &\XLeftrightarrow{30pt}{\ref{cor:flag}} t^{(k)} \in \invs_\refl(\cwm{\coxsf}) \\
      &\XLeftrightarrow{30pt}{\eqref{rem:almost_pos_are_these}} t \in W_{\langle s \rangle} \Leftrightarrow \phi_c(t^{(k)})=\beta_t^{(k)} \in \Phipm,
  \end{align*}
  implying the first defining property for~$\parallel'_c$.
  The second defining property directly follows from~\eqref{eq:ass_shift} using~\eqref{rem:almost_pos_are_these}.

  As shown in \Cref{prop:unique_compatible}, these properties uniquely determine the relation.   It follows from \Cref{lem:shift_eq_taum} that
  \[
    t^{(k)} \parallel'_c r^{(\ell)} \Leftrightarrow \toap_c(t^{(k)}) \parallel_c \toap_c(r^{(\ell)}).\qedhere
  \]
\end{proof}

\christianside{If we had initial, the arguments here would be even more technical. But we are finally done with it.}
\begin{corollary+}
  The $m$-eralized $c$-cluster complex $\Assocm(W,c)$ is the simplicial complex given by all collections of pairwise~$\parallel_c$ compatible $m$-colored almost positive roots.
\end{corollary+}

\subsection{Recovering generalized cluster complexes}
\label{sec:fomin_reading}

In~\cite{FR2005}, S.~Fomin and N.~Reading gave an $m$-eralization of \Cref{def:reading_compatibility} for bipartite Coxeter elements.
Let $\sref = \sref_L \sqcup \sref_R$ be the bipartition of the simple reflections from \Cref{sec:coxeter_elements}.
The \defn{Fomin-Reading map}\footnote{We have chosen to place the ``negative'' simple roots in the separate color~$m$, while S.~Fomin and N.~Reading place the negative simple roots along with the full copy of positive roots in color~$0$.  We have modified their definitions accordingly.} $\FRm : \mPhipm\bij \mPhipm$ is defined by
\[
  \FRm(\beta^{(k)}) =
  \begin{cases}
    \beta^{(k+1)}                & \text{if } 0 \leq k < m-1 \\
    \FR(\beta)^{(0)}  & \text{if } k = m-1 \\ 
	\FR(-\beta)^{(0)}  & \text{if } k = m 
  \end{cases}\ ,
\]
where $\FR \eqdef \prod_{s \in \sref_L} \tau_s \prod_{s \in \sref_R} \tau_s$ and 
$\gamma^{(0)}$ is interpreted as $[-\gamma]^{(m)}$ for $\gamma \in -\Delta$.
%$\FR(\beta)^{(0)} \eqdef -\FR(\beta)^{(m)}$ if $\FR(\beta) \in -\Delta$.

\begin{definition}
  \label{def:fr_compatibility}
The \defn{$\FR$ $m$-eralized compatibility relation} is the unique relation $\parallel^{\FR}$ on $\mPhipm$ characterized by the following two properties~\cite[Theorem~3.4]{FR2005}:
\begin{enumerate}[(i)]
  \item\label{it:frinitial} for $\alpha \in \Delta$ and $\beta^{(k)} \in \mPhipm$,
    $
      \alpha^{(m)} \parallel^{\FR} \beta^{(k)} \Leftrightarrow \beta \in \Phi_{\langle \alpha \rangle},
    $
  \item\label{it:frinduction} for $\beta^{(k)}, \gamma^{(\ell)} \in \mPhipm$,
    $
      \beta^{(k)} \parallel^{\FR} \gamma^{(\ell)} \Leftrightarrow \FRm(\beta^{(k)}) \parallel^{\FR} \FRm(\gamma^{(\ell)}).
    $
\end{enumerate}
\end{definition}

We show that the compatibility relation~$\parallel_c$ recovers $\parallel^{\FR}$ for bipartite $c = c_Lc_R$.

\begin{proposition}
\label{prop:FRmap}
  The Fomin-Reading map $\FRm$ satisfies
  \[
    \big(\FRm\big)^{-1} = \psi \circ \taum[s_N] \circ \cdots \circ \taum[s_1],
  \]
  where $\cwo[c_Rc_L] = \s_1\cdots\s_N$ is the $c_Rc_L$-sorting word for~$\wo$, and where $\psi:\PhiP\rightarrow \PhiP$ is the involution on positive roots given by $\psi(\beta)\eqdef -\wo(\beta)$ and $\psi(\beta^{(k)})\eqdef \psi(\beta)^{(k)}$.
\end{proposition}
\begin{proof}
  It is a straightforward check that this composition satisfies the two cases in the definition of the Fomin-Reading map~$\FRm$.
\end{proof}

\begin{example}
  Let $c = c_Lc_R = st = (12)(23) \in \WA[3]$ and let $m = 2$.
  In this case, $\cwo[c_Rc_L] = \t\s\t$, so we consider the identity
  \[
    \big({\FR}^{(2)}\big)^{-1} = \psi \circ \tau^{(2)}_{t} \circ \tau^{(2)}_s \circ \tau^{(2)}_{t}.
  \]
 We compute the images of this composition on $\mPhipm[2]$:
  \begin{center}
    \begin{tikzpicture}
      \node (A) at (0, 0) {$\alpha^{(0)}\ \gamma^{(0)}\ \beta^{(0)}\hspace*{10pt}
      \alpha^{(1)}\ \gamma^{(1)}\ \beta^{(1)}\hspace*{10pt}
      \alpha^{(2)}\ \beta^{(2)}\ $};

      \node (B) at (0,-1) {$\gamma^{(0)}\ \alpha^{(0)}\ \beta^{(2)}\hspace*{10pt}
      \gamma^{(1)}\ \alpha^{(1)}\ \beta^{(0)}\hspace*{10pt}
      \alpha^{(2)}\ \beta^{(1)}\ $};

      \node (C) at (0,-2) {$\beta^{(0)}\ \alpha^{(2)}\ \beta^{(2)}\hspace*{10pt}
      \beta^{(1)}\ \alpha^{(0)}\ \gamma^{(0)}\hspace*{10pt}
      \alpha^{(1)}\ \gamma^{(1)}\ $};

      \node (D) at (0,-3) {$\beta^{(2)}\ \alpha^{(2)}\ \beta^{(1)}\hspace*{10pt}
      \beta^{(0)}\ \gamma^{(0)}\ \alpha^{(0)}\hspace*{10pt}
      \gamma^{(1)}\ \alpha^{(1)}\ $};

      \node (E) at (0,-4) {$\alpha^{(2)}\ \beta^{(2)}\ \alpha^{(1)}\hspace*{10pt}
      \alpha^{(0)}\ \gamma^{(0)}\ \beta^{(0)}\hspace*{10pt}
      \gamma^{(1)}\ \beta^{(1)}\ $};

      \node (F) at (0,-5) {$\alpha^{(0)}\ \gamma^{(0)}\ \beta^{(0)}\hspace*{10pt}
      \alpha^{(1)}\ \gamma^{(1)}\ \beta^{(1)}\hspace*{10pt}
      \alpha^{(2)}\ \beta^{(2)}\ $};

      \draw[|->, looseness=1] (A.south west) to[out=245,in=115] (B.north west);
      \node at (-3.7,-0.5) {$\tau^{(2)}_t$};
      \draw[|->, looseness=1] (B.south west) to[out=245,in=115] (C.north west);
      \node at (-3.7,-1.5) {$\tau^{(2)}_s$};
      \draw[|->, looseness=1] (C.south west) to[out=245,in=115] (D.north west);
      \node at (-3.7,-2.5) {$\tau^{(2)}_t$};
      \draw[|->, looseness=1] (D.south west) to[out=245,in=115] (E.north west);
      \node at (-3.5,-3.5) {$\psi$};
      \draw[|->, looseness=1] (E.south west) to[out=245,in=115] (F.north west);
      \node at (-3.9,-4.5) {${\FR}^{(2)}$};
    \end{tikzpicture}.
  \end{center}
\end{example}

\begin{corollary}
\label{cor:m-c-assoc-m-assoc}
  For $c = c_Lc_R$ a bipartite Coxeter element, $\beta^{(i)} \parallel_c \gamma^{(j)}$ if and only if $\beta^{(i)} \parallel^{\FR} \gamma^{(j)}$.
\end{corollary}
\begin{proof}
  The defining condition \Cref{def:fr_compatibility}\eqref{it:frinitial} for $\parallel^{\FR}$ holds for $\parallel_c$.
  We must show that $\parallel_c$ also satisfies the defining condition \Cref{def:fr_compatibility}\eqref{it:frinduction}.
  By~\cite[Lemma~6.2]{FR2005}, the $m$-eralized compatibility relation is invariant under $\psi$:
  \[
    \beta^{(k)} \parallel^{\FR} \gamma^{(\ell)} \Leftrightarrow \psi(\beta^{(k)}) \parallel^{\FR} \psi(\gamma^{(\ell)}).
  \]%  [-\wo(\beta)]^{(k)} \parallel^{\FR} [-\wo(\gamma)]^{(\ell)}
  \Cref{prop:FRmap} implies that \Cref{def:fr_compatibility}\eqref{it:frinduction} can be expressed as
  \begin{equation*}
    \beta^{(k)} \parallel^{\FR} \gamma^{(\ell)} \Leftrightarrow \big(\taum[s_N] \circ \cdots \circ \taum[s_1]\big)^{-1}(\beta^{(k)}) \parallel^{\FR} \big(\taum[s_N] \circ \cdots \circ \taum[s_1]\big)^{-1}(\gamma^{(\ell)}). \label{eq:FRequation}
  \end{equation*}
  By applying \Cref{def:parallelc}\eqref{it:mcompatibility2}~$N$ times, we conclude for $c=c_Lc_R$ that the same recurrence relation holds for $\parallel_c$.  It follows that
  \[
    \beta^{(k)} \parallel^{\FR} \gamma^{(\ell)} \Leftrightarrow \beta^{(k)} \parallel_c \gamma^{(\ell)}.\qedhere
  \]
%  for $c = c_Lc_R$ 
\end{proof}

% \begin{remark}
%   We give a more conceptual explanation for the relation given in \Cref{prop:FRmap} in the case $m=1$.
%   The situation for general~$m$ is similar.
%   The operation~$\tau_s$ corresponds to the shift~\eqref{eq:shifts} using \Cref{lem:shift_eq_taum}.
%   In this subword complex, the inverse rotation of the initial~$n$ letters induced an automorphism on $\Assoc(W,c)$,
%   \[
%     \beta^{(k)} \parallel_c \gamma^{(\ell)} \Leftrightarrow \big(\taum[s_1] \circ \cdots \circ \taum[s_n]\big)^{-1}(\beta^{(k)}) \parallel_c \big(\taum[s_1] \circ \cdots \circ \taum[s_n]\big)^{-1}(\gamma^{(\ell)}).
%   \]
%   as does the inverse rotation of the final~$N$ letters in~\eqref{eq:FRequation}.
%   Now, these two automorphisms are clearly obtained from each other by twisting by the action of~$\wo$.
%   The equality of the two operations in \Cref{prop:FRmap} should be thought of as an $m$-eralization of this observation.
%   The automorphism on $\Assocm(W)$ induced by $\FRm$ is, by \Cref{prop:FRmap}, given by $\taum[s_1] \circ \cdots \circ \taum[s_N]$, after twisting by the action of~$\wo$.
% \end{remark}

%%%%%%%%%%%%%%%%%%%%%%%%%%%%%%%%%%%%%%%%%%%%%%%%%%%%%%%%%%%%%%%%%%%%%%%%%%%%%%%%%%%%%
\section{Topological properties}
\label{sec:clustertopology}

In this section, we apply the general results from \Cref{sec:topologysubword} to $m$-eralized cluster complexes.
In~\cite{AT2008}, C.~Athanasiadis and E.~Tzanaki showed that the generalized cluster complex $\Assocm(W,c_Lc_R)$ of Fomin-Reading is vertex-decomposable.
For $m=1$, \cite{CLS2011} gives an elegant proof of vertex-decomposability by realizing the cluster complex as a subword complex and invoking~\cite[Theorem 2.5]{KM2004}, which states that \emph{all} subword complexes are vertex-decomposable.
Based on \Cref{thm:vertex-decomposability}, we also obtain the vertex-decomposability for $m$-eralized $c$-cluster complexes.
\christianside{Aren't we lucky to have Coxeter initial complexes?}
\nathanside{Thank you, dual braid monoid!}

\begin{theorem}
\label{cor:m-c-shelling-order}
 $\Assocm(W,c)$ is vertex-decomposable, hence shellable.
 The lexicographic order of the facets (as sorted lists of positions) of~$\Assocm(W,c)$ is a shelling order.
\end{theorem}

\begin{proof}
  This is a direct consequence of \Cref{thm:vertex-decomposability} given that $\cwm{c}$ is initial in $\c^\infty$.
\end{proof}

As a consequence we also obtain the following uniform theorem first proven by S.~Fomin and N.~Reading in~\cite[Proposition~11.1]{FR2005}.

\begin{theorem}
\label{thm:assoc_homotopy}
  The $m$-eralized $c$-cluster complex $\Assocm(W,c)$ has the homotopy type of a wedge of $\Catm[m-1](W)$ spheres of dimension~$n{-}1$.
\end{theorem}

\begin{proof}
  This follows from \Cref{cor:wedgeofspheres} by observing that the subcomplex of $\Assocm(W,c)$ not containing positions corresponding to a letter in the initial copy of $\cwm{c}$ is equal to the cluster complex $\Assocm[m-1](W,\psi(c))$.
  By \Cref{prop:mclusterisomorphism} and \Cref{cor:m-c-assoc-m-assoc}, $\Assocm[m-1](W,\psi(c))$ is isomorphic to $\Assocm[m-1](W,c)$.
  Since the latter is known to be counted by $\Catm[m-1](W)$, we conclude the theorem.
\end{proof}

\begin{corollary}
\label{cor:m-c-h-vector}
  The $h$-polynomial $h_0 + h_1q + \cdots + h_dq^d$ of the $m$-eralized $c$-cluster complex $\Assocm(W,c)$ is given by the upper and lower covers generating function of $\Cambassocm(W,c)$,
  \begin{align*}
    h_i & = \bigset{ F \in \Cambassocm(W,c) }{ F \text{ has exactly } i \text{ upper covers } }.
  \end{align*}
\end{corollary}
The coefficients in this corollary are called \defn{$W$-Fu\ss-Narayana numbers.}

\begin{proof}
  Since the lexicographic order is a linear extension of $\Cambassocm(W,c)$, the statement follows from \Cref{cor:m-c-shelling-order} together with the fact that the reverse of a shelling order is again a shelling order.
\end{proof}

%%%%%%%%%%%%%%%%%%%%%%%%%%%%%%%%%%%%%%%%%%%%%%%%%%%%%%%%%%%%%%%%%%%%%%%%%%%%%%%%%%%%%
\section{Clusters and noncrossing partitions}
\label{sec:m-assoc-ncp}
%%%%%%%%%%%%%%%%%%%%%%%%%%%%%%%%%%%%%%%%%%%%%%%%%%%%%%%%%%%%%%%%%%%%%%%%%%%%%%%%%%%%%

The connection between $\DeltaAssocm(W,c)$ and $\DeltaNCm(W,c)$ is based on the $m$-eralization of the $m=1$ results in~\cite[Proposition~6.20]{PS20112}.  This recovers, in crystallographic types, the representation-theoretic bijections given by A.~Buan, I.~Reiten, and H.~Thomas~\cite{BRT2011,buan2009m}.% in~\cite{buan2009m}, rephrasing the results in the language of colored factorizations.

\medskip

The root configurations of the facets of $\DeltaAssocm(W,\c)$ satisfy an \defn{$m$-eralized $c$-Cambrian recurrence}.

\begin{proposition}
\label{prop:rootconfiguration}
  Let~$s$ be initial in~$c$ and let~$I$ be a facet of $\DeltaAssocm(W,c)$.
  Then the root configuration satisfies
  \[
    \Roots{I} =
    \begin{cases}
      \big\{\alpha_s^{(0)}\big\} \cup \Roots{I_{\langle s \rangle}} &\text{if } s^{(0)} \in \Roots{I} \\
      s\big(\Roots{\Shift_s(I)}\big) &\text{otherwise }
    \end{cases}\ ,
  \]
  with $I_{\langle s \rangle} \in \DeltaAssocm(\Wres,\coxsn)$ and $\Shift_s(I)\in \Assocm(W,\coxrn)$ as defined in \Cref{prop:assoc_cambrian_recurrence}.
\end{proposition}

\begin{proof}
  Both cases follow from \Cref{prop:assoc_cambrian_recurrence} and~\eqref{eq:root_configuration}.
\end{proof}

%\nathan{I'm a little uneasy about moving between the colored factorization and position specification of a facet, so do make sure that this seems clear}

As discussed, we use the term \defn{natural} to mean that a bijection respects the Cambrian recurrence.

\begin{theorem}
\label{thm:rootconf_skiptset}
  The root configuration induces a natural bijection
  \begin{align*}
    \DeltaAssocm(W,\c) & \cambbij\DeltaNCm(W,c) \\
    I & \longmapsto \t_1^{(i_1)} \cdots \t_n^{(i_n)},
  \end{align*}
for $\Roots{I} = \big\{ \beta_{1}^{(i_1)}, \ldots, \beta_{n}^{(i_n)}\big\}$ and $t_k = s_{\beta_k}$ for $1 \leq k \leq n$.
\end{theorem}\nathanside{And\dots scene.}

\begin{proof}
  Immediate from the Cambrian recurrence for noncrossing partitions in \Cref{prop:nc_cambrian_recurrence} and the recurrence for root configurations in \Cref{prop:rootconfiguration}.
\end{proof}

\begin{theorem}
\label{thm:cambnccambassoc}
  The natural bijection in \Cref{thm:rootconf_skiptset} induces a poset isomorphism $\Cambassocm(W,c) \cong \Cambncm(W,c)$.%$\DeltaAssocm(W,c)\bij\DeltaNCm(W,c)$
\end{theorem}

\begin{proof}
By \Cref{lem:rootupdate} and \Cref{eq:increasingflip}, flips in $\Assocm(W,c)$ are sent to flips in $\DeltaNCm(W,c)$ under the bijection of \Cref{prop:rootconfiguration}.
\end{proof}

\begin{corollary}
\label{cor:m-c-h-vector2}
  The $h$-polynomial of $\Assocm(W,c)$ is given by $\sum q^{\lengthR(w_m)}$, where the sum ranges over all $(w_1,\ldots,w_m) \in \NCm(W,c)$.
\end{corollary}

\begin{proof}
  This follows from \Cref{cor:m-c-h-vector}, from \Cref{thm:cambnccambassoc}, and from \Cref{prop:cambnc}\eqref{eq:cambncitem3}.
\end{proof}

%%%%%%%%%%%%%%%%%%%%%%%%%%%%%%%%%%%%%%%%%%%%%%%%%%%%%%%%%%%%%%%%%%%%%%%%%%%%%%%%%%%%%
\chapter{Sortable elements}
\label{sec:sortable_elements}
%%%%%%%%%%%%%%%%%%%%%%%%%%%%%%%%%%%%%%%%%%%%%%%%%%%%%%%%%%%%%%%%%%%%%%%%%%%%%%%%%%%%%

In this chapter, we study $m$-eralized sortable elements as certain elements in the $m$-eralized weak order.
After giving the general definition (\Cref{sec:m_sort}), we review the known theory of sortable elements in Coxeter groups (\Cref{sortable_elements_classical}).
Detailed background can be found in~\cite{Rea2006,Rea2007,Rea20072}.
We discuss the Cambrian recurrence on $m$-eralized sortable elements (\Cref{sec:m-sort-cambrian-recurrence}).
We describe sortable elements in terms of their Garside factors (\Cref{sec:factorsort}), define the Cambrian poset on sortable elements, and show that this poset is a lattice (\Cref{sec:sort_cambrian_lattices}).  We construct natural bijections between $m$-eralized sortable elements, $m$-eralized noncrossing partitions, and $m$-eralized clusters (\Cref{sec:m-sort-ncp}).
We connect $m$-eralized sortable elements with chains in the shard intersection order (\Cref{sec:m-sort-shard}).   We conclude with a comparison of our $m$-eralized Cambrian lattices of type $A_n$ and $m$-Tamari lattices (\Cref{sec:tamari}).\nathanside{This section might get a bit technical, so please buckle up.}

%%%%%%%%%%%%%%%%%%%%%%%%%%%%%%%%%%%%%%%%%%%%%%%%%%%%%%%%%%%%%%%%%%%%%%%%%%%%%%%%%%%%%
\section{\mhead-eralized sortable elements}
\label{sec:m_sort}
%%%%%%%%%%%%%%%%%%%%%%%%%%%%%%%%%%%%%%%%%%%%%%%%%%%%%%%%%%%%%%%%%%%%%%%%%%%%%%%%%%%%%

N.~Reading introduced $c$-sortable elements in~\cite{Rea2006} as a subset of elements of the Coxeter group $W$.  His notion extends verbatim to elements of the positive Artin monoid $\Artinmon$.

Let~$c$ be a Coxeter element with reduced $\sref$-word $\c$.  For $w$ an element of $W$ or $\Artinmon$, recall from~\Cref{def:c-sorting} that the $\c$-sorting word is the lexicographically first subword of $\c^\infty$ that is a reduced $\sref$-word for~$w$.

\begin{definition}
\label{def:msort_sortable}
  An element $w$ of $W$ or $\Artinmon$ is \defn{$c$-sortable} if its $\c$-sorting word $\sw{w}{\c}=\restri{\c}{I_1} \restri{\c}{I_2} \cdots \restri{\c}{I_k}$ satisfies $I_1 \supseteq I_2 \supseteq \cdots \supseteq I_k$.
\end{definition}
\christianside{This is the definition you probably know, but now in the positive Artin monoid.}

By~\Cref{lem:wosortingword}, $\cwo$ is initial in $\c^\infty$.
Hence, $\wo$ is $c$-sortable.
Since $\bwom[2] = \bc^h \in \Artinmon$, we conclude that that $\bwom$ is $c$-sortable for any positive integer $m$ and any Coxeter element $c$.

\begin{example}
\label{ex:msortA3}
  We illustrate \Cref{def:msort_sortable} with the following non-example.
  In $\Weakm[2](\WA[4])$ with $\c=\s_1\s_2\s_3$, the element $w=s_1s_2s_3s_1s_2 \cdot s_3s_2s_1$ (where the dot denotes the separation of Garside factors), has $\c$-sorting word
  \[
    \sw{w}{c} = \left( \begin{array}{ccc|ccc|ccc|ccc}  1 & 2 & 3  &1 & 2 & 3  & 1 & 2 & 3 &1 & 2 & 3 \\  s_1& s_2 & s_3 &s_1& s_2 & s_3 & - & s_2 & - & s_1 & - & -  \end{array} \right).
  \]
  It is not $c$-sortable since $s_1$ occurs in the fourth but not the third copy of~$\c$.
  Note that the vertical bars here serve only to distinguish the copies of~$\c$.
\end{example}

Although the word $\sw{w}{c}$ depends on a particular choice of a reduced word~$\c$ for the Coxeter element~$c$, the property of being $c$-sortable does not by \Cref{lem:reflection_order}\eqref{it:reflection_order1}.

We write $\Sort(W,c)$ for the set of $c$-sortable elements in the Coxeter group, $\Sort(\Artinmon,c)$ for those in the positive Artin monoid, and $\Sortm(W,c)$ for the restriction  of $\Sort(\Artinmon,c)$ to the interval $\Wm=[\bone,\bwom]$.
We characterize $m$-eralized sortable elements using Garside factorizations in \Cref{def:factorsort}.  Since $\Sortm(W,c) \subseteq \Artinmon$, sortable elements inherit the notion of \defn{support} from \Cref{sec:cox_length_support}.%\nathan{last sentence needed?}

\medskip
\Cref{fig:weakA3} illustrates the weak order of $\WA[4]$, with the $s_1s_2s_3$-sortable elements shaded.   $\Sortm[2](\WA[3])$ is illustrated in \Cref{fig:weakA22} on page~\pageref{fig:weakA22}, and the first, second, and fifth columns of \Cref{fig:sortA22} list the Garside factorizations, sorting words, and supports of the $12$ elements in $\Sortm[2](\WA[2],st)$. % $\Sortm[2](\WA[3])$ can be found in
% An example for $\BA[3]$ is given in \Cref{ex:msortA3}.
%We postpone an individual detailed example to \Cref{ex:msortA3} after we have given the $m$-eralized definition.

%We denote the restriction of $\Sort(\Artinmon,c)$ to the interval $\Wm=[e,\wo^m]$ by $\Sortm(W,c)$.  

\begin{figure}[t]
  \begin{center}
    \resizebox{\textwidth}{!}{
    \begin{tikzpicture}[yscale=1.2]
      \tikzstyle{rect}=[rectangle,draw,opacity=.5,fill opacity=1]
      \tikzstyle{sort}=[fill=black!20]
      \node[rect,sort] (e)   at ( 0,0) {$\one$};

      \node[rect,sort] (s)   at (-3,1) {$s_1$};
      \node[rect,sort] (t)   at ( 0,1) {$s_2$};
      \node[rect,sort] (u)   at ( 3,1) {$s_3$};

      \node[rect,sort] (st)  at (-5,2) {$s_1s_2$};
      \node[rect] (ts)  at (-2,2) {$s_2|s_1$};
      \node[rect,sort] (su)  at ( 0,2) {$s_1s_3$};
      \node[rect,sort] (tu)  at ( 2,2) {$s_2s_3$};
      \node[rect] (ut)  at ( 5,2) {$s_3|s_2$};

      \node[rect,sort] (stu)  at (-6,3) {$s_1s_2s_3$};
      \node[rect,sort] (sts)  at (-4,3) {$s_1s_2|s_1$};
      \node[rect] (sut)  at (-1,3) {$s_1s_3|s_2$};
      \node[rect] (tus)  at ( 1,3) {$s_2s_3|s_1$};
      \node[rect] (uts)  at ( 4,3) {$s_3|s_2|s_1$};
      \node[rect,sort] (tut)  at ( 6,3) {$s_2s_3|s_2$};

      \node[rect,sort] (stus)  at (-5,4) {$s_1s_2s_3|s_1$};
      \node[rect,sort] (stut)  at (-2,4) {$s_1s_2s_3|s_2$};
      \node[rect] (tust)  at ( 0,4) {$s_2s_3|s_1s_2$};
      \node[rect] (suts)  at ( 2,4) {$s_1s_3|s_2|s_1$};
      \node[rect] (tuts)  at ( 5,4) {$s_2s_3|s_2|s_1$};

      \node[rect,sort] (stust)   at (-3,5) {$s_1s_2s_3|s_1s_2$};
      \node[rect] (stuts)   at ( 0,5) {$s_1s_2s_3|s_2|s_1$};
      \node[rect] (tusts)   at ( 3,5) {$s_2s_3|s_1s_2|s_1$};

      \node[rect,sort] (stusts)   at ( 0,6) {$s_1s_2s_3|s_1s_2|s_1$};

      \draw (e) to (s);
      \draw (e) to (t);
      \draw (e) to (u);

      \draw (s) to (st);
      \draw (s) to (su);
      \draw (t) to (ts);
      \draw (t) to (tu);
      \draw (u) to (su);
      \draw (u) to (ut);

      \draw (st) to (stu);
      \draw (st) to (sts);
      \draw (ts) to (sts);
      \draw (ts) to (tus);
      \draw (su) to (sut);
      \draw (tu) to (tus);
      \draw (tu) to (tut);
      \draw (ut) to (uts);
      \draw (ut) to (tut);

      \draw (stu) to (stus);
      \draw (sts) to (stus);
      \draw (stu) to (stut);
      \draw (sut) to (stut);
      \draw (sut) to (suts);
      \draw (tus) to (tust);
      \draw (uts) to (suts);
      \draw (uts) to (tuts);
      \draw (tut) to (tuts);

      \draw (stus) to (stust);
      \draw (tust) to (stust);
      \draw (stut) to (stuts);
      \draw (suts) to (stuts);
      \draw (tust) to (tusts);
      \draw (tuts) to (tusts);

      \draw (stust) to (stusts);
      \draw (stuts) to (stusts);
      \draw (tusts) to (stusts);
    \end{tikzpicture}
    }
  \end{center}
  \caption{The weak order $\Weak(\WA[4])$.  The $s_1s_2s_3$-sortable elements are shaded.}
  \label{fig:weakA3}
\end{figure}
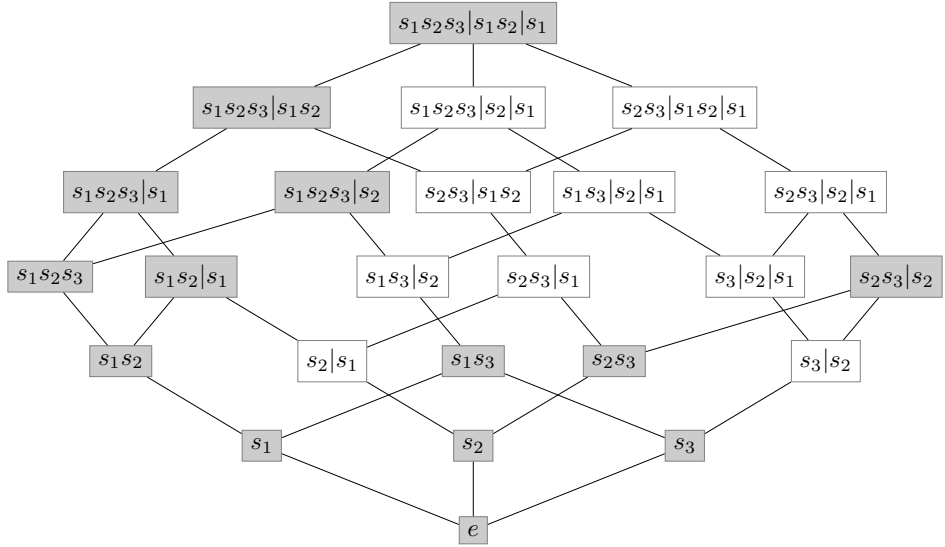
%We $m$-eralize the definiton of $c$-sortable elements certain elements in $\Wm$ in \Cref{def:msort_sortable}.

%%%%%%%%%%%%%%%%%%%%%%%%%%%%%%%%%%%%%%%%%%%%%%%%%%%%%%%%%%%%%%%%%
\section{Classical sortable elements}
\label{sortable_elements_classical}
%%%%%%%%%%%%%%%%%%%%%%%%%%%%%%%%%%%%%%%%%%%%%%%%%%%%%%%%%%%%%%%%%

In this section, we review properties of sortable elements of finite Coxeter groups.  

\subsection{Cambrian recurrence}
There is an immediate defining recurrence for $\Sort(W,c)$ from \Cref{def:msort_sortable}, called the \defn{$c$-Cambrian recurrence}.
\begin{proposition+}[{N.~Reading~\cite[Lemma~2.1 and~2.2]{Rea2006}}]
\label{prop:sort_cambrian_recurrence}
  Let~$s$ be initial in~$c$.
  Then
  \[
    w \in \Sort(W,c) \Leftrightarrow
    \begin{cases}
      \phantom{\sninv}w  \in \Sort(W_{\langle s \rangle}, \coxsn)  & \text{if } s \in \AscSet_L(w) \\
      \sninv w \in \Sort(W, \coxrn)  & \text{if } s \in \DesSet_L(w)
    \end{cases}\ .
  \]
  \qedherecases
\end{proposition+}

\subsection{Lattice properties of sortable elements}

The following properties of sortable elements are most easily explained using the following characterization of the sortable property in the language of biclosed sets from~\eqref{eq:biclosed}.

\begin{theorem+}[{\cite[Theorem~4.3]{RS2011}}]
\label{thm:caligned}
  An element $w \in W$ is $c$-sortable for a Coxeter element~$c \in W$ if and only if, for $\alpha,\beta,\gamma \in \PhiP$ such that $\gamma = a\alpha+b\beta$ with $a,b\in\R_+$, we have
  \begin{equation*}
    \gamma \in \inv(w) \Rightarrow {\min}_{\leq_c}\{\alpha,\beta\} \in \inv(w),
  \end{equation*}
  where $\leq_c$ is the root order induced by the $c$-sorting word~$\cwo$.
  In particular, $\wo \in \Sort(w,c)$.
\end{theorem+}
The following lemma is an immediate consequence of \Cref{thm:caligned}.

\begin{lemma+}
  \label{eq:scfa_parabolic}
  For $w \in \Sort(W,c)$ and $J \subseteq \sref$,
  \[
    w_{J} \in \Sort\left(W_{J},\restri{c}{J}\right). \qedhere
  \]
\end{lemma+}

Under the inclusion $W_{\langle s \rangle} \subset W$, a sortable element in $W_{\langle s \rangle}$ is again sortable in $W$.  The following lift from $W_{\langle s \rangle}$ to $W$ is more subtle.

\begin{lemma+}[{N.~Reading~\cite[Lemma~2.8 and~2.9]{Rea2007}}]
\label{lem:readingjoininparabolic}
  For~$s$ initial in~$c$, if $w \in \Sort(W_{\langle s \rangle},\coxsn)$, then $w \vee s$ is both $c$-sortable and $\coxrn$-sortable with
  \[
    \coveredref(w \vee s) = \coveredref(w) \cup \{s\}. \qedhere
  \]
\end{lemma+}\nathanside{The proofs of analogous statements to~\Cref{eq:scfa_parabolic,lem:readingjoininparabolic} in the Fu\ss case will occupy many pages.}

\subsection{The Cambrian lattice}
The \defn{$c$-Cambrian lattice} $\Cambsort(W,c)$ is the restriction of the weak order $\Weak(W)$ to $\Sort(W,c)$.
The inversion set of the meet of two $c$-sortable elements in the weak order on~$W$ is simply the intersection of their inversion sets, which again directly follows from \Cref{thm:caligned}.

\begin{proposition+}
\label{prop:scfa_inv}
 Let $w,u\in \Sort(W,c)$.  Then $u \wedge v \in \Sort(W,c)$ and
  \[
    \inv(u \wedge v) = \inv(u) \cap \inv(v). \qedhere
  \]
\end{proposition+}

Although the join does not enjoy such a simple description in terms of inversion sets, if $w,u\in \Sort(W,c)$, then also $u \vee v \in \Sort(W,c)$.  We conclude the following theorem~\cite[Theorem~7.3]{RS2011}.  %Indeed, one central property of $c$-sortable elements is the following theorem which can be found for example in~\cite[Theorem~7.3]{RS2011}.

\begin{theorem+}
\label{eq:sublattice}
  $\Cambsort(W,c)$ is a sublattice of $\Weak(W)$.
\end{theorem+}

We $m$-eralize~\Cref{prop:scfa_inv} in \Cref{thm:sort_intersection} and $m$-eralize \Cref{eq:sublattice} in \Cref{thm:sort_is_lattice}.

\subsection{Cambrian rotation}
The following \defn{$c$-Cambrian rotation} is well-defined by~\Cref{prop:sort_cambrian_recurrence} and~\Cref{lem:readingjoininparabolic}:
\begin{align*}
  \Shift_s: \Sort(W,c)  &\longrightarrow \Sort(W,\coxrn) \\
                      w &\longmapsto
                        \begin{cases}
                          w \vee s  & \text{if } s \in \AscSet_L(w) \\
                          \sninv w  & \text{if } s \in \DesSet_L(w)
                        \end{cases}\ .
\end{align*}

\subsection{Sortable elements, noncrossing partitions, and clusters}

By associating a sortable element to its cover reflections, N.~Reading gave a natural bijection between $c$-sortable elements and $c$-noncrossing partitions~\cite[Theorem 6.1]{Rea20072}.  This bijection is given by
\begin{align}
  \Sort(W,c) &\cambbij \NC(W,c) \nonumber \\
  w &\longmapsto r_1 \cdots r_k,\label{eq:sort_to_nc_by_covers}
\end{align}
where $\coveredref(w) = \{r_1,\ldots,r_k\}$ with $r_1 <_{\c} \cdots <_{\c} r_k$.
N.~Reading and D.~Speyer gave an alternative description of this bijection using the notion of a \defn{skip set}~\cite[Section 5]{RS2011}, which we $m$-eralize in \Cref{thm:bij_sort_to_nc}.

N.~Reading proved the existence of the compatibility~$\parallel_c$ by showing that it arises naturally from his theory of sortable elements~\cite{Rea20072}, providing a bijection between sortable elements and clusters.

%%%%%%%%%%%%%%%%%%%%%%%%%%%%%%%%%%%%%%%%%%%%%%%%%%%%%%%%%%%%%%%%%%%%%%%%%%%%%%%%%%%%%
\section{The Cambrian recurrence}
\label{sec:m-sort-cambrian-recurrence}
%%%%%%%%%%%%%%%%%%%%%%%%%%%%%%%%%%%%%%%%%%%%%%%%%%%%%%%%%%%%%%%%%%%%%%%%%%%%%%%%%%%%%

%\medskip
There is a simple inductive characterization of $\Sortm(W,c)$ called the \defn{$m$-eralized $c$-Cambrian recurrence}, which follows directly from \Cref{lem:reflection_order}\eqref{it:reflection_order7} and \Cref{lem:reflection_order4}.
% Rather than take the join with $s^m$ in the first case of the definition of the shift operator, we could have sent~$w$ to itself, viewed as an element of a parabolic subgroup.

\begin{proposition+}
\label{prop:msort_cambrian_recurrence}
  Let~$s$ be initial in~$c$.
  Then
  \[
    w \in \Sortm(W,c) \Leftrightarrow
    \begin{cases}
      \phantom{\sninv}w  \in \Sortm(W_{\langle s \rangle}, \coxsn)  & \text{if } s \in \AscSet_L(w) \\
      \sninv w \in \Sortm(W, \coxrn)  & \text{if } s \in \DesSet_L(w)
    \end{cases}\ .
  \]
  \qedherecases
\end{proposition+}

\begin{example}
\label{ex:sortmcambrian}
  Parallel to \Cref{ex:ncdeltacambrian,ex:assoccambrian}, the Cambrian recurrence for $sts \cdot s \in \Sortm[2](\WA[3],st)$ is
  \[
      \underbrace{sts \cdot s}_{st}
      \mapsto
      \underbrace{ts \cdot s}_{ts}
      \mapsto
      \underbrace{s \cdot s}_{st}
      \mapsto
      \underbrace{s}_{ts}
      \mapsto
      \underbrace{s}_{s}
      \mapsto
      \underbrace{\one}_{s}
      \mapsto
      \underbrace{\one}_{\one},
  \]
where the subscript identifies the (parabolic) Coxeter element.
\end{example}

%\begin{example}
%  Parallel to \Cref{fig:ncA22}, \Cref{fig:sortA22} shows all $12$ elements of
%  \[
%    \Sortm[2](\WA[3],st) \cong \deltaSortm[2](\WA[3],st),
%  \]
%  with their support.
%  The notion $\deltaSortm[2](\WA[3],st)$ will be defined and studied in \Cref{sec:m-sort-shard}.
  \begin{figure}[t]
    \begin{center}
      \begin{tabular}{c|c|c|c|c}
        $\Sortm[2](\WA[3],st)$ & $\Sortma[2](\WA[3],st)$ & $\deltaSortm[2](\WA[3],st)$ & $\Skipset(w)$ & $\supp(w)$ \\
        \hline
         $\S\T|\S\T|\S\T$ &
       $e$ &
        $\one \geqsh \one$ &
        $\alpha^{(0)},\beta^{(0)}$ &
        $-$
        \\
        $\s\t|\s\t|\s\t$ &
        $sts\cdot tst$ &
        $sts \geqsh sts$ &
        $\alpha^{(2)},\beta^{(2)}$ &
        $s,t$
        \\
        $\s\t|\s\t|\S\T$ &
        $sts\cdot t$&
        $sts \geqsh s$ &
        $\gamma^{(1)},\alpha^{(2)}$ &
        $s,t$
        \\
        $\s\t|\S\T|\S\T$ &
        $st$ &
        $st \geqsh e$ &
        $\beta^{(0)},\gamma^{(1)}$ &
        $s,t$
        \\
        \hline
        $\s\T|\S\T|\S\T$ &
       $s$ &
        $s \geqsh e$ &
        $\gamma^{(0)},\alpha^{(1)}$ &
       $s$
        \\
        $\S\t|\S\t|\S\T$ &
        $t\cdot t$ &
        $t \geqsh t$ &
        $\alpha^{(0)},\beta^{(2)}$ &
       $t$
        \\
        $\s\t|\s\t|\s\T$ &
        $sts\cdot ts$&
        $sts \geqsh st$ &
        $\beta^{(1)},\gamma^{(2)}$ &
        $s,t$
        \\
        $\s\t|\s\T|\S\T$ &
        $sts$ &
        $sts \geqsh e$ &
        $\alpha^{(1)},\beta^{(1)}$ &
        $s,t$
        \\
        \hline
        $\S\t|\S\T|\S\T$ &
       $t$ &
        $t \geqsh e$ &
        $\alpha^{(0)},\beta^{(1)}$ &
       $t$
        \\
        $\s\t|\S\t|\S\T$ &
        $st\cdot t$ &
        $st \geqsh st$ &
        $\beta^{(0)},\gamma^{(2)}$ &
        $s,t$
        \\
        \hline
        $\s\T|\s\T|\S\T$ &
        $s\cdot s$ &
        $s \geqsh s$ &
        $\gamma^{(0)},\alpha^{(2)}$ &
       $s$
        \\
        $\s\t|\s\T|\s\T$ &
        $sts\cdot s$ &
        $sts \geqsh t$ &
        $\alpha^{(1)},\beta^{(2)}$ &
        $s,t$
      \end{tabular}
    \end{center}
    \caption{The three variants of the $m$-eralized $st$-sortable elements for $\WA[3]$ with $m=2$, together with their skip sets and their supports.  They are arranged according to their orbits under Cambrian rotation, defined in \Cref{sec:cambrian_rotation}.}
% of $\Sortm[2](\WA[3],st)$ and $\deltaSortm[2](\WA[3],st)$
    \label{fig:sortA22}
  \end{figure}
%\end{example}

%%%%%%%%%%%%%%%%%%%%%%%%%%%%%%%%%%%%%%%%%%%%%%%%%%%%%%%%%%%%%%%%%%%%%%%%%%%%%%%%%%%%%
\section{Factorwise sortable elements}
\label{sec:factorsort}
%%%%%%%%%%%%%%%%%%%%%%%%%%%%%%%%%%%%%%%%%%%%%%%%%%%%%%%%%%%%%%%%%%%%%%%%%%%%%%%%%%%%%

We now give a description of $m$-eralized $c$-sortable elements using their Garside factorizations.
Since each Garside factor may be interpreted as an element of the group~$W$, this relates $m$-eralized $c$-sortable elements to tuples of elements of~$W$.\nathanside{Didn't see that one coming, did you.}

Let~$c$ be a Coxeter element with word $\c$, and let~$w \in W$ with $\DesSet_R(w)=\{s_{i_1},s_{i_2},\ldots,s_{i_k}\}$ ordered so that
\[
  {}^w s_{i_1} <_\c {}^w s_{i_2} <_\c \cdots <_\c {}^w s_{i_k}.
\]
Define the \defn{twisted restriction} of a Coxeter element~$c$ with respect to the element~$w$ to be the Coxeter element $\restr{c}{w}\eqdef s_{i_1} s_{i_2} \cdots s_{i_k}$ of the parabolic subgroup $W_{\DesSet_R(w)}$.

\begin{definition}
\label{def:factorsort}
  An element~$w\in \Wm$ with $\garside(w) = w^{(1)} \cdot\ \cdots\ \cdot w^{(m)}$ is \defn{factorwise $c$-sortable} if $w^{(i)}$ is $c^{(i)}$-sortable for all $1 \leq i \leq m$, where we set $\c^{(1)} = \c$ and $\c^{(i)}\eqdef\restr{\c^{(i-1)}}{w^{(i-1)}}$ for $1 < i \leq m$.  We denote the set of factorwise $c$-sortable elements by $\Sortma(W,c)$.
\end{definition}\nathanside{Let this definition gently wash over you.}

\begin{example}
\label{ex:crestricted}
  Let $\c=\s_1\s_2\s_3 \in \WA[4]$ and let~$w \in \BA[4]$ have the Garside factorization
  \[
    w^{(1)}\cdot w^{(2)} = s_1s_2s_3s_2\cdot s_3s_2=s_1s_3s_2s_3 \cdot s_3s_2.
  \]
  Since $\DesSet(w^{(1)})=\{s_3,s_2\}$ with $s_3^{w^{(1)}} = (13)<_{\c^{(1)}} (34) = s_2^{w^{(1)}}$, we have
  \[
    \c^{(2)}=\restr{\c^{(1)}}{w^{(1)}}=\s_3\s_2.
  \]
  Then~$w$ is factorwise $c$-sortable, because $w^{(1)}$ is $c^{(1)}$-sortable and $w^{(2)}$ is $c^{(2)}$-sortable.
\end{example}

\subsection{Sortable and factorwise sortable elements}
We show that the factorwise sortable elements $\Sortma(W,c)$ coincides with $\Sortm(W,c)$ by proving that factorwise $c$-sortable elements satisfy the Cambrian recurrence.

\christianside{Be prepared: there is a beautiful beach to reach, but the path goes through a deep jungle.}\nathanside{I roll my D20 to cast \textbf{teleportation}.}
\begin{proposition}
\label{prop:aligned_cambrian_recurrence}
  Let~$s$ be initial in~$c$.
  Then
  \[
    w \in \Sortma(W,c) \Leftrightarrow
    \begin{cases}
      \phantom{\sninv}w  \in \Sortma(W_{\langle s \rangle}, \coxsn)  & \text{if } s \in \AscSet_L(w) \\
      \sninv w \in \Sortma(W, \coxrn)  & \text{if } s \in \DesSet_L(w)
    \end{cases}\ .
  \]
\end{proposition}

\begin{proof}
  Let $w \in \Sortma(W,c)$ and suppose
  \[
    \garside(w) = w^{(1)}\cdot w^{(2)} \cdot\ \cdots\ \cdot w^{(m)},
  \]
  where some of these factors may be the identity.

  If $s \in \AscSet_L(w) = \AscSet_L(w^{(1)})$, then $w^{(1)}$ is $c$-sortable by assumption, and \Cref{prop:sort_cambrian_recurrence} implies that $s \notin \supp(w^{(1)})$, so that $s \notin \DesSet_R(w^{(1)})$.
  We conclude that $s \notin \supp(w)$ since all further Garside factors $w^{(2)},\ldots,w^{(m)}$ are inside $W_{\DesSet_R(w^{(1)})}$, implying that~$w$ is factorwise $\coxsn$-sortable as an element of $\Wm_{\langle s \rangle}$ by \Cref{lem:reflection_order4}.
  The converse direction also follows from \Cref{lem:reflection_order4}.

  \medskip

  Otherwise, $s \in \DesSet_L(w)$ so that $w=su$ is reduced.
  We set $\garside(u) = u^{(1)}\cdot u^{(2)}\cdot\ \cdots\ \cdot u^{(m)}$, and thus need to show that
  \[
    w \in \Sortma(W,c) \Leftrightarrow u \in \Sortma(W,\coxrn).
  \]

  If $s u^{(1)} \leq \wo$, then $\garside(w)=su^{(1)} \cdot u^{(2)} \cdot\ \cdots\ \cdot u^{(m)}$ by \Cref{thm:garsidefactorization}, since $\DesSet_R(s u^{(1)}) \supseteq \DesSet_R(u^{(1)})$.
  By \Cref{prop:msort_cambrian_recurrence}, $w^{(1)}$ is $c$-sortable if and only if $\sninv w^{(1)}=u^{(1)}$ is $\coxrn$-sortable, and we conclude this case.

  \medskip

  We consider both implications individually in the final case $s u^{(1)} \not \leq \wo$, so that $s \in \DesSet_L(u^{(1)})$.

  Suppose $u \in \Sortma(W,\coxrn)$, and let $\coxr$ be a word for $\coxrn$ ending in the letter~$s$.
  Since~$s$ is final in the reflection order associated to~$\coxr$ and $u^{(1)}$ is $\coxrn$-sortable, the simple reflection~$s$ is the final reflection in the inversion sequence $\invs(\u^{(1)}(\coxr))$, and so corresponds to the final simple reflection~$r$ in $\sw{u^{(1)}}{\coxr}$.
  More succinctly, we have the reduced expression $u^{(1)}r=su^{(1)}$, and obtain $u^{(1)}=su^{(1)}r^{-1}$.
  In this case, set $w^{(1)}  = u^{(1)} = s u^{(1)} r^{-1}$, which is $c$-sortable by \Cref{prop:msort_cambrian_recurrence}.

  We drop the first Garside factor and repeat the preceding argument on the element $r u^{(2)} \cdots u^{(m)}$, obtaining an element of $\Sortma[m-1](W,c)$ with Garside factorization $w^{(2)}\cdot\ \cdots\ \cdot w^{(m)}$.
  We now claim that $w^{(1)} \cdot w^{(2)}\cdot\ \cdots\ \cdot w^{(m)}$ is the Garside factorization of~$w$.
  We did not change the first Garside factor $w^{(1)} = u^{(1)}$ since $r \in \DesSet_R(u^{(1)})$ and $\DesSet_L(r u^{(2)}) \subseteq \{r\} \cup \DesSet_L(u^{(2)})$ (by \Cref{prop:des_contained}).
  $w^{(2)}$ thus lies in $\restr{W}{u^{(1)}}$.
  We conclude that $w^{(1)}$ was indeed the first Garside factor of~$w$.
  The result follows by induction on the number of Garside factors of~$w$.

  \medskip

  Now suppose $w \in \Sortma(W,c)$.
  Running the argument above in reverse, we obtain the candidate Garside factorization of $u$, $u^{(1)}=w^{(1)}$ and $u^{(2)}\cdot\ \cdots\ \cdot u^{(m)} = r^{-1}w^{(2)}  \cdots w^{(m)}$, where $u^{(1)}r=s u^{(1)}$.
  Since $\DesSet_R(u^{(1)})=\DesSet_R(w^{(1)})$, and since $w^{(2)}$ only uses simple reflections in $\DesSet_R(w^{(1)})$, $u^{(1)}$ is indeed the first Garside factor.
  The result again follows by induction on the number of Garside factors of~$w$.
\end{proof}

\begin{corollary}
\label{cor:factorsort}
    $\Sortm(W,c) = \Sortma(W,c).$
\end{corollary}\nathanside{Like, as sets of elements of $\Artinmon$.}

\begin{proof}
  Since both $\Sortm(W,c)$ and $\Sortma(W,c)$ satisfy the same recurrence and initial conditions, and element $w \in \Wm$ is $c$-sortable if and only if it is factorwise $c$-sortable.
\end{proof}

\begin{example}
 By~\Cref{ex:crestricted}, the element $w=w^{(1)}\cdot w^{(2)} = s_1s_2s_3s_2 \cdot s_3s_2 \in \WA[4]$ is factorwise $c$-sortable for $c = s_1s_2s_3$.
  Its $c$-sorting word is
  \[
    \sw{w}{c} = \left( \begin{array}{ccc|ccc|ccc|ccc}  1 & 2 & 3  &1 & 2 & 3  & 1 & 2 & 3 &1 & 2 & 3 \\  s_1& s_2 & s_3 & - & s_2 & s_3 & - & s_2 & - & - & - & -  \end{array} \right).
  \]
  By \Cref{def:msort_sortable}, this element is also $c$-sortable.
\end{example}

\subsection{Sorting and factorwise sorting words}
Not only do the notions of sortable and factorwise sortable agree, it turns out that the $\c$-sorting word of a $\c$-sortable element~$w$ is \emph{commutation equivalent} to the concatenation of the individual sorting words of the Garside factors of~$w$.

For $w \in \Sortm(W,c)$, define
\[
  \sw{\garside(w)}{c}\eqdef \left[\sq{w}^{(1)}(\c^{(1)})\right] \cdot \left[\sq{w}^{(2)}(\c^{(2)})\right] \cdot\ \cdots\ \cdot \left[\sq{w}^{(m)}(\c^{(m)})\right]
\]
to be the concatenation of the $\c^{(i)}$-sorting words defined from the Garside factors of~$w$ as in \Cref{def:factorsort}.

\begin{proposition}
\label{prop:sortingwordgarsidefactorization}
  For $w \in \Sortm(W,c)$ with sorting word $\sw{w}{c}$, we have
  \[
    \sw{\garside(w)}{c} \equiv \sw{w}{c}.
  \]
\end{proposition}
\begin{proof}
  Let $w \in \Sortm(W,c)$ with~$s$ initial in~$c$ and $s \in \DesSet_L(w)$.
  If we let $w = su$, then the proof of \Cref{prop:aligned_cambrian_recurrence} shows that the Garside factorization of~$u$ is obtained by removing the initial~$s$ from the Garside factorization of~$w$ and performing commutations.
  The same relationship is trivially true for the sorting words of~$w$ and $u$.
  The result now follows from the Cambrian recurrence.
\end{proof}

This proposition is a generalization of \Cref{lem:reflection_order}\eqref{it:reflection_order5} to arbitrary sortable elements---applying \Cref{prop:sortingwordgarsidefactorization} to $\bwom[2]\in\Sortm[2](w,c)$ gives
\[
  \sw{\wo}{c}\cdot \sw{\wo}{\psi({c})} = \sw{\garside(\bwom[2])}{c} \equiv \sw{\wom[2]}{c}= \c^h.
\]

\begin{example}
  In $\WA[4]$ with $\c=\s_1\s_2\s_3$, we compute
  \[
    \begin{array}{rcl}
      \sw{\garside(\wom[2])}{c}=\sw{\wo}{c}\cdot \sw{\wo}{\psi({c})}= &
      \left(\s_1\s_2\s_3|\s_1\s_2|\s_1\right) \cdot \left(\s_3\s_2\s_1|\s_3\s_2|\s_3\right) & = \sw{\wom[2]}{c} \\
      \equiv & \s_1\s_2\s_3|\s_1\s_2\s_3|\s_1\s_2\s_3|\s_1\s_2\s_3 &= \c^4
    \end{array}
  \]
\end{example}

As a corollary to \Cref{prop:sortingwordgarsidefactorization}, the color of a reflection in $\invs_\refl(\sw{w}{c})$ for $w \in \Sortm(W,c)$ matches the Garside factor containing the corresponding letter in $\sw{\garside{w}}{c}$.

\begin{corollary+}
\label{cor:orderedcolors}
  Let $w \in \Sortm(W,c)$ with $\sw{w}{c}=\s_1\s_2\cdots\s_p$ and $\invs(\sw{w}{c}) = \left(\beta_1^{(i_1)}, \ldots, \beta_p^{(i_p)}\right)$.
  Then the letter $\s_a$ belongs to $\sq{w}^{(j+1)}(\sq{c}^{(j+1)})$ under the identification $\sw{w}{c}\equiv \sw{\garside(w)}{c}$ if and only if $i_a = j$.
\end{corollary+}

\section{Lattice properties of \mhead-eralized sortable elements}
\label{sec:technical}

In this section, we $m$-eralize both \Cref{eq:scfa_parabolic} and \Cref{lem:readingjoininparabolic}.  The proofs of these theorems are somewhat technical.

\subsection{Projections of sortable elements to parabolic subgroups}
We first $m$-eralize \Cref{eq:scfa_parabolic} to $m$-eralized sortable elements and their colored inversion sets.%, showing that sortable elements behave well under projection to parabolic subgroups $w_J\eqdef w \wedge \wom[k](J)$, where $k=\deg(w)$ is the number of Garside factors of $w$..

\begin{theorem}
\label{prop:sort_parabolic}
  For $w \in \Sortm(W,c)$, define $w_J\eqdef w \wedge \wom(J)$.  Then
  \[
    w_J \in \Sortm(W_J,\restri{c}{J}) \text{ and } \inv(w_J)=\restri{\inv(w)}{J}.
  \]
\end{theorem}

\christianside{You may also come back to the proof later.}\nathanside{I'm kind of with Christian on this one.}
\begin{proof}
  For each $m \geq k \geq 0$, define $w_k^m(J) \eqdef \wom[k]\wo(\psi^k(J))^{m-k}$ and
  \[
    w(k)\eqdef w \wedge w_k^m(J)
  \]
  with Garside factorization $w(k)=w_k^{(1)} \cdot \cdots \cdot w_k^{(m)}$.
  By construction of Garside factorization, the first $k$ Garside factors of $w(k)$ agree with those of $w$.

  Define $J_k$ to be the set of right descents $s_i$ of $w_k^{(k)}$ such that $w_k^{(1)} \cdots w_k^{(k)} (\alpha_{s_i}) \in \Phi_J$.  By convention, we set $w_m^{(m+1)}=e$ and $J_m=\emptyset$.  We prove the following statements by decreasing induction on $k$:
\begin{itemize}
	\item $w_k^{(k+1)} \cdots w_k^{(m)}$ uses only letters from $J_k$,
	\item the inversions in $\inv(w(k))$ of color weakly greater than $k$ are exactly those inversions in $\inv(w)$ of color weakly greater than $k$ that lie in $\Phi_J$, and
	\item the element $w(k)$ is factorwise $c$-sortable.
\end{itemize}
	
The base case $k=m$ follows from the fact that $w(m)=w$ and $w$ has no inversions of color greater than $m$.  Suppose we have shown the statements for $k$; we will show them for $k-1$.
%we have $w(k-1) = w(k) \wedge w_{k-1}^m(J)$.

   Writing the parabolic decomposition of $w^{(k)}$ with respect to $J_{k-1}$ from~\Cref{eq:parabolic_decomposition_W} as $w^{(k)} = (w^{(k)})_{J_{k-1}} (w^{(k)})^{J_{k-1}}$, we have
\begin{align*}
w(k) &=w^{(1)} \cdot \cdots \cdot w^{(k-1)} \cdot w^{(k)} \cdot w_{k}^{(k+1)} \cdot \cdots \cdot w_{k}^{(m)}\\
&=w^{(1)} \cdot \cdots \cdot w^{(k-1)} \cdot \left[(w^{(k)})_{J_{k-1}} (w^{(k)})^{J_{k-1}}\right] w_{k}^{(k+1)} \cdot \cdots \cdot w_{k}^{(m)}.
\end{align*}

As $w_{k-1}^m(J)$ is initial in~$w_k^m(J)$, we have that $w(k-1)=w(k) \wedge w_{k-1}^m(J)$.   Since $w(k-1)$ is an initial segment of $w(k)$, $\restri{\inv(w(k-1))}{J} \subseteq \restri{\inv(w(k))}{J}$.
  We will conclude equality by showing that there are exactly the right number of inversions in $\inv(w(k-1))$.  Since $w(k)$ and $w(k-1)$ share their first $k-1$ Garside factors, we are only concerned with the remaining $m-k+1$ factors.

By induction, every inversion corresponding to a letter in $w_k^{(k+1)} \cdots w_k^{(m)}$ lies in $\Phi_J$.   To show that these are again inversions in $w(k-1)$, we must check that every letter in $w_k^{(k+1)} \cdots w_k^{(m)}$ can move past $(w^{(k)})^{J_{k-1}}$ as a simple reflection in $\sref_{J_{k-1}}$.

By induction, $w_k^{(k+1)} \cdots w_k^{(m)}$ uses only letters from $J_k$ and $w(k)$ is factorwise $c$-sortable.  Fix $s_i$ a letter used in $w_k^{(k+1)} \cdots w_k^{(m)}$.  By definition of factorwise $c$-sortablility, $s_i \in \DesSet_R(w^{(k)})$.  By the inductive hypothesis, when computing $\inv(w(k))$, the uncolored root corresponding to this $s_i$ lies in $\Phi_J$.

Since $s_i \in \DesSet_R(w^{(k)})$, it must be that $(w^{(1)}\cdot \cdots \cdot w^{(k)})(\alpha_{s_i})$ occurs among the roots corresponding to $(w^{(k)})_{J_{k-1}}$ in $\inv(w(k))$.  Furthermore, since this root can be removed from $\inv(w^{(k)})$ while preserving the biclosed property, it can also be removed from $\inv((w^{(k)})_{J_{k-1}})$.

We conclude that for any letter $s_i \in J_k$, there exists $s_i' \in \sref_{J_{k-1}}$ with $s_i' \in \DesSet_R((w^{(k)})_{J_{k-1}})$ and \[\left(w^{(1)} \cdots w^{(k-1)} (w^{(k)})_{J_{k-1}}\right) (\alpha_{s_i'})=\left(w^{(1)} \cdots w^{(k)}\right) (\alpha_{s_i})\in \Phi_J.\]  We may therefore move every letter in $w_k^{(k+1)} \cdots w_k^{(m)}$ past $(w^{(k)})^{J_{k-1}}$ as a simple reflection in $\sref_{J_{k-1}}$.    Finally, moving all letters of $w_k^{(k+1)} \cdots w_k^{(m)}$ past $(w^{(k)})^{J_{k-1}}$ (in order, from left to right) and then dropping the trailing $(w^{(k)})^{J_{k-1}}$ preserves factorwise $c$-sortability.  
\end{proof}

\begin{example}
  Neither \Cref{cor:orderedcolors} nor \Cref{prop:sort_parabolic} hold for general non-sortable elements.
  For example, let $w = sst \in \BA[3]$ with Garside factorization $s \cdot st$.
  Then \Cref{cor:orderedcolors} fails for $w$.
  Since $\invs(s \cdot st) = \left( \alpha^{(0)},\alpha^{(1)},\beta^{(0)} \right)$, we see that $\t$ appears in the second Garside factor, but the corresponding root~$\beta$ has color~$0$.
  \Cref{prop:sort_parabolic} also fails for $w$ with $J=\{t\}$, since $\restri{\inv(w)}{J}=\{\beta^{(0)}\}$ but $\inv(w_J)=\inv(\one)=\emptyset$.
\end{example}

\subsection{Joins of sortable elements with initial simple reflections}

The following theorem generalizes part of~\Cref{lem:readingjoininparabolic}, and will be needed for $m$-eralized Cambrian rotation to be well-defined.  
Although \Cref{prop:joininparabolic} does not describe the change in cover reflections from $w$ to $w \vee s^k$ (because we have not defined cover reflections for elements of $\Artinmon$), the related \Cref{lemma:skipjoininparabolic} indicates how a related set of colored roots changes.

\begin{theorem}
\label{prop:joininparabolic}
  For~$s$ be initial in~$c$, if $w \in \Sortm(W_{\langle s \rangle},\coxsn)$, then $w \vee s^k$ is both $c$-sortable and $\coxrn$-sortable for any $0 \leq k \leq m$.
\end{theorem}

\medskip

\christianside{If you enjoyed the proof on the previous page, take a deep breath and continue.}\nathanside{This one is actually a little harder, but the idea is still simple.}
Set $w(0) = w^{(1)}_0 \cdot\ \cdots\ \cdot w^{(m)}_0$ to be the Garside factorization of the element~$w \in \Sortm(W_{\langle s \rangle},\coxsn)$ in the statement of~\Cref{prop:joininparabolic}.  For $1 \leq k \leq m$ inductively set
\begin{align*}
  w(k)
    =& w^{(1)}_{k-1} \cdot w^{(2)}_{k-1} \cdot\ \cdots\ \cdot w^{(k-1)}_{k-1} \cdot \left(w^{(k)}_{k-1} \vee s_k\right) \cdot (w^{(k+1)}_{k-1})^{v_k} \cdot\ \cdots\ \cdot (w^{(m)}_{k-1})^{v_k} \\
    =& w^{(1)}_k \cdot w^{(2)}_k \cdot\ \cdots\ \cdot w^{(m)}_k
\end{align*}
where $s_k, v_k\in W$ are given by
\[
  s_k = s_{k-1}^{\left(w^{(k-1)}_{k-1}\right)^{-1}} \quad \text{ and } \quad
  v_k = \left(w^{(k)}_{k-1}\right)^{-1} \left(w^{(k)}_{k-1} \vee s_k\right).
\]
We will show that the decomposition $w(k) = w^{(1)}_k \cdot\ \cdots\ \cdot w^{(m)}_k$ is the Garside factorization of $w \vee s^k$, and that this factorization is factorwise $c$- and $\coxrn$-sortable.

\begin{lemma}
\label{lem:inductivesk}
  The element $w(k)$ is $c$-sortable and $\coxrn$-sortable with Garside factorization $w^{(1)}_k \cdot\ \cdots\ \cdot w^{(m)}_k$.
\end{lemma}

As the proof of \Cref{lem:inductivesk} is a somewhat lengthy induction on~$k$, we extract the base case into a separate lemma for readability.

\begin{lemma}
\label{lem:joininparabolic}
  The element $w(1)$ is $c$-sortable and $\coxrn$-sortable with Garside factorization
  \[
    w(1) = (w^{(1)}_0 \vee s) \cdot (w^{(2)}_{0})^{v} \cdot\ \cdots\ \cdot (w^{(m)}_{0})^{v},
  \]
  where $v = \left(w^{(1)}_{0}\right)^{-1} \left(w^{(1)}_{0} \vee s\right)  \in W$.  Furthermore,  $\coveredref\left(w^{(1)}_0 \vee s\right) = \coveredref\left(w^{(1)}_0\right)\cup\{s\}$.% (w^{(1)}_{1})^{-1} w^{(1)}_{0} 
\end{lemma}

\begin{proof}
  Since $w(0)$ is $c$-sortable (and thus factorwise $c$-sortable) by assumption, its first factor $w^{(1)}_0$ is also $c$-sortable.
  This element $w^{(1)}_0$ satisfies the assumption of \Cref{lem:readingjoininparabolic}, and we obtain that $w^{(1)}_0 \vee s$ is both $c$-sortable and $\coxrn$-sortable, and that
  \begin{equation}
    \coveredref(w^{(1)}_0 \vee s) = \coveredref(w^{(1)}_0)\cup\{s\}. \label{eq:coverstocovers}
  \end{equation}
  The factorwise $c$-sortability also implies that the Garside factors $w^{(2)}_{0},\ldots, w^{(m)}_{0}$ all live in the parabolic subgroup $W_{\DesSet_R(w^{(1)}_0)}$.
  Now, conjugating all these Garside factors by~$v$ simply takes those right descents of $w^{(1)}_0$, maps them to the cover reflection $\coveredref(w^{(1)}_0)$ by conjugating with $w^{(1)}_0$, and then turns these cover reflections back to the corresponding right descents of $w^{(1)}_{1} = w^{(1)}_{0} \vee s$ using~\eqref{eq:coverstocovers}.
  Since this rearrangement of the right descents is clearly compatible with the defining property of factorwise $c$-sortability and factorwise $\coxrn$-sortability, the statements follow.
\end{proof}

\begin{proof}[Proof of \Cref{lem:inductivesk}]
  We write $c^{(i)}_j$ for the Coxeter element $c^{(i)}$ for $w(j)$, as defined at the beginning of \Cref{sec:factorsort}.
  We will prove the statement of \Cref{lem:inductivesk} along with
  \begin{itemize}
    \item $\coveredref(w^{(k)}_{k}) = \coveredref(w^{(k)}_{k-1} \vee s_{k}) = \coveredref(w^{(k)}_{k-1}) \cup \{s_{k}\}$; and
    \item  $s_{k}$ is initial in $c^{(k-1)}_{k-1}$,
  \end{itemize}
  by induction on~$k$.
  These are established for $k=1$ by \Cref{lem:joininparabolic}; we interpret $c^{(0)}_0$ as $c$.
%  \hugh{Added a definition of $c_0^{(0)}$ because it is no longer defined in at the beginning of 5.4.}
  It remains to conclude the statements for~$k$, assuming that they hold for $k-1$.

  The first $k-1$ Garside factors have not changed, and so are still Garside factors, and each of them is sortable in its corresponding parabolic subgroup given by the definition of factorwise sortability.

  \medskip

  We next show that $s_k \in \DesSet_R(w_{k-1}^{(k-1)})$ (which, in particular, shows that it is a simple reflection) and $s_k \notin \supp(w_{k-1}^{(k)})$.
  We can assume by induction that
  \[
    \coveredref(w^{(k-1)}_{k-1}) = \coveredref(w^{(k-1)}_{k-2} \vee s_{k-1}) = \coveredref(w^{(k-1)}_{k-2}) \cup \{s_{k-1}\}.
  \]
  Therefore, $s_k = s_{k-1}^{w^{(k-1)}_{k-1}}$ is the right descent of $w^{(k-1)}_{k-1}$ corresponding to its cover reflection $s_{k-1}$, implying the first property $s_k \in \DesSet_R(w_{k-1}^{(k-1)})$.
  Moreover, $w^{(k)}_{k-1} = \big(w^{(k)}_{k-2}\big)^{v_{k-1}}$ sits inside the right descents of $w^{(k-1)}_{k-1}$ in the same way as $w^{(k)}_{k-2}$ sits in the right descents of $w^{(k-1)}_{k-2}$.
  Since $s_{k-1}$ was not a covered reflection of $w^{(k-1)}_{k-2}$, the right descent $s_k$ of $w^{(k-1)}_{k-1}$ corresponding to this covered reflection cannot be contained in the support of $w_{k-1}^{(k)}$, yielding the second property $s_k \notin \supp(w_{k-1}^{(k)})$.

  The induction hypothesis gives us that $s_{k-1}$ is initial in $c^{(k-2)}_{k-2}$.  Therefore $s_k$ is initial in $c^{(k-1)}_{k-1}$ by the definition of $c^{(k-1)}_{k-1}$ since $s_{k-1} \in \coveredref(w^{(k-1)}_{k-1})$ is the cover reflection corresponding to $s_k \in \DesSet_R(w^{(k-1)}_{k-1})$.% and $s_{k-1}$.
%\hugh{I don't understand ``and $s_{k-1}$''.  It looks like it might be the beginning of a clause that is missing.}
%\nathan{I removed it.  I'm not sure what it was doing.}

  We can therefore apply \Cref{lem:readingjoininparabolic} to the $c^{(k-1)}_{k-1}$-sortable element $w^{(k)}_{k-1}$ to obtain that
  $w^{(k)}_{k} = w^{(k)}_{k-1} \vee s_k$ is again $c^{(k-1)}_{k-1}$- and $(s_k^{-1}c^{(k-1)}_{k-1}s_k)$-sortable with
\[
\coveredref(w^{(k)}_{k}) = \coveredref(w^{(k)}_{k-1} \vee s_{k}) = \coveredref(w^{(k)}_{k-1}) \cup \{s_{k}\}.
\]

Therefore $w^{(k)}_{k-1}$ lives in the parabolic subgroup generated by $\coveredref(w^{(k-1)}_{k}) = \coveredref(w^{(k-1)}_{k-1})$.

  The final part of the proof is to conjugate the remaining Garside factors $w_{k-1}^{(k+1)}$ through $w_{m}^{(k+1)}$ by~$v_k$.
  This part is completely analogous to the argument given in the proof of \Cref{lem:joininparabolic}.
\end{proof}

\begin{proof}[Proof of \Cref{prop:joininparabolic}]
  We show that $w(k) = w \vee s^k$, and again first consider the case $k=1$.

  Clearly, $s \leq w(1)$ since the Garside factorization begins with $w^{(1)}_0 \vee s$ which is above~$s$ in~$W$ and therefore has a reduced $\sref$-word starting with~$s$.
  Also $w \leq w(1)$ since
  \begin{align*}
    w(1)=w^{(1)}_1 \cdot w^{(2)}_1\cdot\ \cdots\ \cdot w^{(m)}_1
      &= \big(w^{(1)}_0 v_1\big)\cdot\big(v_1^{-1} w^{(2)}_0 v_1\big)\cdot\ \cdots\ \cdot \big(v_1^{-1} w^{(m)}_0 v_1\big) \\
      &= w^{(1)}_0 \cdots w^{(m)}_0 v_1 = w v_1,
  \end{align*}
  where we write $v_1 = (w^{(1)}_{0})^{-1} w^{(1)}_{1} = (w^{(1)}_{0})^{-1} w^{(1)}_{0} \vee s_k$ as before, and write $v_1 = (w^{(1)}_{0} \vee s_k)^{-1} (w^{(1)}_{0}) \in B^+$.
  The first equality is given by the definition of $w(1)$ in terms of $w(0)$.
  Then \Cref{lem:joininparabolic} implies that $w^{(1)}_1\cdot\ \cdots\ \cdot w^{(m)}_1$ is indeed the Garside factorization~$w(1)$.

  It remains to show that~$w(1)$ is minimal among all elements above~$s$ and~$w$.  Although the colored inversion set of an element of $B^+$ is not necessarily unique to that element, the number of inversions still tells us its length.  Any element above~$w$ must contain all inversions of~$w$, and any element above $w^{(1)}_0$ and~$s$ must contain the inversions of $w^{(1)}_0 \vee s$.  The inversion set of $w(1)$ contains all these inversions and no others, and therefore has the minimal desired length; we conclude that $w(1) = w \vee s$.

  \medskip

  For the case of general~$k$, we first check that $w(k-1) \leq w(k)$ and that $s^k \leq w(k)$.
  For $w(k-1)$, we have
  \begin{align*}
    w(k)  =& w^{(1)}_{k-1} \cdot w^{(2)}_{k-1} \cdot\ \cdots\ \cdot w^{(k-1)}_{k-1} \cdot (w^{(k)}_{k-1} \vee s_k) \cdot (w^{(k+1)}_{k-1})^{v_k} \cdot\ \cdots\ \cdot (w^{(m)}_{k-1})^{v_k} \\
          =&
    w^{(1)}_{k-1} w^{(2)}_{k-1} \cdots w^{(k-1)}_{k-1} w^{(k)}_{k-1} \left( (w^{(k)}_{k-1})^{-1} (w^{(k)}_{k-1} \vee s_k) \right) (w^{(k+1)}_{k-1})^{v_k} \cdots (w^{(m)}_{k-1})^{v_k} \\
          =& w^{(1)}_{k-1} \cdots w^{(m)}_{k-1} \left( (w^{(k)}_{k-1})^{-1} (w^{(k)}_{k-1} \vee s_k) \right).
  \end{align*}
  For $s^k$, we have
  \begin{align*}
    w(k)  =& w^{(1)}_{k-1} w^{(2)}_{k-1} \cdots w^{(k-1)}_{k-1} (w^{(k)}_{k-1} \vee s_k) (w^{(k+1)}_{k-1})^{v_k} \cdots (w^{(m)}_{k-1})^{v_k} \\
          =& w^{(1)}_{k-1} w^{(2)}_{k-1} \cdots w^{(k-1)}_{k-1} s_k \cdots = w^{(1)}_{1} w^{(2)}_2 \cdots w^{(k-1)}_{k-1} s_k \cdots \\ =& s (w^{(1)}_{1} w^{(2)}_2 \cdots w^{(k-1)}_{k-1} ) \cdots = s s^{k-1} \cdots = s^k \cdots.
  \end{align*}

We now show that $w(k)$ is the minimal element above $s^k$ and $w(k-1)$.  Let $u=(w_{k-1}^{(1)}w_{k-1}^{(2)}\cdots w_{k-1}^{(k-1)})$.  We claim that $s^k \vee w(k-1)=(su) \vee w(k-1).$  We first show that $s^k \vee u = su$.  The element $s^k \vee u$ is divisible by $s^k$, and since~$u$ is divisible by $s^{k-1}$ but not by $s^k$, the length of $s^k \vee u$ is at least one more than the length of $u$.  As the element $su=us_k$ is divisible by both $s^k$ and $u$, we conclude that $su=s^k \vee u$.  Therefore, since~$u$ is initial in~$w(k-1)$, we conclude that
\[su \vee w(k-1) = s^k \vee u \vee w(k-1) = s^k \vee w(k-1).\]
Now $su=us_k$, so that $us_k$ and $w(k-1)$ share their first $k-1$ Garside factors.  Therefore, \[su \vee w(k-1) = us_k \vee w(k-1) = u(s_k \vee (w_{k-1}^k \cdot\ \cdots\ \cdot w_{k-1}^{(m)})).\]  By the inversion set argument used above for $k=1$, we may now conclude that the final $m-k+1$ Garside factors of $s_k \vee (w_{k-1}^k \cdot\ \cdots\ \cdot w_{k-1}^{(m)})$ are of the specified form, so that $w(k)=s^k \vee w(k-1).$
\end{proof}

%%%%%%%%%%%%%%%%%%%%%%%%%%%%%%%%%%%%%%%%%%%%%%%%%%%%%%%%%%%%%%%%%%%%%%%%%%%%%%%%%%%%%
\section{Cambrian lattices}
\label{sec:sort_cambrian_lattices}
%%%%%%%%%%%%%%%%%%%%%%%%%%%%%%%%%%%%%%%%%%%%%%%%%%%%%%%%%%%%%%%%%%%%%%%%%%%%%%%%%%%%%

\begin{definition}
  The \defn{$m$-eralized $c$-Cambrian poset} $\Cambsortm(W,c)$ is the restriction of $\Weakm(W)$ to $\Sortm(W,c)$.
\end{definition}\nathanside{Time for some Cambrian recurrence!}

\Cref{fig:cambsortA22} shows all $12$ $st$-sorting elements in $\Cambsortm[2](\WA[3],st)$.
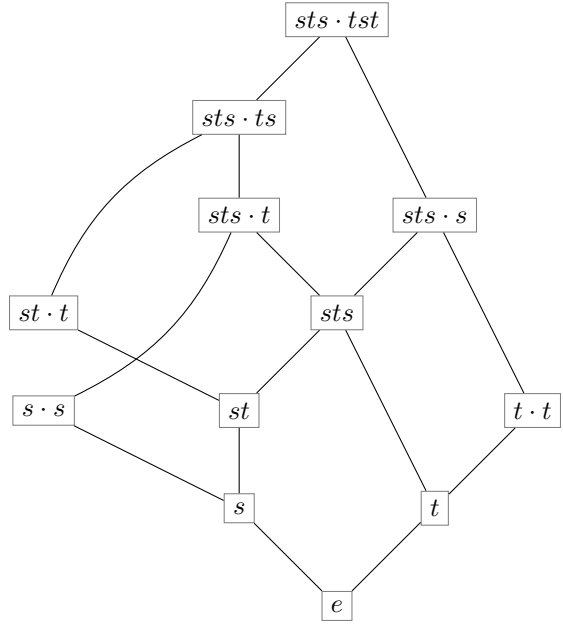
\begin{figure}[t]
  \begin{center}
    \begin{tikzpicture}[scale=1.3]
      \tikzstyle{rect}=[rectangle,draw,opacity=.5,fill opacity=1]
      \tikzstyle{sort}=[]
      \node[rect,sort] (e)   at (0,0) {$\one$};
      \node[rect,sort] (s)   at (-1,1) {$s$};
      \node[rect,sort] (t)   at ( 1,1) {$t$};
      \node[rect,sort] (st)  at (-1,2) {$st$};
      \node[rect,sort] (sts) at ( 0,3) {$sts$};
      \node[rect,sort] (stss)   at ( 1,4) {$sts\cdot s$};
      \node[rect,sort] (stst)   at (-1,4) {$sts\cdot t$};
      \node[rect,sort] (ststs)  at (-1,5) {$sts\cdot ts$};
      \node[rect,sort] (stssts) at ( 0,6) {$sts\cdot tst$};
      \node[rect,sort] (ss)   at (-3,2) {$s\cdot s$};
      \node[rect,sort] (tt)   at ( 2,2) {$t\cdot t$};
      \node[rect,sort] (stt)  at (-3,3) {$st\cdot t$};

      \draw (e) to (s) to (st) to (sts) to (stss) to (stssts);
      \draw (e) to (t) to (sts) to (stst) to (ststs) to (stssts);
      \draw (s) to (ss) to[bend right=20] (stst);
      \draw (t) to (tt) to (stss);
      \draw (st) to (stt) to[bend left=20] (ststs);
    \end{tikzpicture}
  \end{center}
  \caption{The Cambrian lattice $\Cambsortm[2](\WA[3],st)$.  Each sortable element is represented by its Garside factorization.}
  \label{fig:cambsortA22}
\end{figure}
By~\Cref{eq:rightweakorder}, weak order on~$W$ is characterized as containment of inversion sets.  Although comparison of colored inversion sets does not recover $\Weakm(W)$ for $m \geq 2$, it does capture relations among $c$-sortable elements.

%The characterization of the weak order on~$W$ as containment of inversion sets extends to $\Sortm(W,c)$---but the analogue of \Cref{thm:sort_componentwise_inversion} does \emph{not} hold for general elements in $\Weakm(W)$ when $m \geq 2$.

\begin{theorem}
\label{thm:sort_componentwise_inversion}
  For $w,u\in \Sortm(W,c)$,
  \[
    w \leq u \text{ if and only if } \inv(w) \subseteq \inv(u).
  \]
\end{theorem}
\christianside{Inversion sets for sortables behave in the positive Artin monoid in the same way inversion sets for elements in the Coxeter group do---good to know!}

\begin{proof}
 If $w \leq u$ then it is clear that $\inv(w)$ is contained in $\inv(u)$, since~$w$ is initial in~$u$.  We now argue the converse.  Suppose $\inv(w) \subseteq \inv(u)$ and let $s$ be initial in $c$.
  \begin{enumerate}[$\circ$]
    \item Suppose~$s \in \AscSet_L(w)$ and $s \in \AscSet_L(u)$.
    Then we are done by restriction to $W_{\langle s \rangle}$.

    \item The case~$s \in \DesSet_L(w)$ and $s \in \AscSet_L(u)$ is not possible as $\DesSet_L(w) \subseteq \DesSet_L(u)$.

    \item Suppose $s \in \AscSet_L(w)$ and $s \in \DesSet_L(u)$.
    It is clear that $u_{\langle s \rangle} \leq u$.
    Since $w \in \Sortm(W_{\langle s \rangle},\coxsn)$, we have that $\inv(w) \subseteq \inv(u_{\langle s \rangle})$ by \Cref{prop:sort_parabolic}, so that by induction on rank (since both $w,u_{\langle s \rangle} \in \Sortm(W_{\langle s \rangle},\coxsn)$), $w \leq u_{\langle s \rangle}$.
    Since $u_{\langle s \rangle} \leq u$, we conclude that $w \leq u$.

    \item Suppose finally that~$s \in \DesSet_L(w)$ and $s \in \DesSet_L(u)$.
    Then we get the statement for $\sninv u=u'$ and $\sninv w=w'$ by induction on length.
    Multiplying by~$s$ does not change containment of inversion sets (since multiplication by~$s$ just multiplies all inversions by~$s$, and then adds~$\alpha_s^{(0)}$).\qedhere%\hugh{Since we are talking about inversion sets I changed it to talk about roots not reflections.}
  \end{enumerate}
\end{proof}

\medskip

In fact, the intersection of colored inversion sets of two $c$-sortable elements is again the inversion set of a $c$-sortable element, $m$-eralizing \Cref{prop:scfa_inv}.
\begin{lemma}
  \label{thm:sort_intersection}
  Let $u,v\in \Sortm(W,c)$.  Then $u \wedge v \in \Sortm(W,c)$ and
  \[\inv(u \wedge v)= \inv(u) \cap \inv(v).\]
\end{lemma}
\begin{proof}
	Let $s$ be initial in $c$.  We consider again the four possible cases.
  \begin{enumerate}[$\circ$]
    \item If~$s \in \AscSet_L(u)$ and $s \in \AscSet_L(v)$, then we are done by restriction to $W_{\langle s \rangle}$ and induction on rank.

    \item If $s \in \DesSet_L(u)$ and $s \in \AscSet_L(v)$, then $v \in \Sortm(W_{\langle s \rangle},\coxsn)$.  Then $u \wedge v = (u \wedge (\wom(J) \wedge v)) = ((u \wedge \wom(J)) \wedge v) = u_{\langle s \rangle} \wedge v.$  By \Cref{prop:sort_parabolic}, $u_{\langle s \rangle}$ is $\coxsn$-sortable, so that $u \wedge v$ is $c$-sortable with inversion set given by the previous case.

    \item The case  $s \in \DesSet_L(v)$ and $s \in \AscSet_L(u)$ follows by symmetry.

    \item Finally, if~$s \in \DesSet_L(u)$  and $s \in \DesSet_L(v)$, then by~\Cref{prop:msort_cambrian_recurrence}, $\sninv u$ and $\sninv  v$ are both $\coxrn$-sortable.
    By induction on length, $(\sninv  u) \wedge (\sninv  v)$ is $\coxrn$-sortable.
    Then $u \wedge v = s( \sninv  u \wedge \sninv  v)$, which is again $c$-sortable by \Cref{prop:msort_cambrian_recurrence}.   Furthermore, we have
      \[
        \inv(\sninv u \wedge \sninv v)= \inv(\sninv u) \cap \inv(\sninv v)
      \]
      by induction on length.
      This gives
      \begin{align*}
        \inv(w \wedge u) &= \{\alpha_s^{(0)}\} \cup s (\inv(\sninv w \wedge \sninv u)) \\
                         &= \{\alpha_s^{(0)}\} \cup s \left(\inv(\sninv w) \cap \inv(\sninv u)\right) \\
                         &= \inv(w)\cap \inv(u). \qedhere
      \end{align*}
  \end{enumerate}
\end{proof}

An element $w' \in W$ is called \defn{$c$-antisortable} if $w'\wo$ is $c^{-1}$-sortable.  N.~Reading showed that every element $w \in \Weak(W)$ lies above a unique largest $c$-sortable element~$\pid(w)$ in weak order, and lies below a unique smallest $c$-antisortable element~$\piu(w)$~\cite{Rea2007}.  He proved that the fibers of the $\pid: \Weak(W) \to \Sort(W,c)$ are given by the intervals $[\pid(w),\piu(w)]_{\Weak(W)}$, so that $\pid$ defines a lattice congruence~\cite[Proposition~3.1]{Rea2007}.% in the sense that every fiber is an interval and the maps $\pid$ and $\piu$ are order-preserving, see~\cite[Proposition~3.1]{Rea2007}.
%\christian{rewrote and shortened paragraph}
%\hugh{``a lattice congruence in the sense that...'' seems weird.  It is pretty clear that a lattice congruence should be a congruence that respects the lattice structure.  It's then a theorem that this is equivalent to the condition you state.  If we need that, I think we should say it.  Or if you want to say ``a lattice congruence in the sense that it is a congruence that respects the lattice operations,'' I would be fine with that.  }
\medskip

% N.~Reading's $c$-sortable elements are the key to understanding certain order congruences on the weak order that respect the lattice structure of the weak order (and are therefore lattice congruences).
% We briefly summarize some results of~\cite{Rea2007}.  N.~Reading defined an order-preserving projection $\pid: \Weak(W) \to \Sort(W,c)$ sending an element~$w$ to the largest $c$-sortable element less than or equal to~$w$.
% Likewise, there is a related order-preserving map $\piu$ that maps~$w$ to the smallest $c$-sortable greater than or equal to~$w$.\Nathan{um, no?  Also, all we really want here is that he showed it was a lattice via lattice congruence...REWORK}
% N.~Reading showed that the fibers of $\pid$ and $\piu$ are equal and that the fiber containing~$w$ is the interval $[\pid(w),\piu(w)]_{\Weak(W)}$.
% This turns out to be enough to conclude that the $c$-sortable elements form a lattice quotient of the weak order.
% 
% \medskip
% 
In contrast, the $m$-eralized $c$-sortable elements no longer form a lattice quotient of $\Weakm(W)$ for the reason that the maps $\pid$ and $\piu$ are not well-defined.%, as indicated in the following example.

For $\WA[3]$ with $m=3$, the elements of $\Weakm[3](\WA[3])$ lying above~$t$---but not above any larger $st$-sortable element---do not form an interval.  This is illustrated in \Cref{fig:notaquotient}.
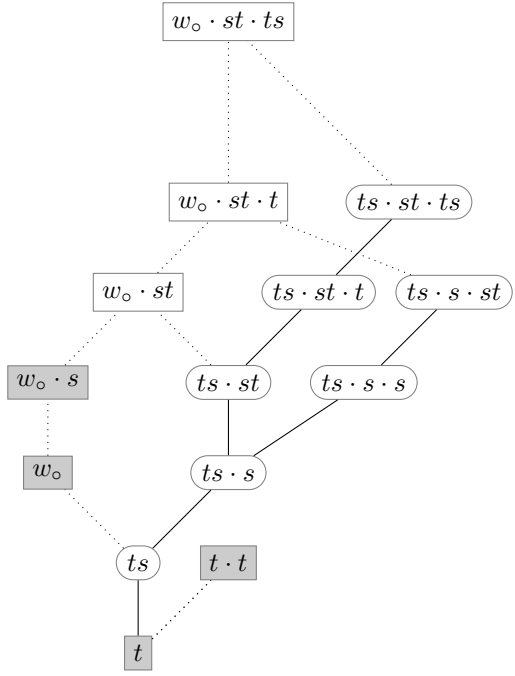
\begin{figure}[t]
  \begin{center}
    \begin{tikzpicture}[scale=1.2]
      \tikzstyle{rect}=[rectangle,draw,opacity=.5,fill opacity=1]
      \tikzstyle{rrect}=[rounded rectangle,draw,opacity=.5,fill opacity=1]
      \tikzstyle{sort}=[fill=black!20]
      \node[rect,sort] (t)   at (0,0) {$t$};
  \node[rect,sort] (tt)   at (1,1) {$t\cdot t$};
  \node[rrect] (ts)   at (0,1) {$ts$};
  \node[rrect] (tss)   at (1,2) {$ts\cdot s$};
  \node[rrect] (tsst)   at (1,3) {$ts\cdot st$};
  \node[rrect] (tsstt)   at (2,4) {$ts\cdot st\cdot t$};
  \node[rrect] (tsstts)   at (3,5) {$ts\cdot st\cdot ts$};

  \node[rrect] (tsss)   at (2.5,3) {$ts\cdot s \cdot s$};
  \node[rrect] (tssst)   at (3.5,4) {$ts\cdot s \cdot st$};

  \node[rect,sort] (sts)   at (-1,2) {$\wo$};
  \node[rect,sort] (stss)   at (-1,3) {$\wo\cdot s$};
  \node[rect] (stsst)   at (0,4) {$\wo\cdot st$};
  \node[rect] (stsstt)   at (1,5) {$\wo\cdot st\cdot t$};
  \node[rect] (stsstts)   at (1,7) {$\wo\cdot st\cdot ts$};

      \draw (t) to (ts) to (tss) to (tsst) to (tsstt) to (tsstts);
  \draw (tss) to (tsss) to (tssst);

  \draw[dotted] (t) to (tt);
  \draw[dotted] (ts) to (sts) to (stss) to (stsst) to (stsstt) to (stsstts);
  \draw[dotted] (tsst) to (stsst);
  \draw[dotted] (tsstts) to (stsstts);
  \draw[dotted] (tssst) to (stsstt);
    %\draw (ts) to (sts) to (stss) to (stsst) to (stsstt) to (stsstts);
  %\draw (ts)  to (stsst);
  %\draw (tss) to (tsss) to (tssst) to (stsstt);
  %\draw (tsst) to (tsstt) to (tsstts) to (stsstts);
  \draw[dotted] (t) to (tt);

    \end{tikzpicture}
  \end{center}%lies above a unique largest $c$-sortable element
  \caption{A part of $\Weakm[3](\WA[3])$.  The $st$-sortable elements are shaded.  The rounded boxes indicate those elements lying above $t$, but not lying above any larger $st$-sortable element.  There are two such maximal elements.}%: $ts\cdot st\cdot ts$ and $ts\cdot s\cdot st$.}
  \label{fig:notaquotient}
\end{figure}

Although $\Sortm(W,c)$ is no longer a lattice quotient of $\Weakm(W)$, the restriction of $\Weakm(W)$ to $\Sortm(W,c)$ is still a lattice.    The proof is analogous to the proof of~\cite[Theorem~1.2]{Rea2007}, except that we do not have a projection map $\pi_\downarrow^c$, and so cannot rely on its properties to compute the join.

\begin{theorem}
\label{thm:sort_is_lattice}
  $\Cambsortm(W,c)$ is a sublattice of $\Weakm(W)$.
\end{theorem}
\christianside{This is maybe the cleanest way of seeing this.}

\begin{proof} 
 \Cref{thm:sort_intersection} shows that $u \wedge v \in \Wm$ is $c$-sortable for $u,v \in \Sortm(W,c)$.  A similar argument proves that $u \vee v \in \Wm$ is $c$-sortable for $u,v \in \Sortm(W,c)$.
  Again, let~$s$ be initial in~$c$.

  \begin{enumerate}[$\circ$]
    \item If $s \in \DesSet_L(u)$ and $s \in \DesSet_L(v)$, then $s \in \DesSet_L(u \vee v)$.
    By induction on length, $(\sninv u) \vee (\sninv v)$ is $\coxsn$-sortable.
    Then $u \vee v = s( \sninv u \vee \sninv v)$ is $c$-sortable.

    \item If $s \in \DesSet_L(v)$ and $s \in \AscSet_L(u)$, then $u \in  \Sortm(W_{\langle s \rangle},\coxsn)$ and $s \vee u$ is $c$-sortable by \Cref{prop:joininparabolic}.
    We compute that $u \vee v = s \vee (u \vee v) = (s \vee u) \vee v$, so that $u \vee v$ is $c$-sortable by the previous case.

    \item The case $s \in \DesSet_L(u)$ and $s \in \AscSet_L(v)$ follows by symmetry.

    \item Finally, suppose $s \in \AscSet_L(u)$ and $s \in \AscSet_L(v)$.  Then $u,v \in \Sortm(W_{\langle s \rangle})$ and we conclude the result by induction on rank.\qedhere
  \end{enumerate}
\end{proof}

\section{The Cambrian rotation}
\label{sec:cambrian_rotation}

For~$s$ initial in~$c$, define the bijection
\begin{align}
  \Shift_s: \Sortm(W,c) &\bij \Sortm(W,\coxrn) \nonumber \\
    w &\longmapsto \begin{cases}
    w \vee s^m  & \text{if } s \in \AscSet_L(w) \\
    \sninv w  & \text{if } s \in \DesSet_L(w)
  \end{cases}\ ,\label{eq:shift_on_sorts}
\end{align}
where $w \vee s^m$ denotes the join in $\Weakm(W)$.
The first case is well-defined by~\Cref{prop:joininparabolic} from which it also follows that $w \vee s^m \neq u \vee s^m$ for $w \neq u$.

\begin{example}
\label{ex:sortshiftorbit}
  Parallel to \Cref{ex:ncshiftorbit,ex:assocshiftorbit}, alternately applying $\Shift_s$ and $\Shift_t$ to $\one \in \Sortm[2](\WA[3],st)$ gives the orbit
  \[
    \begin{array}{ccccc}
      \one
      &\xmapsto{\Shift_s}&
      s \cdot s
      &\xmapsto{\Shift_t}&
      sts \cdot sts \\
      &\xmapsto{\Shift_s}&
      tst \cdot st
      &\xmapsto{\Shift_t}&
      sts\cdot t \\
      &\xmapsto{\Shift_s}&
      tst
      &\xmapsto{\Shift_t}&
      st \\
      &\xmapsto{\Shift_s}&
      t
      &\xmapsto{\Shift_t}&
      e
    \end{array}.
  \]
\end{example}

\begin{definition}
\label{def:sort_cambrian_rotation}
  The \defn{$m$-eralized $c$-Cambrian rotation} $\Camb_c: \Sortm(W,c) \to \Sortm(W,c)$ is given by
  \[
    \Camb_c = \Shift_{s_n} \circ \cdots \circ \Shift_{s_1}
  \]
  for any reduced $\sref$-word $\s_1 \s_2 \cdots \s_n$ for~$c$.
\end{definition}

This composition evidently does not depend on the chosen reduced word.  The elements in \Cref{fig:sortA22} are arranged according to their orbits under Cambrian rotation.

%%%%%%%%%%%%%%%%%%%%%%%%%%%%%%%%%%%%%%%%%%%%%%%%%%%%%%%%%%%%%%%%%%%%%%%%%%%%%%%%%%%%%
\section{Sortable elements, noncrossing partitions, and clusters}
\label{sec:m-sort-ncp}
%%%%%%%%%%%%%%%%%%%%%%%%%%%%%%%%%%%%%%%%%%%%%%%%%%%%%%%%%%%%%%%%%%%%%%%%%%%%%%%%%%%%%

In this section, we use skip sets to relate sortable elements to noncrossing partitions (\Cref{thm:bij_sort_to_nc}) and to clusters (\Cref{thm:m-assoc_to_m_ncp}), where we recall that the term \emph{natural} means that a bijection respects the Cambrian recurrence.
These bijections were introduced by N.~Reading and D.~Speyer for $m=1$ in~\cite{RS2011}.
\christianside{Are you ready to see the parts of the theory coming together naturally?}\nathanside{Thanks, Reading and Speyer!}

\subsection{Sortable elements and noncrossing partitions}

We define a natural bijection between $m$-eralized $c$-sortable elements and $m$-eralized $c$-noncrossing partitions, $m$-eralizing the constructions in~\cite{Rea20072,RS2011}.

\medskip

Let the $\c$-sorting word of $w \in \Sortm(W,c)$ be $\sw{w}{c} = \s_1\cdots\s_p$, and let $s \in \sref$.
We say that~$w$ \defn{skips}~$s \in \sref$ in position~$k+1$ if the leftmost instance of~$s$ in $\c^\infty$ not used in $\sw{w}{c}$ occurs between~$\s_k$ and~$\s_{k+1}$.
The \defn{skip set} of colored positive roots is defined as
\[
  \Skipset(w) = \bigset{ \beta_s^{(\ell_s)} }{ s \in \sref },
\]
where $\beta_s^{(\ell_s)}=s_1\cdots s_k(\alpha_{s}^{(0)})$ when~$w$ skips~$s$ in position~$k+1$, and we say that the skip for~$s$ has color~$\ell_s$.  The skip set is ordered by the indices of the skipped positions.% \Cref{thm:bij_sort_to_nc}.

\begin{example}
\label{ex:skipset}
  The skip set of $s_1s_2s_3|s_2s_3|s_2 \in \Sortm[2](\WA[4],s_1s_2s_3)$ is
  \[
    \begin{array}{ccc|ccc|ccc|ccc}
      s_1 & s_2 & s_3 & {\lightgrey s_1} & s_2 & s_3 & - & s_2 & {\lightgrey s_3} & - & {\lightgrey s_2} & - \\
      & & & (23)^{(0)} & & & & & (34)^{(1)} & & (14)^{(2)} &
    \end{array}.
  \]
  The skip for $s_1$ has color~$0$, the skip for $s_3$ has color~$1$, and the skip for $s_2$ has color~$2$.  \Cref{fig:sortA22} on page~\pageref{fig:sortA22} shows the skip sets of all elements in $\Sortm[2](\WA[3],st)$.
\end{example}

\begin{remark}
\label{rem:skipset}
  Let $w \in \Sortm(W,c)$.  We recover its $c$-sorting word $\sw{w}{c}$ from its skip set $\Skipset(w)$ by reading the word $\c^\infty=\s_1\s_2\ldots$ from left to right, deleting letters as follows: if there is no next letter to read in the copy of $\c^\infty$ from which we have deleted some letters, the remaining letters spell $\sw{w}{c}$.   Otherwise, let~$\s$ be the letter in the current position, and let~$u$ be the product of the undeleted letters strictly to its left.  If $u(\alpha_s^{(0)}) \in \Skipset(u)$, delete the current letter~$\s$ and all occurrences of~$\s$ to the right.
\end{remark}

The skip sets satisfy an \defn{$m$-eralized $c$-Cambrian recurrence}, proven for $m=1$ by N.~Reading and D.~Speyer in~\cite{RS2011}.
\begin{proposition}
\label{prop:skipset}
  Let~$s$ be initial in~$c$ and let $w \in \Sortm(W,c)$.
  Then
  \[
    \Skipset(w) =
    \begin{cases}
      \big\{\alpha_s^{(0)}\big\} \cup \Skipset[\coxsn](w) &\text{if } s \in \AscSet_L(w) \\
      s\big(\Skipset[\coxrn](\sninv w)\big) &\text{if } s \in \DesSet_L(w)
    \end{cases}\ .
  \]
\end{proposition}

\begin{proof}
  If $s \in \AscSet_L(w)$ then~$w \in \Sortm(\Wres,\coxsn)$ with skip set $\Skipset[\coxsn](w)$.
  Treating~$w$ as an element of $\Sortm(W,c)$ does not change the positions of the skips $t \neq s$.  But~$w$ now skips~$s = s_1$ in position~$1$, which adds $\big\{\alpha_s^{(0)}\big\}$ to its skip set.
  If $s \in \DesSet_L(w)$, no simple reflection is skipped in position~$1$ and so each skip
 $s_1\cdots s_k(\alpha_{t}) \in \Skipset(w)$ can be identified with $s(s_2\cdots s_k(\alpha_{t}^{(0)})) \in s\Skipset[\coxrn](\sninv w)$.
\end{proof}

\begin{example}
  Parallel to \Cref{ex:sortmcambrian}, the sequence of skip sets for the Cambrian recurrence on $sts \cdot s \in \Sortm[2](\WA[3],st)$ is
  \[
      \underbrace{\alpha^{(1)},\beta^{(2)}}_{st}
      \hspace*{-2pt}\mapsto\hspace*{-2pt}
      \underbrace{\alpha^{(0)},\gamma^{(2)}}_{ts}
      \hspace*{-2pt}\mapsto\hspace*{-2pt}
      \underbrace{\gamma^{(0)},\alpha^{(2)}}_{st}
      \hspace*{-2pt}\mapsto\hspace*{-2pt}
      \underbrace{\beta^{(0)},\alpha^{(1)}}_{ts}
      \hspace*{-2pt}\mapsto\hspace*{-2pt}
      \underbrace{\alpha^{(1)}}_{s}
      \hspace*{-2pt}\mapsto\hspace*{-2pt}
      \underbrace{\alpha^{(0)}}_{s}
      \hspace*{-2pt}\mapsto\hspace*{-2pt}
      \underbrace{-}_{\one}.
  \]
\end{example}

We deduce the following theorem.

\begin{theorem}
\label{thm:bij_sort_to_nc}
  The map $\Skipset$ induces a natural bijection
  \begin{align*}
    \Sortm(W,c)&\cambbij\DeltaNCm(W,c)\\
    w &\longmapsto \t_1^{(i_1)} \cdots \t_n^{(i_n)},
  \end{align*}
  for $\Skipset(w) = \big\{ \beta_{1}^{(i_1)}, \ldots, \beta_{n}^{(i_n)}\big\}$ and $t_k = s_{\beta_k}$.
\end{theorem}

\begin{proof}
  Immediate from \Cref{prop:skipset} and \Cref{prop:nc_cambrian_recurrence}: the inductive structure on $\Sortm(W,c)$ is sent to the inductive structure on $\DeltaNCm(W,c)$, and both sides have the same initial conditions.
\end{proof}

Comparing rows in \Cref{fig:ncA22,fig:sortA22} illustrates for $\WA[3]$ the natural bijection $\Sortm[2](\WA[3],st) \bij \DeltaNCm[2](\WA[3],st)$.
We use the bijection of \Cref{thm:bij_sort_to_nc} to prove that $\Cambsortm(W,c)$ and $\Cambncm(W,c)$ are isomorphic.

\begin{theorem}
\label{thm:cambnccambsort}
  $\Cambsortm(W,c) \cong \Cambncm(W,c)$.
\end{theorem}

To prove this theorem, we require control over the skip set of covers of the elements in $\Sortm(W_{\langle s \rangle},\coxsn)$.

\begin{lemma}
\label{lemma:skipjoininparabolic}
  Let $w \in \Sortm(W_{\langle s \rangle},\coxsn)$ with
  \[
    \Skipset(w) = \{ \alpha_s^{(0)}, \beta_2^{(\ell_2)}, \ldots,\beta_i^{(\ell_i)},\beta_{i+1}^{(\ell_{i+1})},\ldots, \beta_n^{(\ell_n)}\}
  \]
  for $0\leq \ell_2\leq \cdots\leq \ell_i <k \leq \ell_{i+1} \leq \cdots \leq \ell_n$, then
  \[
    \Skipset(w \vee s^k) = \{[s(\beta_2)]^{(\ell_2)}, \ldots,[s(\beta_i)]^{(\ell_i)},\alpha_s^{(k)},\beta_{i+1}^{(\ell_{i+1})},\ldots, \beta_n^{(\ell_n)}\}.
  \]
\end{lemma}
\begin{proof}
  The proof relies on the description of how the Garside factorization of $w \vee s^{k}$ is obtained from the Garside factorization of~$w \vee s^{k-1}$ in \Cref{lem:inductivesk}.
  In particular, the colored inversion sequence of the $c$-sorting word only changes within the $k$\th\ Garside factor.
  \cite[Proposition~5.4]{RS2011} shows that the roots in this Garside factor change as described in the statement of the lemma for $k=1$ and $m=1$, so that the skip set $\Skipset(w \vee s^k)$ is obtained from the skip set $\Skipset(w \vee s^{k-1})$ by
  \begin{itemize}
    \item replacing $\alpha_s^{(k-1)}$ by $\alpha_s^{(k)}$;

    \item replacing all other $(k-1)$-colored roots $\beta^{(k-1)}$ by $[s(\beta)]^{(k-1)}$; and

    \item leaving all other colored roots unchanged.
  \end{itemize}
  The lemma follows by applying this procedure~$k$ times.
\end{proof}

\begin{proof}[Proof of \Cref{thm:cambnccambsort}]
  We show that a cover relation \[u \lessdot v  \text{ in } \Cambsort(W,c)\] corresponds under the map $\Skipset$ to a cover relation \[I = \r_1^{(i_1)}\cdots\r_n^{(i_n)} \lessdot \t_1^{(j_1)}\cdots\t_n^{(j_n)}=\Cambupflip[r](I) \text{ in } \Cambnc(W,c).\]
  \begin{enumerate}[$\circ$]
    \item If $s \in \AscSet_L(u)$ and $s \in \AscSet_L(v)$ (equivalently, $\r_1^{(i_1)}=\t_1^{(j_1)}=\s^{(0)}$), the statement follows by the Cambrian recurrences in \Cref{prop:nc_cambrian_recurrence,prop:sort_cambrian_recurrence} and \Cref{thm:bij_sort_to_nc}.

    \item The case $s \in \DesSet_L(u)$ and $s \in \AscSet_L(v)$ (equivalently, $\r_1^{(i_1)} \neq \s^{(0)}$ and $\t_1^{(j_1)}=s^{(0)}$) is impossible if we are starting from a cover $u\lessdot v$ in $\Cambsort(W,c)$ since $u \leq v$ in weak order, and it is impossible if we are starting from a cover in $\Cambnc(W,c)$ because the increasing flip of $\r_j^{(i_j)}$ changes only reflections that appear after it in $\c^\infty$, and it changes them into reflections which still appear after it in $\c^\infty$.

    \item If $s \in \DesSet_L(u)$ and $s \in \DesSet_L(v)$ (equivalently, $\r_1^{(i_1)} \neq \s^{(0)}$ and $\t_1^{(j_1)} \neq \s^{(0)}$), the statement follows again by the Cambrian recurrences.

    \item Then suppose $s \in \AscSet_L(u), s \in \DesSet_L(v)$ (equivalently, $\r_1^{(i_1)} = \s^{(0)}$ and $\t_1^{(j_1)} \neq \s^{(0)}$).    Since $u \leq v$ in weak order and $s \in \supp(v)$, we have that $s \leq v$ in weak order, so that $u \vee s \leq v$.
    Since $s \leq vs \notin \supp(u)$, \Cref{lemma:skipjoininparabolic} for $k=1$ implies that $u \vee s$ is $c$-sortable.
    Therefore, $v = u \vee s$ and the skip set of~$v$ may be obtained from the skip set of~$u$ by \Cref{lemma:skipjoininparabolic}.  On the other hand,~\eqref{eq:increasingflip} shows that $\Cambupflip[r](I)$ behaves in exactly the same way.\qedhere
  \end{enumerate}
\end{proof}

\subsection{Sortable elements and clusters}

Using \Cref{thm:rootconf_skiptset}, we obtain a bijection between $m$-eralized $c$-sortable elements and the $m$-eralized $c$-cluster complex by identifying the skip set of a sortable element and root configuration of a facet.

\begin{theorem}
\label{thm:m-assoc_to_m_ncp}
  There is a natural bijection
  \begin{align*}
    \Sortm(W,c) &\cambbij\DeltaAssocm(W,c)\\
    w & \longmapsto I,
  \end{align*}
  given by $\Skipset(w) = \Roots{I}$.
\end{theorem}
% \Christian{Here and elsewhere, you removed the sentence "... that respects Cambrian rotation, the Cambrian recurrence, and support." How do we handle that these in fact do, or is that already somewhere?}

\begin{proof}
  This is a direct consequence of \Cref{prop:skipset,prop:rootconfiguration}.
\end{proof}

We now $m$-eralize~\cite[Theorem~8.1]{Rea20072} to give a satisfyingly direct bijection between $\Sortm(W,c)$ and $\Assocm(W,c)$.
Write the $c$-sorting word of $w \in \Sortm(W,c)$ as $\sw{w}{c} = \s_1\cdots\s_p$.

For $s \in \supp(w)$, let $\beta_s^{(\ell_s)}$ be the colored positive root $\s_1\cdots\s_{k-1}(\alpha_{s_k}^{(0)})$, where $s_k$ is the \emph{last} occurrence of the letter~$\s$ in $\sw{w}{c}$.  Define
\[
  \Lastset(w) \eqdef \bigset{ \beta_s^{(\ell_s)} }{ s \in \supp(w) } \cup \bigset{ \alpha_s^{(m)} }{ s \notin \supp(w) }.
\]

\begin{proposition}
\label{prop:lastsetrecurrence}
  Let~$s$ be initial in~$c$ and let $w \in \Sortm(W,c)$.
  Then
  \[
    \Lastset(w) =
    \begin{cases}
      \Lastset[\coxsn](w) \cup \big\{\alpha_s^{(m)}\big\} &\text{if } s \in \AscSet_L(w) \\
      \left(\taum \right)^{-1}\Lastset[\coxrn](\sninv w) &\text{if } s \in \DesSet_L(w)
    \end{cases}\ .
  \]
\end{proposition}

\begin{proof}
  The first case $s \in \AscSet_L(w)$ is clear, since $s \notin \supp(w)$.
  If $s \in \DesSet_L(w)$, note that $\left(\taum \right)^{-1}$ acts as the simple reflection~$s$ on $\bigset{ \beta_s^{(\ell_s)} }{ s \in \supp(w) }$, while fixing $\bigset{ \alpha_s^{(m)} }{ s \notin \supp(w) }$.
  This is exactly how $\Lastset(\sninv w)$ is obtained from $\Lastset(w)$.
\end{proof}

\begin{theorem}
\label{thm:sortcluster}
  The map $\Lastset$ is a natural bijection
  \begin{align*}
    \Sortm(W,c) &\cambbij \Assocm(W,c) \\
    w &\longmapsto \Lastset(w).
  \end{align*}
\end{theorem}

\begin{proof}
This follows from the Cambrian recurrences in \Cref{prop:sort_cambrian_recurrence} and in \Cref{prop:lastsetrecurrence}.
\end{proof}

%%%%%%%%%%%%%%%%%%%%%%%%%%%%%%%%%%%%%%%%%%%%%%%%%%%%%%%%%%%%%%%%%%%%%%%%%%%%%%%%%%%%%
\section{Sortable elements and the shard intersection order}
\label{sec:m-sort-shard}
%%%%%%%%%%%%%%%%%%%%%%%%%%%%%%%%%%%%%%%%%%%%%%%%%%%%%%%%%%%%%%%%%%%%%%%%%%%%%%%%%%%%%

In this section, we give an alternative description of the bijection in \Cref{thm:bij_sort_to_nc}.  This reconciles the $m$-eralization of $c$-sortable elements with D.~Armstrong's $m$-eralization of noncrossing partitions using the shard intersection order.

\medskip

N.~Reading gave a beautiful proof of the lattice property of the noncrossing partition lattice $\NCL(W,c)$, showing that the restriction of $\Shard(W)$ to the sortable elements $\Sort(W,c)$ is isomorphic to $\NCL(W,c)$ (\cite[Theorem~8.5]{Rea2011}).

\begin{theorem+}
\label{thm:shard_nc_lattice_isomorphism}
  The restriction of $\Shard(W)$ to $\Sort(W,c)$ is isomorphic to $\NCL(W,c)$ under the natural bijection of \Cref{thm:bij_sort_to_nc}.
\end{theorem+}

\begin{example}
  \Cref{fig:ncA2} illustrates $\NCL(\WA[3],(123))$, as well as $\Shard(\WA[3])$ restricted to $\Sort(\WA[3],(123))$.   These should be compared with the corresponding shard intersection order, drawn in \Cref{fig:shardorderA2}.  Sortable elements are specified by their inversion sets---inversions are black circles, non-inversions are white circles, covered reflections are circled in gray, and other inversions in the parabolic subgroup generated by the covered reflections are circled in white.    Circles are indexed by reflections in the order
\raisebox{-0.5\height}{\begin{tikzpicture}[scale=.5]
	\rootposetfill{ 0,0}{2}{{12,23},{13}}{.8}{(e)}{};
\end{tikzpicture}}.

\christianside{A brief glimpse from the shining nonnesting world.}
\Cref{fig:ncA3shard} depicts the restriction of $\Shard(\WA[4])$ to $\Sort(\WA[4],(1234))$, and should be compared with the noncrossing partition lattice in \Cref{fig:ncA3}.  Circles are indexed by reflections in the order
\raisebox{-0.5\height}{\begin{tikzpicture}[scale=.5]
\rootposetfill{ 0,0}{3}{{12,23,34},{13,24},{14}}{.8}{(e)}{};
\end{tikzpicture}}.
  %The same convention% As predicted by \Cref{thm:shard_nc_lattice_isomorphism}, the resulting poset is isomorphic to the poset shown in \Cref{fig:ncA3}.
\label{ex:root_posets}
\end{example}

\begin{figure}[t]
  \begin{center}
    \begin{tikzpicture}[scale=1]
      \tikzstyle{rect}=[rectangle,draw,opacity=.5,fill opacity=1,inner sep=3pt,outer sep=0pt,rounded corners=0.1cm];
      \node[rect] (e)   at ( 0,0) {$(12)(23) \cdot \one$};
      \node[rect] (1)   at (-2,2) {$(13) \cdot (12)$};
      \node[rect] (2)   at ( 0,2) {$(12) \cdot (23)$};
      \node[rect] (121) at ( 2,2) {$(23) \cdot (13)$};
      \node[rect] (12)  at ( 0,4) {$\one \cdot (12)(23)$};
      \draw (e) -- (1) -- (12);
      \draw (e) -- (2) -- (12);
      \draw (e) -- (121) -- (12);
    \end{tikzpicture}
    \qquad
    \begin{tikzpicture}[scale=1]
      \tikzstyle{rect}=[rectangle,draw,opacity=.5,fill opacity=1]
      \rootposet{ 0,0}{2}{{0,0},{0}}{0.2}{(e)}{};
      \rootposet{-2,2}{2}{{2,0},{0}}{0.2}{(1)}{};
      \rootposet{ 0,2}{2}{{1,0},{2}}{0.2}{(2)}{};
      \rootposet{ 2,2}{2}{{0,2},{0}}{0.2}{(121)}{};
      \rootposet{ 0,4}{2}{{2,2},{3}}{0.2}{(12)}{};
      \draw (e) -- (1) -- (12);
      \draw (e) -- (2) -- (12);
      \draw (e) -- (121) -- (12);
    \end{tikzpicture}
  \end{center}
  \caption{\label{fig:ncA2} The isomorphic lattices $\NCL(\WA[3],st)$ and $\Shard(\WA[3])$ restricted to $\Sort(\WA[3],st)$.  Sortable elements are specified by their inversion sets as described in~\Cref{ex:root_posets}.}
%On the left is $\NCL(\WA[3],st)$, labelled by .  On the right is $\Shard(\WA[3])$ restricted to $\Sort(\WA[3],st)$.  }
\end{figure}
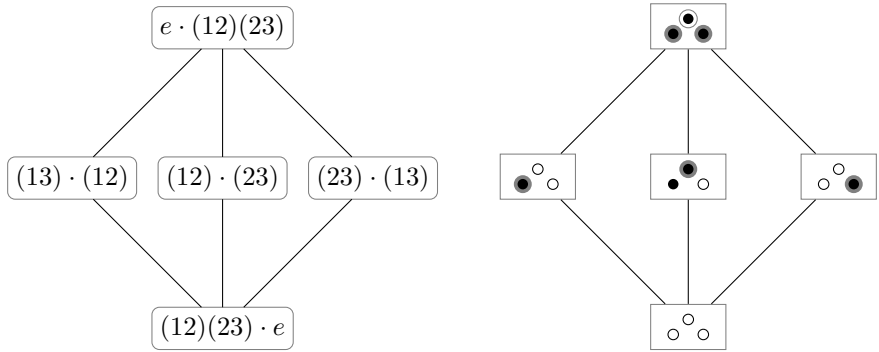

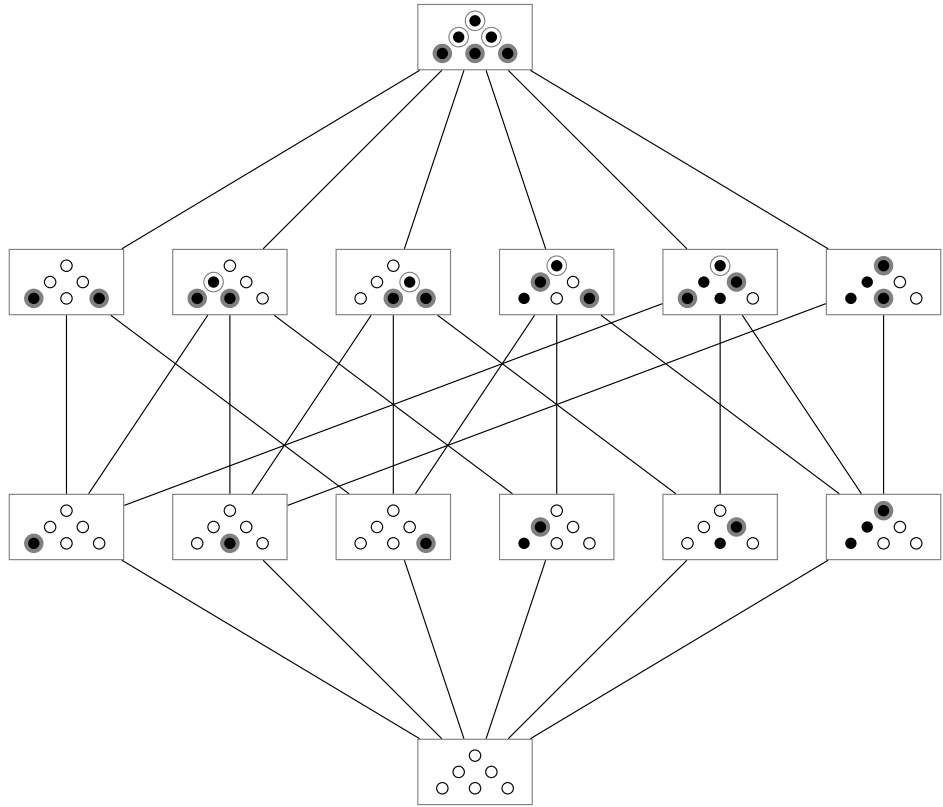
\begin{figure}[t]
  \begin{center}
    \resizebox{\textwidth}{!}{
    \begin{tikzpicture}
      \tikzstyle{rect}=[rectangle,draw,opacity=.5,fill opacity=1]

      \rootposet{ 0,0}{3}{{0,0,0},{0,0},{0}}{0.2}{(e)}{};

      \rootposet{-5,3}{3}{{2,0,0},{0,0},{0}}{0.2}{(s)}{};
      \rootposet{ 1,3}{3}{{1,0,0},{2,0},{0}}{0.2}{(st)}{};
      \rootposet{-1,3}{3}{{0,0,2},{0,0},{0}}{0.2}{(u)}{};
      \rootposet{-3,3}{3}{{0,2,0},{0,0},{0}}{0.2}{(t)}{};
      \rootposet{ 3,3}{3}{{0,1,0},{0,2},{0}}{0.2}{(tu)}{};
      \rootposet{ 5,3}{3}{{1,0,0},{1,0},{2}}{0.2}{(stu)}{};

      \rootposet{-5,6}{3}{{2,0,2},{0,0},{0}}{0.2}{(su)}{};
      \rootposet{-3,6}{3}{{2,2,0},{3,0},{0}}{0.2}{(sts)}{};
      \rootposet{ 1,6}{3}{{1,0,2},{2,0},{3}}{0.2}{(stut)}{};
      \rootposet{-1,6}{3}{{0,2,2},{0,3},{0}}{0.2}{(tut)}{};
      \rootposet{ 3,6}{3}{{2,1,0},{1,2},{3}}{0.2}{(stust)}{};
      \rootposet{ 5,6}{3}{{1,2,0},{1,0},{2}}{0.2}{(stus)}{};

      \rootposet{ 0,9}{3}{{2,2,2},{3,3},{3}}{0.2}{(stusts)}{};

      \draw [black] (stusts) to (stust);
      \draw [black] (stust) to (stu);
      \draw [black] (stus) to (t);
      \draw [black] (stu) to (e);
      \draw [black] (stut) to (u);
      \draw [black] (stusts) to (su);
      \draw [black] (tu) to (e);
      \draw [black] (stut) to (stu);
      \draw [black] (tut) to (tu);
      \draw [black] (stusts) to (sts);
      \draw [black] (stust) to (tu);
      \draw [black] (su) to (s);
      \draw [black] (sts) to (s);
      \draw [black] (st) to (e);
      \draw [black] (tut) to (t);
      \draw [black] (u) to (e);
      \draw [black] (stusts) to (tut);
      \draw [black] (s) to (e);
      \draw [black] (sts) to (t);
      \draw [black] (sts) to (st);
      \draw [black] (stus) to (stu);
      \draw [black] (tut) to (u);
      \draw [black] (stusts) to (stus);
      \draw [black] (su) to (u);
      \draw [black] (stust) to (s);
      \draw [black] (stut) to (st);
      \draw [black] (stusts) to (stut);
      \draw [black] (t) to (e);
    \end{tikzpicture}
    }
  \end{center}
  \caption{\label{fig:ncA3shard}$\Shard(\WA[4])$ restricted to $\Sort(\WA[4],s_1s_2s_3)$.  Sortable elements are specified by their inversion sets as described in~\Cref{ex:root_posets}.}
\end{figure}

Combining \Cref{thm:shard_nc_lattice_isomorphism} with \Cref{def:nc_multichains} suggests a definition of~$m$-eralized sortable elements as~$m$-multi\-chains of sortable elements in $\Shard(W)$.

\nathanside{But that's not organic.}
\begin{definition}
\label{def:sort_multichains}
  The \defn{shard $c$-sortable elements} are the~$m$-multichains
  \[
    \deltaSortm(W,c) \eqdef \bigset{ (w_1 \geqsh w_2 \geqsh \cdots \geqsh w_m) }{ w_i \in \Sort(W,c) }.
  \]
\end{definition}

By construction, the shard $c$-sortable elements are in bijection with the $m$-eralized $c$-noncrossing partitions.

\begin{theorem+}
\label{thm:bij_sort_multichains_to_nc}
  There is a natural bijection%\nathan{support-preserving?}
  \[
    \deltaSortm(W,c) \cambbij \NCm(W,c). \qedhere
  \]
\end{theorem+}

\medskip

%in~\eqref{eq:ncl_meet_semilattice}
The main theorem of this section draws an analogy between the $m$-eralized Cambrian lattices and D.~Amrstrong's $m$-eralization of the noncrossing partition lattice from~\Cref{sec:ncm_chains_deltas_subwords}---just as multichains of noncrossing partitions should be ordered by componentwise absolute order, multichains of sortable elements should be ordered by componentwise weak order.

\begin{theorem}
\label{thm:bij_multichain_sort_to_sort}
  There is a natural bijection
  \[
    \deltaSortm(W,c)\cambbij\Sortma(W,c)
  \]
  that sends $\deltaSortm(W,c)$ under componentwise weak order to $\Cambsortm(W,c)$.
\end{theorem}

\begin{remark}
  A naive guess for a bijection from $\deltaSortm(W,c)$ to $\Sortm(W,c)$ would be to simply multiply the individual factors of $\deltaSortm(W,c)$ (as an element of~$\Artinmon$).
  This naive guess is wrong.
  For example, for $c=s_1s_2s_3 \in \WA[4]$, this map would send the chain of sortable elements in shard order
  \[
    \left(s_1 s_2 s_3 s_1 \geqsh s_1s_2s_3\right) \text{ to the element } s_1s_2s_3|s_1 \cdot \cdot |s_1 s_2 s_3 \in \BA[4].
  \]
  This element is evidently not $c$-sortable.
\end{remark}

\begin{proof}[Proof of \Cref{thm:bij_multichain_sort_to_sort}]
  Given a multichain $(w_1\geqsh w_2 \geqsh \cdots \geqsh w_m) \in \deltaSortm(W,c)$, we produce the Garside factorization of an element in $\Sortm(W,c)$ as follows.

  \Cref{prop:shard_interval_isomorphism} gives a bijection $[\one,w]_{\Shard(W)}\cong \Shard(W_{\DesSet_R(w)})$.
  If $u \leqsh w$ with $w,u \in \Sort(W,c)$,~$u$ is sent to a $\restr{c}{w}$-sortable element in $W_{\DesSet_R(w)}$ under this bijection using \Cref{thm:shard_nc_lattice_isomorphism}.  

  We iterate this procedure to obtain elements $w^{(1)},w^{(2)},\ldots,w^{(m)}$ such that $w^{(1)}\eqdef w_1 \in \Sort(W,c)$, and for all $1 < i \leq m$, $w^{(i)} \in \Sort(W_{\DesSet_R(w_{(i-1)})},c^{(i)})$, where $c^{(i)}\eqdef\restr{c^{(i-1)}}{w^{(i-1)}}$.
  Then $w^{(1)} \cdot w^{(2)} \cdot\ \cdots\ \cdot w^{(m)}$ satisfies the condition in \Cref{thm:garsidefactorization} to be a Garside factorization, it satisfies the factorwise conditions of \Cref{def:factorsort}, and so it is an element of $\Sortm(W,c)=\Sortma(W,c)$ by \Cref{cor:factorsort}.

  The inverse of this bijection is given by explicitly describing the inverse of \Cref{prop:shard_interval_isomorphism} on $c$-sortable elements: given $u \in W_{\DesSet_R(w)}$, we conjugate it by~$w$ to an element of $W_{\coveredref(w)}$.  Since $c$-sortable elements are uniquely defined by their cover reflections by~\eqref{eq:sort_to_nc_by_covers}, there is a unique way to complete the inversion set in $W_{\coveredref(w)}$ to the inversion set of a $c$-sortable element in~$W$.

  \Cref{thm:sort_componentwise_inversion} implies that the componentwise weak order on $\deltaSortm(W,c)$ recovers $\Cambsortm(W,c)$.
\end{proof}

\begin{theorem}
\label{thm:bij_nc_to_sort}
   Starting at any corner of the square below and traveling around the square via the given bijections results in the identity map.
   \[\xymatrix{ 
\NCm(W,c) \ar@{->}[d]|{\text{\Cref{thm:bij_sort_multichains_to_nc}}} \ar@{<-}[rrr]|{\text{\Cref{prop:nc_fuss_subwords}}}    & & &  \DeltaNCm(W,c) \ar@{<-}[d]|{\text{\Cref{thm:bij_sort_to_nc}}} \\
 \deltaSortm(W,c) \ar@{->}[rrr]|{\text{\Cref{thm:bij_multichain_sort_to_sort}}}    & & &  \Sortma(W,c)
  }\]
\end{theorem}

\begin{proof}
  It suffices to check that the composition of the bijections described in \Cref{prop:nc_delta}, \Cref{thm:bij_sort_multichains_to_nc}, and \Cref{thm:bij_multichain_sort_to_sort} is the inverse of the bijection $\Skipset$ given in \Cref{thm:bij_sort_to_nc}.

  Choose $(u_1 \geq_\refl u_2 \geq_\refl \cdots \geq_\refl u_m) \in \NCm(W,c)$, write the corresponding chain of sortable elements as $(w_1 \geqsh w_2 \geqsh \cdots \geqsh w_m) \in \deltaSortm(W,c)$, and denote the corresponding factorwise $c$-sortable element by $w = w^{(1)}w^{(2)}\cdots w^{(m)}\in \Sortma(W,c)$.  Let $\Skipset \left(w \right) = (\delta_0,\delta_1,\delta_2,\ldots,\delta_m) \in \deltaNCm(W,c)$.

  We must show that $(\delta_0,\delta_1,\delta_2,\ldots,\delta_m)  = (c u_1^{-1}, u_1 u_2 ^{-1},u_2 u_3^{-1}, \cdots, u_m)$.  We argue by induction on $m$, the base case following from the $m=1$ theory.

  First, any skips of color~$0$ that could be attributed to the $w^{(2)}\cdots w^{(m)}$ piece of $\Skipset(w)$ are already skips of color~$0$ from the $w^{(1)}$ piece of the product, since the support of $w^{(2)}\cdots w^{(m)}$ is contained in the support of $w^{(1)}$.
  By \Cref{prop:sortingwordgarsidefactorization}, these skips of color~$0$ are not affected by the addition of the piece $w^{(2)}\cdots w^{(m)}$, and therefore account for the first piece of the delta sequence, $c u_1^{-1}$ by the $m=1$ bijection between sortable elements and delta sequences.

  Second, since the covering reflections of $w^{(1)}$ correspond to a parabolic Coxeter element, conjugating by $w^{(1)}$ to map $W_{\DesSet_R(w^{(1)})}$ to $W_{\coveredref(w^{(1)})}$, we have by induction that the skip set of $w^{(2)}\cdots w^{(m)}$ is indeed $(u_1 u_2 ^{-1},u_2 u_3^{-1}, \cdots, u_m)$.

  These two pieces together conclude the proof.
\end{proof}

\begin{example}
  We compute an extended example of~\Cref{thm:bij_nc_to_sort}, starting and ending at  $\NCm(\WA[4],c)$.
  Fix $m=2$ and $c=s_1s_2s_3 \in \WA[4]$.  The element $(23)^{(0)}(34)^{(1)}(14)^{(2)} \in \DeltaNCm(\WA[4],c)$ corresponds under the map in \Cref{prop:nc_delta} to the chain
  \[
    (134) \geq_\refl (14) \in \NCm(\WA[4],c).
  \]
By \Cref{thm:bij_sort_multichains_to_nc}, this chain corresponds to the shard sortable element
  \[
    \left(w_1 \geqsh w_2\right) = \left(s_1s_2s_3s_2\geqsh s_1s_2s_3\right) \in \deltaSortm(\WA[4],c).
  \]
  We use \Cref{thm:bij_multichain_sort_to_sort} to find the Garside factorization of an element of $\Sortm(W,c)$ from this shard sortable element.  We compute
  \[
    \inv_\refl(w_1)=\big\{(12),(13),(14),(34)\big\} \supseteq \big\{(12),(13),(14)\big\}=\inv_\refl(w_2).
  \]
  The covered reflections $\coveredref(w_1)=\{(13),(34)\}$ generate the (nonstandard) parabolic subgroup $W_{ \coveredref(w_1)}$ with reflections $\{(13),(14),(34)\}$.  The associated descents $\DesSet_R(w_1)=\{s_3,s_2\}$ generate the standard parabolic subgroup $W_{\DesSet_R(w_1)}$.

  Under the first isomorphism $[e,w_1]_{\Shard(W)}\cong \Shard(W_{\coveredref(w_1)})$ of \Cref{prop:shard_interval_isomorphism},  the element $(w_2)_{\coveredref(w_1)} \in W_{\coveredref(w_1)}$ has inversions $\{(13),(14)\}$.
  Passing to the standard parabolic $\Shard(W_{\DesSet_R(w_1)})$ by the second isomorphism of \Cref{prop:shard_interval_isomorphism}, the corresponding element of $W_{\DesSet_R(w_1)}$ has inversions $\{(34),(24)\}$ and reduced $\sref$-word $s_3s_2$.  We conclude that
  \[
    \left(s_1s_2s_3s_2\geqsh s_1s_2s_3 \right) \mapsto s_1s_2s_3s_2\cdot s_3s_2 \in \Sortma[2](\WA[4]).
  \]
 This element is factorwise $c$-sortable by \Cref{ex:crestricted} and its skip set was computed in \Cref{ex:skipset}, recovering the initial colored factorization $(23)^{(0)}(34)^{(1)}(14)^{(2)} \in \DeltaNCm(\WA[4],c)$.  
%\[
%   (23)^{(0)}(34)^{(1)}(14)^{(2)} \in \DeltaNCm(\WA[4],c)
%  \]
%  maps to
\end{example}

%%%%%%%%%%%%%%%%%%%%%%%%%%%%%%%%%%%%%%%%%%%%%%%%%%%%%%%%
\section{The linear \mhead-eralized Cambrian and the \mhead-Tamari lattice}
\label{sec:tamari}
%%%%%%%%%%%%%%%%%%%%%%%%%%%%%%%%%%%%%%%%%%%%%%%%%%%%%%%%

For this section, we fix the linear Coxeter element
\[
  c = (1,2,\ldots,n) = s_1\cdots s_{n-1} \in \WA[n]
\]
and the corresponding $m$-eralized $c$-Cambrian lattice $\Cambsortm(\WA[n],c)$.
We conclude this chapter with a brief comparison of $\Cambsortm(\WA[n],c)$ with the $m$-Tamari lattice~\cite{BPR2012}.

An \defn{$m$-Dyck path} is a walk by east and north steps from $(0,0)$ to $(mn,n)$, never going below the line betwen them.
The \defn{$m$-Tamari lattice} on Dyck paths is defined by cover relations:
whenever there is an east step followed immediately by a north step, a line is drawn with slope $1/m$ from their common endpoint upwards until it first return to the Dyck path.
The cover relation consists of swapping the east step with the subpath cut out by the line segment.

% We refer to~\cite{Ber2012} for further discussion of the $m$-Tamari lattice (including the fact that it is a lattice).  
The $1$-Tamari lattice is isomorphic to the 1-Cambrian lattice with linear Coxeter element, but this is no longer the case for $m>1$ and $n>2$.
\medskip

The $m$-Tamari lattice is illustrated for $n=3$ and $m=2$ in \Cref{fig:tamari2}, while the $m$-eralized Cambrian lattice was shown in \Cref{fig:cambsortA22} on page~\pageref{fig:cambsortA22}.
\begin{figure}[t]
  \begin{center}
    \begin{tikzpicture}[scale=1.3]
      \tikzstyle{rect}=[rectangle,draw,opacity=.5,fill opacity=1,minimum width=50pt,minimum height=28pt]
      \tikzstyle{sort}=[]
        \node[rect] (e)   at (0,0) {};
        \dyckpath{(e)}{0,1,1,0,1,1,0,1,1};
        \node[rect] (s)   at (-1.5,1) {};
        \dyckpath{(s)}{0,1,0,1,1,1,0,1,1};
        \node[rect] (t)   at ( 1,1) {};
        \dyckpath{(t)}{0,1,1,0,1,0,1,1,1};
        \node[rect] (st)  at (-1.5,2) {};
        \dyckpath{(st)}{0,1,0,1,1,0,1,1,1};
        \node[rect] (sts) at ( 0,3) {};
        \dyckpath{(sts)}{0,1,0,1,0,1,1,1,1};
        \node[rect] (stss)   at ( 1,4) {};
        \dyckpath{(stss)}{0,1,0,0,1,1,1,1,1};
        \node[rect] (stst)   at (-2,4) {};
        \dyckpath{(stst)}{0,0,1,1,0,1,1,1,1};
        \node[rect] (ststs)  at (-1,5) {};
        \dyckpath{(ststs)}{0,0,1,0,1,1,1,1,1};
        \node[rect] (stssts) at ( 0,6) {};
        \dyckpath{(stssts)}{0,0,0,1,1,1,1,1,1};
        \node[rect] (ss)   at (-3,2) {};
        \dyckpath{(ss)}{0,0,1,1,1,1,0,1,1};
        \node[rect] (tt)   at ( 2,2) {};
        \dyckpath{(tt)}{0,1,1,0,0,1,1,1,1};
        \node[rect] (stt)  at (-3,3) {};
        \dyckpath{(stt)}{0,0,1,1,1,0,1,1,1};

        \draw (e) to (s) to (st) to (sts) to (stss) to (stssts);
        \draw (e) to (t) to (sts) to (ststs) to (stssts);
        \draw (s) to (ss) to (stt);
        \draw (t) to (tt) to (stss);
        \draw (st) to (stt) to (stst) to (ststs);
    \end{tikzpicture}
  \end{center}
  \caption{The $2$-Tamari lattice on the $2$-Dyck paths.}
  \label{fig:tamari2}
\end{figure}
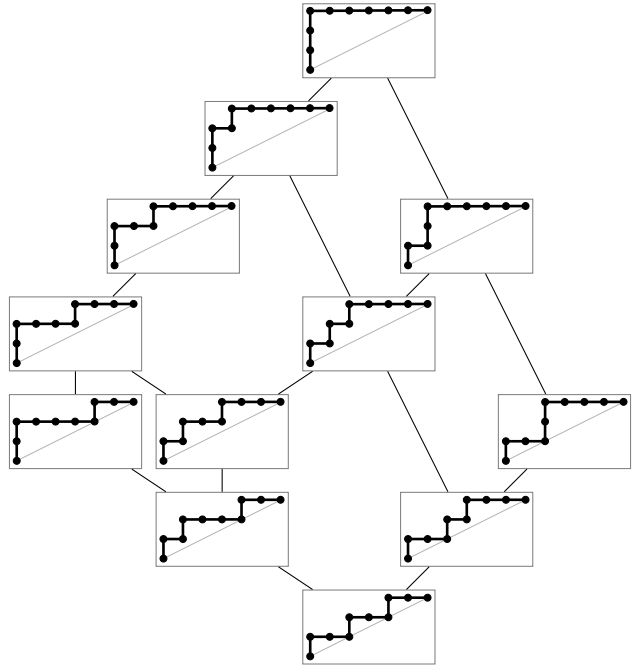
While these two lattices are generally not isomorphic, there is nevertheless the following surprising connection between the two lattices.

\begin{theorem}
  We have
  \[
  \sum_I q^{\#\text{upper covers of }I} = \sum_D q^{\#\text{upper covers of }D},
  \]
  where the first sum ranges over all sortable elements using the cover relation in $\Cambsortm\big(\WA[n],(12\cdots n)\big)$, and the second sum ranges over all $m$-Dyck paths using the cover relation in the $m$-Tamari lattice of Dyck paths from $(0,0)$ to $(mn,n)$.
\end{theorem}

\begin{proof}
  The number of upper covers of an $m$-Dyck path in $m$-Tamari lattice is given by its number of valleys.
  The generating function of all $m$-Dyck paths by their number of valleys is well-known to be the $m$-Narayana polynomial.
  The $m$-Narayana polynomial is also equal to the $h$-polynomial of $\Assocm\big(\WA[n],(12\cdots n)\big)$, see~\cite{FR2005}.
  And we have seen in \Cref{cor:m-c-h-vector}, that this equals the generating function of the number of upper covers in $\Cambassocm\big(\WA[n],(12\cdots n)\big)$.
  (Alternatively, one can deduce this second Narayana count also from \Cref{cor:m-c-h-vector2} and~\cite[Section~3.5]{Arm2006}.)
\end{proof}

We conjecture that there is an even deeper relationships between the two lattices.
M.~Bousquet-M\'{e}lou, E.~Fusy, and L.-F.~Pr\'{e}ville-Ratelle~\cite{BMFPR2012} proved that the number of intervals in the $m$-Tamari lattice is
\[
  \frac{m+1}{n(mn+1)} \binom{(m+1)^2 n +m}{n-1}.
\]
In a subsequent paper, M.~Bousquet-M\'{e}lou, G.~Chapuy, and L.-F.~Pr\'{e}ville-Ratelle proved that the number of labelled intervals in the $m$-Tamari lattice---where the top element of the interval is labelled by a parking function compatible with the corresponding Dyck path---equals $(m+1)^n (mn+1)^{n-2}$~\cite{BMCPR2013}.
% The construction in \Cref{prop:other_m_parking_construction} allows us to label $(w_1 \geqref w_2 \geqref \cdots \geqref w_m) \in \NCm(W,c)$ by a compatible noncrossing parking functions $\bigset{ (U,w_m) }{ U \in W^{\fixd({w_m})} }$.\nathan{parking}
Labelling $(w_1 \geqref w_2 \geqref \cdots \geqref w_m) \in \NCm(W,c)$ by cosets of the parabolic subgroup $W_{\fixd(w_m)}$~\cite{Rho2014}, 
%Using the bijection between $\NCm(W,c)$ and $\Sortm(W,c)$ from \Cref{thm:bij_sort_to_nc} transfers this coset to sortable elements which are the elements in $\Cambsortm(W,c)$.
%\christian{where in the below do we use this additional going to the m-cambrian on sortables? shouldn't we just stay on the m-cambrian on nc, or am I missing a point?}
% 
% \medskip
% 
we conjecture the following relationship between the number of intervals and labelled intervals.

\begin{conjecture}\label{conj:intervals_m_tamari}
  Let $c=(1,2,\ldots,n) \in \WA[n]$ be the linear Coxeter element.
  Then
  \begin{itemize}
    \item the number of intervals in the $m$-Tamari lattice equals the number of intervals in the $m$-eralized Cambrian lattice, and
    \item the number of intervals in the lattices also coincide if each interval is weighted by the number of parking functions or, respectively, by the number of cosets of the top element.
  \end{itemize}
\end{conjecture}\nathanside{Noncrossing $q,t$-Catalan combinatorics?}

We conjecture the following weighted version of the interval enumeration. 

\begin{conjecture}\label{conj:intervals_m_tamari2}
For any integer partition $\lambda \vdash n$, there are the same number of intervals $\alpha \leq \beta$
  \begin{itemize}
    \item in the $m$-eralized Cambrian lattice for which $\lambda$ records the sorted block sizes of the last component~$\delta_m$ of the $m$-delta sequence $\beta = (\delta_0,\delta_1,\ldots,\delta_{m})$, as
    \item in the $m$-Tamari lattice for which $\lambda$ records the sorted sizes of the vertical runs of the $m$-Dyck path $\beta$.
  \end{itemize}
\end{conjecture}
\christianside{If this is true, there must be some structural connection between the two!}
\nathanside{Hi Fr\'ed\'eric!}

Both conjectures have been verified for $\WA$ for $n=3,4$ and $m=1,2,3,4$.

%%%%%%%%%%%%%%%%%%%%%%%%%%%%%%%%%%%%%%%%%%%%%%%%%%%%%%%%%%%%%%%%%%%%%%%%%%%%%%%%%%%%%
\chapter{Positive \mhead-eralized structures}
\label{sec:rat_positive}
%%%%%%%%%%%%%%%%%%%%%%%%%%%%%%%%%%%%%%%%%%%%%%%%%%%%%%%%%%%%%%%%%%%%%%%%%%%%%%%%%%%%%

In this chapter, we study \emph{positive $m$-eralized analogues} of the structures we have previously considered.
We define and discuss their embeddings (\Cref{sec:positiveembedding}), study their enumeration (\Cref{sec:positivecount}) and their symmetries under Kreweras complements and Cambrian rotations (\Cref{sec:positivesymmetry}).\nathanside{Home stretch---two easy chapters, then representation theory!}

%%%%%%%%%%%%%%%%%%%%%%%%%%%%%%%%%%%%%%%%%%%%%%%%%%%%%%%%%%%%%%%%%%%%%%%%%%%%%%%%%%%%%
\section{The positive \mhead-eralized structures and their embeddings}
\label{sec:positiveembedding}
%%%%%%%%%%%%%%%%%%%%%%%%%%%%%%%%%%%%%%%%%%%%%%%%%%%%%%%%%%%%%%%%%%%%%%%%%%%%%%%%%%%%%

Recall that the support $\supp(w)$ of $w \in \Artinmon$ is the set $\{ s_1,\ldots,s_p\}$ of simple reflections contained in any reduced $\sref$-word $\s_1\cdots\s_p$ for~$w$.
The natural bijections
\[
  \Sortm(W,c) \cambbij \Assocm(W,c) \cambbij \NCm(W,c) 
\]
preserve the various notion of support given in \Cref{def:nc_multichains,,def:m-assoc-support}.

\medskip

The uniform underlying idea is to restrict attention from all Fu\ss-Catalan objects to those with full support.
This is a well-worn path for $m=1$, and has also been developed for $m$-eralized $c$-cluster complexes with bipartite~$c$ in~\cite[Section~12]{FR2005}.

We define the following positive analogues together with natural embeddings:
\begin{align*}
  \Sortmpos(W,c)  &\hookrightarrow \Sortm(W,c) \\
  \NCmpos(W,c)    &\hookrightarrow \NCm(W,c) \\
  \Assocmpos(W,c) &\hookrightarrow \Assocm(W,c)\ .
\end{align*}

An element of $\NCm(W,c) \cambbij \Assocm(W,c) \cambbij \Sortm(W,c)$ is said to have \defn{full support} if its support is all of $\sref$.  Restricting to those elements of full support gives a quick, uniform definition for positive versions of each construction.   To better describe the symmetries of these positive structures, we prefer to give slightly different interpretations.

\medskip

An element $w \in \Sortm(W,c)$ has full support if and only if the sorting word $\sw{w}{c}$ starts with an initial copy of~$\c$.  We define the \defn{positive $m$-eralized $c$-sortable elements} to be
\[
  \bigset{ w \in [\bone,\overline{\bc}\bwom]_{\Weak(\Artinmon)} }{ w \text{ is $c$-sortable} }.
\]
We therefore recover the sortable elements of full support using the embedding
\begin{align*}
  \Sortmpos(W,c) &\hookrightarrow \Sortm(W,c) \\
	\bw &\mapsto \bc\bw.
\end{align*}
\christianside{Take care, this is \emph{not} the embedding as a subset.}
Since multiplying an element less than $\overline{\bc}\bwom$ by $\bc$ evidently does not increase its Garside degree, this embedding has the correct image by \Cref{prop:weak_degree_characterization}.
Some caution is warranted---it is already the case that $\Sortmpos(W,c) \subseteq \Sortm(W,c)$, but this is not the intended embedding.  \Cref{fig:sortposA22} illustrates the example in~$\WA[3]$ with $m=2$.

\begin{figure}[t]
  \begin{center}
    \begin{tabular}{ccc}
      $\Sortmpos[2](\WA[3],st)$ & $\hookrightarrow$ & $\Sortm[2](\WA[3],st)$ \\
      \hline
      $\cdot\cdot|\s\t|\s\t$ & $\mapsto$ &
      $\s\t|\s\t|\s\t$
      \\
      $\cdot\cdot|\s\T|\S\T$ & $\mapsto$ &
      $\s\t|\s\T|\S\T$
      \\
      $\cdot\cdot|\S\t|\S\T$ & $\mapsto$ &
      $\s\t|\S\t|\S\T$
      \\
      $\cdot\cdot|\s\t|\S\T$ & $\mapsto$ &
      $\s\t|\s\t|\S\T$
      \\
      $\cdot\cdot|\s\T|\s\T$ & $\mapsto$ &
      $\s\t|\s\T|\s\T$
      \\
      $\cdot\cdot|\s\t|\s\T$ & $\mapsto$ &
      $\s\t|\s\t|\s\T$
      \\
      $\cdot\cdot|\S\T|\S\T$ & $\mapsto$ &
      $\s\t|\S\T|\S\T$
    \end{tabular}
  \end{center}
  \caption{The positive $m$-eralized $st$-sortable elements $\Sortmpos[2](\WA[3],st)$, with their embedding into $\Sortm[2](\WA[3],st)$.}
  \label{fig:sortposA22}
\end{figure}

\medskip

A facet~$I$ of the dual subword complex $\DeltaNCm(W,c)$ has full support if and only if it does not contain any letter from the first copy of~$c$ of the initial copy of~$\wo$ of the search word $\cwom[c][m+1]$.  That is, $I$ has full support if and only if
  \[
    I \cap \set{ \r^{(0)} }{ \r \in \invs_\refl(c)} = \emptyset.
  \]
We define the \defn{positive $m$-eralized $c$-noncrossing partitions} to be
\[
  \NCmpos(W,c) \eqdef \subwordsR(\overline{\c}\cwom[c][m+1], c).
\]

On the level of the colored reflection sequences, the inclusion $\overline{c}\cwom[c][m+1] \hookrightarrow \cwom[c][m+1]$ conjugates every reflection by~$c$.  Therefore, reinserting the initial missing copy of $\c$ defines an embedding
\begin{align*}
  \NCmpos(W,c) &\hookrightarrow \NCm(W,c)\\
 I &\mapsto \set{i+n}{i \in I}.
\end{align*}
  Since the product of the reflections in a facet of $\NCmpos(W,c)$ is an $\refl$-word for~$c$, this conjugation gives another $\refl$-word for~$c$.  The image is therefore exactly those facets with full support.  \Cref{fig:sortposA22} illustrates the example in~$\WA[3]$ with $m=2$.

\begin{figure}[t]
  \begin{center}
    \begin{tabular}{ccc}
      $\NCmpos[2](\WA[3],st)$ & $\hookrightarrow$ & $\NCm[2](\WA[3],st)$ \\
      \hline
      $\cdot\cdot\T.\S\U\T.\s\U\t$ & $\mapsto$ &
      $\S\U\T.\S\U\T.\s\U\t$
      \\
      $\cdot\cdot\T.\s\U\t.\S\U\T$ & $\mapsto$ &
      $\S\U\T.\s\U\t.\S\U\T$
      \\
      $\cdot\cdot\t.\S\U\T.\S\u\T$ & $\mapsto$ &
      $\S\U\t.\S\U\T.\S\u\T$
      \\
      $\cdot\cdot\T.\S\u\T.\s\U\T$ & $\mapsto$ &
      $\S\U\T.\S\u\T.\s\U\T$
      \\
      $\cdot\cdot\T.\s\U\T.\S\U\t$ & $\mapsto$ &
      $\S\U\T.\s\U\T.\S\U\t$
      \\
      $\cdot\cdot\T.\S\U\t.\S\u\T$ & $\mapsto$ &
      $\S\U\T.\S\U\t.\S\u\T$
      \\
      $\cdot\cdot\t.\S\u\T.\S\U\T$ & $\mapsto$ &
      $\S\U\t.\S\u\T.\S\U\T$
    \end{tabular}
  \end{center}
  \caption{The positive $m$-eralized $st$-noncrossing partitions $\NCmpos[2](\WA[3],st)$.}%  They form a single orbit under the positive Kreweras complement.}
  \label{fig:ncposA22}
\end{figure}

%\begin{remark}
%\label{rem:simoes}
%  Analogously to the representation-theoretic description the $m$-eralized nocrossing partitions in~\cite{IT2009}, R.~Coelho Sim{\~o}es has given a similar interpretation of $\NCmpos(W,c)$ in~\cite{Sim2012}.
%\end{remark}

As in the case of $\DeltaNCm(W,c)$, a facet~$I$ of the subword complex $\DeltaAssocm(W,c)$ has full support if and only if it does not contain any letter from the first copy of~$c$ of the initial copy of~$\wo$ in the search word $\c\cwom$.  
We define the \defn{positive $m$-eralized $c$-cluster complex} to be
\begin{align*}
  \Assocmpos(W,c) \eqdef \subwordsS(\cwom, \overline{c}\wom, mN-n) = \subwordsSB(\cwom, \overline{\bc}\bwom).
\end{align*}
The inclusion $\cwom \hookrightarrow \c\cwom$ gives the embedding
\begin{equation}
\label{eq:assopos}
  \begin{aligned}
    \Assocmpos(W,c) &\hookrightarrow \DeltaAssocm(W,c) \\
      I &\mapsto \set{i+n}{i \in I}.
  \end{aligned}
\end{equation}
By \Cref{def:m-assoc-support}, the image of this embedding recovers those elements of full support.  \Cref{fig:assocposA22} illustrates the example in~$\WA[3]$ with $m=2$.
\begin{figure}[t]
  \begin{center}
    \begin{tabular}{ccc}
      $\Assocmpos[2](\WA[3],st)$ & $\hookrightarrow$ & $\DeltaAssocm[2](\WA[3],st)$ \\
      \hline
      $\cdot\cdot.\S\T\S.\T\s\t$ & $\mapsto$ &
      $\S\T.\S\T\S.\T\s\t$
      \\
      $\cdot\cdot.\S\t\s.\T\S\T$ & $\mapsto$ &
      $\S\T.\S\t\s.\T\S\T$
      \\
      $\cdot\cdot.\s\T\S.\T\s\T$ & $\mapsto$ &
      $\S\T.\s\T\S.\T\s\T$
      \\
      $\cdot\cdot.\S\T\s.\t\S\T$ & $\mapsto$ &
      $\S\T.\S\T\s.\t\S\T$
      \\
      $\cdot\cdot.\S\t\S.\T\S\t$ & $\mapsto$ &
      $\S\T.\S\t\S.\T\S\t$
      \\
      $\cdot\cdot.\S\T\S.\t\s\T$ & $\mapsto$ &
      $\S\T.\S\T\S.\t\s\T$
      \\
      $\cdot\cdot.\s\t\S.\T\S\T$ & $\mapsto$ &
      $\S\T.\s\t\S.\T\S\T$
      \\
    \end{tabular}
  \end{center}
  \caption{The positive $m$-eralized $st$-clusters $\Assocmpos[2](\WA[3],st)$.}
  \label{fig:assocposA22}
\end{figure}

%%%%%%%%%%%%%%%%%%%%%%%%%%%%%%%%%%%%%%%%%%%%%%%%%%%%%%%%%%%%%%%%%%%%%%%%%%%%%%%%%%%%%
\section{Enumeration of positive \mhead-eralized structures}
\label{sec:positivecount}
%%%%%%%%%%%%%%%%%%%%%%%%%%%%%%%%%%%%%%%%%%%%%%%%%%%%%%%%%%%%%%%%%%%%%%%%%%%%%%%%%%%%%

Before discussing some structural properties of the positive constructions, we recall that they are uniformly enumerated by the \defn{positive Fu\ss-Catalan} (or \defn{Fu\ss-Dogolon}\footnote{This nomenclature is due to D.~Bessis or comes from the fact that these numbers \emph{lay doggo} without being discovered for far longer than the classical Catalan numbers.}) \defn{numbers of type~$W$}, defined by\nathanside{I'm almost positive this nomenclature has my full support.}
\begin{equation}
  \Catplusm(W) \eqdef \prod_{i=1}^n \frac{mh+d_i-2}{d_i}.
  \label{eq:pos_cat_num}
\end{equation}
We refer to~\cite[Corollary~12.4]{FR2005} for a proof of this is the counting formula for $\Assocmpos(W,c)$.
Formally extending \eqref{eq:fuss_cat_num} to accommodate negative numbers, we obtain $\Catplusm(W) = (-1)^n \Catm[-m-1](W)$.
This purely enumerative identity may be seen as an instance of combinatorial \emph{reciprocity}~\cite[Corollary~1.3]{Ath2005}.

\begin{theorem}
  We have
  \[
    \big| \Sortmpos(W,c) \big| = \big| \NCmpos(W,c) \big| = \big| \Assocmpos(W,c) \big| = \Catplusm(W).
  \]
\end{theorem}

\begin{proof}
  By the embedding $\Assocmpos(W,c) \hookrightarrow \DeltaAssocm(W,c) \bij \Assocm(W,c)$, \Cref{cor:m-c-assoc-m-assoc} gives a bijection between $\Assocmpos(W,c)$ and facets in the Fomin-Reading generalized cluster complex that do not use $m$-colored simple roots.  The enumeration for $\Assocmpos(W,c)$ now follows from \cite[Proposition~12.4]{FR2005}.
  As the bijections
  \[
    \NCm(W,c) \lrbij \Assocm(W,c) \lrbij \Sortm(W,c)
  \]
  respect support, we conclude the result for the remaining objects.
\end{proof}

We obtain the following counting formula analogous to \Cref{thm:m-counting}.

\begin{theorem}
\label{thm:m-pos-counting}
  We have
  \begin{equation*}
    (1-q)^{n+1}\sum_{m=0}^\infty  \big|\NCmpos(W,c)\big| q^m = \sum_{\r_1\cdots \r_n \in \reds(c)} q^{\mathrm{asc}(\r_1\cdots \r_n)}
  \end{equation*}
  where $\asc(\r_1,\ldots,\r_n)$ is the number of ascents $r_i <_c r_{i+1}$ in the reflection order induced by the Coxeter element~$c$.
\end{theorem}
\christianside{Look at the two formulas, aren't they cute.}

\begin{proof}
  This is a direct consequence of~\eqref{eq:assopos}, after observing that this formula is equivalent to
  \[
    \big|\NCmpos(W,c)\big| = \sum_{\r_1\cdots\r_n \in \reds(c)} \binom{m+\des(\r_1\cdots\r_n)}{n}
  \]
  where $\des(\r_1,\ldots,\r_n)$ is the number of ascents $r_i >_c r_{i+1}$.
  The only difference to the proof of \Cref{thm:m-counting} is that we now how only $m$ copies of $\invs_\refl(\cwo)$ in the search word, and we search for~$c^{-1}$ rather than for~$c$.
\end{proof}

%%%%%%%%%%%%%%%%%%%%%%%%%%%%%%%%%%%%%%%%%%%%%%%%%%%%%%%%%%%%%%%%%%%%%%%%%%%%%%%%%%%%%
\section{Symmetries of the positive \mhead-eralized structures}
\label{sec:positivesymmetry}
%%%%%%%%%%%%%%%%%%%%%%%%%%%%%%%%%%%%%%%%%%%%%%%%%%%%%%%%%%%%%%%%%%%%%%%%%%%%%%%%%%%%%

Cambrian rotation realizes a remarkable symmetry on the $m$-eralized $c$-cluster complex $\Assocm(W,c)$ (\Cref{def:cs_c_cambrian_rotation}).
This symmetry is lost when studying the positive $m$-eralized $c$-cluster complex $\Assocmpos(W,c)$---in fact, $\Assocmpos(W,c)$ and $\Assocmpos(W,\coxrn)$ are not necessarily isomorphic simplicial complexes.%, since $\Shift_s$ no longer defines a map from $\Assocmpos(W,c)$ to $\Assocmpos(W,\coxrn)$. 

The situation is reversed for the complexes $\NCm(W,c)$ and $\NCmpos(W,c)$.  Although the dual subword complexes $\NCm(W,c)$ and $\NCm(W,\coxrn)$ are not necessarily isomorphic, it turns out that $\NCmpos(W,c)$ and $\NCmpos(W,\coxrn)$ \emph{are} isomorphic
Analogously to the situation for $\Assocm(W,c)$, composing this isomorphism~$n$ times in the order induced by~$\c$ realizes a symmetry on the complex.  This symmetry is illustrated in \Cref{fig:swflipgraphs}.

\begin{figure}[t]
  \begin{center}
      \begin{tikzpicture}
      \node[inner sep=0pt] (11) at (-3,0)
        {\includegraphics[width=.48\textwidth]{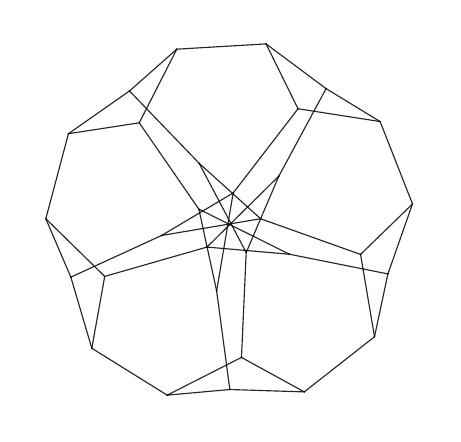}};
      \node[inner sep=0pt] (21) at (-3,-2.7)
        {$\NCmpos[2](\WA[4],c)$ with $30$ facets};
    \end{tikzpicture}
    \quad
    \begin{tikzpicture}
      \node[inner sep=0pt] (12) at ( 3,0)
        {\includegraphics[width=.48\textwidth]{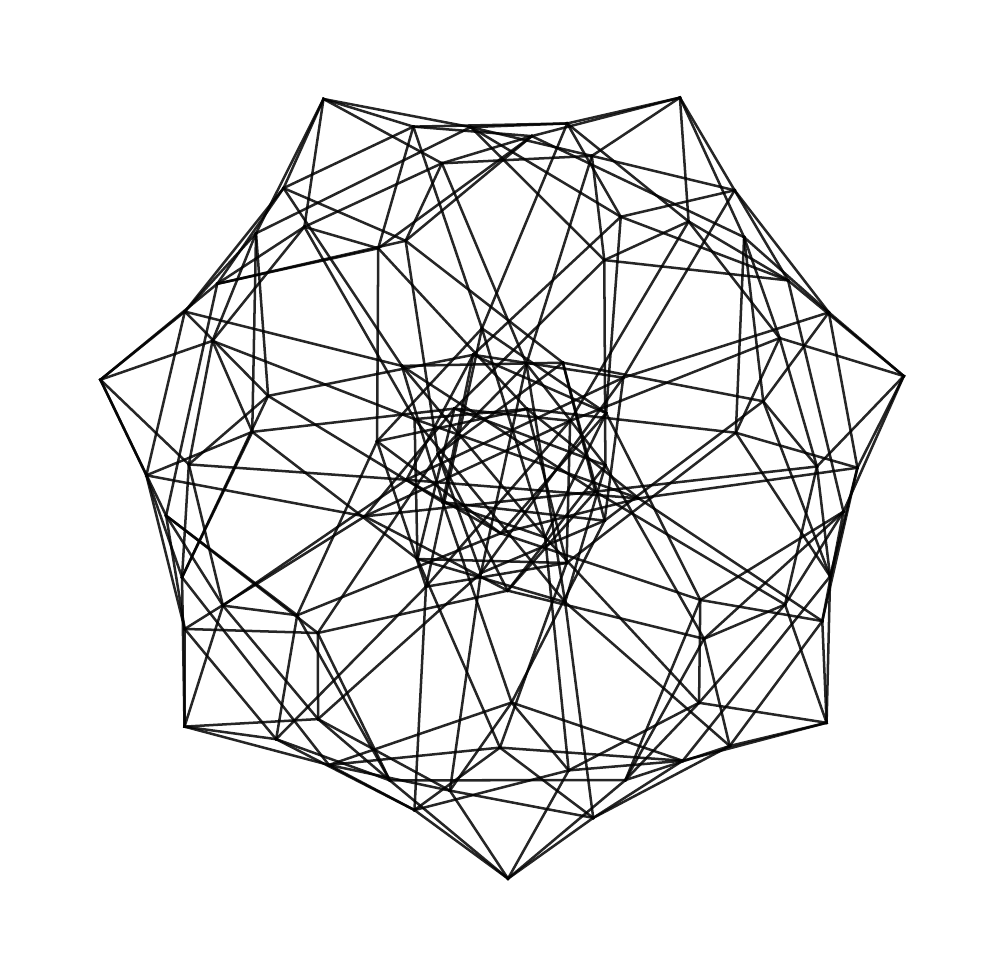}};
      \node[inner sep=0pt] (22) at ( 3,-2.7)
        {$\NCmpos[3](\WA[4],c)$ with $91$ facets};
    \end{tikzpicture}
    \begin{tikzpicture}
      \node[inner sep=0pt] (31) at(-3,-6)
        {\includegraphics[width=.48\textwidth]{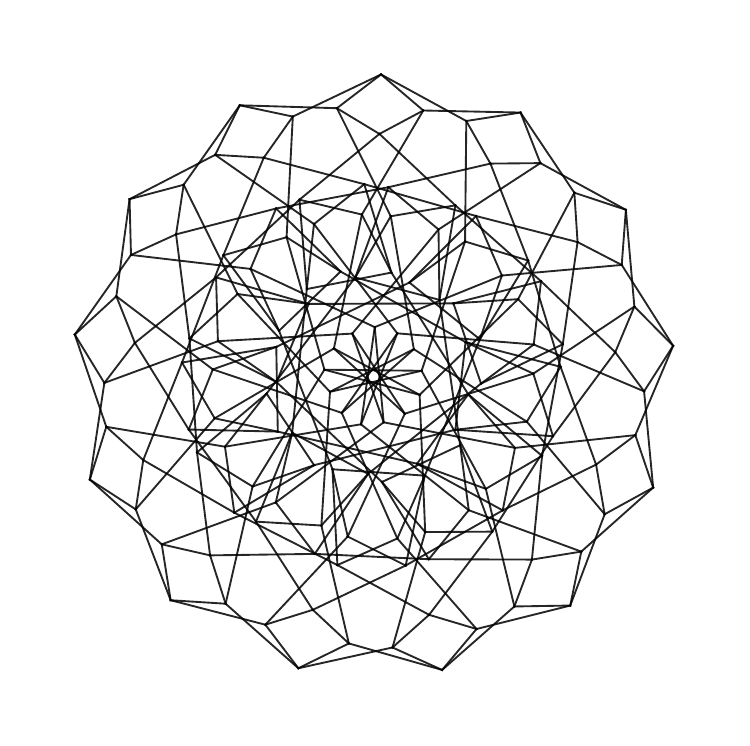}};
      \node[inner sep=0pt] (41) at (-3,-8.7)
        {$\NCmpos[2](\WA[4],c)$ with $143$ facets};
    \end{tikzpicture}
    \quad
    \begin{tikzpicture}
      \node[inner sep=0pt] (32) at ( 3,-6)
        {\includegraphics[width=.48\textwidth]{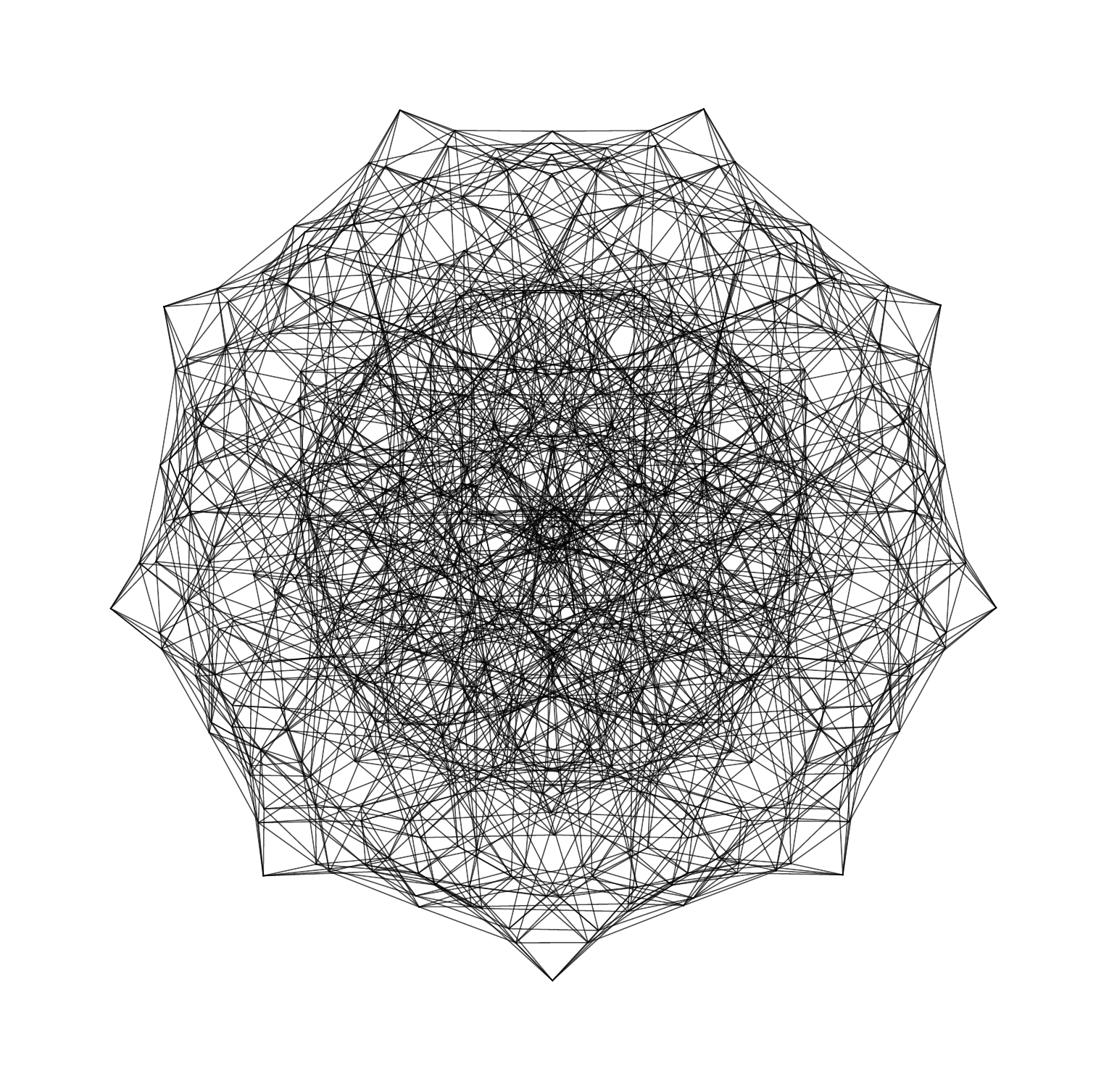}};
      \node[inner sep=0pt] (42) at ( 3,-8.7)
        {$\NCmpos[3](\WA[4],c)$ with $612$ facets};
    \end{tikzpicture}
  \end{center}
\caption{Some examples of the symmetries of $\NCmpos(W,c)$ under the positive Kreweras complement $\Krewpos$.}
%Some positive Fu{\ss}-Catalan associahedra and their symmetries.}
\label{fig:swflipgraphs}
\end{figure}

\medskip

We explain this symmetry by defining a shift operation on $\NCmpos(W,c)$, analogously to~\eqref{eq:ass_shift} for $\Assoc(W,c)$.  For $s$ initial in $c$, \Cref{lem:reflection_order} implies that the search words $\inv_\refl(\Q)$ for $\NCmpos(W,c)$ and $\inv_\refl(\Q')$ for $\NCmpos(W,\coxrn)$ are related by
\begin{equation}
\label{eq:bigrotation}
  \sninv\Q\psi^{m+1}(\s) \equiv \Q'.
\end{equation}

This relationship between the search words $\Q$ and $\Q'$ induces a bijection on the facets of $\NCmpos(W,c)$ and $\NCmpos(W,\coxrn)$.  For $s$ initial in $c$, define 
\begin{align*}
  \Shift_s: \NCmpos(W,c) &\bij \NCmpos(W,\coxrn)\\
I &\longmapsto \Shift_s(I)
\end{align*}
using the identification provided by~\eqref{eq:bigrotation}.

\medskip

We check that this is a bijection.  Recall that every facet $I$ of $\NCmpos(W,c)$ spells out a reduced $\refl$-word for~$c$.  Write the reflections corresponding to the positions of $I$ as the reduced $\refl$-word $\t_1 \t_2 \t_3 \cdots \t_n$ for $c$.  If the first letter of $\inv_\refl(\Q)$ is not in $I$, then $\Shift_s(I)$ spells out the reduced $\refl$-word $\t_1^s \t_2^s \t_3^s \cdots \t_n^s$ for $\coxrn$, and so specifies a facet in $\NCmpos(W,\coxrn)$.   Otherwise, the first letter of $\Q$---corresponding to the reflection $s$---is in $I$.  The reflection corresponding to the last letter in $\sninv\Q\psi^{m+1}(\s)$ is $s^{cs}$, so that the map
$\Shift_s(I)$ spells out the $\refl$-word $\t_2^s \t_3^s \cdots \t_n^s \s^{cs}$.  We compute
$
  t_2^s t_3^s \cdots t_n^s s^{cs} = (c s) s^{c s} = \coxrn,
$
as desired.  It is now clear that $\Shift_s$ is an isomorphism.

\medskip

\christianside{Symmetries everywhere, beautiful.}
We summarize the preceding discussion with the following theorem.

\begin{theorem+}
\label{thm:nceqncdog}
  For any two Coxeter elements $c,c'$, the simplicial complexes $\NCmpos(W,c)$ and $\NCmpos(W,c')$ are isomorphic.
\end{theorem+}

Composing shifts now defines a map from $\NCmpos(W,c)$ to itself.

\begin{definition}
  Let $\cwo = \s_1 \s_2 \cdots \s_N$ be the~$\c$-sorting word for~$\wo$.
  Define the \defn{positive Kreweras complement} $\Krewpos: \NCmpos(W,c) \to \NCmpos(W,c)$ by \[
    \Krewpos \eqdef \Shift_{s_N}\circ\cdots\circ\Shift_{s_2}\circ\Shift_{s_1}.
  \]
\end{definition}
\christianside{Am I supposed to understand this naming scheme?}

\begin{remark}
  Combinatorial interpretations of $m$-eralized noncrossing partitions in classical types in terms of set partitions have been studied in~\cite{Arm2006}.
  Although evocative of Cambrian rotation, it is more appropriate to call this composition the \emph{positive Kreweras complement}---it may be shown to act as a rotation of combinatorial models for positive Fu{\ss}-Catalan noncrossing partitions in classical types~\cite{KS2018}.
\end{remark}

We now compute the order of $\Krewpos$.  Since conjugation by~$\wo$ is an involution, $h = 2|\refl|/|\sref|$, and there are exactly $(m+1)|\refl|-|\sref|$ letters used to construct $\NCmpos(W,c)$,
\[
  \left(\Krewpos\right)^{(m+1)h-2} \equiv \id.
\]

\begin{theorem}
  The order of the positive Kreweras complement is given by
  \[
    \order(\Krewpos) =
    \begin{cases}
      (m+1)h/2 - 1, & \text{if } \psi \equiv \id \text{ or } \big( n = 2 \text{ and } m=1 \big), \\
      (m+1)h - 2, & \text{otherwise.}
    \end{cases}
  \]
\end{theorem}\christianside{This is the same order as the positive Panyushev map on order ideals in root posets!}\nathanside{Did you mean rowmotion?}

\begin{proof}
  By the discussion above, $\left(\Krewpos\right)^{(m+1)h-2} \equiv \id$.  When $\psi \equiv \id$, the symmetry of order~$2$ gives $\left(\Krewpos\right)^{(m+1)h/2-1} \equiv \id$.  We conclude by showing that for any two letters in~$\Q$, there is a positive $m$-divisible noncrossing partition containing exactly one.
  Using the positive Kreweras complement, it is simple to see that this is the case if we are in rank bigger than~$2$.
  Thus, the only special cases are the dihedral groups for which $\psi\not\equiv\id$.
  For those, the situation is fine if $m>1$, while we also have a symmetry of order~$2$ otherwise.
  This completes the proof.
\end{proof}

\begin{theorem}
\label{thm:fussdogregular}
  The flip graph of $\NCmpos(W,c)$ is regular with degree $n(m-1)$.
\end{theorem}\nathanside{Does this embed in the complexified hyperplane complement?}

\begin{proof}
  Choose a facet $I$ of the dual subword complex $\NCmpos(W,c)$ and a letter in the facet to flip.    If the rank of~$W$ is one, this letter can flip to any other letter and so has the desired degree.
  Otherwise the rank is greater than one and we may find another letter in the facet.
  By the symmetry of the $\Shift_s$ maps, we can move this second letter to be the leftmost reflection~$\s$ in the leftmost copy of $\cwo$ in the search word.
  It follows from~\Cref{lem:reflection_order4} that the facet $I$ without this new first letter is naturally a facet of $\NCmpos(W_{\langle s \rangle},\coxsn)$.  We conclude the result by induction on the rank of~$W$.
\end{proof}

\chapter{Conjectures on rational structures}
\label{sec:rat}

In this chapter, we give an elementary definition of the rational Catalan numbers (\Cref{sec:rat_cat_num}).  We conclude with a possible approach and its limitations for constructing noncrossing combinatorial objects counted by the rational Catalan numbers (\Cref{sec:possible_approach}).

\section{Rational Catalan numbers}
\label{sec:rat_cat_num}

From a purely enumerative perspective, it is believed that noncrossing Catalan objects ought to have combinatorial generalizations beyond their $m$-eralization and positive $m$-eralization.  Such generalizations have been studied in type~$A$, and we refer to~\cite{ARW2013,ALW2014,GM2015} for details.

\begin{definition}
\label{def:ratcat}  Let~$W$ be a finite Coxeter group and let~$p$ be a positive integral parameter coprime to the Coxeter number~$h$.
The \defn{rational Catalan number of type~$W$} is given by
  \begin{equation*}
    \Catrat(W)\eqdef \prod_{i=1}^n \frac{ p + ( p e_i \mod h ) }{ d_i },
  \end{equation*}
  where $e_1 = d_1 - 1 \leq \ldots \leq e_n = d_n - 1$ are the \defn{exponents} of~$W$.
\end{definition}

For Coxeter groups and $p = mh{\pm}1$, this formula recovers the Fu\ss-Catalan and positive Fu\ss-Catalan numbers from~\eqref{eq:fuss_cat_num} and~\eqref{eq:pos_cat_num}.  \Cref{fig:rationalABtables} lists some rational Catalan numbers in the classical types.

\begin{figure}[t]
  \begin{tabular}{c|ccccccccccccc}
    $p$   & $1$ & 2 & 3 & 4  & 5  & 6  & 7   & 8  & 9   & 10 & 11  & 12 & 13 \\
    \hline
    \hline
    $A_1$ & 1& $\cdot$& 2& $\cdot$& 3& $\cdot$& 4& $\cdot$& 5& $\cdot$& 6& $\cdot$& 7\\
    $A_2$ &  1& 2& $\cdot$& 5& 7& $\cdot$& 12& 15& $\cdot$& 22& 26& $\cdot$& 35\\
    $A_3$ &  1& $\cdot$& 5& $\cdot$& 14& $\cdot$& 30& $\cdot$& 55& $\cdot$& 91& $\cdot$& 140\\
    $A_4$ &  1& 3& 7& 14& $\cdot$& 42& 66& 99& 143& $\cdot$& 273& 364& 476\\
    $A_5$ &  1& $\cdot$& $\cdot$& $\cdot$& 42& $\cdot$& 132& $\cdot$& $\cdot$& $\cdot$& 728& $\cdot$& 1428\\
    $A_6$ &  1& 4& 12& 30& 66& 132& $\cdot$& 429& 715& 1144& 1768& 2652& 3876 \\
    \hline
    \hline
    $B_2$ &  1& $\cdot$& 3& $\cdot$& 6& $\cdot$& 10& $\cdot$& 15& $\cdot$& 21& $\cdot$& 28\\
    $B_3$ &  1& $\cdot$& $\cdot$& $\cdot$& 10& $\cdot$& 20& $\cdot$& $\cdot$& $\cdot$& 56& $\cdot$& 84\\
    $B_4$ &  1& $\cdot$& 5& $\cdot$& 15& $\cdot$& 35& $\cdot$& 70& $\cdot$& 126& $\cdot$& 210\\
    $B_5$ &  1& $\cdot$& 6& $\cdot$& $\cdot$& $\cdot$& 56& $\cdot$& 126& $\cdot$& 252& $\cdot$& 462\\
    $B_6$ &  1& $\cdot$& $\cdot$& $\cdot$& 28& $\cdot$& 84& $\cdot$& $\cdot$& $\cdot$& 462& $\cdot$& 924 \\
    \hline
    \hline
    $D_4$ &  1& $\cdot$& $\cdot$& $\cdot$& 20& $\cdot$& 50& $\cdot$& $\cdot$& $\cdot$& 196& $\cdot$& 336\\
    $D_5$ &  1& $\cdot$& 7& $\cdot$& 27& $\cdot$& 77& $\cdot$& 182& $\cdot$& 378& $\cdot$& 714\\
    $D_6$ &  1& $\cdot$& 8& $\cdot$& $\cdot$& $\cdot$& 112& $\cdot$& 294& $\cdot$& 672& $\cdot$& 1386
  \end{tabular}
  \caption{
    Some rational Catalan numbers of classical types.
    \label{fig:rationalABtables}
  }
\end{figure}

\medskip

If the Coxeter group~$W$ is crystallographic, it is well-known that the two sets $\{pe_1, \ldots, pe_n\}$ and $\{e_1,\ldots,e_n\}$ coincide modulo~$h$ (for $p$ coprime to $h$).  In this case, \Cref{def:ratcat} simplifies to
\begin{equation}
  \Catrat(W) = \prod_{i=1}^n \frac{ p + e_i }{ d_i }. \label{eq:catratcrystallographic}
\end{equation}
This formula has a combinatorial interpretation---for crystallographic $W$, M.~Haiman showed that $\Catrat(W)$ counts~$W$-orbits of $Q / pQ$ (where $Q$ is the root lattice)~\cite[Theorem~7.4.4]{Hai1994}.

\medskip

A general formula for the rational Catalan numbers was first considered by I.~Gordon and S.~Griffeth in the context of \emph{rational Cherednik algebras} associated to complex reflection groups~\cite{GG2012}.
We refer to~\cite{LT2009} for all needed background on complex reflection groups and extend \Cref{def:ratcat} to all well-generated complex reflection groups.
To see that this definition coincides with the definition in~\cite{GG2012}, we recall some notions from T.~A.~Springer's seminal paper on regular elements~\cite{Spr1974}.
We refer to that paper for further definitions and background on complex reflection groups.

\medskip

\christianside{We are getting a bit sidetracked here.}
Let~$W$ be a well-generated complex reflection group acting irreducibly on a complex $n$-dimen\-sional vector space~$V$.
It is well-known that~$W$ acts on the polynomial ring $\C[V]$ and that the coinvariant algebra $\C[V] / \C[V]^W_+$ carries the regular representation of~$W$.
For an irreducible representation~$\varphi \in \operatorname{Irred}(W)$, the \defn{$\varphi$-exponents} $e_1(\varphi) \leq \ldots \leq e_{\dim(\varphi)}(\varphi)$ are defined to be the degrees of the graded components of $\C[V] / \C[V]^W_+$ in which the $\dim(\varphi)$ copies of~$\varphi$ live.
Its generating function
\[
  f(\varphi,q) = \sum_{i=1}^{\dim(\varphi)} q^{e_i(\varphi)} = \sum_{k \geq 0} \big[ \C[V] / \C[V]^W_+; \varphi\big]_k q^k
\]
is called the \defn{fake degree polynomial} of~$\varphi$, where $\big[ \C[V] / \C[V]^W_+; \varphi\big]_k$ denotes the multiplicity of $\varphi$ inside the~$k$\th\ graded component of the coinvariant algebra.

Let~$h=e_n(V) + 1$ be the Coxeter number of~$W$, let $\zeta$ be a primitive $h$\th\ root of unity, and for~$p$ coprime to~$h$ let $g: \mathbb{C} \to \mathbb{C}$ be the automorphism sending $\zeta$ to $\zeta^p$.  In~\cite{GG2012}, the \defn{rational Catalan number of type~$W$} is defined as
\begin{equation}\label{eq:ggrational}
  \Catrat(W) \eqdef \prod_{i=1}^n \frac{ p + e_i(g(V) ) }{ d_i }.
\end{equation}
These rational Catalan numbers are dimensions of certain modules over the rational Cherednik algebra associated to~$W$ and the rational parameter~$p/h$.
In particular, these numbers are indeed integral---which was not obvious from \Cref{def:ratcat}.

\bigskip

We deduce from Springer theory that~\Cref{def:ratcat} and \eqref{eq:ggrational} agree.
\begin{proposition}
Let~$W$ be a well-generated complex reflection group acting irreducibly on a complex $n$-dimen\-sional vector space~$V$.  The fake degree polynomials of $V$ and of $g(V)$ at primitive $h$\th\ roots of unity are related by
  \[
  f( g(V), \zeta ) = f( V, \zeta^p ).
  \]
  In particular, the two sets $\{e_1(g(V)), \ldots, e_n(g(V))\}$ and $\{ p e_i, \ldots,p e_n\}$ coincide modulo~$h$.
\end{proposition}
\begin{proof}
  Let~$c$ be a $\zeta$-regular element of $W$ (which is known to exist in well-generated groups).
  Then $g(c)$ is $\zeta^p$-regular, and the eigenvalues of $g(c)$ are obtained from the eigenvalues of~$c$ by replacing every $\zeta$ by $\zeta^p$.

  Since both~$c$ and~$c^p$ satisfy the assumption in~\cite[Section~2.5]{Spr1974}, we apply \cite[Proposition~4.5]{Spr1974} for both~$\zeta$ and~$\zeta^p$, expressing the respective fake degree polynomials in terms of the eigenvalues.  This gives the desired relation between the fake degree polynomials of~$V$ and of~$g(V)$.
\end{proof}

\begin{remark}
\label{rem:rotationaction}
  This definition of the rational Catalan numbers has the $q$-analogue
  \[
    \Catrat(W;q) \eqdef \prod_{i=1}^n \frac{ \big[ p + ( p e_i \mod h ) \big]_q }{ \big[d_i \big]_q },
  \]
  where $[a]_q = 1 + q + \ldots q^{a-1}$ is the usual $q$-integer.
  This $q$-number appears as a graded Hilbert series in the context of rational Cherednik algebras~\cite{GG2012}.
\end{remark}

%%%%%%%%%%%%%%%%%%%%%%%%%%%%%%%%%%%%%%%%%%%%%%%%%%%%%%%%%%%%%%%%%%%%%%%%%%%%%%%%%%%%%
\section{A possible approach in the classical types}
\label{sec:possible_approach}
%%%%%%%%%%%%%%%%%%%%%%%%%%%%%%%%%%%%%%%%%%%%%%%%%%%%%%%%%%%%%%%%%%%%%%%%%%%%%%%%%%%%%

\emph{Rational Dyck paths} and their combinatorial properties have been the subject of considerable recent research~\cite{ARW2013,ALW2014,GM2015,Bod2017}.
Such paths are, by construction, counted by the rational Catalan numbers of type~$A$.
As such, they belong to the \emph{nonnesting} rational Fu\ss-Catalan structures and will not play a prominent role here.
On the other hand, the search for \emph{noncrossing} rational Fu\ss-Catalan structures was initiated for type$~A$ in~\cite{ARW2013}.  In this section, we propose an approach to such structures in classical types.
\christianside{It's not quite working, even though the definitions seem to be tailored also for the rational structures.}\nathanside{Yeah, it seems that Coxeter-initial complexes ought to be the right thing.}

\subsection{\mhead-eralizations}

We now follow the Catalan to Fu\ss-Catalan to positive Fu\ss-Catalan progression to its logical conclusion.  The central results of this monograph establish that this progression replaces $\wo \in W$ by $\bwom \in \Artinmon$ to produce $m$-eralizations, and then further by $\overline{\bc}\bwom$ for positive $m$-eralizations.

For $p$ coprime to $h$, it therefore seems reasonable to search for initial segments $\cwop$ of $\c^\infty$, from which we could define \defn{rational $c$-sortable elements}, \defn{rational $c$-noncrossing partitions}, and the \defn{rational $c$-cluster complex} by
\begin{align*}
  \Sortrat(W,c) &\eqdef \bigset{ \bw \in [\bone,\bwop]_{\Weak(\Artinmon)} }{ \bw \text{ is $c$-sortable} }, \\ 
  \NCrat(W,c) &\eqdef \subwordsR(\invs_\refl(\cwop[p+h](\c)), c), \text{ and}\\
  \Assocrat(W,c) &\eqdef \subwordsSB(\c \cwop(\c), \bwop).
\end{align*}
The goal is to find $\bwop \in \Artinmon$ so that all three definitions give objects counted by the rational Catalan number $\Catrat(W)$.

\begin{theorem}
  For any initial segment $\cwop$ of $\c^\infty$, the natural bijections between $\Sortm[\infty](W,c)$, $\NCm[\infty](W,c)$, and $\Assocm[\infty](W,c)$ restrict to bijections
  \[
    \Sortrat(W,c) \lrbij \NCrat(W,c) \lrbij \Assocrat(W,c).
  \]
\end{theorem}
\begin{proof}
  This follows directly from the observation that the bijections
  \begin{itemize}
    \item $\Assocm[\infty](W,c) \lrbij \Sortm[\infty](W,c)$ given by the skip set, and
    \item $\Assocm[\infty](W,c) \lrbij \NCm[\infty](W,c)$ given by the root configuration
  \end{itemize}
  have the desired property.
\end{proof}

Even though $m$-eralized and the positive $m$-eralized structures both fit into this framework, our expectations were somewhat diminished by the following example (checked by computer).

\begin{observation}
\label{obs:A3}
  For $\WA[6]$ with bipartite~$c = (12)(34)(56) \cdot (23)(45)$,
  there does not exist an initial $\cwop$ of $\c^\infty$ such that
  \[
    |\subwordsSB(\c \cwop,\bwop)|=\Catrat[5](\WA[6]) = 66.
  \]
\end{observation}

One still might maintain the following hope.

\begin{hope}
\label{conj:rationalwords}
  Let $(W,\sref)$ be a finite Coxeter system with $p$ coprime to $h$.
  Then there exists \emph{some} Coxeter element~$c$ and a word $\cwop$ initial in $\cwo$ for which
  \[
    |\Sortrat(W,c)| = |\NCrat(W,c)| = |\Assocrat(W,c)| = \Catrat(W),
  \]
with $\cwop[p+h](\c)\equiv \cwop(\c) \cwo$.
\end{hope}

Since such a $\cwop$ is initial in $\c^\infty$, $\Assocrat(W,c)$ is a $c$-initial subword complex, and hence is vertex decomposable by \Cref{thm:vertex-decomposability}.   The relation between $\cwop[p+h](\c)$ and $\cwop(\c)$ implies that for $p>h$, $\Assocrat(W,c)$ has the homotopy type of a wedge of $\Cat^{[p-h]}(W)$ spheres of dimension $n-1$.

\subsection{Conjectural rational constructions}
The constructions in this section provide conjectural combinatorial models for rational noncrossing structures in all infinite families of finite Coxeter systems.

\begin{theorem}
\label{thm:counting}
  \Cref{conj:rationalwords} holds for $I_2(h)$ with Coxeter element~$c$ given by
    \begin{center}
      \begin{tikzpicture}[scale=.4]
          \node[draw,circle,inner sep=1pt] (A) at (0,0) {$1$};
          \node[draw,circle,inner sep=1pt] (B) at (3,0) {$2$};
          \node at (1.5,0.5) {$h$};
          \draw [->]        (A) -- (1.5,0);
          \draw             (1.5,0) -- (B);
      \end{tikzpicture}
    \end{center}
    and $\bwop = \bc^{(p-1)/2}$.
\end{theorem}\nathanside{Amazing.}

\begin{proof}
  Write $s,t$ for the simple reflections, and fix the Coxeter element $c=st$.
  We prove the statement by explicitly counting
  \[
    \NCrat(W,c) = \subwordsR\big(\invs_\refl(\cwop[p+h](\c)), c \big).
  \]
  We immediately obtain that
  \begin{align*}
    |\NCrat(W,c)| &= |\NCrat[p-h](W,c)| + \lengthS(\bwop) + 1 \\
                  &= |\NCrat[p-h](W,c)| + p
  \end{align*}
  where $\lengthS(\bwop) = p-1$ and $|\NCrat[p-h](W,c)|=0$ for $p\leq h$.
  The first summand in this expression comes from those facets using no reflection from the initial copy of $\invs_\refl(\cwo)$, the second summand from those using a single reflection in this initial copy, and the final~$1$ comes from the single facet entirely contained in the initial copy.

 We compute for any $p$ that
  \begin{align*}
    \Catrat(W)  &= \frac{1}{2h}\Big[\big(p+(p \mod h)\big)\big(p+(-p \mod h)\big)\Big] \\
                &= \frac{1}{2h}\Big[p^2+ph+(p \mod h)(-p \mod h)\Big] \\
                &= \frac{1}{2h}\Big[p^2-ph+(p \mod h)(-p \mod h)\Big] + p \\
                &= \frac{1}{2h}\Big[p^2-2ph+h^2+(p-h)h+(p \mod h)(-p \mod h)\Big] + p\\
                &= \frac{1}{2h}\Big[(p-h)^2+(p-h)h+(p \mod h)(-p \mod h)\Big] + p\\
                &= \frac{1}{2h}\Big[\big((p-h)+(p \mod h)\big)\big((p-h)(-p \mod h)\big)\Big] + p\\
                &= \Catrat[p-h](W) + p,
  \end{align*}
  with $\Catrat(W) = p$ for $1 \leq p \leq h$.
\end{proof}

We conjecture \Cref{conj:rationalwords} to hold in types $A_n$, $B_n$, $D_n$, and $H_3$.

\begin{conjecture}
\label{conj:counting}
  \Cref{conj:rationalwords} holds for
  \begin{enumerate}[(1)]
    \item type $A_n$ with Coxeter element $c = (1,\ldots,n+1)$ given by
      \begin{center}
        \begin{tikzpicture}[scale=.4]
            \node[draw,circle,inner sep=1pt] (A) at (0,0) {$1$};
            \node[draw,circle,inner sep=1pt] (B) at (3,0) {$2$};
            \node[draw,circle,inner sep=1pt] (C) at (6,0) {$3$};
            \node[draw,circle,inner sep=1pt] (D) at (12,0) {$n$};
            \draw [->]        (A) -- (1.5,0);
            \draw             (1.5,0) -- (B);
            \draw [->]        (B) -- (4.5,0);
            \draw             (4.5,0) -- (C);
            \draw [dotted,->] (C) -- (10.5,0);
            \draw [dotted   ] (10.5,0) -- (D);
        \end{tikzpicture}
      \end{center}
      and for $p=mh+r$ coprime to $h$ with $0\leq r < h$
      \[
        \bwop[p] = \bc_{a_{r-1}} \cdots \bc_{a_{1}} \bwo^m, \quad a_i \eqdef \lfloor ih/r \rfloor, \quad \bc_j \eqdef \bs_1 \bs_2 \cdots \bs_j,
      \]\christianside{Who drew these?}\nathanside{\dots}
    \item type $B_n$ with Coxeter element~$c = (1,\ldots,n,-1,\ldots,-n)$ given by
      \begin{center}
        \begin{tikzpicture}[scale=.4]
            \node[draw,circle,inner sep=1pt] (A) at (0,0) {$1$};
            \node[draw,circle,inner sep=1pt] (B) at (3,0) {$2$};
            \node[draw,circle,inner sep=1pt] (C) at (6,0) {$3$};
            \node[draw,circle,inner sep=1pt] (D) at (12,0) {$n$};
            \node at (1.5,0.5) {$4$};
            \draw [->]        (A) -- (1.5,0);
            \draw             (1.5,0) -- (B);
            \draw [->]        (B) -- (4.5,0);
            \draw             (4.5,0) -- (C);
            \draw [dotted,->] (C) -- (10.5,0);
            \draw [dotted   ] (10.5,0) -- (D);
        \end{tikzpicture}
      \end{center}
      and
      \[
        \bwop = \bc^{(p-1)/2},
      \]
    \item type $D_n$ with any Coxeter element~$c = (1,-1)(2,\ldots,n,-2,\ldots,-n)$ given by
      \begin{center}
        \begin{tikzpicture}[scale=.4]
            \node[draw,circle,inner sep=1pt] (A1) at (0,1) {$1$};
            \node[draw,circle,inner sep=1pt] (A2) at (0,-1) {$2$};
            \node[draw,circle,inner sep=1pt] (B) at (3,0) {$3$};
            \node[draw,circle,inner sep=1pt] (C) at (6,0) {$4$};
            \node[draw,circle,inner sep=1pt] (D) at (12,0) {$n$};
            \draw [->]        (A1) -- (1.5,0.5);
            \draw             (1.5,0.5) -- (B);
            \draw [->]        (A2) -- (1.5,-0.5);
            \draw             (1.5,-0.5) -- (B);
            \draw [->]        (B) -- (4.5,0);
            \draw             (4.5,0) -- (C);
            \draw [dotted,->] (C) -- (10.5,0);
            \draw [dotted   ] (10.5,0) -- (D);
        \end{tikzpicture}
      \end{center}
      and
      \[
        \bwop = \bc^{(p-1)/2},
      \]
    \item type $H_3$ with Coxeter element~$c$ given by
      \begin{center}
        \begin{tikzpicture}[scale=.4]
            \node[draw,circle,inner sep=1pt] (A) at (0,0) {$1$};
            \node[draw,circle,inner sep=1pt] (B) at (3,0) {$2$};
            \node[draw,circle,inner sep=1pt] (C) at (6,0) {$3$};
            \node at (1.5,0.5) {$5$};
            \draw [->]        (A) -- (1.5,0);
            \draw             (1.5,0) -- (B);
            \draw [->]        (B) -- (4.5,0);
            \draw             (4.5,0) -- (C);
        \end{tikzpicture}
      \end{center}
      and
      \[
        \bwop = \bc^{(p-1)/2},
      \]
  \end{enumerate}
\end{conjecture}

\christianside{We have computed it several times, it seems to \emph{really} fail.}
\begin{remark}
  Unfortunately, this hope fails to work in at least some other types.  For example, in types~$F_4$ and~$H_4$, there do not exist Coxeter elements~$c$ and words~$\cwop$ initial in $\c^\infty$ such that \Cref{conj:rationalwords} holds.
\end{remark}

\chapter{\mhead-eralized structures in representation theory}
\label{sec:representationtheory}
%%%%%%%%%%%%%%%%%%%%%%%%%%%%%%%%%%%%%%%%%%%%%%%%%%%%%%%%%%%%%%%%%%%%%%%%%%%%%%%%%%%%%

\newcommand{\Endo}{\operatorname{End}}
\newcommand{\add}{\operatorname{add}}
\newcommand{\oco}{{{}^\infty\c^\infty}}

In this chapter, we exhibit the connection between the Catalan combinatorics
which we have been studying and the representation theory of finite-dimensional
hereditary Artin algebras.  Our main references for this representation
theory are \cite{dlabringel76,H1987,ARS1997}.
\hughside{If you were wondering where I was, I've mostly been hanging out in this chapter.}
%\christianside{Ask not what representation theory can do for you---ask what you can do for representation theory.}

We give a more detailed introduction to the connections to be explored in this chapter (\Cref{sec:rep-overview}) and present some combinatorial constructions which match directly with the combinatorics built into the representation theory (\Cref{sec:combcon}).
We then give a quick
introduction to the representation theory of Artin algebras (\Cref{sec:artin}).
We introduce exceptional sequences, and explain their equivalence to
$\refl$-factorizations of Coxeter elements (\Cref{sec:fact}).  We
use exceptional sequences to give representation-theoretic interpretations
of $m$-eralized clusters and $m$-eralized noncrossing partitions (\Cref{sec:cnc}), and
of the natural bijection between them (\Cref{sec:cncbij}).  We show that these interpretations can be extended to the positive
setting (\Cref{sec:pos}).  We show that certain symmetries of 
$\NablaAssocm(W,c)$ and $\DeltaNCmpos(W,c)$ can be accounted for
representation-theoretically (\Cref{sec:ar-trans}).  We give a
representation-theoretic interpretation of $m$-eralized $c$-sortable
elements (\Cref{sec:aisles}), and a representation-theoretic analogue of
the bijection to $m$-eralized clusters (\Cref{sec:bij}).

\section{Overview} \label{sec:rep-overview}
$\NablaAssocm(W,c)$ and $\DeltaNCm(W,c)$
%\christian{we use $m$-clusters and $m$-nc only in this chapter. should use globally valid terms}
are known to
admit interpretations in terms of representation theory of hereditary
Artin algebras.
%\christian{here and everywhere: we should clarify if the repr th works / has been done for simply-laced / crystrallographic / all types}
See, for example, \cite{BRT2012}, where
this perspective was employed to give the first definition of a
uniform bijection between these two sets. 
In view of these existing interpretations, it is satisfying to
provide a similar interpretation for $m$-eralized sortable elements and
$m$-eralized Cambrian lattices.
In this chapter, we show that the elements of
$\Sortm(W,c)$ correspond bijectively to a certain natural class of
co-aisles in the bounded derived category of a hereditary Artin algebra
defined in terms of~$W$ and~$c$.  The co-aisles in question all contain a
certain ``standard'' co-aisle, and the inversion sets of the elements of
$\Sortm(W,c)$ designate the indecomposable objects in the corresponding
co-aisle other than those in the standard co-aisle.  It follows that
$\Cambsortm(W,c)$ is the inclusion order on this class of co-aisles.

Further, the combinatorial bijections which we have constructed in
previous chapters also have corresponding representation-theoretic versions in crystallographic types.
We will show that
the bijection between $\NablaAssocm(W,c)$ and $\DeltaNCm(W,c)$ in
\Cref{thm:rootconf_skiptset} agrees (up to a choice of convention)
with the bijection from \cite{BRT2012}.
We will also show that
the bijection between $\Cambsortm(W,c)$ and $\NablaAssocm(W,c)$ is a
combinatorial version of a bijection between (co)-aisles and
silting objects which goes back to B.~Keller and D.~Vossieck \cite{KV}.
%\Hugh{is this really camb-sort or camb-nc?}

%We also use the representation-theoretic perspective to explain the rotational symmetry
%on $\NablaAssocm(W,c)$ and $\DeltaNCmpos(W,c)$.  This can also be viewed as
%providing an explanation for the fact that the otherwise similar
%$\NablaAssocmpos(W,c)$ and $\DeltaNCm(W,c)$ do not.  

We omit proofs of most of the purely representation-theoretic statements
in this chapter, preferring to refer the reader to the existing
literature.  On the other hand, we include some proofs either because
we consider them to be instructive or because it was not easy to find
the desired results in the literature.  \hughside{I spent a pleasant day reading all the algebra books in the Oberwolfach library so you don't have to.}

%%%%%%%%%%%%%%%%%%%%%%%%%%%%%%%%%%%%%%%%%%%%%%%%%%
\section{Combinatorial constructions}\label{sec:combcon}
%%%%%%%%%%%%%%%%%%%%%%%%%%%%%%%%%%%%%%%%%%%%%%%%%%%%%%%%%%%%%
In this section, we give some combinatorial constructions, based on considerations from previous chapters, in order to prepare for our applications to the representation
theory of hereditary Artin algebras.
\christianside{This is the best way to convince a combinatorialist to read this chapter.}
\nathanside{What, use the word ``combinatorial'' in red italics?}

Let $(W, \sref)$ be a Coxeter system with $\sref=\{s_1,\dots,s_n\}$.  
For a (possibly infinite or bi-infinite) $\sref$-word $\Q$, we define the
\defn{combinatorial AR quiver} of $\Q$ to be the quiver whose vertex set is the letters of
$\Q$, and such that for each letter $\s_i$ in $Q$, and each $s_j\in\sref$ such that
$s_i$ and $s_j$ do not commute, there is an arrow from that $\s_i$ to the next
occurrence of $\s_j$, if any.  (The poset whose Hasse diagram is the combinatorial AR
quiver is known in the literature as the \emph{heap} of $\Q$ \cite{XV}.)

The \defn{combinatorial AR translation} is a partially defined map from
the letters of $\Q$ to the letters of $\Q$, sending each $\s_i$ to the previous instance
of $\s_i$, if any.

If $\Q=\s_1\s_2\dots$ is a (possibly infinite, but not bi-infinite) $\sref$-word, then there is an
inversion sequence of colored positive roots $\inv(Q)$, namely
$\alpha^{(0)}_1, s_1(\alpha_2^{(0)}),\dots$, as defined in \Cref{sec:inv-set-pos-Art}.
We refer to the colored positive root corresponding to a letter in $\Q$ as
its colored root label.  

We are especially interested in three cases of the combinatorial AR quiver
construction, depending on a
choice of $\c=\s_1\dots\s_n$.  \begin{enumerate}[$\circ$]

\item $\Q=\cwo$.
Recall that $\cwo$ is the $c$-sorting word for $\wo$; see \Cref{def:c-sorting} for the
definition and \Cref{lem:reflection_order} for some of its important properties.
Two examples of these quivers, for the Coxeter group of type~$A_3$, with $\c=\s_1\s_2\s_3$
and $\c=\s_2\s_1\s_3$, are given in~\Cref{combarquiv}.  (As always, when we give examples in type $A_n$, we use $s_i$ to represent the adjacent transposition $(i\ i+1)$.)

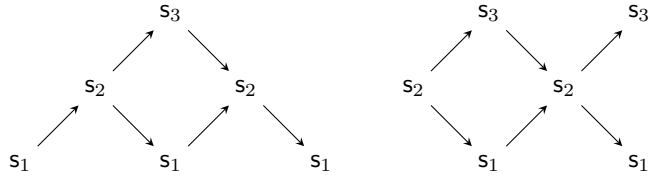
\begin{figure}
$$
  \begin{tikzpicture}
  \node (a) at (0,0) {$\s_1$};
  \node (b) at (1,1) {$\s_2$};
  \node (c) at (2,2) {$\s_3$};
  \node (d) at (2,0) {$\s_1$};
  \node (e) at (3,1) {$\s_2$};
  \node (f) at (4,0) {$\s_1$};
  \draw[-stealth] (a) -- (b);
  \draw[-stealth] (b) -- (c);
  \draw[-stealth] (b) -- (d);
  \draw[-stealth] (c) -- (e);
  \draw[-stealth] (d) -- (e);
  \draw[-stealth] (e) -- (f);
  \end{tikzpicture}
\qquad
    \begin{tikzpicture}
  \node (b) at (1,1) {$\s_2$};
  \node (c) at (2,2) {$\s_3$};
  \node (d) at (2,0) {$\s_1$};
  \node (e) at (3,1) {$\s_2$};
  \node (f) at (4,0) {$\s_1$};
  \node (g) at (4,2) {$\s_3$};
  \draw[-stealth] (b) -- (c);
  \draw[-stealth] (b) -- (d);
  \draw[-stealth] (c) -- (e);
  \draw[-stealth] (d) -- (e);
  \draw[-stealth] (e) -- (f);
    \draw[-stealth] (e) -- (g);
\end{tikzpicture}$$
  \caption{Combinatorial AR quivers for $\wo(\s_1\s_2\s_3)=\s_1\s_2\s_3\s_1\s_2\s_1$, and for $\wo(\s_2\s_1\s_3)=\s_2\s_1\s_3\s_2\s_1\s_3$.\label{combarquiv}
}
\end{figure}

\item $\Q=\c^\infty$.  The beginning of the combinatorial AR quiver for
$\c^\infty$ in the Coxeter group of type $A_3$ with $\c=\s_1\s_2\s_3$ is given in \Cref{combarinf}. Note that, as can be seen in the example, and as shown by \Cref{lem:reflection_order}, $\cwo$ is initial in $\c^\infty$.

\begin{figure}
$$
  \begin{tikzpicture}
  \node (a) at (0,0) {$\s_1$};
  \node (b) at (1,1) {$\s_2$};
  \node (c) at (2,2) {$\s_3$};
  \node (d) at (2,0) {$\s_1$};
  \node (e) at (3,1) {$\s_2$};
  \node (f) at (4,0) {$\s_1$};
  \node (g) at (4,2) {$\s_3$};
  \node (h) at (5,1) {$\s_2$};
  \node (i) at (6,2) {$\phantom{\s_3}$};
  \node (j) at (6,0) {$\phantom{\s_1}$};
  \draw[-stealth] (a) -- (b);
  \draw[-stealth] (b) -- (c);
  \draw[-stealth] (b) -- (d);
  \draw[-stealth] (c) -- (e);
  \draw[-stealth] (d) -- (e);
  \draw[-stealth] (e) -- (f);
    \draw[-stealth] (e) -- (g);
    \draw[-stealth] (f) -- (h);
    \draw[-stealth] (g) -- (h);
    \draw[-stealth] (h) -- (i);
    \draw[-stealth] (h) -- (j);
\end{tikzpicture}$$
  \caption{Initial subquiver of the combinatorial AR quiver of $\c^\infty$ for $\c=\s_1\s_2\s_3$.\label{combarinf}}
\end{figure}
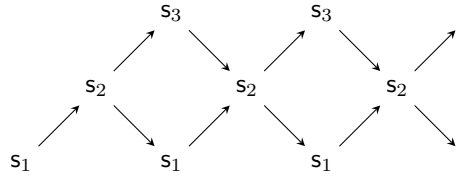

\item $\Q=\oco$.  By definition, $\oco$ is the bi-infinite word consisting of
  repetitions of $\c$.  Part of the combinatorial AR quiver for $\oco$ is given
  in \Cref{combardinf}. Note that, up to commutations, the word $\oco$ does not depend
  on the choice of $\c$, and the corresponding combinatorial AR quiver does not depend on the choice of $\c$ at all.  

  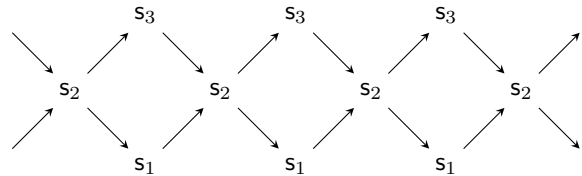
\begin{figure}
$$
    \begin{tikzpicture}
        \node (w) at (-2,2) {$\phantom{\s_3}$};
  \node (z) at (-2,0) {$\phantom{\s_1}$};
      \node (y) at (-1,1) {$\s_2$};
      \node (x) at (0,2) {$\s_3$};
  \node (a) at (0,0) {$\s_1$};
  \node (b) at (1,1) {$\s_2$};
  \node (c) at (2,2) {$\s_3$};
  \node (d) at (2,0) {$\s_1$};
  \node (e) at (3,1) {$\s_2$};
  \node (f) at (4,0) {$\s_1$};
  \node (g) at (4,2) {$\s_3$};
  \node (h) at (5,1) {$\s_2$};
  \node (i) at (6,2) {$\phantom{\s_3}$};
  \node (j) at (6,0) {$\phantom{\s_1}$};
  \draw[-stealth] (a) -- (b);
  \draw[-stealth] (b) -- (c);
  \draw[-stealth] (b) -- (d);
  \draw[-stealth] (c) -- (e);
  \draw[-stealth] (d) -- (e);
  \draw[-stealth] (e) -- (f);
    \draw[-stealth] (e) -- (g);
    \draw[-stealth] (f) -- (h);
    \draw[-stealth] (g) -- (h);
    \draw[-stealth] (h) -- (i);
    \draw[-stealth] (h) -- (j);
        \draw[-stealth] (x) -- (b);
        \draw[-stealth] (y) -- (x);
            \draw[-stealth] (y) -- (a);
    \draw[-stealth] (z) -- (y);
    \draw[-stealth] (w) -- (y);

\end{tikzpicture}$$
  \caption{Subquiver of the combinatorial AR quiver of $\oco$ for $\c=\s_1\s_2\s_3$.\label{combardinf}}
\end{figure}
\end{enumerate}

By \Cref{lem:reflection_order}, $\cwo\psi(\cwo)\equiv \c^h$.  We can therefore view
$\c^\infty$ as an infinite alternating repetition of $\cwo$ and $\psi(\cwo)$.
Because $\cwo$ is initial in $\c^\infty$, the colored root labelling of
letters from $\cwo$ are the same in $\cwo$ and $\c^\infty$.  The set of these labels
are exactly the 0-colored positive roots.  

\medskip
We define the \defn{combinatorial shift} sending a letter of $\cwo$ to the
corresponding letter in the following $\psi(\cwo)$, and similarly sending letters of
$\psi(\cwo)$ to the corresponding letter of the following copy of $\cwo$.  This
sends the reflection corresponding to the colored almost positive root
$\beta^{(i)}$ to the reflection corresponding to $\beta^{(i+1)}$.  

Similarly, we can view $\oco$ as a bi-infinite word consisting of alternating copies
of $\cwo$ and $\psi(\cwo)$, although this factorization is not canonical.
Using such a factorization, we can also define the combinatorial shift, and inverse
combinatorial shift, on
the letters of $\oco$.  This definition does not depend on the choice of factorization.

Let us define the set of \defn{$\mathbb Z$-colored positive roots} to be $\PhiP\times \mathbb Z$.    
Once we fix a copy of $\cwo$ inside $\oco$, there is a well-defined labelling of
the letters of $\oco$ by $\mathbb Z$-colored positive roots.  The chosen copy of
$\cwo$ corresponds to the roots labelled with color 0; other reflections are
labelled so that the combinatorial shift sends the reflection corresponding to
$\beta^{(i)}$ to the reflection corresponding to $\beta^{(i+1)}$.

%%%%%%%%%%%%%%%%%%%%%%%%%%%%%%%%%%%%%%%%%%%%%%%%%%%%%%%%%%%%%%%%%%%%%%%%%%%%%%%%%%%%%
\section{A brief introduction to hereditary Artin algebras}\label{sec:artin}
%%%%%%%%%%%%%%%%%%%%%%%%%%%%%%%%%%%%%%%%%%%%%%%%%%%%%%%%%%%%%%%%%%%%%%%%%%%%%%%%%%%%%

We now introduce the representation-theoretic setting in which we shall
work. 
Let~$\Bbbk$ be a field, and let~$H$ be a finite-dimensional basic hereditary
algebra over~$\Bbbk$, with~$n$ isomorphism classes of simple objects. We consider the category $\mod H$ of finite-dimensional left $H$-modules. We suppose $H$ to be representation-finite, i.e., having only finitely many isomorphism classes of indecomposable modules.
Introductions to the representation theory of finite-dimensional hereditary Artin algebras
can be found in \cite{ARS1997,ASS2006,DDPW2008} and \cite[Chapter 7]{SY17}.

\subsection{The Grothendieck group} The
Grothendieck group of~$H$, denoted $K_0(H)$, is the
free abelian group
on the set of isomorphism classes of~$H$-modules, modulo the
subgroup generated by all expressions of the form
$[M]-[L]-[N]$ for any short exact sequence $0\rightarrow L
\rightarrow M \rightarrow N \rightarrow 0$.
The Grothendieck group is isomorphic to $\mathbb Z^n$;
the classes of the simple~$H$-modules form a
$\mathbb Z$-basis for $K_0(H)$.
The expansion of $[M]$ in the basis
of the simple modules records the multiplicity with which each simple
module appears in a composition series for $M$.

$K_0(H)$ is naturally equipped with a symmetric, bilinear form
(the symmetrization of the Euler form) which is
positive definite.  The map sending a module to its
class in the Grothendieck group defines a bijection from
the isomorphism classes of indecomposable modules to the positive roots of a
root system~$\Phi$ compatible with this symmetric, bilinear form.

The classes in $K_0(H)$ of the simple modules are the simple roots.
The Coxeter group corresponding to the root system,~$W$, acts naturally on $K_0(H)$.
Any finite crystallographic root system can be realized in this way for some
choice of~$H$.
If $\Phi$ is a simply-laced root system, we can take~$H$ to be the path algebra
of an orientation of the corresponding Coxeter diagram
over an arbitrary
ground field $\Bbbk$.  For more details
on constructing algebras corresponding to non-simply-laced root systems,
see for example~\cite{DDPW2008}.

For a simple reflection~$s_i \in W$, write~$S_i$ for the corresponding simple
$H$-module,~$P_i$ for the projective cover of $S_i$, $I_i$ for the injective
envelope of $S_i$, and $e_i$ for the corresponding idempotent in $H$.
Let~$c_H$ be the product of the simple reflections taken in an order such
that~$s_i$ precedes~$s_j$ if $\Ext^1(S_j,S_i)\ne 0$.  
Although there may be
different orders which satisfy this condition, the product $c_H$ is
well-defined as an element of~$W$.   The Coxeter combinatorics associated to
the element~$c_H$ turns out to encode a great deal of the structure of the
category of $H$-modules.  We will write $c$ for $c_H$ if there is no
risk of confusion.

\subsection{The AR translation and the AR quiver}
If $M$ is an indecomposable non-projective module, the Auslander-Reiten
translation of $M$, written~$\tau M$, is the indecomposable module
characterized by the fact that $[\tau M]=c[M]$.  Similarly,
the inverse Auslander-Reiten translation of $M$ is
characterized by $[\tau^{-1}M]=c^{-1}[M]$ if $M$ is a
non-injective module
\cite[Proposition VIII.2.2]{ARS1997}.

The isomorphism classes of
indecomposable $H$-modules are naturally organized as the
vertices of what is
called its \defn{Auslander-Reiten quiver}, or AR quiver for short.  By definition,
we draw an arrow from $X$ to $Y$ if there is a morphism from $X$ to $Y$
which is not a sum of morphisms factoring through other indecomposable objects (and is not
an isomorphism).
This is a slightly simplified version of the AR quiver,
which suffices for our purposes.  We are ignoring the multiplicities of the
arrows, which require some additional book-keeping if~$\Bbbk$ is not
algebraically closed.  We are also avoiding the discussion of AR sequences,
because we do not actually need them for our purposes

\medskip
The following theorem shows that the structure of the AR quiver of $\mod H$
corresponds to the structure
of $\cwo$.  This theorem is already known; the proof we give consists of a sequence of references to the literature for its different components.  \nathanside{My favorite kind of proof!}

\begin{theorem} Let $H$ be a finite-dimensional Artin algebra, and let $c=c_H$ be the
  corresponding Coxeter element as defined above.  
  \begin{enumerate}
  \item   \label{ptone}
    The AR quiver of $\mod H$ coincides with the combinatorial AR quiver of $\cwo$.  

  \item \label{ptoneb} AR translation in $\mod H$
    is given by the combinatorial AR translation.
    
  \item  \label{pttwo} The indecomposable module $M$ corresponds to the simple
    reflection in
    the combinatorial AR quiver whose corresponding inversion is the 0-colored root
    $[M]^{(0)}$.

\item  \label{ptthree} The projective indecomposable $P_i$ correponds to the first instance of~$\s_i$ in $\cwo$.  Its class in the Grothendieck group is $[P_i]=s_{1}\dots s_{i-1}(\alpha_{i})$.

\item  \label{ptfour} The injective indecomposable $I_i$ corresponds to the final instance of~$\s_i$ in $\cwo$.  Its class in the Grothendieck group is $-c^{-1}[P_i]$.
\end{enumerate}
  \end{theorem}

\begin{proof}
We start by considering the combinatorial AR quiver for $\c^\infty$.  
The AR quiver
of $H$ is a subquiver of the combinatorial AR quiver for $\c^\infty$
with $P_i$ corresponding to the first occurrence of
$\s_i$, by \cite[Proposition VIII.1.15]{ARS1997} and the discussion before it.
This also establishes that the AR translation in $\mod H$ agrees with the
combinatorial AR translation in $\c^\infty$.

The class in the Grothendieck group of $P_i$ is $s_1\dots s_{i-1}(\alpha_i)$,
by \cite[Lemma 1.6]{dlabringel76}.  
The subset of $\c^\infty$ corresponding to the AR quiver is 
as follows: take the first $k_i$ copies of $\s_i$, where $k_i$ is maximal
such that $[P_i]$, $c[P_i]$, $c^2[P_i]$, $\dots$, $c^{k_i}[P_i]$ are all positive roots,
by \cite[Proposition VIII.1.15]{ARS1997}.
These are the roots corresponding to successive copies of $\s_i$ in $\c^\infty$, so
$k_i$ can also be described as maximal such that the roots corresponding to the
first $k_i$ copies of $\s$ in $\c^\infty$ are positive.  This establishes that the
subset of $\c^\infty$ corresponding to the AR quiver agrees with the subset
corresponding to $\cwo$, proving (\ref{ptone}) and (\ref{ptoneb}).

We know that the $j$-th copy of $\s_i$ corresponds to the module
$\tau^{-j+1}P_i$, whose class in the Grothendieck group is $c^{j-1}s_1\dots s_{i-1}(\alpha_i)$.  This root, with color zero, is 
the colored positive root labelling the $j$-th copy of $\s_i$ in
$\c^\infty$. This shows (\ref{pttwo}).

%As we have shown in \Cref{sec:combcon}, agrees with the colored positive root
%labelling
%This is not the same as the formula given in (\ref{pttwo}); in (\ref{pttwo}, we only
%use reflections appearing in $\cwo$, whereas the statement we already know uses all
%reflections in $\c^\infty$.  Since $\cwo$ is initial in $\c^\infty$ by \Cref{lem:reflection_order}, the two answers agree. This establishes (\ref{pttwo}).

%To verify that the subset of $\c^\infty$ which corresponds to the AR quiver
%coincides with the subset corresponding to $\cwo$, we verify the second claim of the theorem, that the positive
%roots associated to the vertices of the AR quiver (i.e. the classes of the
%modules) agree with the
%roots associated by $\c^\infty$.  

%By \cite[Lemma 1.6]{dlabringel76}, the positive roots associated to the first copy of $\c$ in $\cwo$, i.e. $s_1\dots s_{i-1}\alpha_i$
%correspond to the classes of the projective modules.  The class of
%$\tau^{-p+1}P_i$ is $c^{-p+1}[P_i]$ which is the root associated to the $p$-th
%instance of $s_i$ in $\c^\infty$.  This establishes the second statement,
%and thus also the first statement.

We have already established (\ref{ptthree}).
Now (\ref{ptfour}) follows from \cite[Proposition VIII.2.2]{ARS1997}.
\end{proof}

\subsection{The bounded derived category}
The bounded derived category of $H$, denoted $D^b(H)$,
is a triangulated category, with shift functor denoted $[1]$.
Material on derived categories can
be found in \cite{H1987}, in textbooks on homological algebra,
or in \cite{K2007}.

We write $\ind~D^b(H)$
%\christian{the space after $\ind, \modu, \Ext, \Hom$ are not given automatically, or is that on purpose?}
for the indecomposable objects of $D^b(H)$.  Any
indecomposable object of $D^b(H)$ is of the form $M[i]$, for~$i\in \mathbb Z$
and $M$ an indecomposable~$H$-module \cite[Theorem I.5.2]{H1987}; the hypothesis in this section
of \cite{H1987} that the ground field is algebraically close is not used for this
result.
%(There are additional
%indecomposable objects in the bounded derived category of a
%non-hereditary algebra.)
  The indecomposable objects in $D^b(H)$ are in bijection with the $\mathbb Z$-colored
  positive roots.
  For $M$ an indecomposable object in $\mod H$ and $i\in \mathbb Z$, we define
  $$\udim\ M[i] := [M]^{(i)}.$$
  If $X$ is the direct sum of indecomposable objects $X_1\oplus \dots\oplus X_r$, we define $\udim X$ to be the set $\{\udim X_1,\dots,\udim X_r\}.$

There is also an AR quiver defined for the bounded derived category.
As before, we define the AR quiver by saying that there is a vertex for each
isomorphism class of indecomposable modules, and there is an arrow between two
vertices when there is an irreducible morphism between the corresponding
modules.  (This agrees with the usual definition by \cite[Proposition I.4.3]{H1987}, up to the fact that we are not keeping track of valuation data on the
arrows of the AR quiver.)  There is also an Auslander-Reiten translation which sends indecomposable objects to indecomposable objects.  

The next theorem shows that the AR quiver of $D^b(H)$ is the combinatorial AR quiver
of $\oco$.  Recall that we showed in \Cref{sec:combcon} that $\oco$ can also be
described as a bi-infinite alternating sequence of $\cwo$ and $\psi(\cwo)$.
Again, we include the known proof for the convenience of the reader.

\begin{theorem} The AR quiver of $D^b(H)$ is the combinatorial AR quiver of
  $\oco$, and the AR translation is given by the combinatorial AR translation.
  Fixing a copy of $\cwo$ inside $\oco$, we can identify it with the
  AR quiver of $\mod H$ inside $D^b(H)$; the following copy of $\psi(\cwo)$ is then
  identified with $\mod H[1]$, the following copy of $\cwo$ is identified with
  $\mod H[2]$, and so on in both directions.  
\end{theorem}
  
  \begin{proof} 
    \cite[I.4.7]{H1987} shows that the AR quiver and the AR translation
    restricted to each shift of
the module category, are the same as in the module category.
To understand the AR quiver, we need only understand how successive copies of the module category
are connected.

The arrows between successive copies of the module category
are given by \cite[Lemma I.5.4]{H1987}: there is an arrow from $I_i[-1]$ to
each (projective) vertex in the module category with an arrow to $P_i$, and an arrow to $P_i$ from each vertex in the $[-1]$-shift of the module category with an arrow from $I_i[-1]$.  (This lemma is proved under the assumption that~$\Bbbk$
is algebraically closed; the same proof applies in the general case, though
somewhat heavier notation is needed.)
This establishes the statements in the theorem about the AR quiver.

The statement about AR translation follows from the corresponding statement about AR translation in the module category, together with the fact that
$\tau P_i = I_i[-1]$ from \cite[Lemma I.5.4]{H1987} and the fact that
$[I_i]=-c^{-1}[P_i]$.
\end{proof}

  \begin{example} We give the representation theory corresponding to
   \Cref{ex:a2roots}.
   Let~$H$ be the path algebra of the quiver
   $$ 1 \longleftarrow 2. $$\christianside{What about $H_2$?} There
are three indecomposable $H$-modules, the simples $S_1$ and $S_2$, and
$P_2$, the projective at 2.  Their corresponding classes in $K_0(H)$ can
be identified as the positive roots of the $A_2$ root system, respectively
$\alpha$, $\beta$, and $\gamma$.  The Coxeter group is isomorphic to the
symmetric group on three letters, and
is generated by simple reflections~$s$ (corresponding to $S_1$) and
$t$ (corresponding to $S_2$).  The Coxeter element is $c_H=st$, and $\cwo=\s\t\s$, which gives us the
shape of the AR quiver of the module category.  
The AR quiver of the
bounded derived category can either be thought of as repeated copies of
$Q$ with connecting arrows, or repeated copies of the module category, with
connecting arrows.  The AR quiver is \hughside{Well, technically, not all of the AR quiver.}
\begin{center}\begin{tikzpicture}
\node (da) at (-2.7,0.5) {$\cdots$};
\node (ponen) at (-2,0) {$P_2[-1]$};
\node (stwon) at (-1,1) {$S_2[-1]$};
\node (sone) at (0,0) {$S_1$};
\node (pone) at (1,1) {$P_2$};
\node (stwo) at (2,0) {$S_2$};
\node (sonep) at (3,1) {$S_1[1]$};
\node (ponep) at (4,0) {$P_2[1]$};
\node (db) at (4.6,0.5) {$\cdots$};
\draw [-stealth] (sone) -- (pone);
\draw [-stealth] (pone) -- (stwo);
\draw [-stealth] (ponen) -- (stwon);
\draw [-stealth] (stwon) -- (sone);
\draw [-stealth] (stwo) -- (sonep);
\draw [-stealth] (sonep) -- (ponep);
\end{tikzpicture}.\end{center}

\end{example}
%%%%%%%%%%%%%%%%%%%%%%%%%%%%%%%%%%%%%%%%%%%%%%%%%%%%%%%%%%%%
\section{Factorizations of Coxeter elements and exceptional sequences}
%%%%%%%%%%%%%%%%%%%%%%%%%%%%%%%%%%%%%%%%%%%%%%%%%%%%%%%%%%%%%%%%
\label{sec:fact}

%\christian{we tried to not have math mode in section titles, what about ``factorizations of Coxeter elements and exceptional sequences''?}

%\christian{what about ``Exceptional sequences''?}
A key structure underlying the interplay between the representation theory of finite type Artin algebras and Catalan combinatorics is the
correspondence between exceptional sequences and $\refl$-factorizations of $c^{-1}$ (i.e., $\reds(c^{-1})$, as in \Cref{sec.dual_braid_rel}. We now explain this connection.%\hugh{Christian, you suggested we should refer to dual subword complexes, but this is actually something different, since these factorizations are not necessarily ordered in any particular way.}
\hughside{Fun fact: exceptional sequences and factorizations of Coxeter elements were developed by disjoint sets of mathematicians who were unaware they were studying the same objects.}

\subsection{Exceptional sequences in the module category}
\begin{definition}
A sequence $X_1,\dots,X_r$ of $H$-modules is called
an \defn{exceptional sequence} if each $X_j$ is
indecomposable and $\Ext^i(X_\ell,X_j)=0$ for $\ell > j$ and $i=0,1$.  An
exceptional sequence is called \defn{complete} if its length is $n$.
\end{definition}

Exceptional sequences of length $n$ are called complete because that is the
maximum possible length of an exceptional sequence.
Basic references for exceptional sequences are \cite{CB93} and (over an
arbitrary ground field) \cite{R94}.

Recall that indecomposable $H$-modules correspond bijectively to positive
roots, which correspond bijectively to reflections in $W$.  A sequence of
indecomposable modules may therefore be considered as a
sequence of reflections.  It turns out that there is a very nice way to
characterize the sequences of reflections which correspond to
exceptional sequences of modules.  

The following result was shown in the simply-laced case
in \cite{IT2009}.  A proof in a much more general setting, which,
in particular, drops the simply-laced assumption, can be found in
\cite{IS2010}.  See also \cite{HK2015} for another presentation.

\begin{theorem+}
\label{exceptionalisfactor} If $X_1,\dots,X_n$ is a sequence of~$n$ indecomposable modules,
then it is an exceptional sequence if and only if $s_{[X_1]}\dots s_{[X_n]}=c^{-1}$.
\end{theorem+}%\nathan{OK, so this is the subword complex, basically}
%\christian{this is the repr th version of the $t_1 \cdots t_k$ is nc iff $t_i t_j$ is nc for $i<j$ for subwords $t_1 \cdots t_k$ of $\c^\infty$.}
%\hugh{I would say these are chains in the lattice of noncrossing partitions}

\begin{example} Continuing the $A_2$ example which we began above, there
are three exceptional sequences in $H$-mod: $(S_1,P_2)$, $(P_2,S_2)$, and
$(S_2,S_1)$. They correspond respectively to the $\refl$-factorizations
$$c^{-1}=su=ut=ts.$$
%The exceptional sequences in $D^b(H)$ are obtained from
%exceptional
%sequences in $H$-mod by applying arbitrary shifts to each term.
\end{example}

The braid group on~$n$ strands acts on $\reds(c^{-1})$ by the Hurwitz action, which we have recalled in \Cref{sec.dual_braid_rel}. For
$1\leq i \leq n-1$, let $\boldsymbol\sigma_i$
be the generator taking the $(i+1)$\th\ strand over the $i$\th.
Given a factorization
$c^{-1}=r_1\dots r_n$, the result of acting by $\boldsymbol\sigma_i$ is to move $r_i$ past
$r_{i+1}$, at the cost of conjugating $r_i$, i.e.,
$(r_i,r_{i+1}) \mapsto (r_{i+1},r_{i+1}r_ir_{i+1})$.
The corresponding moves on
exceptional sequences are known as \defn{mutations}. %\nathan{Does it have to be mutation?}\hugh{It was mutation before mutation was cool! (I think the terminology goes back to the late 80s.)}
(Right) mutation replaces
a subsequence $(X,Y)$ by $(Y,R_YX)$.  $R_YX$ can be defined as the unique
module in the abelian, extension-closed subcategory generated by $X$ and $Y$
such that $(Y,R_YX)$ is an exceptional sequence.  
By Theorem \ref{exceptionalisfactor}, $R_YX$ can also be characterized by the fact
that $[R_YX]=|s_{[Y]}[X]|$.  

\begin{example}\label{mut-ex} We continue the same example.
If we apply right mutation to $(P_2,S_2)$, we obtain
$(S_2,S_1)$.  If we apply right mutation to $(S_2,S_1)$, we obtain
$(S_1,P_2)$.
If we apply right mutation to $(S_1,P_2)$, we obtain
$(P_2,S_2)$.
\end{example}

\subsection{Exceptional sequences in the derived category}
It turns out that it is useful to consider exceptional sequences in
$D^b(H)$, and in fact, from now on, when we refer to exceptional sequences,
we mean exceptional sequences in $D^b(H)$.
The definition is almost the same: a sequence $X_1,\dots,X_r$ of indecomposable objects from $D^b(H)$ is
exceptional if 
$\Ext^i(X_\ell,X_j)=0$ for $\ell > j$ and all $i\in \mathbb Z$.

It turns out that exceptional sequences in the derived category are very
closely related to exceptional sequences in the module category.  To aid the reader,
we provide a proof of the following simple lemma.  

\begin{lemma} Let $M_1,\dots,M_r$ be $H$-modules, and let $i_1,\dots,i_r \in\mathbb Z$.  Then $(M_1,\dots,M_r)$ is an exceptional sequence of $H$-modules if and only if $(M_1[i_1],\dots,M_r[i_r])$ is an exceptional sequence in $D^b(H)$.
  \end{lemma}

\begin{proof} The definition of exceptional sequence in $D^b(H)$ is clearly
  insensitive to applying $[1]$ to any of the terms.  It is therefore
  enough to check that if $(M_1,\dots,M_r)$ is a sequence of modules, then
  it is an exceptional sequence in $D^b(H)$ if and only if it is an
  exceptional sequence in $\mod H$.  On the face of it, the two definitions
  look different, because for an exceptional sequence in $D^b(H)$, we check
  the vanishing of $\Ext^i(M_\ell,M_j)$ for $\ell>j$ and all $i$, whereas in the definition in $\mod H$, we only check this vanishing for $i=0$ and $i=1$.  The
  point is that, since $M_\ell$ and $M_j$ are $H$-modules, we know that
  $\Ext^i(M_\ell,M_j)=0$ for all $i\ne 0,1$.
  \end{proof}

By Theorem \ref{exceptionalisfactor}, we
can think of complete exceptional sequences in $D^b(H)$ as $\refl$-factorizations
of $c^{-1}$ where each factor receives an (arbitrary) color in $\mathbb Z$.

For exceptional sequences in $D^b(H)$, there is a way to specify
precisely how mutations act with respect to shift so that mutations still induce a braid group action.  There is an
easy way to state the outcome concretely:
$R_YX$ is shifted so that it weakly precedes $X$ in the AR quiver partial order,
but it is far to the right as possible subject to that condition.

More algebraically, it can be described as follows.
 %   
%Consider the right thick($Y$) approximation to $X$\christian{many undefined terms. are those clear anyway, or maybe a reference?}.
%
There will be at most one $b\in \mathbb Z$ such that $\Hom(X,Y[b])\ne 0$.  Having found this $b$, take an $\Endo(Y)$-basis
$f_1,\dots,f_a$ of $\Hom(X,Y[b])$.  Then define the map $f$ from $X$ to the sum
of $a$ copies of $Y[b]$ such that the map into the $i$-th copy is given by
$f_i$.  (This map is known as the left thick $\add Y$ approximation to $X$.)  
Then complete this map to a triangle. (This is a routine operation in a triangulated category, akin to taking the kernel or cokernel of a morphism in an abelian category.)
The third term of the triangle is defined to be
$R_YX$, which is determined up to isomorphism,
$$ R_YX \rightarrow X \rightarrow Y^a[b] \rightarrow R_YX[1].$$
(See the proof of \cite[Lemma 7.1]{BRT2012} for more on this.)

\begin{example} Let us redo \Cref{mut-ex} in the derived category.
  If we apply right mutation to $(P_2,S_2)$, we obtain
$(S_2,S_1)$.  If we apply right mutation to $(S_2,S_1)$, we obtain
  $(S_1,P_2)$.  (In these two cases, the results were the same as before.)
  But
if we apply right mutation to $(S_1,P_2)$, we obtain
$(P_2,S_2[-1])$.
\end{example}

The following useful lemma can be found, for example, in \cite{BRT2012}.

\begin{lemma+}
\label{full-mutation}
Given a complete exceptional sequence $(X_1,\dots,X_n)$, if we successively mutate it so as to move $X_1$ to the other end, the result is a sequence $(X_2,\dots,X_n,\tau^{-1}X_1[-1])$.
%\christian{$\tau$ is not defined}
\end{lemma+}

\begin{corollary+}
\label{trapped}
Given a complete exceptional sequence $(X_1,\dots,X_n)$, if we successively mutate it so as to move $X_1$ part way through the sequence, the result is a sequence $(X_2,\dots,X_i,Y,X_{i+1},\dots,X_n)$ where $Y$ is located between $\tau^{-1}X_1[-1]$ and $X_1$ in the AR-quiver.
\end{corollary+}

%%%%%%%%%%%%%%%%%%%%%%%%%%%%%%%%%%%%%%%%%%%%%%%%%%%%%%%%%%%%%%%%%%%%%%%%%%%%%%%%%%%%%
\section{\mhead-eralized clusters and noncrossing partitions}
%%%%%%%%%%%%%%%%%%%%%%%%%%%%%%%%%%%%%%%%%%%%%%%%%%%%%%%%%%%%%%%%%%%%%%%%%%%%%%%%%%%%%
\label{sec:cnc}

In this section we define silting objects and $\Hom_{\leq 0}$ configurations
in the derived category, and show how they correspond to $m$-eralized clusters and $m$-eralized noncrossing partitions, respectively.

In order to do this, it is
convenient to begin in the $m=\infty$ setting.
We view $\c\cwo^m$ as an initial subsequence of $\c\cwo^{m+1}$, and thereby
interpret
$\NablaAssocm(W,c)$ as a subset of $\NablaAssocmp(W,c)$. We then
define $\NablaAssocinf(W,c)=\bigcup_{m\geq 1} \NablaAssocm(W,c)$.

Similarly, define
\begin{align*}
  \DeltaNCinf(W,c) &\eqdef \bigcup_{m\geq 1} \DeltaNCm(W,c) \\
  \Sortinf(W,c)    &\eqdef \bigcup_{m\geq 1} \Sortm(W,c),
\end{align*}
and define $\Cambsortinf(W,c)$ to be the containment order on $\Sortinf(W,c)$.
(We could also equivalently define $\Cambncinf$ using $\DeltaNCinf$.)

\subsection{\mhead-eralized noncrossing partitions}

\begin{definition}\label{def:hom} An object~$X$ in $D^b(H)$ is called a \defn{$\Hom_{\leq 0}$ configuration} if
its indecomposable direct summands can be ordered as $X_1,\dots,X_n$
such that the sequence is exceptional and $\Hom(X_i,X_j)=0$ for $i\ne j$.
\begin{enumerate}[$\circ$]
\item We write $\Hom_{\leq 0}(H)$ for the $\Hom_{\leq 0}$ configurations in $D^b(H)$.  
\item We write $\Homp(H)$ for the $\Hom_{\leq 0}$ configurations all of whose
indecomposable summands are of the form $M[i]$ with $i\geq 0$.  \item We write
$\Homm(H)$ for the $\Hom_{\leq 0}$ configurations all of whose summands are
of the form $M[i]$ with $m\geq i\geq 0$.\end{enumerate}
\end{definition}

\begin{proposition} \label{homisforwards}~$X$ is a $\Hom_{\leq 0}$ configuration
 if and only if the direct summands of~$X$ read
in some (or equivalently, any) order compatible with the AR quiver of
$D^b(H)$, yields a sequence of reflections whose product is~$c$.
\end{proposition}

\begin{proof} Suppose~$X$ is a $\Hom_{\leq 0}$ configuration.
Let the indecomposable summands
of~$X$ be $X_1,\dots,X_n$, ordered in the reverse of
any order compatible with the AR
quiver.  By definition, then, $\Ext^{i}(X_\ell,X_j)=0$ for $\ell>j$
and $i>0$  (because $X_\ell$ will be to the left of $X_j[i]$).
On the other hand, because~$X$ is a $\Hom_{\leq 0}$ configuration, we know that
$\Ext^i(X_\ell,X_j)=0$ for all $\ell\ne j$ and $i\leq 0$.  This establishes that
$X_1,\dots,X_n$ is an exceptional sequence, and, by
\Cref{exceptionalisfactor}, the corresponding product of reflections is
$c^{-1}$.  Thus, if we read the reflection in the AR order, rather
than in the reverse order, we obtain~$c$.

Conversely, let $X_n,\dots,X_1$ be the summands of~$X$ in
some order compatible with the AR quiver,
and suppose that the
corresponding product of reflections $t_{[X_n]}\dots t_{[X_1]}=c$, so
$t_{[X_1]}\dots t_{[X_n]} = c^{-1}$, and thus,
by \Cref{exceptionalisfactor}, we know that $X_1,\dots,X_n$ is an exceptional sequence.  We now reverse
the previous argument to deduce that~$X$ is a $\Hom_{\leq 0}$ configuration.
\end{proof}

\begin{example} We continue the same example. $\Hom_{\leq 0}$-configurations
  are of one of the three following forms:\begin{itemize}
  \item $S_1[i]\oplus P_2[j]$ with $j<i$,
    \item $P_2[i] \oplus S_2[j]$ with $j<i$,
    \item $S_2[i]\oplus S_1[j]$ with $j\leq i$.\end{itemize}
  \end{example}

\begin{theorem} \label{thm:hom-deltanc} There is a natural bijection
  \begin{eqnarray*}
\Homp(H) &\cambbij &\DeltaNCinf(W,c)\\
 X & \longmapsto & \udim\ X. \end{eqnarray*}
This bijection restricts to 
a bijection from $\Homm(H)$ to $\DeltaNCm(W,c)$.
\end{theorem}

\begin{proof}
Reading $\invs_\refl(\cwoinf)$ is equivalent to (a particular way of) reading the classes in the Grothendieck group of the objects in the AR quiver of $D^b(H)$.
The theorem now follows from \Cref{homisforwards}.
Because the bijection is induced from a bijection between indecomposables
and letters in $\c^\infty$, the fact that the Cambrian recursion is satisfied
is clear. 
  The statement about the restrictions is also clear.
  \end{proof}

\subsection{\mhead-eralized clusters}
\begin{definition}\label{def:silting} An object~$X$ in $D^b(H)$ is called \defn{silting} if $\Ext^i(X,X)=0$ for
$i>0$, and $X$ is the direct sum of $n$ pairwise non-isomorphic indecomposables.
  \begin{enumerate}[$\circ$]
\item 
%$n$ different isomorphism classes
%\christian{do we want to avoid abbreviations?}
  %of summands.
  We write $\Silt(H)$ for the set of silting objects in $D^b(H)$.  
\item We write $\Siltp(H)$ for the
silting objects in $D^b(H)$ all of whose indecomposable
summands are of the
form $M[i]$ with~$M$ an $H$-module and $i\geq 0$, or $I[-1]$ with~$I$ an indecomposable
injective.
\item We write $\Siltm(H)$ for the silting objects in $D^b(H)$ all of whose indecomposable summands are of the form $M[i]$ with $m>i\geq 0$, or $I[-1]$ for $I$
an indecomposable injective.
\end{enumerate}
  \end{definition}

Statements similar to the following can be found in \cite{IS2010} and
\cite{BRT2011}; we omit the proof, which is very similar to that of
\Cref{homisforwards}.

\begin{proposition+}
\label{siltingisbackwards}
$X$ is silting if and only if the direct summands of~$X$ read in some (or equivalently, any) order compatible with the AR quiver of $D^b(H)$, yields a sequence of reflections whose product is $c^{-1}$.
\end{proposition+}

%\begin{proof} Suppose~$X$ is silting.  Let the indecomposable summands
%of~$X$ be $X_1,\dots,X_n$, ordered in any order compatible with the AR
%quiver.  By definition, then, $\Ext^{i}(X_\ell,X_j)=0$ for $\ell>j$
%and $i\leq 0$  (because $X_\ell$ will be to the right of $X_j[i]$).
%On the other hand, because~$X$ is silting, we know that
%$\Ext^i(X_\ell,X_j)=0$ for all $\ell\ne j$ and $i>0$.  This establishes that
%$X_1,\dots,X_n$ is an exceptional sequence, and, by
%\Cref{exceptionalisfactor}, the corresponding product of reflections is~$c^{-1}$.

%Conversely, let $X_1,\dots,X_n$ be the summands of~$X$ in
%some order compatible with the AR quiver, and suppose that the
%corresponding product of reflections is~$c^{-1}$.  By \Cref{exceptionalisfactor}, we know that $X_1,\dots,X_n$ is an exceptional sequence.  We now reverse
%the previous argument to deduce that~$X$ is silting.
%\end{proof}

\begin{example}
  We continue the same example.
  Silting objects are of one of the three following forms:
  \begin{itemize}
    \item $S_1[i] \oplus P_2[j]$ with $j\geq i$,
    \item $P_2[i] \oplus S_2[j]$ with $j\geq i$,
    \item $S_2[i] \oplus S_1[j]$ with $j>i$.
  \end{itemize}
\end{example}

%Write $D^\sharp(H)$ for the full subcategory of
%$D^b(H)$
%whose indecomposables are the indecomposables of $\mod H[i]$ for $i\geq 0$.
%For~$M$ an indecomposable $H$-module and $i\geq 0$, we associate to the indecomposable
%object $M[i]$ in $D^\sharp(H)$ a corresponding
%colored almost positive root $\udim(M[i])=[M]^{(i)}$.

%For~$m$ an indecomposable $H$-module and $i\geq 0$, define $\udim(M[i])$ to be
%the colored positive root $[M]^{(i)}$.

\medskip

Now, we can state the following theorem.

\begin{theorem}\label{prop:silt-to-assoc} There is a natural bijection
  \begin{eqnarray*}
    \Siltp(H) & \cambbij &\NablaAssocinf(W,c) \\
    X & \longmapsto & \udim \ \tau^{-1}X
    \end{eqnarray*}
 This bijection restricts to a bijection from $\Siltm(H)$ to $\NablaAssocm(W,c)$.  
\end{theorem}

\begin{proof} The proof is the analogue of the proof of \Cref{thm:hom-deltanc},
  using \Cref{siltingisbackwards}.\end{proof}

% \section[A representation-theoretic bijection]{A representation-theoretic version of the bijection between \mhead-eralized clusters and \mhead-eralized noncrossing partitions}
\section[A representation-theoretic bijection]{The bijection between \mhead-eralized clusters and noncrossing partitions}
\label{sec:cncbij}

\cite{BRT2012} gives a bijection from $\Siltp(H)$ to $\Homp(H)$.
%This is described
%there in terms of mutation of exceptional sequences, but we shall limit
%ourselves here to giving the combinatorial version of the bijection.  Let
%$X$ be a silting object, and let
%$X_1,\dots,X_n$ be an ordering of the summands of~$X$ in the AR-quiver order.
%Let the corresponding colored positive roots be $\beta_i^{(c_i)}$.  Define
%$\gamma_i^{(d_i)}$ by successively applying $t_{\beta_{i+1}}, \dots, t_{\beta_n}$
%to $\beta_i^{(c_i)}$, where, if the reflection changes the sign of the
%root, we decrement the color, and otherwise we leave the color the same.
%Then the collection $\beta_n^{(d_n)},\dots,\beta_1^{(d_1)}$ is the sequence
%of colored dimension vectors encoding a $\Hom_{\leq 0}$-configuration.  This
%defines a bijection from silting objects to $\Hom_{\leq 0}$-configurations.
%The above map does not restrict to give a map from $\Siltp$ to $\Homp$.
%In order to get such a map, we first multiply, in the same way, successively,
%by $s_1,\dots,s_n$, and then apply $[1]$, after which we apply the
%previously-defined bijection.
Define $\Rev$ of an exceptional sequence $(E_1,\dots,E_n)$
to be obtained by applying the fundamental element of the braid group $\Braidgrp_{n}$, and then applying
[1].  More explicitly, $\Rev(E_1,\dots,E_n)$ is obtained by moving $E_1$ to
the end of the sequence (changing it as it moves),
moving $E_2$ to just before the modified $E_1$, moving $E_3$ to just before
the modified $E_2$, etc., and then applying [1].

In order to apply $\Rev$ to a silting object or a $\Hom$-configuration, we order its indecomposable summands into an exceptional
sequence, apply $\Rev$, and then take the direct sum of the resulting terms, forgetting their order.  The result is well-defined independent of the choice of
ordering.

\begin{theorem}[{\cite{BRT2012}}]\label{thm:rev}  \label{BRTbij}
The map $\Rev$ is a natural bijection from silting objects to
$\Hom_{\leq 0}$ configurations
\begin{eqnarray*} \Silt(H) &\cambbij& \Hom_{\leq 0}(H) \\
  X&\longmapsto& \Rev(X).
\end{eqnarray*}
This bijection restricts to a bijection   from $\Siltp(H)$ to $\Homp(H)$,
    and further restricts to a bijection from $\Siltm(H)$ to $\Homm(H)$.
\end{theorem}

\begin{proof} We refer to \cite{BRT2012} for the proof of the existence of
  the bijection.  We prove here that it respects the Cambrian recurrence.

Let $\gamma_1^{(k_1)},\dots,\gamma_n^{(k_n)}$ be in $\NablaAssocinf(W,c)$.
Let the corresponding indecomposable objects be $X_1,\dots,X_n$, with
$X$ their direct sum, so $X\in \Siltp(H)$.

Let~$s_1$ be initial in~$c$, and suppose that
$\gamma_1^{(k_1)}$ is not equal to $\alpha_{1}^{(0)}$.  Equivalently, this means
that $\tau S_1$ is not a summand of $X$.
In this case, it is clear
that $\Revp$ respects the Cambrian recurrence:  $\Revp$ is defined in
terms of the derived category, so it is not sensitive to the difference
between~$H$ and the algebra obtained by applying a reflection functor at $1$.
(See \cite{dlabringel76} for reflections functors over general fields.)

%The leftmost $\Revp(X_1,\dots,X_n)$
%It is
%clear that the leftmost element of its root configuration is also
%$\gamma_1^{(c_1)}$.  The corresponding fact for $\Revp$ follows immediately
%from \cite[Lemma 3.3(b)]{BRT2012}.

Now suppose that $\gamma_1^{(c_1)} = \alpha_1^{(0)}$.  In representation-theoretic
terms, it means that $\tau S_1$ is a summand of~$X$.  The summand of
$\Rev(X)$ corresponding to $\tau S_1$, is $S_1$, by \Cref{full-mutation}.
The other summands of~$X$ will satisfy that $\Ext^i( X_j,\tau S_s)=0$ for
$i>0$, because~$X$ is silting.  By Auslander-Reiten duality, it follows that
$\Hom^i(S_1, X_j) =0$ for $i\geq 0$.  This means that the other summands of
$X_j$ can be identified with objects in $D^b(H/He_1H)$.  The definitions of
$\Rev$ in $D^b(H)$ and $D^b(H/He_1H)$ agree. This shows that $\Rev$ satisfies
the Cambrian recurrence.
\end{proof}

\begin{theorem}
  The following diagram commutes

  $$\begin{tikzpicture}[xscale=6,yscale=2]
    \node (a) at (0,0) {$\Siltp(H)$};
    \node (b) at (1,0) {$\Homp(H)$};
    \node (c) at (0,-1) {$\NablaAssocinf(W,c)$};
    \node (d) at (1,-1) {$\DeltaNCinf(W,c)$.};
\draw [stealth-stealth] (a) -- (b) node[midway,above] {\Cref{thm:rev}};
\draw [stealth-stealth] (a) -- (c) node[midway,left]{\Cref{prop:silt-to-assoc}};
\draw [stealth-stealth] (b) -- (d) node [midway,right]{\Cref{thm:hom-deltanc}};
\draw [stealth-stealth] (c) -- (d) node[midway, below]{\Cref{thm:rootconf_skiptset}};
  \end{tikzpicture}$$
  The bijections all restrict to the appropriate $m$-eralized versions.
\end{theorem}

%  The bijection $\Rev$ defined above coincides with that from \Cref{thm:rootconf_skiptset}. 
%\end{theorem}

%\Hugh{Christian has a direct proof of this which doesn't go through the
%Cambrian recurrence (which was how I was going to prove it.}
%\Christian{Please see \Cref{prop:cluster product}. Is that what you wanted?}
%\Hugh{Almost!  I would like to control the colors of the roots, not just
%the roots.  I guess it is just a matter of acting on colored roots by
%reflections.  Do you think it would be a lot of work to do this?}
\begin{proof} We have already shown that all of the bijections in this
  diagram respect the Cambrian recurrence, so the diagram commutes.
  \end{proof}

%\christian{so far, I haven't seen sortables appearing. is that on purpose?}

%%%%%%%%%%%%%%%%%%%%%%%%%%%%%%%%%%%%%%%%%%%%%%%%%%%%%%%%%%%%%%%%%%%%%%%%%%%%%%%%%%%%%
\section{Positive Fu\ss-Catalan combinatorics in representation theory}
%%%%%%%%%%%%%%%%%%%%%%%%%%%%%%%%%%%%%%%%%%%%%%%%%%%%%%%%%%%%%%%%%%%%%%%%%%%%%%%%%%%%%
\label{sec:pos}

We can proceed in a similar fashion to define representation-theoretic objects
corresponding to positive analogues of noncrossing partitions or of
clusters. \hughside{This section is basically nothing but definitions, but it's put to some use in the next section.}

\begin{definition}
Define $\Hompos(H)$ to consist of $\Hom_{\leq 0}$ configurations
in $D^b(H)$ all of whose indecomposable summands are of the form
$M[i]$ with~$M$ an $H$-module and $i\geq 0$, and where,
if $i=0$,~$M$ is not allowed to be
projective. Define $\Homposm(H)$ to consist of those
$\Hom_{\leq 0}$ configurations in $\Hompos(H)$ such that, in addition, all
of the indecomposable summands are located in shift at most $m$.
\end{definition}

The following lemma is immediate.

\begin{lemma+}
  The bijection $\udim$ from $\Homp(H)$ to $\DeltaNCinf(W,c)$ restricts to bijections from $\Hompos(H)$ to $\DeltaNCinfpos(W,c)$ and from $\Homposm(H)$ to $\DeltaNCmpos(W,c)$.
\end{lemma+}

\begin{definition}
Define $\Siltpos(H)$ to consist of silting objects in $D^b(H)$ all of whose
indecomposable summands are of the form $M[i]$ with~$M$ an $H$-module and
$i\geq 0$.  Similarly, define $\Siltposm(H)$ to consist of silting objects
in $D^b(H)$ all of whose indecomposable summands are of the form
$M[i]$ with~$M$ an $H$-module and $m> i \geq 0$.
\end{definition}

In other words, compared
to $\Siltp(H)$ and $\Siltm(H)$ respectively,
we are simply forbidding the indecomposable
injectives in shift $[-1]$ as possible summands.
The following lemma is immediate.

\begin{lemma+}
  The bijection $\udim \circ \tau^{-1}$ from $\Siltp(H)$ to $\NablaAssocinf(W,c)$ restricts to bijections from $\Siltpos(H)$ to $\NablaAssocinfpos(W,c)$ and from $\Siltposm(H)$ to $\NablaAssocmpos(W,c)$.
\end{lemma+}

The positive analogue of \Cref{BRTbij} is also shown in \cite{BRT2012}.
We omit the easy proof.

\begin{theorem+}
  $\Rev$ defines a bijection 
  \begin{eqnarray*} \Siltpos(H)&\longleftrightarrow& \Hompos(H)\\
    X&\longmapsto&\Rev(X).\end{eqnarray*}
    This bijection
    restricts to a bijection from $\Siltposm(H)$ to $\Homposm(H)$.
\end{theorem+}

%\begin{proof} $X \in \Siltp(H)$ belongs to $\Siltpos(H)$ if and only if
%  $\tau(X)$
%also belongs to $\Siltp(H)$.  Similarly, $Y\in \Hom(H)$ belongs to
%$\Hompos(H)$ if and only if $\tau(Y)$ also belongs to $\Hom(H)$.  The theorem now
%follows from \Cref{BRTbij} together with the fact that $\Rev$ commutes
%with $\tau$.
%\end{proof}
%\todo{Notation in this section should be modified to match with Chapter 6
%once that notation has stabilized.}

%%%%%%%%%%%%%%%%%%%%%%%%%%%%%%%%%%%%%%%%%%%%%%%%%%%%%%%%%%%%%%%%%%%%%%%%%%%%%%%%%%%%%
\section{Representation theory as a source of symmetries}
%%%%%%%%%%%%%%%%%%%%%%%%%%%%%%%%%%%%%%%%%%%%%%%%%%%%%%%%%%%%%%%%%%%%%%%%%%%%%%%%%%%%%
\label{sec:ar-trans}

There are cyclic group actions on $\NablaAssocm(W,c)$, 
and $\DeltaNCmpos(W,c)$, as discussed in \Cref{sec:m-assoc-cambrian-recurrence} and \Cref{sec:positivesymmetry} respectively.  In this section,
we explain how they can be seen as arising out of the representation
theory.
\christianside{Excitement!}

\subsection{Orbit categories} We begin with a quick introduction to
orbit categories, which we will make use of for studying both clusters and
noncrossing partitions.

\begin{definition}
Let $G$ be an autoequivalence of $D^b(H)$.  The \defn{orbit category} of $D^b(H)$ with
respect to $G$, which we denote $D^b(H)/G$, is a category whose objects are the objects of $D^b(H)$, and whose
morphisms are defined by
\begin{equation}\Hom_{D^b(H)/G}(X,Y)=\bigoplus_{i\in \mathbb Z} \Hom_{D^b(H)}(X,G^iY).
  \label{eq:homs}\end{equation}
\end{definition}

Define an autoequivalence of $D^b(H)$ by $F^{(m)}=[m]\tau^{-1}$.
Let $\Cm$ be the orbit category with respect to $F^{(m)}$.
It is easy to see that $\Cm$ has finite-dimensional $\Hom$-spaces, because
there are only a finite number of non-zero summands in the direct sum (\ref{eq:homs}).
It is a deeper result that this category is triangulated (\cite{K2005}, see also \cite{Amiot}).

A triangulated category $\mathcal C$ is called \defn{$k$-Calabi--Yau} if for any $X$ and
$Y$ objects of $\mathcal C$, we have that $\Ext^i_{\mathcal C}(X,Y)$ is naturally isomorphic to the dual of $\Ext^{k-i}_{\mathcal C}(Y,X)$.  

\begin{lemma}[\cite{K2005}] $\Cm$ is $(m+1)$-Calabi--Yau. \end{lemma}

\begin{proof} After unwinding the definitions, this follows from Auslander-Reiten duality. \end{proof}
%\begin{eqnarray*}
%\Ext_\Cm^i(X,Y)&\cong&\Hom_\Cm(X,Y[i])\\
%&\cong&\bigoplus_j \Hom_{D^b(H)}(X,(F^{(m)})^jY[i])\\
%&\cong&\bigoplus_j \Hom_{D^b(H)}(X,[mj+i]\tau^{-j}Y)\\
%&\cong&D\bigoplus_j \Ext^1_{D^b(H)}(\tau^{-j-1}Y[mj+i],X)\\
%&\cong&D\bigoplus_j \Hom_{D^b(H)}(Y,(F^{(m)})^{j+1}X[m+1-i])\\
%&\cong&D\bigoplus_j \Ext^{m+1-i}(Y,(F^{(m)})^{j+1}X)\\
%&\cong&D\Ext_\Cm^{m+1-i}(Y,X)
%\end{eqnarray*}
%Here $D$ is vector space duality, and the fourth isomorphism comes from
%Auslander-Reiten duality.
%\end{proof}

We now divide into two cases, to treat clusters and noncrossing
partitions.

\subsection{\mhead-eralized clusters}
Let $m\geq 0$.  The orbit category $\Cm$ is called the \defn{$m$-cluster category}.
An object $X$ in $\Cm$ is called a (basic) \defn{$m$-cluster tilting object}
if no two indecomposable direct summands of $X$ are isomorphic, 
$\Ext^i(X,X)=0$ for $1\leq i \leq m$, and $X$ is maximal with respect to
this property, i.e., if there is some $Y$ such that $\Ext^i(X\oplus Y,X
\oplus Y)=0$, then $Y$ is a direct sum of summands of $X$.  The maximality property
can also be captured by saying that $X$ is the direct sum of $n$ pairwise non-isomorphic
indecomposable objects.  

Let $\Fm$ denote the full additive category of $D^b(H)$ whose indecomposable
modules are of the form either $M[i]$ for $0\leq i <m$ or $I[-1]$, with
~$I$ an indecomposable injective.  $\Fm$ is a fundamental domain for the
action of $F^{(m)}$ on $\ind~D^b(H)$.

\begin{lemma}[\cite{BRT2011}]\label{lem:silt-m-clust} Let $X$ be an object in $\Fm$.  $X$ is a silting object if
and only if the image of $X$ in $\Cm$ is an $m$-cluster tilting object.
\end{lemma}

\begin{proof} Suppose that $X$ and $Y$ are indecomposable objects in
$\Fm$.  It suffices to show that they can permissibly occur together as
summands of an $m$-cluster tilting object if and only if they can permissibly occur
together as summands of a silting object .  In other words,
we must check that
$\Ext_{\Cm}^i(X,Y)=0=\Ext_{\Cm}^i(Y,X)$ for all $1\leq i \leq m$
if and only if
$\Ext_{D^b(H)}^i(X,Y)=0=\Ext_{D^b(H)}^i(Y,X)$ for all $i>0$.
We may assume that $X$ is
weakly to the left of $Y$ in the AR quiver for $D^b(H)$.

We first consider the forwards implication.
Since $X$ is weakly to the left of $Y$, it follows that
$\Ext^i(X,Y)=0$ for all $i>0$.  Because $X$ and $Y$ both lie in
$\Fm$, $\Ext^i(Y,X)=0$ for $i>m$.  The fact that $\Ext^i(Y,X)=0$ for
$1\leq i \leq m$ follows from the corresponding statement for $\Cm$.

For the reverse implication, we observe that, because $X$ and $Y$
are both in $\Fm$, and $X$ is weakly to the left of $Y$, the only summand of $\Ext^i_{\Cm}(Y,X)$ which could be
nonzero is $\Ext^i_{D^b(H)}(Y,X)$, but this is zero by assumption.  The vanishing
of $\Ext^i_\Cm(X,Y)$ now follows from the $(m+1)$-Calabi-Yau property of $\Cm$.
\end{proof}

\begin{proof}
Since $\tau$ commutes with $F^{(m)}$, it descends to an autoequivalence of
the $m$-cluster category.
Acting on $\Cm$, it is clear that $\tau$ preserves the property of
being an $m$-cluster tilting object.  Thus, it defines a cyclic action on the
set of silting objects in $\Fm$, and thus on $\NablaAssocm(W,c)$. Since
$\tau$ corresponds to the combinatorial AR translation, it gives rise to
the Cambrian rotation $\Camb_c$.
\end{proof}

\subsection{Positive \mhead-eralized noncrossing partitions}
Let  $m\geq 0$. It turns out that the suitable orbit category in which to understand the
$m$-eralized noncrossing partitions is $\Cmp={\mathcal C}^{(-m-1)}$, that is to say, we use
the same construction as above, but with a negative value for the
parameter.  Write $\fmp$ for $F^{(-m-1)}$. This approach was introduced by
R. Coelho Simoes for $m=1$ \cite{CS}.

%Similarly, define an autoequivalence of $D^b(H)$ by $\fmp=[m+1]\tau$.
%Let $\Cmp$ be the orbit category with respect to this autoequivalence.
Let $\Fmp$ be the fundamental domain for the action of $\fmp$ on $D^b(H)$
which consists of all the indecomposable objects in shifts 0 to~$m$ of
$D^b(H)$, except for the indecomposable projectives in shift zero.

An object $X$ in $\Cmp$ is called an $\Hom_{\leq 0}$ configuration if
$\Ext^i(X,X)=0$ for $-1\geq i \geq -m$, while $\Hom(X,X)$ consists of linear
combinations of the identity maps on indecomposable summands of $X$, and
$X$ is maximal with respect to this property.

\begin{lemma+} Let $X$ be an object in $\Fmp$.  $X$ is $\Hom_{\leq 0}$
configuration in $D^b(H)$ if
and only if the image of $X$ in $\Cmp$ is a $\Hom_{\leq 0}$ configuration
in $\Cmp$.
\end{lemma+}

We omit the proof, which is identical to the proof of \Cref{lem:silt-m-clust}.

\begin{proposition}
  $[-1]$ induces the positive Kreweras complement action $\Krew$ on $\DeltaNCmpos(W,c)$.
\end{proposition}
\begin{proof}
Similarly to the $m$-cluster case, since $[-1]$ commutes with $\fmp$, it descends to an autoequivalence
of $\Cmp$, where it induces a cyclic action on $\Hom_{\leq 0}$ configurations,
and thus also induces a cyclic action of $\Hom_{\leq 0}$ configurations in
$\Fmp$, and thus on $\DeltaNCmpos(W,c)$.   Since
$[-1]$ corresponds to the combinatorial shift, it gives rise to
the positive Kreweras complement action $\Krew$.
\end{proof}

The Auslander-Reiten translation $\tau$ also gives a symmetry of $\DeltaNCmpos(W,c)$, but the group of symmetries induced by $[1]$ includes the group of
symmetries induced by $\tau$, and sometimes the inclusion is strict.

%%%%%%%%%%%%%%%%%%%%%%%%%%%%%%%%%%%%%%%%%%%%%%%%%%%%%%%%%%%%%%%%%%%%%%%%%%%%%%%%%%%%%
\section[\mhead-eralized sortable elements]{\mhead-eralized sortable elements and co-aisles in the derived category}
%%%%%%%%%%%%%%%%%%%%%%%%%%%%%%%%%%%%%%%%%%%%%%%%%%%%%%%%%%%%%%%%%%%%%%%%%%%%%%%%%%%%%
\label{sec:aisles}

When we speak of a subcategory of $D^b(H)$, we always mean a full
subcategory closed under direct sums and direct summands.
Thus, one way to specify a subcategory of $D^b(H)$ is to specify its
indecomposable objects.

\begin{definition}\label{def:aisle}
A subcategory~$\V$ of $D^b(H)$ is called a \defn{co-aisle}
if the following hold~\cite[Proposition~1.3]{KV}.
\begin{enumerate}
\item $X\in \V$ implies $X[-1]\in \V$. \label{cond:1}
\item \label{cond:2} If we have a triangle $X\rightarrow Y \rightarrow Z \rightarrow$ in
$D^b(H)$ such
that~$X$ and~$Z$ are in~$\V$, then $Y$ is also in~$\V$. 
\item For each object~$Z$ of $D^b(H)$, there is some $X\in \V$ such that
the map $\Hom_\V(X,\cdot) \rightarrow \Hom(Z,\cdot)|_\V$ is an epimorphism.
\label{cond:3}
\end{enumerate}
(Note that~\cite{KV} works with
aisles; we work with the dual notion of co-aisles for convenience.)
Following the terminology of~\cite{KV}, a co-aisle
$\V$ is called \defn{separated} if there is some~$m$ such that~$\V$ is
contained in the additive category generated by $\mod H[i]$ for $i\leq m$.
\end{definition}

Note that if~$\V$ is a subcategory of $D^b(H)$ satisfying this condition,
then, to show that~$\V$ is a separated co-aisle, it suffices to check
conditions~\eqref{cond:1} and~\eqref{cond:2} above.  Condition~\eqref{cond:3} will automatically be
satisfied because for any $Z \in D^b(H)$, there will only be a finite
number of indecomposable modules of~$\V$ admitting a non-zero morphism from~$Z$, so we could take~$X$ to be their direct sum.

The standard
co-aisle~$\sta(H)$ in $D^b(H)$ is additively generated by $M[i]$ for
$M \in \mod H$ and $i<0$.

\begin{definition}\label{def:good-co}
Write $\storf(H)$ for the separated co-aisles in $D^b(H)$ which
contain~$\sta(H)$.  Write $\storfm(H)$ for the separated co-aisles in
$\storf(H)$ which are contained in $\sta(H)[m]$.
\end{definition}

For $\V\in \storf_H$, denote the collection of indecomposable objects in
$\V$ in degrees $\geq 0$ by $\Vs$.  Write $\storfp(H)$ for the collections
which arise as $\Vs$ for some $\V\in \storf(H)$.  For $\TT\in \storfp(H)$, write
$\overline \TT$ for the corresponding separated co-aisle.

For~$\V$ a separated co-aisle containing~$\sta(H)$, define
\[
  \dimp(\V)= \big\{ \udim(X) \mid X\in \Vs \big\}.
\]

We show that the collections which arise as $\dimp(\V)$ for $\V$ a separated
co-aisle containing~$\sta(H)$
are exactly the inversion sets
of~$c$-sortable elements of the Artin monoid, by establishing
the appropriate Cambrian recursion for the sets $\dimp(\V)$.

Let $c=s_1\dots s_n$.  Define $\mu_1(H)$ to be the algebra obtained by applying a reflection functor at vertex 1. % (See \cite{dlabringel76} for reflections functors over general fields.)
As usual, we identify $\ind \mod \mu_1(H)$ with $\ind \mod H$, with~$S_1$ removed, and
$S_1[1]$ added.
We will also need to
consider $H'=H/He_1H$, a hereditary Artin algebra of rank $n-1$.

\begin{proposition}
\label{prop:tf_cambrian_recurrence}
  Let~$s_1$ be initial in~$c_H$.
  Then
  \[
    \TT \in \storfp(H) \Leftrightarrow
    \begin{cases}
      \TT \in \storfp(H') & \text{if } S_1 \not \in \TT\\
      \TT  \in \storfp(\mu_1(H))& \text{if } S_1 \in \TT
    \end{cases}\ .
  \]
\end{proposition}

\begin{proof}
Suppose that $\TT \in \storfp(H)$.  Consider first the case that $S_1 \not\in \TT$.
Let~$M$ be an~$H$-module such that $\Hom(S_1,M)\ne 0$.
We claim that $M[i] \not \in \V$ for any
$i\geq 0$.  Because~$\overline\TT$ is closed under $[-1]$, it suffices to show that
$M \not \in \TT$.
Suppose otherwise.
We have a short exact
sequence of modules $0 \rightarrow S_1 \rightarrow M \rightarrow M'
\rightarrow 0$, which gives rise to a triangle:
$$ M'[-1] \rightarrow S_1 \rightarrow M \rightarrow M'[0]$$
Now $M'[-1]\in \overline \TT$
because $\overline \TT$ contains~$\sta(H)$.  Then \eqref{cond:2} 
implies that $S_1$ is in~$\TT$, contrary to our assumption.

Therefore, all the indecomposable objects of~$\TT$ are of
the form $M[i]$ for~$M$ an $H$-module
satisfying $\Hom(S_s,M)=0$ (and $i\geq 0$);
equivalently,
they are of the form $M[i]$ for~$M$ an $H'$-module (and $i\geq 0$).
Thus, $\TT \in \storfp(H')$.

Next, suppose that $S_1 \in \TT$. In this case,~$\overline\TT$ is a
co-aisle containing $\sta({\mu_1(H)})$, so $\TT \in \storf(\mu_1(H))$.

Conversely, suppose that we are given~$\TT$ in
$\storfp(H')$.  Define~$X$ to be the additive
subcategory of $D^b(H)$ generated by~$\TT$ and~$\sta(H)$.  \eqref{cond:1} obviously holds. Condition \eqref{cond:2}
follows from the 
long exact sequence for homology. Together these imply that $X$ is a co-aisle in $\storf(H)$, so $\TT\in \storfp(H)$.

Finally, suppose we are given~$\TT$ in
$\storfp(\mu_1(H))$.  It is clear that $\TT\cup \{S_1\}$ lies in
$\storfp(H)$.
\end{proof}

%$H^i
%It is clear from the cohomology long exact sequence that

%From the cohomology long exact sequence
%it is clear that the subcategory
%$V=\{X \in D^b(H) \mid H^i(X) \in H/He_1H \text{-mod for } i\geq 0$ is
%a torsion-free class,

%We claim that no element
%of $T^+$ has $S_1$ in its support.  Suppose there were some $M[i]\in T$ with
%$i\geq 0$ and $S_1$ in the support of $M$.  Then

%$T^+$ is supported on $H/He_1H$, where $e_1$ is the idempotent corresponding

%$T^+$ for the set of indecomposable objects of~$T$ not contained in
%the standard aisle.

\begin{theorem}\label{aisle-sort}There is a natural bijection
  \begin{eqnarray*}
    \storf(H) & \cambbij & \Sortinf(W,c)\\
    \V &\longmapsto& \dimp(\V) \end{eqnarray*}
  This bijection restricts to a bijection from
  $\storfm(H)$ to inversion sets
  of $m$-$c$-sortable elements in the Artin monoid.
\end{theorem}\hughside{It was clear from the beginning of this project that this bijection would give us sets of colored roots that should be considered the $m$-eralization of the inversion sets of $c$-sortable elements, but at that point it wasn't clear there was anything interesting to be said about them.}

\begin{proof} Both satisfy the same Cambrian recursion. The statement about restriction is straightforward, since a $c$-sortable elment
is $m$-$c$-sortable if and only if all its inversions have color at most
$m$.
\end{proof}

An analogous result in the abelian case
was shown for $kQ$-mod in~\cite{IT2009}.
Another description of aisles in the hereditary
setting was provided by
\cite{SvR}.

%\christian{I am a little confused about the $m$-eralized setting. Are you not discussing that, or am I missing the point there (which is very possible)?}
%\hugh{I was just doing the $m=\infty$ setting, which is easiest, but I
%guess I should also do finite $m$.}
%\christian{I think it would also be good to say a few words about the positive $m$-eralized setting.}

%%%%%%%%%%%%%%%%%%%%%%%%%%%%%%%%%%%%%%%%%%%%%%
\section{\mhead-eralized clusters and sortable elements}
\label{sec:bij}

\newcommand{\extinj}{\operatorname{ExtInj}}
There is a bijection between silting objects and aisles which goes back
to \cite[Section 5]{KV}.  We follow \cite{KV} except that we use co-aisles instead of aisles.  An object $Y$ in a co-aisle $\V$ is called $\Ext$-injective in $\V$ if $\Ext^i(M,Y)=0$ for all $M \in \V$ and all $i>0$.  We write $\extinj(\V)$ for the direct sum of the $\Ext$-injective indecomposables of $\V$.
\christianside{Does that mean, we are reinventing ideas from the 1980s?}

\begin{theorem}[\cite{KV}]\label{th:silt-co}
  There is a bijection from silting objects to co-aisles containing the standard
  co-aisle given as follows
  \begin{eqnarray*} \Siltp(H) &\longleftrightarrow& \storf(H) \\
    X &\longmapsto& \{ M\in D^b(H)  \mid \Ext^i(M,X)=0 \textrm{ for } i\geq 0\}\\
    \extinj(\V) & \longmapsfrom &    \V.
  \end{eqnarray*}
  This bijection restricts to a bijection between $\Siltm(H)$ and $\storfm(H)$.
\end{theorem}

We can now state the following theorem.

\begin{theorem}
  The following diagram commutes

  $$\begin{tikzpicture}[xscale=6,yscale=2]
    \node (a) at (0,0) {$\Siltp(H)$};
    \node (b) at (1,0) {$\storf(H)$};
    \node (c) at (0,-1) {$\NablaAssocinf(W,c)$};
    \node (d) at (1,-1) {$\Sortinf(W,c)$};
\draw [stealth-stealth] (a) -- (b) node[midway,above] {\Cref{th:silt-co}};
\draw [stealth-stealth] (a) -- (c) node[midway,left]{\Cref{prop:silt-to-assoc}};
\draw [stealth-stealth] (b) -- (d) node [midway,right]{\Cref{aisle-sort}};
\draw [stealth-stealth] (c) -- (d) node[midway, below]{\Cref{thm:sortcluster}.};
  \end{tikzpicture}$$
  The bijections all restrict to the appropriate $m$-eralized versions.
\end{theorem}

\begin{proof}
  All the bijections in the diagram respect the Cambrian recursion.
  \end{proof}
\hughside{And on that note, we leave you.  Thanks for reading!}

\bibliographystyle{amsalpha}
\bibliography{../StumpThomasWilliams}
\cleardoublepage

\end{document}